\documentclass[12pt, letter]{article}

\usepackage[dvipsnames]{xcolor}
\usepackage[pdftex]{graphicx}
\usepackage{amsmath}
\usepackage{amsthm}
\usepackage{amsfonts}
\usepackage{amssymb}
\usepackage[margin=1.5in]{geometry}
\usepackage{tikz-cd}
\usetikzlibrary{cd}
\usepackage{enumitem}

\usepackage{caption}

\usepackage{mathtools}
\usepackage[]{microtype}

\usepackage{nicefrac}

\usepackage{stmaryrd}

\usepackage{tabularx}
\usepackage{tikz}
\usepackage{verbatim}

\usepackage{hyperref}
\makeatletter
\@ifpackageloaded{hyperref}{

\newcommand\myshade{85}
\colorlet{mylinkcolor}{violet}
\colorlet{mycitecolor}{Green}
\colorlet{myurlcolor}{Aquamarine}

\hypersetup{
citecolor = mylinkcolor!\myshade!black,
linkcolor  = mycitecolor!\myshade!black,
urlcolor   = myurlcolor!\myshade!black,
linkbordercolor = mylinkcolor!\myshade!black,
colorlinks = true,
}
}
{
\newcommand{\phantomsection}{} 
}

\def\namedlabel#1#2{\begingroup
  #2%
  \def\@currentlabel{#2}%
  \phantomsection\label{#1}\endgroup
}
\makeatother

\usepackage{amsrefs}
\usepackage{cleveref}

\numberwithin{equation}{section}
\newtheorem{theorem}{Theorem}[section]
\newtheorem{corollary}[theorem]{Corollary}
\newtheorem{lemma}[theorem]{Lemma}
\newtheorem{proposition}[theorem]{Proposition}
\newtheorem{prop}[theorem]{Proposition}

\newtheorem*{theorem*}{Theorem}

\theoremstyle{definition}
\newtheorem{definition}[theorem]{Definition}
\newtheorem{defin}[theorem]{Definition}
\newtheorem{fact}[theorem]{Fact}

\newtheorem{question}[theorem]{Question}

\theoremstyle{remark}
\newtheorem*{remark}{Remark}
\newtheorem{claim}{Claim}[theorem]
\newtheorem*{claim*}{Claim}

\newcommand{\Gianluca}[1]{{\textcolor{ForestGreen}{[\textbf{Gianluca:} #1]}}}
\newcommand{\Andy}[1]{{\textcolor{blue}{[\textbf{Andy:} #1]}}}

\DeclareMathOperator{\id}{id}

\DeclarePairedDelimiter{\set}{\{}{\}}

\renewcommand{\b}[1]{\mathbf{#1}}

\renewcommand{\cal}[1]{\mathcal{#1}}
\renewcommand{\rm}[1]{\mathrm{#1}}
\renewcommand{\sf}[1]{\mathsf{#1}}
\renewcommand{\frak}[1]{\mathfrak{#1}}

\newcommand{\bA}{\mathbf{A}}

\newcommand{\rmA}{\mathrm{A}}
\newcommand{\cA}{\mathcal{A}}

\newcommand{\bB}{\mathbf{B}}

\newcommand{\rmB}{\mathrm{B}}
\newcommand{\cB}{\mathcal{B}}

\newcommand{\bC}{\mathbf{C}}
\newcommand{\bbC}{\mathbb{C}}
\newcommand{\rmC}{\mathrm{C}}

\newcommand{\frakC}{\mathfrak{C}}

\newcommand{\bD}{\mathbf{D}}

\newcommand{\bE}{\mathbf{E}}

\newcommand{\rmE}{\mathrm{E}}

\newcommand{\rmF}{\mathrm{F}}
\newcommand{\cF}{\mathcal{F}}

\newcommand{\bbG}{\mathbb{G}}

\newcommand{\bbH}{\mathbb{H}}
\newcommand{\rmH}{\mathrm{H}}

\newcommand{\bK}{\mathbf{K}}

\newcommand{\rmK}{\mathrm{K}}
\newcommand{\cK}{\mathcal{K}}

\newcommand{\cL}{\mathcal{L}}

\newcommand{\bM}{\mathbf{M}}

\newcommand{\rmM}{\mathrm{M}}
\newcommand{\cM}{\mathcal{M}}

\newcommand{\bN}{\mathbf{N}}
\newcommand{\bbN}{\mathbb{N}}
\newcommand{\rmN}{\mathrm{N}}
\newcommand{\cN}{\mathcal{N}}

\newcommand{\rmP}{\mathrm{P}}
\newcommand{\cP}{\mathcal{P}}

\newcommand{\bbQ}{\mathbb{Q}}

\newcommand{\bbR}{\mathbb{R}}

\newcommand{\rmS}{\mathrm{S}}
\newcommand{\cS}{\mathcal{S}}

\newcommand{\bbT}{\mathbb{T}}

\newcommand{\cU}{\mathcal{U}}

\newcommand{\cV}{\mathcal{V}}
\newcommand{\sfV}{\mathsf{V}}

\newcommand{\sfW}{\mathsf{W}}

\newcommand{\bbZ}{\mathbb{Z}}

\newcommand{\wt}{\widetilde}
\newcommand{\wh}{\widehat}
\newcommand{\ol}{\overline}

\newcommand{\sa}{\mathrm{Sa}}
\newcommand{\la}{\langle}
\newcommand{\ra}{\rangle}
\newcommand{\flim}{\mathrm{Flim}}
\newcommand{\age}{\mathrm{Age}}
\newcommand{\fin}[1]{[#1]^{{<}\omega}}
\newcommand{\fr}{Fra\"iss\'e }
\renewcommand{\phi}{\varphi}
\newcommand{\emb}{\mathrm{Emb}}
\newcommand{\aut}{\mathrm{Aut}}
\newcommand{\homeo}{\mathrm{Homeo}}
\newcommand{\dom}{\mathrm{dom}}

\newcommand{\im}{\mathrm{Im}}

\newcommand{\Int}{\mathrm{Int}}
\newcommand{\st}{\mathrm{St}}
\newcommand{\gl}{\mathrm{Gl}}
\newcommand{\NU}{\mathrm{NU}}
\newcommand{\sn}{\mathrm{SN}}
\newcommand{\op}{\mathrm{op}}

\newcommand{\rc}{\mathrm{rc}}

\newcommand{\mg}{\mathrm{M}(G)}

\newcommand{\sub}{\subseteq}

\newcommand{\address}[1]{\gdef\@address{#1}}
\newcommand{\email}[1]{\gdef\@email{\url{#1}}}

\title{Topological groups with tractable minimal dynamics}
\date{}
\author{Gianluca Basso and Andy Zucker}

\begin{document}
\maketitle

\begin{abstract}

    A Polish group $G$ has the \emph{generic point property} if any minimal $G$-flow admits a comeager orbit, or equivalently if the universal minimal flow (UMF) does. The class $\sf{GPP}$ of such Polish groups is a proper extension of the class $\sf{PCMD}$ of Polish groups with metrizable UMF. Motivated by analogous results for $\sf{PCMD}$, we define and explore a robust generalization of $\sf{GPP}$ which makes sense for all topological groups, thus defining the class $\sf{TMD}$ of topological groups with \emph{tractable minimal dynamics}. These characterizations yield novel results even for $\sf{GPP}$; for instance, a Polish group is in $\sf{GPP}$ iff its UMF has no points of first countability.

    Motivated by work of Kechris, Pestov, and Todor\v{c}evi\'c that connects topological dynamics and structural Ramsey theory, we state and prove an \emph{abstract KPT correspondence} which characterizes the class $\sf{TMD}$ and shows that $\sf{TMD}$ is $\Delta_1$ in the L\'evy hierarchy. We then develop set-theoretic methods which allow us to apply forcing and absoluteness arguments to generalize numerous results about $\sf{GPP}$ to all of $\sf{TMD}$. We also apply these new set-theoretic methods to first generalize parts of Glasner's structure theorem for minimal, metrizable \emph{tame} flows to the non-metrizable setting, and then to prove the \emph{revised Newelski conjecture} regarding definable NIP groups. We conclude by discussing some tantalizing connections between definable NIP groups and $\sf{TMD}$ groups.

    \let\thefootnote\relax
	\let\thefootnote\relax
    	\let\thefootnote\relax\footnote{The first author was partially supported by Istituto Nazionale di Alta Matematica “Francesco Severi”, and wishes to thank the Faculty of Mathematics at the University of Waterloo for its hospitality during a month-long research visit. The second author was supported by NSERC grants RGPIN-2023-03269 and DGECR-2023-00412.}
\end{abstract}

\section{Introduction}

This paper is a contribution to the study of abstract topological dynamics, in particular to the study of the \emph{minimal dynamics} of topological groups. 
We refer to Section~\ref{Section:Background}, in particular Subsection~\ref{Subsection:Top_Dyn}, for more precise statements of definitions and conventions. 

Given a topological group $G$, a $G$-flow is a continuous action of $G$ on a compact space $X$; the $G$-flow $X$ is \emph{minimal} if every orbit is dense, equivalently if $X$ contains no proper subflows. A $G$-map between $G$-flows $X$ and $Y$ is a continuous, $G$-equivariant map. By a fundamental result of Ellis \cite{Ellis}, there is for each topological group $G$ a \emph{universal minimal flow} $\mg$, a minimal $G$-flow which admits a $G$-map onto every other minimal $G$-flow, which is furthermore unique up to isomorphism.

The study of $\mg$ can tell us meaningful information about the class of \emph{all} minimal $G$-flows. This is perhaps most clear for the class $\sf{EA}$ of \emph{extremely amenable} topological groups, those $G$ for which $\mg$ is a singleton. In this case, there are no non-trivial minimal $G$-flows, and every $G$-flow contains a fixed point. One can say that the extremely amenable groups are exactly those with ``\emph{trivial minimal dynamics}.'' By a result of Granirer and Lau \cite{GranLau}*{Lemma 4}, non-trivial locally compact groups never belong to $\sf{EA}$. However, among groups which are not locally compact, examples abound, including the unitary group of an infinite dimensional Hilbert space \cite{GroMil} and the group of orientation-preserving homeomorphisms of the unit interval \cite{Pestov_1998}. For closed subgroups of $S_\infty$, the group of all permutations of $\bbN$ equipped with the pointwise convergence topology, the seminal work of Kechris, Pestov, and Todor\v{c}evi\'c \cite{KPT} established a deep connection between extreme amenability and the \emph{Ramsey property} (Subsection~\ref{Subsection:KPT}), a combinatorial statement that some classes of finite structures possess. Any closed subgroup $G\subseteq S_\infty$ has the form $\aut(\bK)$ for $\bK$ a \emph{\fr structure}, and given a \fr structure $\bK$, the class $\age(\bK)$ of all finite structures which embed into $\bK$ is a \emph{\fr class} (Subsection~\ref{Subsection:AutGroups}). One of the major theorems of \cite{KPT} is that in this situation, $G\in \sf{EA}$ iff $\age(\bK)$ has the Ramsey property. For instance, the ordinary finite Ramsey theorem \cite{Ramsey} is equivalent to the statement that the class of finite linear orders has the Ramsey property, hence the theorem recovers the earlier result of Pestov \cite{Pestov_1998} that $\aut(\bbQ, <)\in \sf{EA}$.

Before mentioning other aspects of \cite{KPT}, we point out two generalizations of the theorem discussed above. In one direction, Barto\v{s}ov\'a \cites{Bartosova_Thesis, Bartosova_2013, Bartosova_More} has provided a Ramsey-theoretic characterization of extreme amenability for closed subgroups of $S_\kappa$, the group of permutations of an infinite cardinal $\kappa$. In particular, \fr classes with the Ramsey property yield examples of automorphism groups of uncountable structures in $\sf{EA}$.  In another direction, Melleray and Tsankov \cite{MT} define a notion of \emph{approximate Ramsey property} for \emph{metric \fr classes} and show that this characterizes extreme amenability of the automorphism group of a \emph{metric \fr structure}. Any Polish group is isomorphic to the automorphism group of some metric \fr structure.
   
Moving beyond extremely amenable groups, there are Polish groups $G$ with the property that $\mg$ (hence every minimal flow) is metrizable. Let us say that a Polish group has \emph{concrete minimal dynamics} in this case and write $\sf{PCMD}$ for this class. Pestov \cite{Pestov_1998} shows for instance that the universal minimal flow of the group of orientation-preserving homeomorphisms of the unit circle equipped with the compact-open topology is its natural action on the circle, while Glasner and Weiss \cite{GW_UMF_S_infty} show that the universal minimal flow of $S_\infty$ is its natural action on the space $\rm{LO}(\bbN)$ of linear orders on $\bbN$.

Kechris, Pestov, and Todor\v{c}evi\'c in \cite{KPT} (see also \cite{Lionel_2013}) introduce a framework, often referred to as \emph{KPT correspondence}, which successfully computes the universal minimal flows of automorphism groups of certain \fr structures. Given a \fr class $\cK$ of $\cL$-structures with \emph{\fr limit} $\bK$, there is often a canonical way of forming an \emph{expansion class} $\cK^*$ in a larger language $\cL^*$ whose $\cL$-reduct is $\cK$ and such that $(\cK^*, \cK)$ is a so called \emph{excellent pair}. Among the requirements of being an excellent pair is that $\cK^*$ is a \fr class satisfying the Ramsey property. Combining a major theorem of \cite{KPT} with its converse proven in \cite{Zucker_Metr_UMF}, writing $G:=\aut(\bK)$, we have $G \in \sf{PCMD}$ iff there is $\cK^*$ making $(\cK^*, \cK)$ an excellent pair, and $\mg$ can be computed explicitly from $\cK^*$. For general Polish groups, by combining the main theorems of \cite{BYMT} and \cite{MNVTT}, we have that $G\in \sf{PCMD}$ iff there is a closed, co-precompact, presyndetic, extremely amenable subgroup $H\leq G$ (Subsection~\ref{Subsection:Topological_groups}). In this case, we have $\mg\cong \wh{G/H}$, where this denotes the completion of the left coset space $G/H$ under the right uniformity. Thus, \textbf{one can think of $H\leq G$ (or similarly $\cK^*)$ as not only a certificate of the metrizability of $\mg$, but also as an explicit set of directions for constructing $\mg$. 
}

A consequence of the above discussion (and the main theorem from \cite{BYMT}) is that if $G\in \sf{PCMD}$, then $\mg$ has a comeager orbit. Indeed, if $\mg = \wh{G/H}$, the orbit $G/H$ is comeager. Let us define the class $\sf{GPP}$ of Polish groups with the \emph{generic point property}, i.e.\ such that $\mg$ has a comeager orbit. This is equivalent by \cite{AKL}*{Proposition~14.1} to every minimal flow having a comeager orbit. It is natural to ask whether the inclusion $\sf{PCMD}\subseteq \sf{GPP}$ is strict. In \cite{ZucThesis}, it is shown that upon dropping one assumption from the definition of excellent pair, namely that of \emph{precompactness}, then if one could find $(\cK^*, \cK)$ a weakly excellent, non-excellent pair in this sense and letting $\bK = \flim(\cK)$ and $G = \aut(\bK)$, we would have $G\in \sf{GPP}\setminus \sf{PCMD}$. Kwiatkowska \cite{Kw_Dendrites} exhibits several examples of such weakly excellent, non-excellent pairs by considering $\cK$ a \fr class whose limit encodes a \emph{generalized Wa\.{z}ewski dendrite}. Basso and Tsankov \cite{MR4549426} give many more examples of Polish groups $G\in \sf{GPP}\setminus \sf{PCMD}$, which are all applications of the \emph{kaleidoscopic group} construction of Duchesne, Monod, and Wesolek \cite{DMW}. Mirroring the structure theorem of \cite{MNVTT}, Zucker \cite{ZucMHP} shows that for a Polish group $G$, we have $G\in \sf{GPP}$ iff there is an extremely amenable subgroup $H\leq G$ such that $\mg\cong \sa(G/H)$, where this is the \emph{Samuel compactification} of $G/H$ equipped with its right uniformity. When $G = \aut(\bK)$ for $\bK = \flim(\cK)$ a \fr structure, this amounts to saying that there is $\cK^*$ making $(\cK^*, \cK)$ a weakly excellent pair. Like before, \textbf{we can think of this $H$ (or similarly $\cK^*$) as a certificate of the fact that $G\in \sf{GPP}$, but where the instructions for forming $\mg$ from this certificate are less explicit} since once needs to form a Samuel compactification.

There are many other characterizations of the class $\sf{PCMD}$ that are useful for proving various closure properties of it. For instance, Jahel and Zucker \cite{JahelZuckerExt} show that $\sf{PCMD}$ is closed under group extensions by considering what is a priori a completely unrelated property of topological groups. Given a $G$-flow $X$, write $\rm{AP}_G(X)\subseteq X$ for the set of \emph{almost periodic} points, i.e.\ those $x\in X$ with $\ol{G\cdot x}$ minimal, and write $\rm{Min}_G(X)$ for the space of minimal subflows of $X$ equipped with the Vietoris topology. Define the class $\sf{CAP}$ (for \emph{closed almost periodic}) of topological groups such that $\rm{AP}_G(X)$ is closed for any $G$-flow $X$, and $\sf{SCAP}$ for the class of ``strongly $\sf{CAP}$'' groups, i.e.\ those where $\rm{Min}_G(X)$ is compact for any $G$-flow $X$. It turns out that for $G$ Polish, we have $G\in \sf{PCMD}$ iff $G\in \sf{CAP}$ iff $G\in \sf{SCAP}$. The direction $\sf{PCMD}$ implies $\sf{SCAP}$ essentially appears in \cite{JahelZuckerExt}, and the direction $\neg\sf{PCMD}$ implies $\neg \sf{CAP}$ is due to Barto\v{s}ov\'a and Zucker and appears in \cite{ZucThesis}.

Revisiting Barto\v{s}ov\'a's work on uncountable structures \cites{Bartosova_Thesis, Bartosova_2013, Bartosova_More},  there it is shown under some mild extra hypotheses that excellent pairs of \fr classes also yield information about $\mg$ for $G$ the automorphism group of an uncountable structure. For instance, if $\kappa > \aleph_0$ is an uncountable cardinal and  $S_\kappa$ denotes the group of all permutations of $\kappa$ equipped with the pointwise convergence topology, then $\rmM(S_\kappa) \cong \rm{LO}(\kappa)$, in direct analogy with the result on $S_\infty$ from \cite{GW_UMF_S_infty}. While the space $\rm{LO}(\kappa)$ is no longer metrizable for uncountable $\kappa$, this is still a nice ``concrete'' description of the universal minimal flow of $S_\kappa$. Motivated by these examples and by the fact that the classes $\sf{PCMD}$, $\sf{CAP}$, and $\sf{SCAP}$ coincide for Polish groups, the current authors in \cite{BassoZucker} work with $\sf{CAP}$ and $\sf{SCAP}$ as potential \emph{definitions} of what ``concrete minimal dynamics'' should mean for general topological groups. In particular, it is shown that $\sf{CAP} = \sf{SCAP}$ and that all of the examples of automorphism groups of uncountable structures whose universal minimal flows are computed in \cites{Bartosova_Thesis, Bartosova_2013, Bartosova_More} belong to this class. In the present paper, we introduce a more neutral name and call this the class of topological groups with \emph{concrete minimal dynamics}, denoted $\sf{CMD}$ (Definition~\ref{Def:Concrete_Min_Dyn}), making $\sf{PCMD}$ exactly the intersection of Polish and $\sf{CMD}$. 

The present paper seeks to address the following questions:
\begin{itemize}
    \item 
    Is there an alternate characterization of the class $\sf{GPP}$ along the lines of the definitions of the classes $\sf{CAP}$ and/or $\sf{SCAP}$? 
    
    \item 
    Can this characterization be used to find the correct generalization of the class $\sf{GPP}$ to all topological groups? Let us call this the class of topological groups with ``tractable minimal dynamics'' and denote it by $\sf{TMD}$.
    \item 
    Can this characterization be used to prove closure properties of the class $\sf{TMD}$ similar to those enjoyed by $\sf{CMD}$ as shown in \cite{BassoZucker}? 
    \item 
    For both of the classes $\sf{CMD}$ and $\sf{TMD}$, are there ``certificates'' (in the spirit of KPT correspondence) that describe membership in these classes that make sense for arbitrary topological groups?
\end{itemize}
We provide affirmative answers to all of these lines of inquiry. We define the class of $\sf{TMD}$ groups to be those $G$ such that $\mg$ satisfies the conclusions of Theorem~\ref{Thm:Rosendal_Minimal}. Item~\ref{Item:Ultra_Thm:RM} of the theorem in particular is equivalent to the following: if $X$ is a $G$-flow and $(Y_i)_{i\in I}$ is a net of minimal subflows with $Y_i\to Y$ in the Vietoris topology, then $Y$ contains a unique minimal subflow. Compare this to the definition of $\sf{SCAP}$, which demands that the Vietoris limit $Y$ be minimal. Theorem~\ref{Thm:Polish_Group_WUEB} asserts that $\sf{GPP}$ is exactly the intersection of Polish and $\sf{TMD}$. The way in which $\sf{TMD}$ generalizes $\sf{GPP}$ is best understood via the \emph{Rosendal criterion} (Definition~\ref{Def:Rosendal_Criterion}), which characterizes when a Polish group $G$ acting on a Polish space $X$ admits a comeager orbit (for a more general statement, see Theorem~\ref{thm:Rosendal-general}). For general topological groups and $G$-spaces, Rosendal's criterion still makes sense, and we show that $\sf{TMD}$ is exactly the class of topological groups such that $\mg$ satisfies the Rosendal criterion. 

Using the various equivalent definitions of the class $\sf{TMD}$, we prove a variety of properties of this class, many of which are new even for the class $\sf{GPP}$. We show that locally-compact non-compact groups are never $\sf{TMD}$, that the class $\sf{TMD}$ is closed under group extensions (Theorem~\ref{Thm:WCAP_Group_Extensions}) and surjective inverse limits (Proposition~\ref{Prop:Inverse_Limits}). In the case of a product $G = H\times K$ with at least one of $H$ or $K$ in $\sf{TMD}$, we show that $\mg$ is ``almost'' $\rmM(H)\times \rmM(K)$, in the sense that $\mg$ is a highly proximal extension of $\rmM(H)\times \rmM(K)$ (Theorem~\ref{Thm:Products}). We show that the class $\sf{GPP}$ is characterized by those Polish groups $G$ such that $\mg$ has a point of first countability (Theorem~\ref{Thm:Polish_Group_WUEB}). In particular, if $G$ is locally compact, non-compact Polish, then $\mg$ does not contain any point of first countability.

Before proceeding, let us mention three novel aspects of this work that particularly inform the content of the other sections:
\begin{itemize}
    \item 
    Ultracoproducts of $G$-flows (Subsection~\ref{Subsection:Ultracoproducts}).
    \item 
    Fattening spaces (Section~\ref{Section:NUlts_Fattenings}).
    \item 
    Forcing and absoluteness in topological dynamics (Section~\ref{Section:Abs}).
\end{itemize}

\noindent
\textbf{Ultracoproducts of $G$-flows:} While ultracoproducts were introduced in the preprint \cite{ZucUlts}, we give a completely self-contained introduction here to the portion of the theory that we will need (\cite{ZucUlts} will be revised and expanded using concepts developed in this paper). An ultracoproduct of a tuple $(X_i)_{i\in I}$ of $G$-flows along an ultrafilter $\cU\in \beta I$ can be thought of as a universal instance of Vietoris convergence in $\rm{Sub}_G(Z)$ for some ambient $G$-flow $Z$ containing all the $X_i$; see Subsection~\ref{Subsection:Ultracoproducts} for the instance of the precise definition we will need. Theorems~\ref{Thm:Concrete}\ref{Item:Ult_Thm:Concrete} and \ref{Thm:Rosendal_Minimal}\ref{Item:Ultra_Thm:RM} give characterizations of both $\sf{CMD}$ and $\sf{TMD}$ using ultracoproducts; $G\in \sf{CMD}$ iff every ultracopower of $\mg$ is minimal (equivalently canonically isomorphic to $\mg$), while $G\in \sf{TMD}$ if every ultracopower of $\mg$ is a highly proximal extension of $\mg$, so in particular contains a unique minimal subflow. 

In Section~\ref{Section:AutMG}, we use ultracoproducts to investigate the automorphism group of $\rmM(G)$, which we denote by $\bbG$. One can equip $\bbG$ with the so-called \emph{tau-topology}, which was introduced by Furstenberg \cite{Furstenberg_Distal} and further investigated in \cites{Ellis_book, EGS_1975, Glasner_Prox, Auslander, KP_2017, Rze_2018}. This is a compact T1 topology with separately continuous multiplication and continuous inversion. Given a tuple $(p_i)_{i\in I}$ of automorphisms of $\rmM(G)$ and an ultrafilter $\cU\in \beta I$, we naturally obtain an automorphism of the ultracopower of $\rmM(G)$ along $\cU$; from this automorphism, and by considering the various minimal subflows of the ultracopower, we obtain a set of automorphisms of $\mg$, which we prove are closely related to tau-ultralimits. Our new characterization of the tau-topology, together with Theorem~\ref{Thm:Rosendal_Minimal}\ref{Item:Ultra_Thm:RM}, allows us to conclude that if $G \in \sf{TMD}$, the tau-topology on $\mg$ is Hausdorff, so a compact Hausdorff topological group. More generally, given any minimal flow $X$, one can form its \emph{Ellis group}, the automorphism group of any minimal subflow of $\rmE(X)$, the enveloping semigroup of $X$. The Ellis group of $X$, denoted $\bbG_X$, can also be equipped with the tau-topology; in particular $\bbG = \bbG_{\mg}$. Theorem~\ref{Thm:Tau_Hausdorff} gives a complete characterization of minimal flows $X$ with $\bbG_X$ Hausdorff: They are factors of minimal flows which are a proximal extension of an equicontinuous extension of a proximal flow. In particular, we obtain a structure theorem for $\mg$ for $G\in \sf{TMD}$ and, at first, a sharper structure theorem when $G\in \sf{GPP}$.
\vspace{3 mm}

\noindent
\textbf{Fattening spaces:} This paper features \emph{near ultrafilters} in two distinct contexts. The first is when constructing the \emph{Samuel compactification} of a uniform space. The second is when constructing the \emph{Gleason completion} of a $G$-space $X$. Similar theory has sprouted up in each context, in particular in equipping each near ultrafilter space with a \emph{topo-uniform structure} \cites{BYMT, ZucMHP, BassoZucker}. We introduce in Section~\ref{Section:NUlts_Fattenings} the concept of a \emph{fattening space} which unifies these two constructions into a general framework. This framework suggests many natural generalizations of concepts from topology to the fattening space context, including versions of being an isolated point, a discrete space, a finite space, or being extremally disconnected, and we generalize many results from general topology to their fattening counterparts. This viewpoint informs many of the vocabulary choices we make when defining certain properties of $G$-flows. Given a $G$-flow $X$, these various fattening versions of topological definitions reflect properties of the ``local action'' of small neighborhoods of $e_G$. In addition to being a useful conceptual framework, fattening spaces also provide a robust technical framework for quickly and accurately performing various topological computations when working with near ultrafilter spaces. 

Comfort with these technical aspects of near ultrafilter spaces is particularly helpful for understanding Section~\ref{Section:KPT}, which develops our ``abstract KPT correspondence.''  The most important concept from this section is that of a \emph{$G$-skeleton}, which in the familiar case that $G = \aut(\flim(\cK))$ becomes a reasonable (Definition~\ref{Def:Expansion_Classes}) \fr expansion class $\cK^*$. 
We state an abstract \emph{Ramsey property} for $G$-skeletons and prove that $G \in \sf{TMD}$ if and only if $G$ admits a $G$-skeleton with the Ramsey, minimality and ED properties (see Definition~\ref{Def:GED_Unfolded}). When this holds, the \emph{folded flow} of such a $G$-skeleton is then isomorphic to $\mg$. We believe that $\sf{TMD}$ is the largest class of topological groups which admit any meaningful form of KPT correspondence describing their minimal dynamics. Upon adding a pre-compactness demand on $G$-skeletons, we obtain a characterization of when $G\in \sf{CMD}$. This yields a new characterization of Polish groups with metrizable universal minimal flow which very roughly can be viewed as a form of KPT correspondence for metric \fr classes. It would be interesting to make this connection more formal.
\vspace{3 mm}

\noindent
\textbf{Forcing and absoluteness in topological dynamics:}
We apply our abstract KPT correspondence in Section~\ref{Section:Abs} to show that the classes $\sf{EA}$, $\sf{CMD}$, and $\sf{TMD}$ are $\Delta_1$ in the L\'evy hierarchy (Theorem~\ref{Thm:TMD_Absolute}); while it is not hard to see that these classes are $\Pi_1$, the abstract KPT correspondence provides $\Sigma_1$ definitions. L\'evy absoluteness then provides yet another characterization of these classes: a group $G$ is in $\sf{CMD}$ (resp.\ $\sf{TMD}$) iff in some forcing extension where $G$ is second countable, its Raikov completion is in $\sf{PCMD}$ (resp.\ $\sf{GPP}$). We apply these absoluteness results to prove (in ZFC) further closure properties of $\sf{CMD}$ and $\sf{TMD}$ (Theorem~\ref{Thm:Abs_Apps}), and we strenthen the $\mg$ structure theorem for $\sf{TMD}$ to have the same form as the one for $\sf{GPP}$ from Section~\ref{Section:AutMG} (Theorem~\ref{Thm:MG_Structure_TMD}). For many of these results, the forcing and absoluteness proof is the only argument we know of. These techniques also motivate some key open questions and conjectures that we have, which primarily ask: Given  $G\not\in\sf{TMD}$, to what extent does $G$ have ``wild'' minimal dynamics? 

The last part of Section~\ref{Section:Abs} and Section~\ref{Section:Definable} are geared towards an application of the set-theoretic techniques developed in Section~\ref{Section:Abs} to the field of \emph{definable} topological dynamics. Pioneered by Newelski~\cite{Newelski_2009}, this is the study of groups definable in first-order structures and their \emph{definable actions} on compact spaces. Much like the Samuel compactification of a topological group, every definable group $G$, say definable in $\bM$, admits a largest externally definable ambit, denoted $\rmS_{G, ext}(\bM)$, and any minimal subflow $\cM\subseteq S_{G, ext}(\bM)$ is the \emph{universal minimal externally definable $G$-flow}. Newelski asked if in the case that $\bM$ is NIP, the automorphism group $\bbG_\cM$ equipped with the tau-topology was isomorphic to $G^\frak{C}/G_{00}$, so in particular, a compact Hausdorff topological group which only depends on the theory of $\bM$ and the first-order definition of $G$. This turns out to be false \cite{GPP_2015}; the  \emph{revised Newelski conjecture} as posed by Krupi\'nski and Pillay \cite{KP_2023} asks if $\bbG_\cM$ is tau-Hausdorff. Recently, Chernikov, Gannon, and Krupi\'nski \cite{DC3} proved the revised Newelski conjecture in the case that $\bM$ is countable, using deep results of Glasner \cite{Glasner_Tame_Gen} on the structure of minimal, metrizable, \emph{tame} flows. In particular, the structure theorem implies that the Ellis group of every minimal, metrizable, tame flow is Hausdorff. We prove that when suitably stated, the assertions that a given minimal flow has Hausdorff Ellis group and/or is tame are both $\Delta_1$. This allows us to use a forcing and absoluteness argument to affirmatively resolve the revised Newelski conjecture with no cardinality constraints. We also obtain a partial version of Glasner's structure theorem valid for all minimal tame flows. We conclude by exploring some tantalizing connections between the existing theory of definable dynamics and the development of the classes $\sf{CMD}$ and $\sf{TMD}$ given here. Upon appropriately generalizing to definable groups, stable groups are examples of ``definably $\sf{CMD}$'' groups, and work of Chernikov and Simon \cite{CS_2018} strongly suggests that definably amenable NIP groups should be ``definably $\sf{TMD}$.''
\section{Background}
\label{Section:Background}

\subsection{Notation and binary relations}
\label{Subsection:Notation}
Set-theoretic notation is standard. We write $\omega = \bbN = \{0, 1, 2,...\}$ for the set of non-negative integers, write $n< \omega$ when $n\in \omega$, and identify $n< \omega$ with the set $\{0,...,n-1\}$. Given sets $X$ and $Y$, we let $Y^X$ denote the set of functions from $X$ to $Y$. 

Given sets $X$ and $Y$ and $R\subseteq X\times Y$, we put $R^{-1} = \{(y, x)\in Y\times X: (x, y)\in R\}$. If $A\subseteq X$ and $x\in X$, we put $R[A] = \{y\in Y: \exists a\in A (a, y)\in R\}$, $R[x] = R[\{x\}]$ (unless $x\leq \omega$), $\dom(R) = R^{-1}[Y]$, and $\im(R) = R[X]$. In particular, if $f\colon X\to Y$ is a function, then $f[A]$ is the image of $A$ under $f$, and given $B\subseteq Y$, $f^{-1}[B]$ is the pre-image of $B$ under $f$. The \emph{fiber image} of $A$ under $R$ is the set $R_{\rm{fib}}(A) = \{y\in Y: \forall x\in X ((x, y)\in R\rightarrow x\in A)\}$. If $Z$ is a set and $S\subseteq Y\times Z$, we put $S\circ R = \{(x, z): \exists y\in Y (x, y)\in R\text{ and }(y, z)\in S\}$.

Given sets $X', Y'$ and a function $f'\colon X'\to Y'$, we often abuse notation and write $f\times f'$ for the function $X\times X'\to Y\times Y'$ given by $(f\times f')(x, x') = (f(x), f'(x'))$.

\subsection{Topological spaces}
\label{Subsection:Topological_spaces}
All topological spaces discussed in this paper are regular and Hausdorff by default, with exceptions pointed out explicitly.

Given $(X, \tau)$ a topological space (with $\tau$ typically understood), we write $\op(X)$ for the set of non-empty open subsets of $X$, and given $x\in X$, we write $\op(x, X)\subseteq \op(X)$ for the set of open neighborhoods of $x$ in $X$. If $A\subseteq X$, we write $\rm{Int}_X(A)$ or just $\rm{Int}(A)$ for the interior of $A$ in $X$, and we write $\rm{cl}_X(A)$, $\rm{cl}_\tau(A)$ or just $\ol{A}$ for the closure of $A$ in $X$. We write $\rmF(X)$ for the set of \emph{non-empty} closed subsets of $X$, $\rmK(X)$ for the set of \emph{non-empty} compact subsets of $X$, $\rm{clop}(X) = (\op(X)\cap \rmF(X))\cup \{\emptyset\}$, $\rc(X) = \{\ol{A}: A\in \op(X)\}\cup \{\emptyset\}$ for the \emph{regular closed} subsets of $X$, i.e, the closure of open sets of $X$, and $\rc(x, X) = \{\ol{A}: A\in \op(x, X)\}$. We recall that $\rm{clop}(X)$ is a Boolean algebra with meet and join given by union and intersection, and that $\rc(X)$ is a complete Boolean algebra with $\bigvee_i A_i = \ol{\bigcup_i A_i}$ and $\bigwedge_i A_i = \ol{\Int(\bigcap_i A_i)}$, and $\neg A = \ol{X \setminus A}$. 

One can equip $\rmK(X)$ with the \emph{Vietoris topology}, where the sub-basic open sets have the following two forms:
\begin{itemize}
    \item 
    Given $A\in \op(X)$, we define $\rm{Meets}(A, X) = \{K\in \rmK(X): K\cap A = \emptyset\}$.
    \item 
    Given $A\in \op(X)$, we define $\rm{Subset}(A, X) = \{K\in \rmK(X): K\subseteq A\}$.
\end{itemize}
Whenever $X$ is compact, $\rmK(X)$ is compact; we equip $\rmK(X)$ with the Vietoris topology unless explicitly mentioned otherwise (see Subsection~\ref{Subsection:Meets}). 

We also mention the following shortcut for verifying that a function is continuous; though we imagine that this result is folklore, we are not aware of a reference aside from \cite{ZucThesis}*{Proposition~0.0.1}

\begin{fact}
    \label{Fact:Continuity}
    Let $X$ and $Y$ be topological spaces. If $D\subseteq X$ is dense and $\phi\colon X\to Y$ is a map with the property that whenever $(d_i)_{i\in I}$ is a net from $D$ with $d_i\to x\in X$, then $\phi(d_i)\to \phi(x)$, then $\phi$ is continuous.
\end{fact}

 If $Y$ is a topological space, a function $f\colon Y\to \bbR$ is \emph{lower semi-continuous} (lsc) if for every $c\in \bbR$, we have $f^{-1}[(-\infty, c]]\subseteq Y$ closed. 

If $X$ is a set, a function $d\colon X\times X\to [0, \infty]$ is a \emph{semi-(pseudo)metric} if $d$ satisfies all properties of being a (pseudo)metric with the possible exception of the triangle inequality. If $d$ is a semi-pseudometric on $X$, $A\subseteq X$, and $c> 0$, we write $\rmB_d(A, c) := \{y\in X: \exists x\in A\, d(x, y)< c\}$. If $x\in X$, we write $\rmB_d(x, c)$ for $\rmB_d(\{x\}, c)$. We can emphasize $X$ if needed by writing $\rmB_d^X(A, c)$, etc. If $X, Y$ are sets equipped with semi-pseudometrics $d_X, d_Y$, respectively, a function $f\colon X\to Y$ is $(d_X, d_Y)$-Lipschitz if $d_Y(f(a), f(b))\leq d_X(a, b)$ for any $a, b\in X$. If $Y = \bbC$ or $\bbR$ equipped with the Euclidean metric, we just say $d_X$-Lipschitz.
Note that we allow semi-pseudometrics to take distance $\infty$.

Given spaces $X$ and $S$, we let $\rmC(X, S)$ denote the set of continuous functions from $X$ to $S$. If $S$ is a subspace of a normed space, we let $\rm{CB}(X, S)\subseteq \rmC(X, S)$ denote the bounded continuous functions. When $S = \bbC$, we omit it from the notation.  In particular, $\rm{CB}(X)$ is a unital commutative $C^*$-algebra. If $d$ is a (not necessarily continuous) semi-pseudometric on $X$, we let $\rmC_d(X, S)\subseteq \rmC(X, S)$ denote those continuous functions from $X$ to $S$ which are $(d, d_S)$-Lipschitz. Similarly for $\rm{CB}_d(X, S)\subseteq \rm{CB}(X, S)$.

\subsection{Stone and Gelfand duality}
\label{Subsection:Stone-and-Gelfand}

We will make extensive use of various duality results for Boolean and $C^*$-algebras that we briefly recall. 

Stone duality asserts that the categories of Boolean algebras and compact zero-dimensional spaces are contravariantly isomorphic. Given a Boolean algebra $B$, write $B^+:= B\setminus \{0_B\}$. An \emph{ultrafilter on $B$} is a subset $p\subseteq B^+$ maximal with respect to being closed under meets. The \emph{Stone space} of $B$, denoted $\rm{St}(B)$, is the space of all ultrafilters on $B$. Given $a\in B$, we let $\rmC_a = \{p\in \rm{St}(B): a\in p\}$. The sets $\{\rmC_a: a\in B^+\}$ form a base for a compact zero-dimensional topology on $\rm{St}(B)$, and we have $\rm{clop}(\rm{St}(B))\cong B$. Dually, given a compact zero-dimensional space $X$, we have $\st(\rm{clop}(X)) \cong X$. If $X$ and $Y$ are zero-dimensional spaces and $\phi\colon X\to Y$ is continuous, we obtain a Boolean algebra homomorphism $\hat{\phi}: \rm{clop}(Y)\to \rm{clop}(X)$ via $\hat{\phi}(A) = \phi^{-1}[A]$. If $B$ and $C$ are Boolean algebras and $\eta\colon B\to C$ is a Boolean algebra homomorphism, we obtain a continuous map $\hat{\eta}\colon \st(C)\to \st(B)$ where given $p\in \st(C)$ we have $\hat{\eta}(p) = \{a\in B: \eta(a)\in p\}$. 

We mention that the Stone duals of \emph{complete} Boolean algebras are exactly the compact \emph{extremally disconnected} spaces; we call $X$ extremally disconnected (ED) if $A\in \op(X)$ implies that $\ol{A}\in \op(X)$. Equivalently, this happens when $\rc(X) = \rm{clop}(X)$. If $X$ is any space, we write $\gl(X):= \st(\rc(X))$ for the \emph{Gleason cover} of $X$.
Since $\rc(X)$, is complete $\gl(X)$ is always ED.

Gelfand duality states that the categories of compact spaces and unital commutative $C^*$-algebras are contravariantly isomorphic.  To each compact Hausdorff space $X$, one associates the algebra $\rm{CB}(X, \bbC)$, and given a unital commutative $C^*$-algebra $\cA$, one forms the \emph{Gelfand space} or \emph{spectrum} of $\cA$, the space $\wh{\cA}$ of multiplicative linear functionals $\cA\to \bbC$ equipped with the topology of pointwise convergence. If $X$ and $Y$ are compact and $\phi\colon X\to Y$ is continuous, then one obtains a $*$-homomorphism $\hat{\phi}\colon \rmC(Y, \bbC)\to \rm{CB}(X)$ via $\hat{\phi}(f) = f\circ \phi$. Conversely, if $\cA$ and $\cB$ are unital commutative $C^*$-algebras and  $\eta\colon \cA\to \cB$ is a $*$-homomorphism, then one obtains a continous map $\hat{\eta}\colon \wh{\cB}\to \wh{\cA}$ via $\hat{\eta}(x)(a) = x(\eta(a))$. 

Using Gelfand duality, one defines the \emph{beta compactification} of a topological space $X$ to be the space $\beta X:= \wh{\rm{CB}(X)}$. There is a natural map $\iota_X\colon X\to \beta X$ given by $\iota_X(x)(f) = f(x)$, and $\iota_X$ satisfies (and is defined by) the following universal property: for any compact space $Y$ and any continuous map $\phi\colon X\to Y$, there is a continuous map $\wt{\phi}\colon \beta X\to Y$ with $\phi = \wt{\phi}\circ \iota_X$. Recall that $X$ is \emph{Tychonoff} if for any $x\in X$ and closed $K\subseteq X$ with $x\not\in K$, there is $f\in \rm{CB}(X)$ with $f(x) = 0$ and $f[K] = \{1\}$. Tychonoff spaces are exactly those $X$ for which $\iota_X$ is an embedding; in this case, we typically suppress $\iota_X$ and just view $X$ as a subspace of $\beta X$. If $I$ is a discrete space, $\beta I$ can be identified with the space of ultrafilters \emph{over} $I$, i.e.\ the space  $\rm{St}(\cP(I))$, and is called the \emph{\v{C}ech-Stone compactification} of $I$. If $\leq_I$ is an upwards-directed partial order on $I$, we say that $\cU\in \beta I$ is \emph{cofinal} if for every $i\in I$, we have $\{j\in I: i\leq_I j\}\in \cU$. If $X$ is a compact space, $(x_i)_{i\in I}$ is a tuple from $X$, $x\in X$, and $\cU\in \beta I$, we write $\lim_\cU x_i = x$ and call $x$ the \emph{ultralimit} of $(x_i)$ along $\cU$ iff for every $A\in \op(x, X)$, we have $\{i\in I: x_i\in A\}\in \cU$. By compactness, the ultralimit of a tuple always exists. In Section~\ref{Section:AutMG}, we will consider compact, not-necessarily-Hausdorff spaces, where the ultralimit becomes a subspace of $X$ rather than a member of it.

\subsection{Uniform spaces and Samuel compactifications}
\label{Subsection:Uniform_spaces}

Recall that a \emph{uniform space} $(X, \cU_X)$ is a set $X$ equipped with a system of \emph{entourages} $\cU_X\subseteq \cP(X\times X)$, such that 
\begin{itemize}
    \item 
    $\cU_X$ is upwards closed and closed under finite intersections (a filter).
    \item 
    $\cU_X$ is Hausdorff, i.e.\ $\bigcap_{U\in \cU_X} U = \Delta_X := \{(x, x): x\in X\}$.
    \item 
    $\cU_X$ admits square roots, i.e.\ for every $U\in \cU_X$, there is $V\in \cU_X$ with $V^2:= V\circ V\subseteq U$.
\end{itemize}
If $(X, d)$ is a metric space, we can equip $X$ with the uniformity $\cU_d$  generated by $\{\{(x, y): d(x, y)<\epsilon\}: \epsilon > 0\}$, and in various notation defined below, we can write $d$ instead of $\cU_d$. 

The \emph{uniform topology of $(X, \cU_X)$} has a base of neighborhoods of $x\in X$ given by $\{U[x]: U\in \cU_X\}$ and is Tychonoff (embeds into a compact space, namely $\beta X$). Conversely, if $X$ is a Tychonoff space, then the topology of $X$ can be described via a uniform space structure, possibly several. We recall that if $X$ is compact, then there is a unique compatible uniformity, namely the one generated by all open neighborhoods of $\Delta_X$, and we denote this uniformity by $\cU_X$.

A net $(x_i)_{i\in I}$ is \emph{$\cU_X$-Cauchy} if for any $U\in \cU_X$, we eventually have $U[x_i]\cap U[x_j]\neq \emptyset$. We let $(\wh{X}^{\cU_X}, \wh{\cU_X})$ denote the \emph{completion} of $(X, \cU_X)$, the uniform space of all equivalence classes of Cauchy nets. Usually, $\cU_X$ is understood, and we just write $\wh{X}$. The uniformity $\wh{\cU_X}$ is given by the upwards closure of $\{\rm{cl}_{\wh X}(U) : U\in \cU_X\}$. 

A function $f\colon X\to \bbC$ is $\cU_X$-\emph{uniformly continuous} (omitting $\cU_X$ if understood) if, for any $\epsilon > 0$, there is $U\in \cU_X$ such that $(x, y)\in U$ implies $|f(x) - f(y)|< \epsilon$. The \emph{Samuel compactification} of $(X, \cU_X)$, denoted $\sa(X, \cU_X)$ (omitting $\cU_X$ if understood) is the Gelfand dual of the algebra of bounded uniformly continuous functions from $X$ to $\bbC$. Up to canonical identifications, we have inclusions $X\subseteq \wh{X}\subseteq \sa(X)$ with $X\subseteq \sa(X)$ dense. In particular, any bounded uniformly continuous $f\colon X\to \bbC$ admits a unique continuous extension to $\sa(X)$; we often abuse notation and simply let $f\colon \sa(X)\to \bbC$ denote the continuous extension. When $I$ is a set equipped with the discrete uniformity, we have $\sa(I) = \beta I$. We will discuss a \emph{near ultrafilter} constuction of $\sa(X, \cU_X)$ that works for any uniform space at the end of Section~\ref{Section:NUlts_Fattenings}.

\subsection{Topological groups}
\label{Subsection:Topological_groups}
If $G$ is a topological group, we denote its identity by $e_G$, and we let $\cN_G$ denote a fixed base of symmetric open neighborhoods of $e_G$. The Birkhoff-Kakutani theorem (see \cite{Berberian}) implies that the topology on any topological group $G$ can be generated by continuous, bounded \emph{semi-norms} on $G$, functions $\sigma\colon G\to \bbR^{{\geq}0}$ with $\sigma(e_G) = 0$, $\sigma(g) = \sigma(g^{-1})$, and  $\sigma(gh)\leq \sigma(g)+\sigma(h)$ for any $g, h\in G$.  We write $\rm{SN}(G)$ for the set of bounded continuous semi-norms on $G$. We remark that bounded continuous semi-norms are in one-one correspondence with bounded, continuous, right-invariant pseudometrics on $G$, where given $\sigma\in \rm{SN}(G)$ and $g, h\in G$, we put $\sigma_r(g, h) = \sigma(gh^{-1})$, and given a bounded, continuous, right-invariant pseudometric, the distance from $e_G$ provides a semi-norm. We will abuse notation and use semi-norms and continuous, right-invariant pseudometrics interchangeably, simply giving a semi-norm two inputs if we want to view it as a pseudometric. Given $\sigma\in \rm{SN}(G)$ and $\epsilon > 0$, we let $\rmB_\sigma(\epsilon)\in \cN_G$ denote the set $\sigma^{-1}[[0, \epsilon )]$. A \emph{base of semi-norms} is any $\cB\subseteq \rm{SN}(G)$ so that $\{\rmB_\sigma(\epsilon): \sigma\in \cB, \epsilon > 0\}$ forms a neighborhood base for $e_G$.

Any topological group $G$ comes with several compatible uniform structures, some of which we now recall: 
\begin{itemize}
    \item 
    The \emph{right uniformity}, generated by $\{\{(g, h)\in G\times G: g\in Uh\}: U\in \cN_G\}$. Equivalently, this is the uniformity generated by the collection of continuous, bounded, right-invariant pseudometrics on $G$. Since we will define $G$-flows as \emph{left} actions, the \emph{right} uniformity is the one we will use the most. Thus $\wh{G}$ and $\sa(G)$ will always refer to the right uniformity on $G$. We let $\rm{RUCB}(G)$ denote the $C^*$-algebra of bounded right-uniformly-continuous functions from $G$ to $\bbC$.
    \item 
    The \emph{left uniformity}, generated by $\{\{(g, h)\in G\times G: g\in hU\}: U\in \cN_G\}$.
    
    \item 
    The \emph{Raikov uniformity}, the meet of left and right, generated by $\{\{(g, h)\in G\times G: g\in Uh\cap hU\}: U\in N_G\}$. We will denote the completion with respect to this uniformity by $\wt{G}$. 
\end{itemize}
We mention that the Raikov completion of a topological group is always a topological group. Polish groups are exactly the Raikov complete, second countable groups \cite{Kechris_Classical}*{Section~9}.

We write $H\leq^c G$ (or $H\trianglelefteq^c G$) when $H$ is a closed (closed and normal) subgroup of $G$, and we equip both $G/H$ and $H\backslash G$ with the quotient topology. The right uniformity on $G$ induces a compatible uniform structure on $G/H$ generated by $\{\{(gH, kH)\in G/H\times G/H: gH\in UkH\}: U\in \cN_G\}$, and all uniform space notation regarding $G/H$ will refer to this uniformity. We call $H\leq^c G$ \emph{co-precompact} if $\wh{G/H} = \sa(G/H)$;  this happens iff for every $U\in \cN_G$, there is $F\in \fin{G}$ with $UFH = G$. For $G$ Polish, this is iff $\sa(G/H)$ is metrizable.

While there is an analogous uniform structure on $H\backslash G$, it is not the one we will most often consider. If $\sigma\in \rm{SN}(G)$, write $\ker(\sigma):= \{g\in G: \sigma(g) = 0\}\leq^c G$. Then $\sigma$ induces a metric on $\ker(\sigma)\backslash G$, and this is the type of uniformity we most often encounter on right coset spaces.

Given a group $G$, a subset $T\subseteq G$ is called \emph{thick} if the collection $\{gT: g\in G\}$ has the finite intersection property, equivalently, if for any $F\in \fin{G}$, there is $g\in G$ with $Fg\subseteq T$. A subset $S\subseteq G$ is \emph{syndetic} if $G\setminus S$ is not thick, equivalently, if there is $F\in \fin{G}$ with $FS = G$. A subset $P\subseteq G$ is \emph{piecewise syndetic} if there is $F\in \fin{G}$ with $FP\subseteq G$ thick. Equivalently, $P$ is piecewise syndetic iff for some thick $T\subseteq G$, the set $(P\cap T)\cup(G\setminus T)\subseteq G$ is syndetic. From this, it follows that the collection of piecewise syndetic subsets of $G$ is a coideal: If $P\subseteq G$ is piecewise syndetic and $P = P_0\cup P_1$, then $P_0$ or $P_1$ is also piecewise syndetic. All three of these largeness notions are $2$-sided invariant.

When $G$ is a topological group, we call a subset $S\subseteq G$ \emph{pre-thick} or \emph{presyndetic} if for any $U\in \cN_G$, the subset $US\subseteq G$ is thick or syndetic, respectively.\footnote{In particular, $H\leq^c G$ is presyndetic iff for every $U\in \cN_G$, there is $F\in \fin{G}$ with $G = FUH$. Note the similarity between the definitions of $H\subseteq G$ being presyndetic versus being co-precompact.}

\subsection{Topological dynamics}
\label{Subsection:Top_Dyn}
Fix a topological group $G$. A \emph{left $G$-space} is a topological space $X$ equipped with a continuous action $a\colon G\times X\to X$. Typically, the action $a$ is understood, and we simply write $g\cdot x$ or $gx$ for $a(g, x)$. One can define right $G$-spaces analogously; we refer to left by default. Generally speaking, topological vocabulary attached to a $G$-space $X$ refers to $X$ as a topological space. Later on, we will decorate certain topological vocabulary with $G$ as a prefix, and this can change the meaning. A key example of a $G$-space is $G$ itself under left multiplication, i.e.\ $a(g, h) = gh$ for $g, h\in G$. Whenever we refer to $G$ as a $G$-space, we by default refer to this action. A \emph{$G$-flow} is a compact $G$-space. We remark that for any $G$-flow $X$, there is a unique continuous extension of the action to $\wt{G}$.

Given a $G$-space $X$, one obtains a \emph{right} action of $G$ on $\rm{CB}(X)$ where given $f\in \rm{CB}(X)$ and $g\in G$, $f{\cdot}g\in \rm{CB}(X)$ is defined via $f{\cdot}g(x) = f(gx)$. However, the action may fail to be norm-continuous. We let $\rm{CB}_G(X)\subseteq \rm{CB}(X)$ denote the subset of $f\in \rm{CB}(X)$ where the map from $G$ to $\rm{CB}(X)$ given by $g\to f\cdot g$ is norm-continuous. Given  $\sigma\in \rm{SN}(G)$, we put $\rm{CB}_\sigma(X) = \{f\in \rm{CB}(X): \|f{\cdot}g - f\|\leq \sigma(g)\}$. Note that $\rm{CB}_G(X) = \bigcup_{\sigma\in \rm{SN}(G)} \rm{CB}_\sigma(X)$ and that $\rm{CB}_G(G) = \rm{RUCB}(G)$. If $X$ is a $G$-flow, we have $\rm{CB}(X) = \rm{CB}_G(X)$.

If $X$ and $Y$ are $G$-spaces, we call $\phi\colon X\to Y$ a \emph{$G$-map} if $\phi$ is continuous and $G$-equivariant. If $\phi$ is also surjective, we call it a \emph{factor $G$-map} and say that $Y$ is a \emph{$G$-factor} of $X$ or that $X$ is a \emph{$G$-extension} of $Y$; one can say \emph{factor} or \emph{extension} if $G$ is understood. A \emph{$G$-subspace} of the $G$-space $X$ is a non-empty $G$-invariant subspace $Y\subseteq X$. A $G$-subspace $Y\subseteq X$  is \emph{proper} if $Y\subsetneq X$, and it is a \emph{$G$-subflow} or just \emph{subflow} if it is compact. We write $\rm{Sub}_G(X)$ for the space of $G$-subflows of $X$ equipped with the Vietoris topology; when $X$ is a $G$-flow, $\rm{Sub}_G(X)$ is compact. A $G$-flow $X$ is called \emph{minimal} if every orbit is dense, equivalently if $X$ contains no proper subflows. Note by Zorn's lemma that every flow contains a minimal subflow. Let us briefly note two criteria for minimality; one is folklore, and the other is a mild strengthening of a folklore result. 

\begin{lemma}
    \label{Lem:Min}
    Fix a $G$-flow $X$.
    \begin{enumerate}[label=\normalfont(\arabic*)]
        \item\label{Item:AP_Lem:Min}
        For $x\in X$, the subflow $\ol{G\cdot x}$ is minimal iff for every $A\in \op(x, X)$, the set $\{g\in G: gx\in A\}$ is syndetic. In particular, if $X$ is minimal and $A\in \op(X)$, then $\{g\in G: gx\in A\}$ is syndetic. 
        \item\label{Item:Min_Lem:Min} 
        $X$ is minimal iff for every $A\in \op(X)$, there is $F\in \fin{G}$ with $FA\subseteq X$ dense.
    \end{enumerate}
\end{lemma}

\begin{proof}
    \ref{Item:AP_Lem:Min}: See \cite{Auslander}*{Theorem 7, p.\ 11}.
    \vspace{3 mm}

    \noindent
    \ref{Item:Min_Lem:Min}: The classical version is that $X$ is minimal iff for every $A\in \op(X)$, there is $F\in \fin{G}$ with $FA = X$, so one direction is clear. In the other, suppose $Y\subsetneq X$ is a minimal subflow. Fix $A\in \op(X)$ with $Y\cap \ol{A} =  \emptyset$. For any $F\in \fin{G}$, $X\setminus \ol{FA}$ is open and contains $Y$. Hence $FA\subseteq X$ cannot be dense for any $F\in \fin{G}$. 
\end{proof}

\begin{corollary}
     \label{Cor:Orbit_Fragment}
    Let $X$ be a minimal $G$-flow and $x\in X$. Then if $T\subseteq G$ is thick, we have $T\cdot x\subseteq X$ dense. If $P\subseteq G$ is piecewise syndetic, then $P\cdot x\subseteq X$ is somewhere dense.
\end{corollary}

\begin{proof}
    It suffices to show the claim about thick sets, which follows immediately from Lemma~\ref{Lem:Min}\ref{Item:AP_Lem:Min}. 
\end{proof}

$\sa(G)$ is a $G$-flow, where given $g\in G$, $p\in \sa(G)$, and $f\in \rm{RUCB}(G)$, we have $gp(f) = p(fg)$. Note that the orbit of $e_G$ in $\sa(G)$ is $G$ and is therefore dense. The flow $\sa(G)$ satisfies the following universal property: If $X$ is any other $G$-flow and $x\in X$, then there is a (necessarily unique) $G$-map $\rho_x\colon \sa(G)\to X$ with $\rho_x(e_G) = x$. In particular, for $g\in G$, we have $\rho_x(g) = gx$, hence our choice of notation $\rho_x$ which suggests ``right multiplication'' by $x$. Hence given $p\in \sa(G)$, we often write $p\cdot x$ or $px$ for $\rho_x(p)$. When $X = \sa(G)$, this is an associative binary operation, giving $\sa(G)$ the structure of a \emph{compact right-topological semigroup (CRTS)}. Every CRTS contains an \emph{idempotent}, an element $u$ with $u\cdot u = u$, and any subflow of $\sa(G)$ is a subsemigroup. If $M\subseteq \sa(G)$ is a minimal subflow and $u\in M$ is idempotent, then $Mu = M$, implying that the map $\rho_u\colon \sa(G)\to M$ is a retraction of $\sa(G)$ onto $M$. For much more on CRTSs, see \cite{HS_Alg_SC}. 

Similarly, if $H\leq^c G$, then $\sa(G/H)$ is a $G$-flow, the orbit of \emph{the element} $H\in \sa(G/H)$ is dense, and given any $G$-flow $X$ and an $H$-fixed-point $x\in X$, there is a $G$-map $\rho_{x, H}\colon \sa(G/H)\to X$ with $\rho_{x, H}(H) = x$. Writing $\pi_H\colon G\to G/H$ for the quotient, it continuously extends to $\wt{\pi}_H\colon \sa(G)\to \sa(G/H)$, and we have $\rho_x = \rho_{x, H}\circ \wt{\pi}_H$.

By a theorem of Ellis \cite{Ellis}, there exists for each topological group $G$ a \emph{universal minimal flow}, which we denote by $\rmM(G)$. This is a minimal $G$-flow with the property that every minimal $G$-flow is a factor of it, and this universal property uniquely characterizes $\rmM(G)$ up to isomorphism. Every minimal subflow of $\sa(G)$ is clearly a universal minimal flow. Ellis then uses techniques from the theory of CRTSs to show that any minimal $M\subseteq \sa(G)$ is \emph{coalescent}, meaning that every $G$-map from $M$ to itself is an isomorphism. The uniqueness of $\rmM(G)$ then follows. One can also give an argument that $\rmM(G)$ is coalescent directly using its universal property \cite{Gutman_Li}. 

\section{Near ultrafilters on fattening spaces}
\label{Section:NUlts_Fattenings}

We will make heavy use of two distinct but very similar near ultrafilter constructions; one from a uniform space builds the Samuel compactification \cite{KocakStrauss}, and the other from a $G$-space $X$ builds a  $G$-flow $\rmS_G(X)$, whose relationship with $X$ depends on the properties of $X$ \cites{ZucDirectGPP, ZucThesis}. In this section, we take the opportunity to unify these constructions by creating the framework of \emph{fattening spaces}. Given a fattening space, we will construct the compact space of \emph{near ultrafilters} on it, and we can use the fattening space structure to define canonical families of lower semi-continuous pseudometrics on the near ultrafilter  space, first explored in \cite{BYMT}. Our proof that these pseudometrics are \emph{adequate} is conceptually much simpler than those appearing in \cite{ZucMHP}.  Uniform spaces are discussed in detail in Subsection~\ref{Subsection:Unif_Fat}, and $G$-spaces in Section~\ref{Section:Gleason}.

\subsection{Lower semi-continuous semi-pseudometrics}
\label{SubSec:LSC}

\begin{defin}
    \label{Def:Topometric}
    Fix a topological space $Y$ and an lsc semi-pseudometric $\partial$ on $Y$.
    
    \begin{itemize}
        \item We say that $y\in Y$ is \emph{$\partial$-compatible} iff $y\in \Int(\rmB_{\partial}(y, c))$ for every $c> 0$,  iff $y\in \Int(\ol{\rmB_{\partial}(y, c)})$ for every $c> 0$. If each $y \in Y$ is $\partial$-compatible, we say that $\partial$ is \emph{compatible}.
        \item We say that $\partial$ is \emph{open-compatible} if whenever $A\in \op(Y)$ and $c> 0$, we have $\ol{A}\subseteq \Int(\ol{\rmB_\partial(A, c)})$. If $\partial$ is compatible then it is open-compatible.
        \item We say that $\partial$ is \emph{adequate} if whenever $A\in \op(Y)$ and $c> 0$, we have $\rmB_{\partial}(A, c)\in \op(Y)$.
    \end{itemize}
\end{defin}

Notice that if $\partial$ is a \emph{discrete} lsc semi-pseudometric on $Y$, then the above properties reduce to topological properties: a point is $\partial$-compatible iff it is isolated, $\partial$ is compatible iff $Y$ is discrete, and $\partial$ is open-compatible iff $Y$ is ED. This observation will inform our choice of terminology in the sequel.

\begin{proposition}
    \label{Prop:EDOpenComp}
    If $\partial$ is an lsc semi-pseudometric on an ED topological space $Y$, then $\partial$ is open-compatible.
\end{proposition}
\begin{proof}
    Given $A\in \op(Y)$ and $c> 0$, we have $\ol{A}\subseteq \ol{\rmB_\partial(A, c)}$, so $\Int(\ol{A})\subseteq \Int(\ol{\rmB_\partial(A, c)})$, and as $Y$ is ED, we have $\Int(\ol{A}) = \ol{A}$. 
\end{proof}

A \emph{compact topometric space} $(Y, \partial)$ is a compact space $Y$ equipped with an lsc metric $\partial$. This implies that $\partial$ is complete and yields a finer uniformity than the compact one \cite{Ben_Yaacov_2008}.

If $Y$ is compact and $\partial$ is an lsc pseudometric on $Y$, then the equivalence relation on $Y$ given by $\partial(x, y) = 0$ is closed. We write $Y/\partial$ for the quotient, $\pi_\partial\colon Y\to Y/\partial$ for the quotient map, and $[y]_\partial$, $[A]_\partial$ for the $\partial$-saturation of some $y\in Y$ or $A\subseteq Y$. Letting $\ol{\partial}$ denote the metric induced by $\partial$ on $Y/\partial$, then $(Y/\partial, \ol{\partial})$ is a compact topometric space. More generally, if $Z$ is another compact space and $\pi\colon Y\to Z$ is a continuous surjection, we say that $\pi$ is \emph{subordinate to $\partial$} if $\{(x, y)\in Y^2: \pi(x) = \pi(y)\}\subseteq \partial^{-1}[\{0\}]$. In this case, we write $\partial^\pi$ for the induced lsc pseudometric on $Z$ (so, in the case of $Y/\partial$, $\ol{\partial} = \partial^{\pi_\partial}$).

\begin{prop}
\label{Prop:Adequate}
    Fix a space $Y$ and $\partial$ an adequate lsc pseudometric on $Y$.  
    \begin{enumerate}[label=\normalfont(\arabic*)]
        \item\label{Item:Dense_Prop:Adequate} 
         If $A\subseteq Y$, $D\subseteq A$ is dense, and $c> 0$, then $\rmB_\partial(D, c)\subseteq \rmB_\partial(A, c)$ is dense.
        \item\label{Item:Comp_Prop:Adequate}
        $y\in Y$ is $\partial$-compatible iff $\Int(\rmB_\partial(y, c))\neq \emptyset$ for every $c> 0$.
    \end{enumerate}
    For the next items, assume $Y$ is compact, $Z$ is another compact space, and $\pi\colon Y\to Z$ is a continuous surjection subordinate to $\partial$.
    \begin{enumerate}[resume, label=\normalfont(\arabic*)]
        \item\label{Item:Quotient_Prop:Adequate}
        $\partial^\pi$ is adequate. If $\partial$ is also open-compatible, then so is $\partial^\pi$. 
        \item\label{Item:CompDown_Prop:Adequate}
        $y\in Y$ is $\partial$-compatible iff $\pi(y)$ is $\partial^\pi$-compatible.
    \end{enumerate}
\end{prop}

\begin{proof}
    \ref{Item:Dense_Prop:Adequate}: Fix $y\in \rmB_\partial(A, c)$ and $B\in \op(y, Y)$. Then $\rmB_\partial(B, c)\cap A\neq \emptyset$. By adequacy, $\rmB_\partial(B, c)$ is open, hence $\rmB_\partial(B, c)\cap D\neq \emptyset$, i.e.\ $B\cap \rmB_\partial(D, c)\neq \emptyset$.  
    \vspace{3 mm}

    \noindent
    \ref{Item:Comp_Prop:Adequate}: One direction is clear. For the other, if $\Int(\rmB_\partial(y, c))\neq \emptyset$, then $y\in \rmB_\partial(\Int(\rmB_\partial(y, c)), c)\subseteq \rmB_\partial(y, 2c)$, and by adequacy $\rmB_\partial(\Int(\rmB_\partial(y, c)), c)$ is open.
    \vspace{3 mm}

    \noindent
    \ref{Item:Quotient_Prop:Adequate}: The first statement is immediate. For the second, fix $c > 0$. It is enough to show that given $A\in \op(Y)$, $x\in \ol{A}$, and $y\in Y$ with $\partial(x, y) = 0$, then $y\in \Int(\ol{\rmB_{\partial}(A, c)})$. Write $\epsilon = c/3$. As $\partial$ is open-compatible, we have $\ol{A}\subseteq \Int(\ol{\rmB_{\partial}(A, \epsilon)})=: B$. By adequacy $\rmB_{\partial}(A, \epsilon)\subseteq B$ is dense. Then using Item~\ref{Item:Dense_Prop:Adequate}, we have $$y\in \ol{\rmB_{\partial}(B, \epsilon)} = \ol{\rmB_{\partial}(\rmB_{\partial}(A, \epsilon), \epsilon)}\subseteq \ol{\rmB_{\partial}(A, 2\epsilon)}.$$
    Hence, once again using that $\partial$ is open-compatible, we have 
    \begin{equation*}
    y\in \Int(\ol{\rmB_{\partial}(\rmB_{\partial}(A, 2\epsilon), \epsilon)})\subseteq \Int(\ol{\rmB_{\partial}(A, 3\epsilon)}).
    \end{equation*}

    \noindent
    \ref{Item:CompDown_Prop:Adequate}: For the forward direction, first note that by Item~\ref{Item:Comp_Prop:Adequate} that if $z\in Y$ with $\partial(y, z) = 0$, then $z$ is also $\partial$-compatible. Now note that $\pi^{-1}[\rmB_{\partial^\pi}(\pi(y), c)] = \rmB_{\partial}(y, c) = \rmB_{\partial}(z, c)$. In particular, for every $z\in \pi^{-1}[\pi(y)]$, we have $z\in \Int(\rmB_\partial(z, c))$. By compactness, we have $\pi(y)\in \Int(\rmB_{\partial^\pi}(\pi(y), c))$ as desired. 

    The reverse follows from the equality $\pi^{-1}[\rmB_{\partial^\pi}(\pi(y), c)] = \rmB_{\partial}(y, c)$ along with Item~\ref{Item:Comp_Prop:Adequate}. 
\end{proof}

\subsection{Fattening spaces}
\label{SubSec:Fat}

Given a topological space $X$ and an adequate lsc pseudometric $\rho$ on $X$, then this equips $\op(X)$ with a natural ``fattening'' operation, where the fattening of $A\in \op(X)$ by some $c> 0$ is just $\rmB_{\rho}(A, c)\in \op(X)$. We also note that by Proposition~\ref{Prop:Adequate}\ref{Item:Dense_Prop:Adequate}, if $A, B\in \op(X)$ and $\ol{A} = \ol{B}$, then also $\ol{\rmB_\rho(A, c)} = \ol{\rmB_\rho(B, c)}$. Hence we can think of this fattening structure as living on $\rc(X)$. Fattening spaces are abstractions of this idea.

\begin{defin}
    \label{Def:Fattening}
    Fix $L$ an upwards-directed partial order. 
    
    An \emph{$L$-fattening space} is a topological space $X$ equipped with a function $\Phi\colon \rc(X)\times L\times (0, \infty)\to \rc(X)$, called an $L$-\emph{fattening system}, satisfying the following properties. We typically suppress $\Phi$ and write $\Phi(A, \sigma, c)=: A(\sigma, c)$. Below, we have  $A, A_i, B\in \rc(X)$, $\sigma \in L$, $c, \epsilon > 0$.
    \begin{itemize}
         \item 
        (Basics): $A(\sigma, c)\supseteq A$, $\emptyset(\sigma, c) = \emptyset$.
        \item 
        (Monotonicity): $(A, \sigma, c)\to A(\sigma, c)$ is increasing in the $\rc(X)$ and $(0, \infty)$ coordinates and decreasing in the $L$ coordinate (both non-strict). For intuition, note that larger pseudometrics have smaller balls.
        \item 
        (Continuity): $A(\sigma, c) = \bigvee_{\epsilon > 0} A(\sigma, c-\epsilon)$. 
        \item 
        (Symmetry): $A(\sigma, c)\wedge B\neq \emptyset$ iff $A\wedge B(\sigma, c)\neq \emptyset$.
        \item 
        (Joins): $\bigvee_i A_i(\sigma, c) = (\bigvee_i A_i)(\sigma, c)$. 
        \item 
        (Triangle inequality): $A(\sigma, c)(\sigma, \epsilon) \subseteq A(\sigma, c+\epsilon)$. 
    \end{itemize}
    Given $A, B\in \rc(X)$, we say $A$ and $B$ are \emph{apart} and write $A\perp B$ if there are $\sigma\in L$ and $c> 0$ with $A(\sigma, c)\cap B = \emptyset$. If $A$ and $B$ are not apart, we call them \emph{near}, and write $A\parallel B$. Given any $\cA\subseteq \rc(X)^+$, we put $[\cA]_{\rm{fat}} = \{A(\sigma, c): A\in \cA, \sigma\in L, c> 0\}$. We remark that we always understand $\neg A(\sigma, c)$ to mean $\neg (A(\sigma, c))$. Topological vocabulary applied to an $L$-fattening space applies to the underlying space.

    We say that $X$ is \emph{$L$-continuous} if for any $x\in X$ and $A\in \op(x, X)$, there is $B\in [\rc(x, X)]_{\rm{fat}}$ with $B\subseteq A$.
\end{defin}

We briefly introduce two examples of fattening spaces as motivation. The first example comes from dynamics.

\begin{defin}
\label{def:SNG-fattening}
    Suppose $G$ is a topological group and $X$ is a $G$-space. Consider $\rm{SN}(G)$ with the pointwise order. Then $X$ admits a $\sn(G)$-continuous fattening system given by letting $\ol{A}(\sigma, c) := \ol{\rmB_\sigma(c)\cdot A}$, 
    for any $A \in \op(X)$, $\sigma\in \sn(G)$, and $c> 0$.
\end{defin}

The second example will be prototypical in a sense we will soon make precise. Let $L$ be any upwards directed poset, let $X$ be a space, and let $\sigma\to d_\sigma$ be an order-preserving map from $L$ to adequate lsc pseudometrics on $X$ (where we equip the set of pseudometrics with the pointwise order). We define an $L$-fattening system on $\rc(X)$ where given $A\in \op(X)$ and $c> 0$, we set $\ol{A}(\sigma, c) := \ol{\rmB_{d_\sigma}(A, c)} = \ol{\rmB_{d_\sigma}(\ol{A}, c)}$, with the second equality by Proposition~\ref{Prop:Adequate}\ref{Item:Dense_Prop:Adequate}. 

\begin{defin}
    \label{Def:AdFat}
    An $L$-fattening space is \emph{adequate} if it is induced by adequate lsc pseudometrics as above. To emphasize the pseudometrics, we can write $(X, (d_\sigma)_\sigma)$ for this fattening space. 
\end{defin}

We will eventually see that the pseudometrics are uniquely determined (Lemma~\ref{Lem:FuncFat}) and that every locally-compact ED fattening space is adequate. 

\begin{remark}
    In both of the examples discussed, there is arguably somewhat more emphasis on open sets rather than regular closed sets. Indeed, in both cases, we actually have a function $\op(X)\times L\times (0, \infty)\to \op(X)$ which induces a well-defined function $\rc(X)\times L\times (0, \infty)\to \rc(X)$. We define things using regular closed sets to take advantage of the fact that $\rc(X)$ is a complete Boolean algebra.    
\end{remark}

We now associate to a fattening space a collection of lsc semi-pseudometrics. For adequate fattening spaces, these will agree with the corresponding adequate lsc pseudometrics.

\begin{defin}
    \label{Def:Fattening_SPM}
    Fix an $L$-fattening space $X$. Given $\sigma\in L$, we define $\partial_\sigma^X\colon X^2\to [0, \infty]$ where given $x, y\in X$ and $c\in [0, \infty)$, we have
    \begin{align*}
    \partial_\sigma^X(x, y)\leq c &\Leftrightarrow \forall \epsilon > 0\,  \forall A\in \op(x, X)\,  \forall B\in\op(y, X) \,\, \ol{A}(\sigma, c+\epsilon)\wedge \ol{B} \ne \emptyset\\
    &\Leftrightarrow  \forall \epsilon > 0\,  \forall A\in \op(x, X)\,\, y\in \ol{A}(\sigma, c+\epsilon).
    \end{align*}
    When $X$ is understood, we omit it from the notation. Remark that for each $\sigma\in L$,  $c, \epsilon>0$, and $A \in \op(X)$, we have
\begin{equation}
    \label{Eq:Basic_Fat_partial}
    \rmB_{\partial_\sigma}(A, c) \subseteq \ol{A}(\sigma, c) \subseteq \rmB_{\partial_\sigma}(A, c+\epsilon).
\end{equation}
\end{defin}

Clearly $\partial_\sigma$ is an lsc semi-pseudometric. In general it need not be a pseudometric. 
A sufficient condition for this to hold is that $X$ satisfies the following weakening of the ED property.

\begin{defin}
    \label{Def:Fattening_Spaces_Props}
   An $L$-fattening space $X$ is \emph{$L$-extremally disconnected}, or $L$-ED, if for every $A\in \rc(X)$,  $\sigma \in L, c>0$, we have $A\subseteq \rm{Int}(A(\sigma, c))$.
\end{defin}

\begin{lemma}
    \label{Lem:ED_Pseudo}
    An $L$-fattening space $X$ is $L$-ED if and only if $\partial_\sigma$ is open-compatible, for each $\sigma\in L$, in which case each $\partial_\sigma$ is a pseudometric.
\end{lemma}

\begin{proof}
Suppose that each $\partial_\sigma$ is open-compatible. Fix $A\in \op(X)$, $\sigma\in L$, and $c> 0$. Then by open-compatibility and \eqref{Eq:Basic_Fat_partial}, we have
\begin{equation*}
    \ol{A}\subseteq \Int(\ol{\rmB_{\partial_\sigma}(A, c)}) \subseteq \Int(\ol{A}(\sigma, c)).
\end{equation*}

Conversely, suppose $X$ is $L$-ED, and fix $\sigma\in L$, $A\in \op(X)$ and $c> 0$. Then by $L$-ED and \eqref{Eq:Basic_Fat_partial}, we have
\begin{equation*}
    \ol{A}\subseteq \Int\left(\ol{A}(\sigma, c/2)\right) \subseteq \Int\left(\ol{\rmB_{\partial_\sigma}(A, c)}\right),
\end{equation*}
so $\partial_\sigma$ is open-compatible. Finally, notice that if $X$ is $L$-ED,
$$\partial_\sigma(x, y)\leq c\Leftrightarrow \forall \epsilon > 0\,  \forall A\in \op(x, X)\,\, y\in \Int(\ol{A}(\sigma, c+\epsilon)),$$
and from this, triangle inequality follows. 
\end{proof}

\begin{defin}
    If $X$ is an $L$-ED fattening space, the uniformity generated by the family of pseudometrics $(\partial_\sigma)_{\sigma\in L}$ is called the \emph{$L$-fattening uniformity} on $X$, or just the \emph{$L$-uniformity}. Notice that the $L$-fattening uniformity is Hausdorff whenever  $X$ is $L$-continuous.

    The topology induced by the $L$-fattening uniformity is called the \emph{$L$-fattening topology} on $X$.
\end{defin}
The pair $(X, (\partial_\sigma)_\sigma)$ form a topo-uniform space in the sense of \cite{BassoZucker}.
If $X$ is a compact $L$-continous $L$-ED fattening space, then the $L$-fattening uniformity is complete and and the $L$-fattening topology is finer than the topology on $X$, as is the case for compact topometric spaces.

We now show that for adequate fattening spaces, the pseudometrics are uniquely determined. 

\begin{lemma}
\label{Lem:FuncFat}
    Suppose that $(X, (d_\sigma)_\sigma)$ is an adequate $L$-fattening space. Then for each $\sigma\in L$, we have $d_\sigma = \partial_\sigma$.  
\end{lemma}

\begin{proof}
    The inequality $\partial_\sigma \leq d_\sigma$ is clear. For the reverse, suppose $x, y\in X$ and $c\geq 0$ is such that $\partial_\sigma(x, y)\leq c$. Thus for any $A\in \op(x, X)$, $B\in \op(y, X)$, and $\epsilon > 0$, we have that $\ol{A}(\sigma, c+\epsilon)\cap B = \ol{\rmB_{d_\sigma}(A, c+\epsilon)}\cap B\neq \emptyset$. So also $\rmB_{d_\sigma}(A, c+\epsilon)\cap B\neq \emptyset$. Using that $d_\sigma$ is lsc, we have $d_\sigma(x, y)\leq c$.  
\end{proof}

In general, we do not know if one can always recover the fattening structure on $X$ just from the $\partial_\sigma$ functions. This can be done if $X$ is locally compact.

\begin{prop}
    \label{Prop:FD}
    Fix $X$ a locally compact $L$-fattening space. Given $A\in \op(X)$, $\sigma\in L$, and $c> 0$, we have
    $$\rmB_{\partial_\sigma}(A, c) = \bigcup\{\ol{B}(\sigma, c-\epsilon): B\in \op(X) \text{ with } \ol{B}\in \rmK(A), \epsilon > 0\}.$$
    In particular, we have $ \ol{A}(\sigma, c) = \ol{\rmB_{\partial_\sigma}(A, c)}$.
\end{prop}

\begin{proof}
    The $\subseteq$ direction follows from local compactness. For the reverse, let $B\in \op(X)$ with $\ol{B}\in \rmK(A)$, $\epsilon > 0$, and $y\in \ol{B}(\sigma, c-\epsilon)$. Using compactness and the joins property, there is $x\in \ol{B}$ with $y\in \ol{E}(\sigma, c-\epsilon)$ for every $E\in \op(x, X)$. So $\partial(x, y)< c$. The equality $\ol{\rmB_{\partial_\sigma}(A, c)} = \ol{A}(\sigma, c)$ follows from the join and continuity properties.
\end{proof}

 We can now show that locally-compact $L$-ED fattening spaces are adequate.

\begin{corollary}
    \label{Cor:ED_Adequate}
    Fix $X$ a locally compact  $L$-ED fattening space. Then for each $\sigma\in L$, $\partial_\sigma$ is adequate.
\end{corollary}
\begin{proof}
    Let $A\in \op(X)$, $c> 0$, and $\epsilon > 0$. Then by Proposition~\ref{Prop:FD}, \eqref{Eq:Basic_Fat_partial} and $L$-ED, we have
    $$\rmB_{\partial_\sigma}(A, c) = \bigcup\left\{\Int\left(\ol{B}(\sigma, c-\epsilon)\right): B\in \op(X) \text{ with } \ol{B}\in \rmK(A), \epsilon > 0\right\}.$$
\end{proof}

When $X$ is compact and topologically ED, Proposition \ref{Prop:FD} gives an explicit description of the $\partial_\sigma$ balls.
\begin{corollary}
    \label{Prop:Fattening_PM_CED}
    Suppose $X$ is a compact ED $L$-fattening space. Then for each $\sigma\in L$,  $A\in \rc(X) = \rm{clop}(X)$ and $c> 0$, we have that:
    $$\rmB_{\partial_\sigma}(A, c) = \bigcup_{\epsilon > 0} A(\sigma, c-\epsilon).$$
\end{corollary}

The prototypical example of a compact ED fattening space is the Gleason cover of a fattening space. 

\begin{defin}
    Given an $L$-fattening space $X$, we endow $\rm{Gl}(X)$ of the $L$-fattening system inherited from $X$ by Stone duality. We call such $L$-fattening space the \emph{Gleason cover} of $X$. 
\end{defin}

\subsection{Near Ultrafilters}

Our goal now is to define the space of \emph{near ultrafilters} on a fattening space $X$, denoted $\NU(X)$. This will be defined by forming a closed equivalence relation $\sim$ on $\gl(X)$ and setting $\NU(X):= \gl(X)/\sim$. Given $Y\subseteq \gl(X)$ a $\sim$-class, the union $\bigcup Y\subseteq \rc(X)$ will be a combinatorial object that we call a \emph{near ultrafilter} on $X$, and we will see that we can work with these objects in similar ways as ultrafilters.

Indeed, by Corollary~\ref{Cor:ED_Adequate} we can form the adequate lsc pseudometrics $\partial_\sigma^{\rm{Gl}(X)} = \partial_\sigma$. Let $\sim$ be the equivalence relation on $\gl(X)$ given by $p\sim q$ iff $\partial_\sigma(p, q) = 0$ for every $\sigma\in L$. This is a closed equivalence relation, and we equip $\gl(X)/{\sim}$ with the quotient topology. We analyze the quotient map $\psi_X \colon \gl(X)\to \gl(X)/{\sim}$ in more detail, giving us a way to view members of $\gl(X)/{\sim}$ as ultrafilter-like objects. 

First a technical lemma, where we observe that the $\partial_\sigma$-distance between two compact sets can be computed using almost the same formula. 

\begin{lemma}
    \label{Lem:FD_Compacts}
    Suppose $X$ is an $L$-fattening space, let $\sigma \in L$, and $P, Q\in \rmK(X)$, and fix $c\geq 0$. Then 
    $$\partial_\sigma(P, Q)\leq c \Leftrightarrow \forall A, B\in \op(X) ((A\supseteq P \text{ and } B\supseteq Q)\rightarrow (\ol{A}(\sigma, c+\epsilon)\wedge \ol{B}\neq \emptyset)).$$
\end{lemma}

\begin{proof}
    The $\Rightarrow$ implication is clear, so assume the right hand side. We first prove that $\partial_\sigma(P, Q)\leq c$ assuming that $P = \{x\}$ for some $x\in X$. Since $\ol{A}(\sigma, c+\epsilon)\cap B \neq \emptyset$ for every open $B\supseteq Q$, we have $\ol{A}(\sigma, c+\epsilon)\cap Q\neq \emptyset$. As $Q$ is compact, there is some $y\in Q$ with $y\in \ol{A}(\sigma, c+\epsilon)$ for every $\epsilon > 0$, i.e.\ $\partial_\sigma(x, y)\leq c$.

    Now let $P\in \rmK(X)$ be arbitrary. Mimicking the above argument, find $y\in Q$ with $y\in \ol{A}(\sigma, c+\epsilon)$ for every open $A\supseteq P$. By the above result, $\partial_\sigma(y, P)\leq c$.
\end{proof}

Recall the notation introduced at the beginning of \Cref{Subsection:Stone-and-Gelfand}.

\begin{lemma}
\label{Lem:NU_Prelims}
    Let $X$ be an $L$-fattening space and let $Y\neq Z\subseteq \gl(X)$ be $\sim$-classes. 
    \begin{enumerate}[label=\normalfont(\arabic*)]
        \item 
        For any $A\in \rc(X)$, $\sigma\in L$, and $c> 0$, if $\rmC_A\cap Y\neq \emptyset$, then $Y\subseteq \rmC_{A(\sigma, c)}$.
        \item 
        For any $A\in \rc(X)$ with $\rmC_A\cap Y = \emptyset$, then there are $\sigma\in L$ and $c> 0$ with $Y\cap \rmC_{A(\sigma, c)} = \emptyset$.
        \item
        There are $A, B\in \rc(X)$, $\sigma\in L$, and $c> 0$ with $Y\subseteq \rmC_A$, $Z\subseteq \rmC_B$, and $A(\sigma, c)\wedge B = \emptyset$.
    \end{enumerate}
\end{lemma}

\begin{proof}
    For the first item, if $\rmC_A\cap Y\neq \emptyset$, we have $\partial_\sigma(y, \rmC_A) = 0$. Hence $y\in \rmC_{A(\sigma, c)} = \ol{\rmB_{\partial_\sigma}(\rmC_A, c)}$. For the second item, by the defintion of $\sim$, we can find $\sigma\in L$ with $\partial_\sigma(Y, \rmC_A) = 2c > 0$, then use Lemma~\ref{Lem:FD_Compacts}. The third item then follows.
\end{proof}

\begin{defin}
    Let $X$ be an $L$-fattening space. 
    Say that $p\subseteq \rc(X)$ \emph{has finite $L$-near intersections} if for every $F\in \fin{p}$ $\sigma\in L$ and $c> 0$, we have $\bigwedge_{A\in F} A(\sigma, c)\neq \emptyset$. An $L$-\emph{near ultrafilter over $X$} is a subset of $\rc(X)$ which is maximal with respect to having finite $L$-near intersections. Let $\NU(X)$ denote the set of near ultrafilters over $X$. 
    
    If $A\in \op(X)$ and $p\in \NU(X)$, we often abuse notation and write $A\in p$ to mean $\ol{A}\in p$, and we sometimes view $p\subseteq \op(X)$ accordingly.
\end{defin}

\begin{prop}
\label{Prop:NU_Equals_Quotient}
    The near ultrafilters over $X$ are exactly of the form $\bigcup Y = \{A\in \rc(X): \rmC_A\cap Y\neq \emptyset\}$, where $Y\subseteq \gl(X)$ is a $\sim$-class. 
\end{prop}

\begin{proof}
    That sets of this form are near ultrafilters is immediate from Lemma~\ref{Lem:NU_Prelims}. Conversely, suppose $p\subseteq \rc(X)$ is a near ultrafilter over $X$. Recall that $\psi_X \colon \gl(X)\to \gl(X)/{\sim}$ denotes the quotient map. 
    We claim that $\{\psi_X[\rmC_A]: A\in p\}$ has the finite intersection property. To see this, consider $F\in \fin{p}$, and observe using Lemma~\ref{Lem:NU_Prelims} that 
    \begin{align*}
    \bigcap_{A\in F} \psi_X[\rmC_A] = \{Y\in \gl(X)/{\sim}: \forall A\in F\, \forall \sigma\in L\, \forall c> 0\,\, \rmC_{A(\sigma, c)}\cap Y\neq \emptyset\}
    \end{align*}
    As $p$ has finite near intersections and as $L$ is directed, the right hand side is non-empty.
    
    Now letting $Y\in \bigcap_{A\in p} \psi_X[\rmC_A]$, we have $p\subseteq \{A\in \rc(X): \rmC_A\cap Y\neq \emptyset\}$.
\end{proof}

\begin{corollary}
    $\NU(X) \cong \gl(X)/{\sim}$. Going forward, we write $\psi_X\colon \gl(X)\to \NU(X)$ for the quotient map and equip $\NU(X)$ with the compact quotient topology. 
\end{corollary}

Given $Y\subseteq \rm{Gl}(X)$ a $\sim$-class, we let $\hat{Y}\in \NU(X)$ denote the corresponding near ultrafilter, and given $p\in \NU(X)$, we let $\hat{p}\subseteq \gl(X)$ denote the corresponding $\sim$-class. 
Given $A\in \rc(X)$, we write $\rmC_A^X  = \{p\in \NU(X): A\in \hat p\}$ and $\rmN_A^X = \NU(X)\setminus \rmC_A^X$. Later, when it is clear that we are discussing $\NU(X)$ and not $\gl(X)$, we will often omit the $X$ superscripts; however, for the remainder of this subsection, we will maintain them. We collect some observations about near ultrafilters and the topology on $\NU(X)$.

\begin{lemma}\label{Lem:NUBasic}    
    Let $X$ be an $L$-fattening space. Fix $p\in \NU(X)$, $k< \omega$ and $A, A_i\in \rc(X)$ $(i < k)$.
    \begin{enumerate}[label=\normalfont(\arabic*)]
        
        \item\label{Item:Not_Lem:NUBasic} 
        If $A\not\in p$, then there are $\sigma\in L$ and $c> 0$ with $A(\sigma, c)\not\in p$.
        \item\label{Item:PR_Lem:NUBasic}
        If $A\in p$ and $A = \bigvee_{i< k} A_i$, then there is $i < k$ with $A_i\in p$.
        \item\label{Item:Filter_Lem:NUBasic}
        If $A, A_0,..., A_{k-1}\in p$, $\sigma\in L$, and $c> 0$, then $A\wedge \bigwedge_{i< k} A_i(\sigma, c)\in p$ and $\bigwedge_{i< k} A_i(\sigma, c)\in [p]_{\rm{fat}}$.
        \item 
        $\{\rmN_B^X: B\in \rc(X)\}$ is a base for the  topology on $\NU(X)$.
        \item \label{Item:Base_Lem:NUBasic}
        For $p\in \NU(X)$, a base of not-necessarily-open neighborhoods of $p$ is given by $\{\rmC_B^X: B\in [p]_{\rm{fat}}\}$.
        \item \label{Item:Fiber_Lem:NUBasic}
        $[p]_{\rm{fat}} = \{B\in \rc(X): \hat{p}\subseteq \rmC_B\}$.
    \end{enumerate}
\end{lemma}

\begin{proof}
    Immediate; the first and third items use Lemma~\ref{Lem:NU_Prelims}.
\end{proof}


The next proposition discusses what members of $\rc(\NU(X))$ look like. 
The following definition will be helpful: Given an $L$-fattening space $X$ and $A\in \rc(X)$, the \emph{strong $L$-complement} of $A$ is 
$$A^*:= \bigvee \{B\in \rc(X): A\perp B\} = \neg \bigwedge_{\sigma, c} A(\sigma, c).$$
The collection of regular closed sets of the form $\bigvee_{A\in S} A^*$ for some  $S\subseteq \rc(X)$, is denoted by $[\rc(X)]_{\rm{thin}}$.

\begin{prop}
    \label{Prop:NU_Fattening}
    Fix an $L$-fattening space $X$ and $A \in \rm{rc}(X)$. 
    \begin{enumerate}[label=\normalfont(\arabic*)]
         \item\label{Item:Image_Prop:NUF} 
        $\psi_X[\rmC_A] = \rmC_A^X$
        \item\label{Item:Joins_Prop:NUF} 
        If $\{A_i: i\in I\}\subseteq \rc(X)$ and $\bigvee_{i\in I} A_i = A$, then $\bigvee_{i\in I} \rmC_{A_i}^X = \rmC_A^X$. 
        \item\label{Item:Open_Prop:NUF}
        $(\psi_X)^{-1}[\rmN_A^X] = \bigcup \{\rmC_B: A\perp B\}$, and $\ol{\rmN_A^X} = \rmC_{A^*}^X$. 
        \item\label{Item:NU_RC:NUF}
        $\rc(\NU(X)) = \{\rmC_{B}^X: B \in [\rc(X)]_{\rm{thin}}\}$
    \end{enumerate}
\end{prop}

\begin{proof}
    \ref{Item:Image_Prop:NUF}:
    Clearly $\psi_X[\rmC_A]\subseteq \rmC_A^X$. Conversely, given $p\in \rmC_A^X$, then $\{A\}\cup [p]_{\rm{fat}}$ has the finite intersection property, so can be extended to $q\in \gl(X)$. 
    \vspace{3 mm}
    
    \noindent
    \ref{Item:Joins_Prop:NUF}:
    $\rmC_A^X = \psi_X[\rmC_A] = \psi_X\left[\ol{\bigcup_{i\in I} \rmC_{A_i}}\right]= \ol{\psi_X\left[\bigcup_{i\in I} \rmC_{A_i}\right]} = \ol{\bigcup_{i\in I} \rmC_{A_i}^X} = \bigvee_{i\in I} \rmC_{A_i}^X$.
    \vspace{3 mm}
    
    \noindent
    \ref{Item:Open_Prop:NUF}:
    The containment $(\psi_X)^{-1}[\rmN_A^X] \supseteq  \bigcup \{\rmC_B: A\perp B\}$ is clear. Conversely, if $q \in \gl(X)$ and $A\not\in \psi_X(q)$, then for some $\sigma\in L$ and $c> 0$, we have $\neg A(\sigma, c)\in \psi_X(q)$, implying $\neg A(\sigma, c/2)\in q$. By compactness, this implies $\ol{\rmN_A^X} = \rmC_{A^*}^X$.
    \vspace{3 mm}
    
    \noindent
    \ref{Item:NU_RC:NUF} follows from the previous items. 
\end{proof}

\begin{definition}
    We endow $\NU(X)$ with the $L$-fattening structure given by letting $\rmC_B^X(\sigma, c) := \rmC_{B(\sigma, c)}^X$ for each $B\in [\rc(X)]_{\rm{thin}}$, $\sigma\in L, c>0$.
\end{definition}

\begin{proposition}
    Fix and $L$-fattening space $X$ and consider $\NU(X)$ endowed with the $L$-fattening structure from above. Then for each $\sigma\in L$, 
    $$\partial_\sigma^{\NU(X)} = \partial_\sigma^{\psi_X},$$
    where $\partial_\sigma^{\psi_X}$ is the quotient pseudometric inherited from $\gl(X)$. In particular, $(\NU(X), (\partial_\sigma^{\NU(X)})_\sigma)$ is an adequate $L$-fattening space.
\end{proposition}
\begin{proof}
    Given $p, q\in \NU(X)$ and $c\geq 0$, Lemmas~\ref{Lem:FD_Compacts} and \ref{Lem:NUBasic} items \ref{Item:Base_Lem:NUBasic} and \ref{Item:Fiber_Lem:NUBasic} yield that
\begin{align*}
    \partial_\sigma^{\NU(X)}(p, q) \leq c &\Leftrightarrow \forall A\in p\,\forall B\in q\, \forall \epsilon > 0\,\, A(\sigma, c+\epsilon)\wedge B\neq \emptyset \\ &\Leftrightarrow \partial_\sigma^{\psi_X}(p, q)\leq c. \qedhere
\end{align*}

\end{proof}

Note that $\NU(X) \cong \varprojlim \NU(X)/\partial_\sigma$, i.e.\ for $p\neq q\in \NU(X)$, there is $\sigma\in L$ with $\partial_\sigma(p, q)> 0$. 

The last goal of this subsection is to show that the near ultrafilter construction is idempotent, i.e.\ that $\NU(X)$ and $\NU(\NU(X))$ are canonically isomorphic (Corollary~\ref{Cor:NU_Idempotent}). First, using Proposition~\ref{Prop:NU_Fattening}, we can show that $\NU(X)$ is $L$-ED.

\begin{prop}
\label{Prop:NU_Fat_props}
    For any $L$-fattening space $X$, the $L$-fattening space $\NU(X)$ is $L$-ED and $L$-continuous.
\end{prop}

\begin{proof}
    To check $L$-ED, fix $B\in [\rc(X)]_{\rm{thin}}$. Fix $\sigma\in L$ and $c> 0$. Now, $\rmC_B^X(\sigma, c) = \rmC_{B(\sigma, c)}^X$, and $\rmC_B^X\subseteq \rmN_{\neg B(\sigma, c)}^X\subseteq \Int(\rmC_{B(\sigma, c)}^X)$. 

     To check continuity, fix $A\in \rc(X)$ and $p\in \NU(X)$ with $p\in \rmN_A^X$. Find $\sigma\in L$ and $c> 0$ with $p\in \rmN_{A(\sigma, c)}^X$. Write $B =  A(\sigma, 2c/3)$. Then $p\in \rmN_B^X$, and using Proposition~\ref{Prop:NU_Fattening}\ref{Item:Open_Prop:NUF} we have $\ol{\rmN_B^X}(\sigma, c/3) = \rmC_{B^*(\sigma, c/3)}^X\subseteq \rmN_A^X$.
\end{proof}

We remind the reader of the conventions on binary relations from Subsection~\ref{Subsection:Notation}. We define
\begin{align*}
    \pi_X &:= \{(p, x)\in \NU(X)\times X: \rc(x, X)\subseteq p\},\\
    \iota_X &:= \pi_X^{-1}.
\end{align*}

We have that $\dom(\pi_X)\subseteq \NU(X)$ is dense, $\im(\pi_X) = X$, and $\pi_X^{-1}[x]$ is closed for every $x\in X$. Given $A\in \rc(X)$, we have $\rmC_A^X\cap \pi_X^{-1}[x]\neq \emptyset$ iff $x\in A(\sigma, c)$ for every $\sigma\in L$ and $c> 0$, and we have $\pi_X^{-1}[x]\subseteq \rmC_A^X$ iff $A(\sigma, c)\in \rc(x, X)$ for every $\sigma\in L$ and $c> 0$.

Any topological space can be seen as a continuous $L$-fattening space by just letting $A(\sigma, c) = A$ for any $A \in \rc(X)$, in which case $\NU(X) = \gl(X)$. The following proposition generalizes a well-known fact in general topology.
\begin{prop}\label{Prop:Fattening_Space_Props}
    Fix an $L$-fattening space $X$.
    \begin{enumerate}[label=\normalfont(\arabic*)]
        \item \label{Item:OC_Prop:FSP}
        $X$ is $L$-ED iff $\iota_X\colon X\to \NU(X)$ is a function, in which case $\iota_X$ is continuous.
        \item \label{Item:Sep_Prop:FSP} 
        If $X$ is $L$-continuous, then $\pi_X\colon \dom(\pi_X)\to X$ is a continuous function. 
        \item \label{Item:Compact_Prop:FSP} 
        If $X$ is compact, then $\dom(\pi_X) = \NU(X)$.
    \end{enumerate}
\end{prop}

\begin{proof}
    \ref{Item:OC_Prop:FSP}: Assume $X$ is $L$-ED. Fix $x\in X$, and suppose $p, q\in \NU(X)$ both satisfied $(x, p), (x, q)\in \iota_X$. Observe that for every $A\in p$, $A_0\in \op(x, X)$, $\sigma\in L$, and $c> 0$ we have $\ol{A_0}\cap A(\sigma, c)\in p$, so non-empty. As $X$ is regular, this implies that $x\in A(\sigma, c)$. Using the triangle inequality and $L$-ED, we have $x\in \rm{Int}(A(\sigma, c))$. Similarly $x\in \rm{Int}(B(\sigma, c))$ for $B\in q$. But this implies $A\parallel B$ for every $A\in p$ and $B\in q$, implying $p = q$. To see $\iota_X$ is continuous, consider $\rmN_A^X$ for some $A\in \rc(X)$. Suppose $x\in X$ satisfies $\iota_X(x):= p\in \rmN_A^X$. Then for some $\sigma\in L$ and $c > 0$, also $p\in \rmN_{A(\sigma, c)}^X$. Using $L$-ED, find $B\in \op(x, X)$ satisfying $B\cap A(\sigma, c/2) = \emptyset$. It follows for every $y\in B$ that $\iota_X(y)\in \rmN_A^X$.  

    Conversely, suppose $X$ is not $L$-ED, witnessed by $A\in \rc(X)$, $\sigma\in L$, and $c> 0$, and $x\in A\setminus \Int(A(\sigma, c))$. Then $\{\neg A(\sigma, c)\}\cup \rc(x, X)$ and $\{A\}\cup \rc(x, X)$ can be extended to $p\neq q\in \NU(X)$ with $(x, p)$, $(x, q)\in \iota_X$. 
    \vspace{3 mm}
   
    \noindent
    \ref{Item:Sep_Prop:FSP}: Let $(p_i)_{i\in I}$ be a net from $\dom(\pi_X)$ with $p_i\to p\in \dom(\pi_X)$. Write $x_i, x$ for $\pi_X(p_i)$ and $\pi_X(p)$, respectively. Given $A\in \op(x, X)$, use the continuity of $X$ to find $B\in \rc(x, X)$, $\sigma\in L$, and $c> 0$ with $B(\sigma, c)\subseteq A$. As $B\in p$, $\rmC_{B(\sigma, c)}^X$ is a neighborhood of $p$, so eventually $p_i\in \rmC_{B(\sigma, c)}^X$. For such 
    $i\in I$, it follows using the continuity of $X$ that $x_i\in B(\sigma, c)\subseteq A$. As $A$ was arbitrary, we have $x_i\to x$ as desired. 
    \vspace{3 mm}

    \noindent
    \ref{Item:Compact_Prop:FSP}: Clear.
\end{proof}

\begin{prop}
    \label{Prop:Distance_X_NUX}
    Given a fattening space $X$ and $x, y\in X$, we have  $\partial_\sigma^X(x, y) = \partial_\sigma^{\NU(X)}(\pi_X^{-1}[x], \pi_X^{-1}[y])$.
\end{prop}

\begin{proof}
    By Lemma~\ref{Lem:FD_Compacts}, and since $p\in \rmC_A^X$ implies $p\in \Int(\rmC_{A(\sigma, \epsilon)}^X)$, the following holds: Given $c> 0$, we have  $\partial_\sigma^{\NU(X)}(\pi_X^{-1}[x], \pi_X^{-1}[y])\leq c$ iff for every $\epsilon > 0$ and every $A\in \rc(X)$ with $\pi_X^{-1}[x]\subseteq \rmC_A^X$, we have $\rmC_{A(\sigma, c+\epsilon)}^X\cap \pi_X^{-1}[y]\neq \emptyset$. This happens iff for every $\epsilon > 0$ and $A\in \rc(x, X)$, we have $y\in A(\sigma, c+\epsilon)$, i.e.\ iff $\partial_\sigma^X(x, y)\leq c$. 
\end{proof}

\begin{corollary}
    \label{Cor:NU_Idempotent}
    Suppose $X$ is an $L$-continuous $L$-ED fattening space. Then $\iota_X\colon X\to \NU(X)$ is a continuous embedding of fattening spaces. 
    If moreover $X$ is compact, then $\iota_X$ is an isomorphism.
    In particular $\NU(Y)\cong \NU(\NU(Y))$, for any $L$-fattening space $Y$.  
\end{corollary}

\begin{proof}
    Propositions~\ref{Prop:NU_Fat_props},  \ref{Prop:Fattening_Space_Props}, and \ref{Prop:Distance_X_NUX}.
\end{proof}

\subsection{More on the $L$-uniformity}
\label{Subsection:L-uniformity}

The following definitions and results generalize results from \cite{ZucMHP} and \cite{BassoZucker} about $G$-flows to the general context of fattening spaces. The proof of Theorem~\ref{thm:ED-first-countable} is a mild generalization of the argument of one of the key results in \cite{BYMT}. Nonetheless, it has profound implications for the underlying topology of certain $G$-flows, in particular when $G$ is Polish (see \Cref{cor:Polish-dichotomy} and \Cref{Section:Polish}). 

\begin{defin}
    \label{Def:Topo_Uniform}
     Given an $L$-fattening space $X$, $\sigma \in L$, and $c > 0$, we define the following binary relations on $X$, where given $x\in X$, we have:
     \begin{itemize}
         \item 
        $Q_{\sigma, c}^X[x] = \bigcap\{A(\sigma, c): A\in \rc(X)\text{ with }x\in A\}$
        \item 
         $R_{\sigma, c}^X[x] = \bigcap\{A(\sigma, c): A\in \rc(x, X)\}$.
      \end{itemize}
     In any of the notation, when $X$ is understood, we may omit it.
\end{defin}

We always have $Q_{\sigma, c}\subseteq R_{\sigma, c}$ and $R_{\sigma, c} = (R_{\sigma, c})^{-1}$.  Note that if $x, y\in X$ and $c\geq 0$, we have
$$\partial_\sigma(x, y)\leq c\Leftrightarrow \forall \epsilon > 0\,\, (x, y)\in R_{\sigma, c+\epsilon}.$$

When $X$ is $L$-ED and  $\epsilon > 0$, we have $R_{\sigma, c}\subseteq Q_{\sigma, c+\epsilon}$. 
In particular, if $X$ is a compact $L$-ED space, then the collections $\{Q_{\sigma, c}: \sigma \in L, \; c >0\}$ and $\{R_{\sigma, c}: \sigma \in L, \; c >0\}$ both generate the  $L$-uniformity on $X$.

\begin{defin}
    \label{Def:L_Isolated}
    For $X$ a $L$-fattening space, we say that $x\in X$ is \emph{$L$-isolated} if $x \in \Int_X(Q_{\sigma, c}[x])$ for all $\sigma \in L$ and $c > 0$. When $X$ is $L$-ED, this happens iff $x\in \Int_X(R_{\sigma, c}[x])$ for all $\sigma \in L$ and $c > 0$ iff the $L$-fattening topology and the ordinary topology coincide at $x$. 

    If each $x \in X$ is $L$-isolated, we say that $X$ is \emph{$L$-discrete}. If $X$ is a $L$-ED space, this is equivalent to the compatibility of the $L$-fattening uniformity. 
\end{defin}

Notice that when the fattening is trivial, such points are exactly the isolated points of $X$, hence the choice of terminology.

\Cref{thm:ED-first-countable} generalizes the topological fact that in compact ED spaces, there are no non-trivial convergent sequences, so that the only points of first-countability are isolated. In the proof, we use the following fact, whose proof can be directly adapted from \cite{BassoZucker}*{Corollary~5.7}).

\begin{fact}
\label{Fact:BetaN}
    Fix a compact $L$-ED fattening space.
    If  $\{A_n: n< \omega\}\subseteq \op(X)$ are such that there are $\sigma \in L, c>0$ such that $\ol{A_n}(\sigma, c) \wedge \ol{A_m}(\sigma, c) = \emptyset$ for all $n < m$, then for any $(x_n)_{n< \omega}$ with  $x_n\in A_n$, the continuous map $\phi\colon \beta\omega\to X$ satisfying $\phi(n) = x_n$ for every $n< \omega$ is injective.
\end{fact}

 \begin{theorem}[Compare to Theorem~2.3 of \cite{BYMT}]
    \label{thm:ED-first-countable}
    Let $X$ be an $L$-ED fattening space, and $x \in X$ a point of first countability. Then $x$ is $L$-isolated. In particular, first-countable $L$-ED fattening spaces are $L$-discrete.  
\end{theorem}
    \begin{proof}
    Suppose $x$ is not $L$-isolated. Then there is a net converging to $x$ which is not convergent in the $L$-uniformity. As $x$ is a point of first countability, this net can actually be chosen to be a sequence $(x_n)_{n< \omega}$. We may also assume that $x\not\in \{x_n: n< \omega\}$. Since the $L$-uniformity is complete, the sequence $(x_n)_{n<\omega}$ is not Cauchy, so there are $\sigma \in L, \delta >0$ such that $\partial_\sigma(x_n, x_m) > \delta$ for all $n \neq m$. Mimicking the argument from \cite{BYMT}*{Lemma~2.5}, one can find a subsequence $(x_{n_k})_{k< \omega}$, and sets $A_k\in \op(x_{n_k}, X)$ such that $\partial_\sigma(A_k, A_\ell) \ge \delta/2$ for all $k \neq \ell$. By \eqref{Eq:Basic_Fat_partial} and \Cref{Fact:BetaN}, the map $\phi\colon \beta \omega\to X$ which continuously extends the map $k\to x_{n_k}$ is injective, and $x\in \phi[\beta\omega\setminus \omega]$. As points in $\beta\omega\setminus \omega$ are never points of first countability, neither is $x\in X$. 
    \end{proof}

\subsection{Uniform spaces as fattening spaces}
\label{Subsection:Unif_Fat}

We end the section by considering uniform spaces. Given a uniform space $(X, \cU)$, just denoted $X$ when $\cU$ is understood, we can view $X$ as a fattening space as follows. We let $L$ be the collection of bounded, uniformly continuous pseudometrics on $X$ equipped with the pointwise order. Every $\sigma\in L$ is trivially lsc, adequate, and open-compatible. In particular, $(X, (\sigma)_{\sigma\in L})$ becomes an $L$-ED adequate fattening space as in Definition~\ref{Def:AdFat}. Since the members of $L$ generate the uniformity on $X$, this fattening is continuous, and thus $\iota_X\colon X\to \NU(X)$ is a continuous embedding onto its image, by Corollary~\ref{Cor:NU_Idempotent}. We will therefore identify $X$ with $\im(\iota_X)$.

\begin{fact}[\cite{KocakStrauss}]
    For $X$ a uniform space, $\iota_X\colon X\to \NU(X)$ is the Samuel compactification of $X$. 
\end{fact}

Going forward, we identify $\sa(X)$ and $\NU(X)$ when $X$ is a uniform space. In the metric case, we can detect when a topometric space is the Samuel compactification of a metric space as follows.

\begin{prop}
    \label{Prop:Topo_Converse}
    Let $Y$ be compact and $\partial$ be an lsc, adequate, open-compatible metric on $Y$. Suppose $X\subseteq Y$ is dense  and every $x\in X$ is $\partial$-compatible. Then, writing $d = \partial|_{X\times X}$, we have $(Y, \partial)\cong (\sa(X, d), \partial_d)$. 
\end{prop}

\begin{proof}
    Since $(Y, \partial)$ is adequate and since $X$ consists of compatibility points, we have for any $A\in \op(Y)$ and $c > 0$ that 
    $$\rmB_{\partial}(A, c)\cap X = \rmB_{\partial}(A\cap X, c)\cap X = \rmB_d(A\cap X, c).$$ 
    Given $y\in Y$, let $\cF_y = \{A\in \op(X): y\in \ol{A}\}$. We claim that $\cF_y\in \sa(X)$, i.e.\ is a near ultrafilter. It suffices to show that $\cF_y$ has the near FIP. First, given $c> 0$ and $A\in \cF_y$, the above adequacy observation gives us $\ol{\rmB_d(A, c)} = \ol{\rmB_\partial(\Int(\ol{A}), c)}$. As $A\in \op(X)$, we have $\ol{A}\subseteq Y$ regular closed. Then since $\partial$ is open-compatible on $Y$ and $y\in \ol{A} = \ol{\Int(\ol{A})}$, we have $y\in \Int(\ol{\rmB_d(A, c)})$. In particular, for any $Q\in \fin{\cF_y}$, we have $\bigcap_{A\in Q} \rmB_d(A, c)\neq \emptyset$.

    Since $\partial$ is lsc, the inclusion $(X, d)\hookrightarrow Y$ is uniformly continuous, where here $Y$ has its compact uniformity. Let $\phi\colon \sa(X)\to Y$. Given $p\in \sa(X)$, we claim $\cF_{\phi(p)} = p$. Let $A\in p$. Then $p\in \rmC_A$, so $\phi(p)\in \ol{A}$. In particular, $\phi$ is a homeomorphism. 
    
    We now show that $\phi$ is a $(\partial_d, \partial)$-isometry. We first observe that on $\sa(X)$, $\partial_d$ is the largest lsc metric with $\partial_d|_X = d$. Hence it suffices to show that on $Y$, $\partial$ is the largest lsc metric with $\partial|_X = d$. Fix $y_0, y_1\in Y$, and suppose $\partial(y_0, y_1) = c$. Given $A_i\in \op(y_i, Y)$ and $\epsilon > 0$, we need to show
    $$d(A_0\cap X, A_1\cap X)< c+\epsilon.$$
    If this is not true, then 
    \begin{align*}
        \rmB_d(A_0\cap X, c+\epsilon)\cap A_1\cap X &= \emptyset\\
        \Rightarrow \rmB_\partial(A_0, c+\epsilon)\cap A_1\cap X &= \emptyset.
    \end{align*}
    However, $\rmB_\partial(A_0, c+\epsilon)\cap A_1\neq \emptyset$ as it contains $y_1$, and it is open by adequacy. The density of $X\subseteq Y$ then gives a contradiciton.
\end{proof}

\section{$G$-extremally disconnected flows}
\label{Section:Gleason}

This section, using the presentation of fattening spaces from the previous section, presents some background material drawn from \cite{ZucDirectGPP}, \cite{ZucMHP}, and \cite{LeBoudecTsankov}. We consider 
$G$-ED flows and define the \emph{Gleason completion} of a $G$-flow $X$.
Using this, we define a particular instance of the \emph{$G$-ultracoproduct} construction. We also take the opportunity to give a characterization of closed subspaces and subflows of Gleason completions in terms of \emph{near filters} on $\op(X)$, generalizing some results from \cite{ZucThesis} about $\sa(G)$. 

We remark that many of the properties and results that only in this and the following sections could be stated in the more general context of $L$-fattening spaces.

\subsection{The Gleason completion}

Recall that we can endow any $G$-space $X$ with a continuous $\sn(G)$-fattening system, as in \Cref{def:SNG-fattening}. 
\emph{Abusing notation, we call this a $G$-fattening system, and write $G$-ED instead of $\sn(G)$-ED, $G$-isolated instead of $\sn(G)$-isolated, and so on.}
When discussing these properties, we can omit the second $G$ and just say \emph{$G$-ED space} or \emph{$G$-ED flow} when referring to such $G$-spaces and $G$-flows, respectively. 
The $G$-fattening uniformity on a $G$-ED space is called the \emph{UEB uniformity} (see \cite{BassoZucker}), and the topology it induces on $X$ is the \emph{UEB topology}.
Examples of $G$-ED flows are the universal minimal flow $\mg$ and the Samuel compactification $\sa(G)$.
Given $X$ a $G$-ED flow and $U \in \cN_G$, we write 
\begin{itemize}
         \item 
        $Q_U^X[x] = \bigcap\{\ol{UA}: A\in \op(X)\text{ with }x\in \ol{A}\}$,
        \item 
         $R_{\sigma, c}^X[x] = \bigcap\{\ol{UA}: A\in \op(x, X)\}$.
      \end{itemize}
Note that when $\sigma\in \sn(G)$, $c> 0$, and $U = \rmB_\sigma(c)$, then the above definition coincides with $Q_{\sigma, c}^X$ and $R_{\sigma, c}^X$ as in Definition~\ref{Def:Topo_Uniform}. As usual, we omit the $X$ superscript when understood.

\begin{remark}
    In \cite{LeBoudecTsankov} and \cite{ZucUlts},  $G$-ED flows are called \emph{Gleason complete}, and in \cite{ZucMHP}, they are called \emph{maximally highly proximal} (MHP). See the discussion around Definition~\ref{Def:HP} regarding the changes in terminology.
\end{remark}

Let us immediately note that $G$-maps respect the fattening structure in the following sense.

\begin{prop}
    \label{Prop:Level_Maps}
    If $X$ and $Y$ are $G$-spaces and $\phi\colon X\to Y$ is a $G$-map, then for each $\sigma\in \rm{SN}(G)$, $\phi$ is $(\partial_\sigma^X, \partial_\sigma^Y)$-Lipschitz. 
\end{prop}

\begin{proof}
    Immediate from Definition~\ref{Def:Fattening_SPM}.
\end{proof}


When $X$ is a $G$-space, then $G$ acts on the $G$-ED flow $\NU(X)$ in the obvious way, this action is continuous, and $\pi_X\subseteq \NU(X)\times X$ is $G$-invariant. Furthermore, the fattening structure, as well as the fattening uniformity, on $\NU(X)$ coming from $X$ are exactly the same as the one arising from $\NU(X)$ as a $G$-flow. Typically for $G$-spaces, we write $\rmS_G(X)$ for $\NU(X)$ and call it the \emph{Gleason completion}\footnote{While the name \emph{$G$-Gleason cover} might be more appropriate, we use \emph{Gleason completion} to agree with \cite{LeBoudecTsankov}.} of $X$. In the references \cite{ZucDirectGPP, ZucMHP, BassoZucker}, members of $\rmS_G(X)$ are viewed as subsets of $\op(X)$ rather than as subsets of $\rc(X)$, and \textbf{we also do this here} (see the remark following \Cref{Def:AdFat}). We remark that when $X = G$, we have $\rmS_G(G)\cong \sa(G)$. 

We discuss the universal property of $\pi_X$ for two key types of $G$-spaces. First, we discuss the case when $X$ is a $G$-flow. Then $\pi_X\colon \rmS_G(X)\to X$ is a $G$-map. 

\begin{defin}
    \label{Def:Irred_Map}
    If $X$ and $Y$ are spaces, a continuous map $\pi\colon Y\to X$ is \emph{irreducible} if $\pi[Y]\subseteq X$ is dense and for every $B\in \op(Y)$, $\Int(\pi_{\rm{fib}}(B))\neq \emptyset$. 
\end{defin}

Let us pause to make a useful observation about fiber images. 

\begin{lemma}
    \label{Lem:Fiber_Image_Open}
    Given a space $X$, a compact space $Y$, and $R\in \rmF(Y\times X)$, then if $B\in \op(Y)$, we have $R_{\rm{fib}}(B)\in \op(X)\cup \{\emptyset\}$.
\end{lemma}

\begin{proof}
    If $(x_i)_{i\in I}$ is a net from $X$ with each $x_i\not\in R_{\rm{fib}}(B)$ and $x_i\to x\in X$, find $y_i\in Y\setminus B$ with $(y_i, x_i)\in R$. Passing to a subnet, let $y_i\to y\in Y\setminus B$. As $R\in \rmF(Y\times X)$, we have $(y, x)\in R$, showing that $x\not\in R_{\rm{fib}}(B)$.
\end{proof}

This definition of irreducible is somewhat more general than the one given in \cite{LeBoudecTsankov}, which discusses the case that both $X$ and $Y$ are compact. If $X$ and $Y$ are both compact, then by Lemma~\ref{Lem:Fiber_Image_Open}, the definition of $\pi$ being irreducible simplifies to demanding that $\pi$ is onto and that $\pi_{\rm{fib}}(B)\neq \emptyset$ for every $B\in \op(Y)$. A sufficient condition for a continuous onto map $\pi \colon Y \to X$ between compact spaces to be irreducible is that is is \emph{weakly almost 1-to-1} \cite{MR3752609}, that is, that the set $\set{y \in Y : f^{-1}[f(y)] = \set{y}}$ is dense in $Y$; if $Y$ is metrizable, then the condition is also necessary.
We justify the increased generality of our definition by pointing out a useful example: If $Y\subseteq X$ are topological spaces with $Y\subseteq X$ dense and $\pi\colon Y\hookrightarrow X$ is the inclusion map, then $\pi$ is irreducible.

\begin{fact}[\cite{ZucDirectGPP}]
	\label{Thm:Universal_Irreducible_Extension}
	Let $X$ be a $G$-flow. Then $\pi_X\colon \rmS_G(X)\to X$ is the \emph{universal irreducible extension} of $X$: $\pi_X$ is irreducible, and if $Y$ is a $G$-flow and $\theta\colon Y\to X$ is an irreducible extension, then there is a $G$-map $\phi\colon \rmS_G(X)\to Y$ such that $\pi_X = \theta\circ \phi$.
\end{fact}

Notice that automatically, the map $\phi\colon \rmS_G(X)\to Y$ from \Cref{Thm:Universal_Irreducible_Extension} is an irreducible extension. 

The other situation we discuss is when $X$ is a $G$-ED space. Then $\iota_X\colon X\to \rmS_G(X)$ is a $G$-map and embedding, and we have the following.
\begin{fact}[\cite{ZucThesis}]
    \label{Fact:Gleason_Max_Comp}
	For $X$ a $G$-ED space, $\iota_X\colon X\to \rmS_G(X)$ is the \emph{universal $G$-equivariant compactification} of $X$. Namely, given any $G$-map $\phi\colon X\to Y$ with $Y$ a $G$-flow, there is $\wt{\phi}\colon \rmS_G(X)\to Y$ with $\phi = \wt{\phi}\circ \iota_G(X)$.
\end{fact}

Note that $X$ is a $G$-ED flow exactly when $\pi_X\colon \rmS_G(X)\to X$ is an isomorphism.

When $G$ is discrete, the property is purely topological: a $G$-space $X$ is $G$-ED exactly when the space $X$ is ED, and $\rmS_G(X)$ is just the Gleason cover of $X$. Many theorems about compact ED spaces admit dynamical versions when appropriately formulated. For instance, recall that a continuous map $\phi\colon X\to Y$ between topological spaces is called \emph{pseudo-open} if whenever $A\in \op(X)$, we have $\Int_Y(\phi[A])\neq \emptyset$. Also recall that any $G$-map between minimal flows is pseudo-open. 

\begin{prop} 
    \label{Prop:MHP_Open} 
    Let $G$ be a topological group and $\phi\colon X\to Y$ a pseudo-open $G$-map between $G$-flows with $Y$ $G$-ED. Then $\phi$ is open.
\end{prop}

\begin{remark}
    When $G$ is trivial, Proposition~\ref{Prop:MHP_Open} recovers the standard topological fact that any continuous pseudo-open map from a compact space to a compact ED space is open.
\end{remark}

\begin{proof}
    Fix $B\in \op(X)$ and $x\in B$, towards showing that $\phi(x)\in \rm{Int}(\phi[B])$. Fix $U\in \cN_G$ and $A\in \op(x, B)$ with $\ol{UA}\subseteq B$. Note that $\phi[UA] = U{\cdot}\phi[A]$ and $\phi\Big[\ol{UA}\Big] = \ol{\phi[UA]}$. Because $\phi$ is pseudo-open, we have $\ol{\rm{Int}(\phi[A])} \supseteq \phi[A]$. As $\phi(x)\in \phi[A]$ and $Y$ is $G$-ED, we have 
    \begin{align*}
        \phi(x) &\in \rm{Int}\Big(\ol{U{\cdot}\rm{Int}(\phi[A])}\Big) = \rm{Int}\Big(\ol{U{\cdot}\phi[A]}\Big)
        = \rm{Int}\Big(\phi\Big[\ol{UA}\Big]\Big)
        \subseteq \rm{Int}(\phi[B]). \qedhere
    \end{align*}
\end{proof}

\subsection{Ultracoproducts and highly proximal maps}
\label{Subsection:Ultracoproducts}

Using Fact~\ref{Fact:Gleason_Max_Comp}, we can discuss a simple case of the \emph{$G$-ultracoproduct} construction. While the preprint \cite{ZucUlts} defines this more generally, we focus here on the case of an ultraco\emph{power} of a $G$-ED flow $X$.

Let $I$ be a set. Form the $G$-space $I\times X$ given by $g\cdot (i, x) = (i, gx)$. Then $I\times X$ is $G$-ED, and $\rmS_G(I\times X)$ is the universal $G$-equivariant compactification of $I\times X$. The projection maps from $I\times X$ to $I$ and $X$ both continuously extend to  $G$-maps $\pi_I\colon \rmS_G(I\times X)\to \beta I$ and $\pi_X\colon \rmS_G(I\times X)\to X$ (the notation for the projection maps should arguably have more decoration, but we will use this when the context is clear). Given an ultrafilter $\cU\in \beta I$, we define the \emph{$G$-ultracopower} of $X$ along $\cU$ to be the $G$-flow $\Sigma_\cU^GX:= \pi_I^{-1}(\{\cU\})$. More combinatorially, given $p\in \rmS_G(I\times X)$, we have $p\in \Sigma_\cU^G X$ iff for every $S\in \cU$, we have $S\times X\in p$. We write $\pi_{X, \cU}\colon \Sigma_\cU^GX\to X$ for the \emph{ultracopower map}, the restriction of $\pi_X$ to $\Sigma_\cU^GX\subseteq \rmS_G(I\times X)$. 

Observe that in $\rm{Sub}_G(\rmS_G(I\times X))$, we have
$$\lim_{i\to \cU} (\{i\}\times X) = \Sigma_\cU^GX.$$
The $\subseteq$ direction follows from  the continuity of $\pi_I$, and the $\supseteq$ direction follows from the density of $I\times X\subseteq \rmS_G(I\times X)$. Thus it is informative to think about ultracopowers as a universal instance of Vietoris convergence of subflows in the following sense. Suppose $Z$ is a flow and $\la Y_i: i\in I\ra\subseteq \rm{Sub}_G(Z)$ is a tuple of subflows with $\displaystyle\lim_{i\to \cU} Y_i = Y\in \rm{Sub}_G(Z)$, then if for every $i\in I$ we have a $G$-map $\phi_i\colon X\to Y_i$, we obtain a $G$-map $\phi_\cU\colon \Sigma_\cU^GX\to Y$ by forming the $G$-map $\phi\colon I\times X\to Z$ given by $\phi(i, x) = \phi_i(x)$, continuously extending to $\alpha_G(I\times X)$, and restricting to $\Sigma_\cU^G X$. We will use this fact repeatedly.

We now turn to a brief discussion of terminology, in particular comparing the notions of irreducible and \emph{highly proximal} $G$-maps. One of our main theorems, Theorem~\ref{Thm:Rosendal_Minimal}, provides several equivalent characterizations of when a minimal $G$-ED flow $X$ has the property that for any ultracopower of $X$, the ultracopower map is highly proximal.  

\begin{defin}
    \label{Def:HP}
    Let $\pi\colon Y\to X$ be a factor map of $G$-flows. We call $\pi$ \emph{highly proximal} if for any $x\in X$, there is a net $(g_i)_{i\in I}$ from $G$ such that $g_i \cdot \pi^{-1}[x]$ converges in $\rmK(X)$ to a singleton. 
\end{defin}

When $X$ is minimal and $(Y, \pi)$ is an irreducible extension, then $Y$ is also minimal and $\pi$ is highly proximal. This is the setting originally considered in \cite{Auslander_Glasner}. Conversely, if $Y$ is minimal and $\pi\colon Y\to X$ is highly proximal, then $\pi$ is irreducible. Thus, \emph{among minimal flows}, $\pi_G(X)\colon \rmS_G(X)\to X$ is the universal highly proximal extension, and \emph{among minimal flows}, $\rmS_G(X)$ is maximally highly proximal, explaining the ``MHP'' terminology used in \cite{ZucMHP}. However, upon considering non-minimal flows, the notions of irreducible and highly proximal extensions become distinct. Indeed, later on we will need to consider non-trivial highly proximal extensions of $\rmM(G)$, and the domain of such an extension can never be minimal. Thus, with the authors of \cite{LeBoudecTsankov}, we have agreed upon the ``Gleason completion'' terminology used both here and in \cite{LeBoudecTsankov}.

We note that highly proximal maps onto minimal flows can be characterized by being ``almost irreducible'' in the following sense.

\begin{prop}
    \label{Prop:HP_Onto_Minimal}
    Suppose $X$ and $Y$ are $G$-flows with $X$ minimal, and let $\pi\colon Y\to X$ be a $G$-map. Then $\pi$ is highly proximal iff both of the following:
    \begin{itemize}
        \item 
        $Y$ contains a unique minimal subflow $Z$,
        \item 
        for every $B\in \op(Y)$ with $B\cap Z\neq \emptyset$, we have that $\pi_{\rm{fib}}(B)\neq \emptyset$.
    \end{itemize}
\end{prop}

\begin{proof}
    Suppose $Z_0, Z_1\subseteq Y$ are distinct minimal flows. For any $x\in X$, we have $\pi^{-1}[x]\cap Z_i = \emptyset$ for each $i< 2$. Hence if $(g_i)_{i\in I}$ is a net from $G$ such that $(\pi^{-1}[g_ix])_{i\in I}$ is Vietoris convergent with limit $K\in \rmK(Y)$, then $K\cap Z_i\neq \emptyset$ for each $i< 2$, implying that $|K|> 1$. Hence $\pi$ cannot be highly proximal.

    For the rest of the proof, assume that $Z\subseteq Y$ is the unique minimal subflow of $Y$. First suppose $\pi$ is highly proximal, and fix $B\in \op(Y)$ with $B\cap Z\neq \emptyset$. To see that $\pi_{\rm{fib}}(B)\neq \emptyset$, fix $x\in X$. Then $\ol{G\cdot \pi^{-1}[x]}\subseteq \rmK(Y)$ has a unique minimal subflow consisting of singletons from $Z$. As $B\cap Z\neq \emptyset$, we have $(G\cdot \pi^{-1}[x])\cap \rm{Subset}(B, Y) \neq \emptyset$. In particular for some $g\in G$, we have $gx\in \pi_{\rm{fib}}(B)$. 

    Now suppose for every $B\in \op(Y)$ with $B\cap Z = \emptyset$ we have $\pi_{\rm{fib}}(B)\neq \emptyset$. Then $\pi_{\rm{fib}}(B)\in \op(X)$. Hence given any $x\in X$, there is $g\in G$ with $gx\in \pi_{\rm{fib}}(B)$. By letting $B$ range over the neighborhoods of some $y\in Y$, we see that  
    $\{y\}\in \ol{G\cdot \pi^{-1}[x]}$. 
\end{proof}

\subsection{Subspaces and subflows of Gleason completions}
\label{Subsection:NUlts}

In this subsection, we explore the structure of $\rmS_G(X)$ in more detail, obtaining some generalizations of results from \cite{ZucThesis}. Throughout this subsection, \textbf{fix a topological group $G$ and a $G$-space $X$}. 

\begin{defin}
    \label{Def:Subsets_of_G_Spaces}
    We say $A\in \op(X)$ is \emph{thick} if for every $F\in \fin{G}$, we have $\bigcap_{g\in F} gA\neq \emptyset$. Equivalently, for every $F\in \fin{G}$, there is $B\in \op(X)$ with $FB\subseteq A$. We say that $A\in \op(X)$ is \emph{pre-thick} if for any $U\in\cN_G$, $UA\subseteq X$ is thick, that is, if $[A]_{\rm{fat}}$ consists of thick sets.
\end{defin}

Note that when $X = G$ and $A\in \op(G)$, then the above definition of thick coincides with the usual definition of a thick subset of a group.

\begin{defin}
    \label{Def:Near_Filters}
    We call $\cF\subseteq \op(X)$ a \emph{near filter} if the following all hold:
    \begin{enumerate}
        \item 
        If $A\in \cF$ and $B\in \op(X)$ satisfies $A\subseteq \ol{B}$, then $B\in \cF$ ($\cF$ is \emph{monotone}).
        \item 
        If $A, B\in \cF$ and $U\in \cN_G$, then $UA\cap UB\in \cF$ ($\cF$ is \emph{closed under finite near intersections}).
        \item 
        If $A\in \op(X)$ and $UA\in \cF$ for every $U\in \cN_G$, then $A\in \cF$ ($\cF$ is \emph{determined by fattenings}).
    \end{enumerate}
    Whenever $\cF\subseteq \op(X)$ is a near filter, write $\rmC_\cF = \bigcap_{A\in \cF} \rmC_A$, and whenever $K\subseteq \rmS_G(X)$ is a closed subspace, write $\cF_K = \{A\in \op(X): K\subseteq \rmC_A\}$. It is routine to check that $\cF_K$ is a near filter.

    We call a near filter \emph{pre-thick} if all of its members are pre-thick.
\end{defin}

\begin{lemma}
    The closed subspaces of $\rmS_G(X)$ are in 1-1 correspondence with near filters on $\op(X)$. 
\end{lemma}

\begin{proof}
    Suppose $\cF\subseteq \op(X)$ is a near filter. We need to show that $\cF_{\rmC_\cF} = \cF$. The $\supseteq$ direction is clear, so towards showing $\subseteq$, suppose $A\in \cF_{\rmC_\cF}$. Then $\rmC_{UA}$ is a neighborhood of $\rmC_\cF$, so as $\cF$ is closed under finite near intersections, we may find $B\in \cF$ with $\rmC_B\subseteq \rmC_{UA}$, implying that $B\subseteq \ol{UA}$. As $\cF$ is monotone, $UA\in \cF$. As $U\in \cN_G$ was arbitrary and $\cF$ is determined by fattenings, we have $A\in \cF$. 
\end{proof}

\begin{lemma}
    \label{Lem:Subflows_of_SGX}
    Given $A\in \op(X)$, then $\rmC_A\subseteq \rmS_G(X)$ contains a subflow iff $A$ is pre-thick. A closed subspace $Y\subseteq \rmS_G(X)$ contains a subflow iff $\cF_Y$ is prethick. Hence minimal subflows of $\rmS_G(X)$ are in 1-1 correspondence with maximal pre-thick near filters.
\end{lemma}

\begin{proof}
    We have that $A$ is pre-thick iff the collection $\{gA: g\in G\}$ has the NFIP iff $\bigcap_{g\in G} C_{gA}\neq \emptyset$, and when non-empty, this is a subflow of $\rmS_G(X)$. 

    If $Y\subseteq \rmS_G(X)$ contains a subflow, then $\cF_Y$ is pre-thick by the above. Conversely, suppose $\cF_Y$ is pre-thick. We will show that $\{gA: A\in \cF_Y\}$ has the NFIP, which will imply that $Y$ contains the subflow $\bigcap_{A\in \cF, g\in G} C_{gA}$. Fix $A_0,..., A_{n-1}
\subseteq \cF_Y$, $g_0,..., g_{n-1}\in G$, and $U\in \cN_G$. Let $V\in \cN_G$ be suitably small (to be determined). Then $\bigcap_{i< n} VA_i\in \cF$, so we have $\bigcap_{i< n} g_i\cdot (V\cdot(\bigcap_{i< n} VA_i))\neq \emptyset$, and $\bigcap_{i< n} g_i\cdot(V\cdot (\bigcap_{i< n} VA_i))\subseteq \bigcap_{i< n} g_iV^2A_i$. Thus if $V$ satisfies $g_iV^2\subseteq Ug_i$ for each $i< n$, we have $\bigcap_{i< n} Ug_iA_i\neq \emptyset$ as desired.  
\end{proof}

\section{Topological groups with tractable minimal dynamics}
\label{Section:TMD}

This section introduces the class of topological groups with \emph{tractable minimal dynamics} (Definition~\ref{Def:Dyn_Tractable}) and provides several equivalent characterizations of this property. Two such characterizations are natural enlargements of the classes $\sf{CAP}$ and $\sf{SCAP}$ from the introduction. Another considers ultracopowers of $\rmM(G)$ and the properties of the ultracopower map. Yet another considers the \emph{Rosendal criterion}, which in the case of a Polish group acting on a Polish space characterizes when the action has a comeager orbit.

\subsection{Concrete minimal dynamics}
\label{Subsection:Concrete}

The classes of (strongly) CAP groups and UEB groups are defined in \cite{BassoZucker}, where one of the main results is that these classes coincide. Here, we take the opportunity to introduce a new, neutral name for this class of topological groups, namely, the groups with \emph{concrete minimal dynamics}. We present some key features of this class of groups, both generalizing the main theorem of \cite{BassoZucker} while also introducing some concepts that will be used in characterizing the broader class of topological groups with \emph{tractable} minimal dynamics.

The following definition is implicit in \cite{ZucDirectGPP}.

\begin{defin}
    \label{Def:Dyn_Bounded}
    A $G$-space $X$ is \emph{$G$-finite} if for any $U\in \cN_G$ and any $\cF\subseteq \op(X)$ with the property that $\{UA: A\in \cF\}$ is pairwise disjoint, then $\cF$ is finite.
\end{defin}

\begin{lemma}
    \label{Lem:Dyn_Bounded}
    Let $X, Y$ be $G$-spaces, and let $\phi\colon X\to Y$ be a $G$-map. Moreover, let $\alpha\colon K \to G$ be a continuous group homomorphism with dense image.  
    \begin{enumerate}[label=\normalfont(\arabic*)]
        \item\label{Item:Factor_Lem:DB} 
        If $\phi[X]\subseteq Y$ is dense and $X$ is $G$-finite, then $Y$ is $G$-finite.
        \item \label{Item:Irred_Lem:DB}
        If $\phi$ is irreducible and $Y$ is $G$-finite, then $X$ is $G$-finite.
        \item \label{Item:Homo_Lem:DB}
        If $X$ is $K$-finite, then $X$ is $G$-finite.
    \end{enumerate}
\end{lemma}

\begin{proof}
    \ref{Item:Factor_Lem:DB}: If $\{A_n: n< \omega\}\subseteq \op(Y)$ and $U\in \cN_G$ are such that $\{UA_n: n< \omega\}$ is pairwise disjoint, then $\{\phi^{-1}[UA_n]: n< \omega\} = \{U\phi^{-1}[A_n]: n< \omega\}\subseteq \op(X)$ is pairwise disjoint.
    \vspace{3 mm}

    \noindent
    \ref{Item:Irred_Lem:DB}: If $\{B_n: n< \omega\}\subseteq \op(X)$ and $U\in \cN_G$ are such that $\{UB_n: n <\omega\}$ is pairwise disjoint, then $\{U\cdot\Int(\phi_{\rm{fib}}(B_n)): n< \omega\}\subseteq \op(Y)$ is pairwise disjoint.
    \vspace{3 mm}

    \noindent
    \ref{Item:Homo_Lem:DB}: If $\{B_n: n< \omega\}\subseteq \op(X)$ and $U\in \cN_G$ are such that $\{UB_n: n <\omega\}$ is pairwise disjoint, then $\{\alpha^{-1}(U)B_n: n <\omega\}$ are pairwise disjoint.
\end{proof}

We next compare $G$-finiteness with the following notion of local topological transitivity implicit in \cite{BYMT}.

\begin{defin}
    \label{Def:U_TT}
    Let $G$ be a topological group and $X$ a $G$-space. Given $U\in \cN_G$ and $A\in \op(X)$, we say that $A$ is \emph{$U$-topologically transitive}, or $U$-TT for short, if whenever $B\in \op(A)$, we have $A\subseteq \ol{UB}$. Note that when $U = G$ and $A = X$, we recover the usual definition of $X$ being a topologically transitive $G$-space.  
\end{defin}

Recall the relations $Q_U^X$ defined in Subsection~\ref{Subsection:L-uniformity}.

\begin{lemma}
    \label{Lem:U_TT}
    Fix a $G$-space $X$. 
    \begin{enumerate}[label=\normalfont(\arabic*)]
        \item \label{Item:UTT_Lem:UTT}
        If $x\in X$, $U\in \cN_G$, and $\Int(Q_U[x])\neq \emptyset$, then $\Int(Q_U[x])$ is $U^2$-TT.
        \item \label{Item:Interior_Lem:UTT}
        If $U\in \cN_G$ and $A \in \op(X)$ is $U$-TT, then for each $x \in A$, we have $A\subseteq \Int(Q_U[x])$.   
    \end{enumerate}
\end{lemma}
    
\begin{proof}
    \ref{Item:UTT_Lem:UTT}: Fix $B\in \op(\rm{Int}(Q_U[x]))$. 
    For any open neighborhood $A$ of $x$, we have $B \sub \ol{UA}$, so in particular $UB \cap A \ne \emptyset$, since $U$ is assumed to be symmetric. It follows that $x\in \ol{UB}$, hence $\rm{Int}(Q_U[x])\subseteq Q_U[x] \sub \ol{U^2B}$, by definition of $Q_U$. Since $B$ was arbitrary, $\rm{Int}(Q_U[x])$ is $U^2$-TT.
    \vspace{3 mm}
        
    \noindent
    \ref{Item:Interior_Lem:UTT}: Fix $x \in A$ and $B \in \op(X)$ with $x \in \ol{B}$. Then $B \cap A$ is nonempty and thus $A \sub \ol{UB}$, since $A$ is $U$-TT. As $B$ was arbitrary, we have $A \sub \Int(Q_U[x])$.
\end{proof}

The next theorem generalizes the main results of \cite{BassoZucker}, which discusses a subset of the following items in the case $X = \mg$. In that case, using the terminology from \cite{BassoZucker}, items \ref{Item:UEB_Thm:Concrete}, \ref{Item:R_Thm:Concrete}, \ref{Item:PsAll_Thm:Concrete} and \ref{Item:PsBase_Thm:Concrete} are ways of saying that $G$ is UEB, item \ref{Item:Ult_Thm:Concrete} can be thought of as saying that $G\in \sf{SCAP}$ \cite{ZucUlts}*{Corollary~6.22}.  
Notice that $\ref{Item:UEB_Thm:Concrete}\Rightarrow  \ref{Item:DB_Thm:Concrete}$ is a dynamical generalization of the fact that compact discrete spaces are finite and remains true without the $G$-ED assumption. 

\begin{theorem}
    \label{Thm:Concrete}
    If $G$ is a topological group and $X$ is a $G$-ED flow, then the following are equivalent.
    \begin{enumerate}[label=\normalfont(\arabic*)]
        \item \label{Item:DB_Thm:Concrete}
        $X$ is $G$-finite. 
        \item \label{Item:UTT_Thm:Concrete}
        For every $U\in \cN_G$, $X$ can be covered by (finitely many) open $U$-TT sets.
        \item \label{Item:UEB_Thm:Concrete}
        $X$ is $G$-discrete.
        \item\label{Item:R_Thm:Concrete}
        For each $U\in \cN_G$ and $x\in X$, $x\in\rm{Int}(R_U[x])$.
        \item \label{Item:Ult_Thm:Concrete}
        For any ultracopower of $X$, the ultracopower map is an isomorphism.
        \item \label{Item:PsAll_Thm:Concrete}
        For each $\sigma\in \rm{SN}(G)$, $\partial_\sigma$ is compatible.
        \item \label{Item:PsBase_Thm:Concrete}
        For some base of semi-norms $\cB\subseteq \rm{SN}(G)$ and for each $\sigma\in \cB$, $\partial_\sigma$ is compatible.
    \end{enumerate}
\end{theorem}

\begin{proof}
    $\ref{Item:DB_Thm:Concrete}\Rightarrow \ref{Item:UTT_Thm:Concrete}$: Say $U\in \cN_G$ and $x\in X$ is not contained in any open $U$-TT set. Let $V\in\cN_G$ satisfy $V^3\subseteq U$. Clearly $x$ cannot be contained in any open $V$-TT set, but now for every $A\in \op(x, X)$, the fact that $A$ is not $V$-TT can be witnessed by some $B\in \op(x, A)$. If not, then for any $C\in \op(A)$, we have $x\in \ol{VC}$, so $x\in\Int(\ol{V^2C})$ as $X$ is $G$-ED, then $A\subseteq \ol{V^3C}\subseteq \ol{UC}$ by our ``if not'' assumption. Therefore we can recursively find $\{A_n: n< \omega\}\subseteq \op(x, X)$ such that $A_{n+1}\subseteq A_n$ and $A_n\not\subseteq \ol{VA_{n+1}}$ for any $n< \omega$. If $W\in \cN_G$ satisfies $W^2\subseteq V$, then setting $B_n = A_n\setminus \ol{VA_{n+1}}$ gives us a family of open sets with $\{WB_n: n< \omega\}$ pairwise disjoint.
    \vspace{3 mm}

    \noindent
    $\ref{Item:UTT_Thm:Concrete}\Rightarrow \ref{Item:DB_Thm:Concrete}$: Fix $U\in \cN_G$, and let $A_0,..., A_{k-1}\in \op(X)$ be $U$-TT sets with $X = \bigcup_{i< k} A_i$. Suppose $\{B_n: n< \omega\}\subseteq \op(X)$. Find $m< n< \omega$ and $i< k$ with both $B_m\cap A_i\neq \emptyset$ and $B_n\cap A_i \neq \emptyset$. Then since $A_i$ is $U$-TT, we have $A_i\subseteq \ol{UB_m}\cap \ol{UB_n}$, implying that $UB_m\cap UB_n\neq \emptyset$. 
    \vspace{3 mm}

    \noindent
    $\ref{Item:UTT_Thm:Concrete}\Leftrightarrow \ref{Item:UEB_Thm:Concrete} \Leftrightarrow \ref{Item:R_Thm:Concrete}$: Follows from Lemma~\ref{Lem:U_TT} and since $X$ is $G$-ED.
    \vspace{3 mm}
    
    \noindent
    $\ref{Item:R_Thm:Concrete}\Rightarrow \ref{Item:PsAll_Thm:Concrete}\Rightarrow \ref{Item:PsBase_Thm:Concrete}\Rightarrow \ref{Item:R_Thm:Concrete}$: Clear.
    \vspace{3 mm}

    \noindent
    $\ref{Item:UTT_Thm:Concrete}\Rightarrow \ref{Item:Ult_Thm:Concrete}$: Fix a set $I$ and $\cU\in \beta I$. We identify $\Sigma_\cU^G X\subseteq \rmS_G(I\times X)$ as discussed before Definition~\ref{Def:HP}, and we let $\pi:= \pi_{X, \cU}$ denote the ultracopower map. Suppose $p, q\in \Sigma_\cU^G X$ satisfy $\pi(p) = \pi(q) = x$. This happens iff for every $A\in p$ and every $C\in \op(x, X)$, we have $\{i\in I: A\cap (\{i\}\times C)\neq \emptyset\}\in \cU$, and similarly for $q$. To show $p = q$, it suffices to show that for every $A\in p$, $B\in q$, and $U\in \cN_G$ that $UA\cap B\neq \emptyset$. Fix $C\in \op(x, X)$ which is $U$-TT. Given $i\in I$, write $A_i\in \op(X)$ for the set satisfying $A \cap (\{i\}\times X) = \{i\}\times A_i$, and likewise for $B_i$. We have $I':= \{i\in I: A_i\cap C\neq \emptyset \text{ and } B_i\cap C\neq \emptyset\}\in \cU$, and for $i\in I'$, we have $C\subseteq \ol{UA_i}$, implying $UA_i\cap B_i\neq \emptyset$. Hence $UA\cap B\neq \emptyset$. 
    \vspace{3 mm}

    \noindent
    $\ref{Item:Ult_Thm:Concrete}\Rightarrow \ref{Item:UTT_Thm:Concrete}$: Suppose $U\in \cN_G$ and $x\in X$ are such that no $A\in \op(x, X)$ is $U$-TT. We will construct a set $I$ and an ultrafilter $\cU\in \beta I$ so that $|(\pi_{X, \cU})^{-1}[x]|\geq 2$. We take $I = \op(x, X)$, and we let $\cU\in \beta I$ be any cofinal ultrafilter. For each $i\in I$, since $i$ is not $U$-TT, we can find $A_i, B_i\in \op(i)$ with $UA_i\cap B_i = \emptyset$. Let $A = \bigcup_{i\in I} \{i\}\times A_i$, and similarly for $B$. Then we can find $p, q\in \Sigma_\cU^G X$ with $A\in p$, $B\in q$, and $\pi_{X, \cU}(p) = \pi_{X, \cU}(q) = x$. Since $UA\cap B = \emptyset$, we have $p\neq q$.
\end{proof}

\begin{defin}
    \label{Def:Concrete_Min_Dyn}
    A topological group $G$ has \emph{concrete minimal dynamics} iff $\mg$ satisfies the equivalent conditions of Theorem~\ref{Thm:Concrete}. Write $\sf{CMD}$ for the class of topological groups with concrete minimal dynamics.
\end{defin}

For Polish groups, concrete minimal dynamics coincides exactly with having metrizable universal minimal flow. This can be seen from the following corollary.

\begin{corollary}
    \label{cor:Polish-dichotomy}
    If $G$ is Polish, then any $G$-ED flow $X$ is either: 
    \begin{itemize}
        \item metrizable, i.e., second-countable, or
        \item not first-countable.
    \end{itemize}
    In particular, if $X$ is not metrizable, then $X$ is not $G$-discrete.
\end{corollary}
\begin{proof}
    If $X$ is first-countable, then it is $G$-discrete by \Cref{thm:ED-first-countable}. By \Cref{Thm:Concrete}, if $\sigma$ is a norm on $G$, then $\partial_\sigma$ is a compatible metric on $X$.

    For the last statement, we note that for any $G$-flow $Y$, as $G$ is metrizable, any $G$-isolated point in $Y$ is a point of first-countability. 
\end{proof}

To end this subsection, we indicate how to recover what is among the more technically challenging results of \cite{BassoZucker}, namely proving that $\sf{CAP}\subseteq \sf{CMD}$. Using ultracoproducts to distill the main ideas of the argument, we obtain the following generalization.

\begin{theorem}
    \label{Thm:CAP}
    Suppose $X$ is a minimal $G$-ED flow which is not $G$-finite. Then there is an ultracopower of $X$ with $\rm{AP}_G(\Sigma_\cU^GX)\subsetneq \Sigma_\cU^GX$. In particular, for some set $I$ with $|I| = \rm{d}(X)$, we have that $\rm{AP}_G(\alpha_G(I\times X))\subseteq \alpha_G(I\times X)$ is not closed.
\end{theorem}

\begin{proof}
    Using Theorem~\ref{Thm:Concrete}, and in particular the proof of $\ref{Item:DB_Thm:Concrete}\Rightarrow \ref{Item:UTT_Thm:Concrete}$, find $x\in X$ and $U\in \cN_G$ such that $\Int(R_U[x])\not\in \op(x, X)$. Then passing to $V\in \cN_G$ with $V^3\subseteq U$ and using that $X$ is $G$-ED, we have $\Int(R_V[x]) = \emptyset$. Let $D\subseteq X$ be a dense subset of $X$ with $|D| = \rm{d}(X)$; we can suppose that $D\cap R_V[x] = \emptyset$.  

    Set $I = \fin{D}$, viewed as a directed partial order under reverse inclusion. For each $i\in I$, let $A_i\in \op(x, X)$ satisfy $\ol{VA_i}\cap i = \emptyset$. Note that $\Int(\bigcap_{i\in I} \ol{VA_i}) = \emptyset$. Let $\cU\in \beta I$ be cofinal, and define $x_\cU = \lim_{i\to \cU} (i, x)$.  Then letting $W\in \cN_G$ satisfy $W^2\subseteq V$ and writing $B = \bigcup_{i\in I}\{i\}\times W A_i$, we have that $\rmC_B$ is a neighborhood of $x_\cU$. We show that $S:= \{g\in G: gx_\cU\in \rmC_B\}$ is not syndetic, which will imply that $x_\cU\not\in \rm{AP}_G(\Sigma_\cU^GX)$ by Lemma~\ref{Lem:Min}\ref{Item:AP_Lem:Min}. If $S$ were syndetic, then $Sx\subseteq X$ would be somewhere dense by Corollary~\ref{Cor:Orbit_Fragment}. However, every $g\in S$ satisfies $gx\in VA_i$ for $\cU$-many $i\in I$, and as $\cU\in \beta I$ is cofinal, we have $Sx\subseteq \bigcap_{i\in I}\ol{V\cdot A_i}$, contradicting that $\Int(\bigcap_{i\in I} \ol{VA_i}) = \emptyset$.
\end{proof}

\subsection{The meets topology}
\label{Subsection:Meets}
Given a compact space $Z$, this subsection introduces a new, not-necessarily-Hausdorff topology on $\rmK(Z)$ which coarsens the Vietoris topology that we call the \emph{meets topology}. Upon restricting our attention to $\rm{Min}_G(Z)$ for $Z$ a $G$-flow, whether the meets topology is Hausdorff will be one of our characterizations of groups with tractable minimal dynamics.

\begin{defin}
    \label{Def:Meets}
    Given a compact space $Z$, the not-necessarily-Hausdorff \emph{meets topology} on $\rmK(Z)$ is generated by the sub-basis $\{\rm{Meets}(A, Z): A\in \op(Z)\}$.
\end{defin}

Most often, $Z$ will be a $G$-flow, and we will want to consider the meets topology restricted to $\rm{Min}_G(Z)$. The Vietoris topology on $\rm{Min}_G(Z)$ is always Hausdorff, but not in general compact. Conversely, while the meets topology on $\rm{Min}_G(Z)$ is not in general Hausdorff, we have the following.

\begin{prop}
    \label{Prop:Meets_Compact}
    If $Z$ is a $G$-flow, $(X_i)_{i\in I}$ is a net from $\rm{Min}_G(Z)$, and $X\in \rm{Min}_G(Z)$, then $X_i\to X$ in the meets topology iff for any Vietoris convergent subnet, $X$ is a subflow of the Vietoris limit. In particular, $\rm{Min}_G(Z)$ equipped with the meets topology is compact.
\end{prop}

\begin{proof}
    Suppose $(X_i)_{i\in I}$ is a net from $\rm{Min}_G(Z)$. By passing to a subnet, we can assume in the Vietoris topology that $X_i\to Y\in \rm{Sub}_G(Z)$. We claim that given $X\in \rm{Min}_G(Z)$, we have $X_i\to X$ in the meets topology iff $X\subseteq Y$. If $X\subseteq Y$, let $A\in \op(Z)$ be such that $X\in \rm{Meets}(A, Z)$. Then also $Y\in \rm{Meets}(A, Z)$, so eventually $X_i\in \rm{Meets}(A, Z)$. Conversely, if $X\not\subseteq Y$, then $X\cap Y = \emptyset$. Find disjoint $A, B\in \op(X)$ with $Y\subseteq A$ and $X\subseteq B$. Eventually we have $X_i\in \rm{Subset}(A, Z)$, implying that eventually $X_i\not\in \rm{Meets}(B, Z)$.
\end{proof}

On $\rm{Min}_G(X)$, the Vietoris topology is always Hausdorff, but not always compact, while the meets topology is always compact, but not always Hausdorff. The CAP groups are exactly those groups for which the Vietoris topology on $\rm{Min}_G(Z)$ is compact Hausdorff, which is the case exactly when the Vietoris and meets topologies coincide. We therefore investigate those topological groups where the meets topology is always compact Hausdorff. Proposition~\ref{Prop:Meets_Compact} tells us that this occurs exactly when Vietoris limits of minimal flows contain a unique minimal subflow. Ultracopowers of $\rmM(G)$ serve as a nice universal setting to investigate properties of Vietoris convergence of minimal subflows. In particular, we have the following.

\begin{prop}
    \label{Prop:WCAP_Ultracoproducts}
    Given a topological group $G$, the following are equivalent.
    \begin{itemize}
        \item 
        For any $G$-flow $Z$, the meets topology on $\rm{Min}_G(Z)$ is compact Hausdorff.
        \item 
        For any $G$-flow $Z$ and for any net $(X_i)_{i\in I}$ from $\rm{Min}_G(Z)$ with $X_i\to Y\in \rm{Sub}_G(Z)$, then $Y$ contains a unique minimal subflow.
        \item 
        Any ultracopower of $\mg$ contains a unique minimal subflow. 
    \end{itemize} 
\end{prop}

\begin{proof}
    The equivalence of the first two bullets follows from Proposition~\ref{Prop:Meets_Compact}. For the third bullet, if there is an ultracopower of $\rmM(G)$ which contains more than one minimal subflow, then since such an ultracopower is a Vietoris limit of copies of $\rmM(G)$, we see using Proposition~\ref{Prop:Meets_Compact} that the first bullet cannot hold. In the other direction, suppose every ultracopower of $\rmM(G)$ contains a unique minimal subflow.  Let $Z$ be a $G$-flow, and let $(Y_i)_{i\in I}$ be a Vietoris convergent net from $\rm{Min}_G(Z)$ with limit $Y\in \rm{Sub}_G(X)$. Let $\cU\in \beta I$ be any cofinal ultrafilter. By Proposition~4.4 of \cite{ZucUlts}, $Y$ is a factor of $\Sigma_\cU^G Y_i$, which in turn is a factor of $\Sigma_\cU^G \rmM(G)$. Since $\Sigma_\cU^G \rmM(G)$ has a unique minimal subflow, so does $Y$.
\end{proof}

To understand which $G$-flows have the property that their ultracopowers contain a unique minimal subflow, we isolate the following property. 

\begin{defin}
    \label{Def:Thickness_Condition}
    A $G$-flow $X$ satisfies the \emph{thickness condition} if for any $U\in \cN_G$, there is $F\in \fin{G}$ such that for any $A, B\in \op(X)$, we have $UFA\cap UFB \neq \emptyset$.
\end{defin}

Note that any flow satisfying the thickness condition has a unique minimal subflow. Also note that the thickness condition is preserved under factor maps and under irreducible extensions. This condition is exactly what ensures that if $I$ is an index set and $X$ is a $G$-flow satisfying the thickness condition, it is impossible to construct two thick members of $\op(I\times X)$ which are uniformly apart. Theorem~\ref{Thm:Rosendal_Minimal} will prove that for $G$-ED flows, the thickness condition is exactly what characterizes when all ultracopowers contain a unique minimal subflow.

\subsection{Rosendal's criterion} 
\label{Subsection:Rosendal}
    
When $G$ is a Polish group, the following criterion due to Rosendal characterizes when a Polish $G$-space has a comeager orbit appears in \cite{BYMT}. For a slightly more general formulation of the result, see \Cref{thm:Rosendal-general}. However, the criterion makes sense without any Polish assumptions.

\begin{defin}
    \label{Def:Rosendal_Criterion}
    Fix $G$ a topological group and $X$ a $G$-space. Given $U\in \cN_G$, we say that $A\in \op(X)$ is \emph{somewhere $U$-TT} if some $A_0\in \op(A)$ is $U$-TT, and $A$ is \emph{nowhere $U$-TT} if it is not somewhere $U$-TT. 
    
    We say that the $G$-space $X$ satisfies the \emph{Rosendal criterion} if $X$ for every $U\in \cN_G$ and $A\in \op(X)$, $A$ is somewhere $U$-TT. Write $\cal{RC}_G$ for the class of $G$-spaces satisfying the Rosendal criterion.
\end{defin}

If $G$ is discrete, then a $G$-space $X$ satisfies the Rosendal criterion if and only if the set of isolated points of $X$ is dense in $X$. We do not know of a commonly used terminology for such spaces, hence the choice to diverge from our naming scheme. 
Moreover, for an arbitrary topological group $G$, the property ``the $G$-space $X$ has a dense set of $G$-isolated points'' is a priori stronger than ``$X \in \cal{RC}_G$'', see question \Cref{Que:G-iso-dense}.

We say that a map $\phi\colon X\to Y$ between topological spaces is \emph{SWD-preserving} if $\phi[A]\subseteq Y$ is somewhere dense whenever $A\subseteq X$ is somewhere dense, equivalently whenever $A\in \op(X)$. This notion is referred to as \emph{category-preserving} in \cite{LeBoudecTsankov}. Pseudo-open maps are always SWD-preserving, and the two notions are equivalent if $\phi$ is closed and $X$ is semiregular.

\begin{lemma}
    \label{Lem:Rosendal_Maps}
    Let $X$ and $Y$ be $G$-spaces, and let $\phi\colon X\to Y$ be a $G$-map. Moreover, let $\alpha\colon K \to G$ be a continuous group homomorphism with dense image.
    \begin{enumerate}[label=\normalfont(\arabic*)]
        \item \label{Item:PsOpDown_Lem:RMaps}
        If $\phi[X]\subseteq Y$ is dense, $\phi$ is SWD-preserving, and $X\in \cal{RC}_G$, then $Y\in \cal{RC}_G$.
        \item \label{Item:IrredUP_Lem:RMaps}
        If $\phi$ is irreducible and $Y\in \cal{RC}_G$, then $X\in \cal{RC}_G$.
        \item \label{Item:Homo_Lem:RMaps}
        If $X \in \cal{RC}_K$, then $X \in \cal{RC}_G$.
    \end{enumerate}
\end{lemma}

\begin{proof}
    \ref{Item:PsOpDown_Lem:RMaps}: Given $U\in \cN_G$ and $B\in \op(Y)$, let $A = \phi^{-1}[B]\in \op(X)$ (non-empty as $\phi[X]\subseteq Y$ is dense). As $X\in \cal{RC}_G$, find $A_0\in \op(A)$ which is $U$-TT. As $\phi$ is SWD-preserving, $\Int_Y(\ol{\phi[A_0]}):= B_0\in \op(B)$ is non-empty. We show that $B_0$ is $U$-TT. Fix $C\in \op(B_0)$. As $\phi^{-1}[C]\cap A_0\in \op(A_0)$, we have $\ol{U\cdot \phi^{-1}[C]}\supseteq A_0$. It follows that $\ol{U\cdot C}\supseteq B_0$.  
    \vspace{3 mm}

    \noindent
    \ref{Item:IrredUP_Lem:RMaps}: Given $U\in \cN_G$ and $A\in \op(X)$, let $B = \Int(\phi_{\rm{fib}}(A))\in \op(Y)$. As $Y\in \cal{RC}_G$, find $B_0\in \op(B)$ which is $U$-TT. We claim that $A_0:= \phi^{-1}[B_0]$ is $U$-TT. If $A_1, A_2\in \op(A_0)$, then for each $i< 2$, set $\Int(\phi_{\rm{fib}}(A_i))=: B_i\in \op(B_0)$. We have $UB_1\cap B_2\neq \emptyset$, so also $U\phi^{-1}[B_1]\cap \phi^{-1}[B_2]\neq \emptyset$, and $\phi^{-1}[B_i]\subseteq A_i$.
    \vspace{3 mm}
    
    \noindent
    \ref{Item:Homo_Lem:RMaps}: If $U\in \cN_G$ and $A\in \op(X)$ are such that $A$ is nowhere $U$-TT, then $A$ is nowhere $\alpha^{-1}(U)$-TT, since for each $A_0 \in \op(A)$ and $B \in \op(A_0)$, it holds that $\ol{UB} \supseteq \ol{\alpha^{-1}(U)B}$.
    \end{proof}

\begin{lemma}
    \label{lem:RC-equivs}
    Let $X$ be a $G$-space. The following are equivalent:
    \begin{enumerate}[label=\normalfont(\arabic*)]
        \item \label{Item:RC_lem:RC-equivs}
        $X \in \cal{RC}_G$.
        \item  \label{Item:Q_U-open-dense_lem:RC-equivs}
        For each $U\in \cN_G$, $\{x\in X: x\in \rm{Int}(Q_U[x])\}\subseteq X$ contains an (open) dense set.
    \end{enumerate}
\end{lemma}
\begin{proof}
    If $X\in \cal{RC}_G$, then given $U\in \cN_G$ and $A\in \op(X)$, there is $A_0\in \op(A)$ which is $U$-TT, so by  Lemma~\ref{Lem:U_TT}\ref{Item:Interior_Lem:UTT}, any $x\in A_0$ satisfies $x\in \rm{Int}(Q_U[x])$, giving an open dense set of such $x$. Conversely, assume for every $U\in \cN_G$ that densely many $x$ satisfy $x\in \Int(Q_U[x])$, and fix $U\in \cN_G$ and $A\in \op(X)$, and find $V\in \cN_G$ with $V^2 \subseteq U$. Find $x\in A$ with $x\in \rm{Int}(Q_V[x])$. So $A_0:= A\cap \rm{Int}(Q_V[x])$ is non-empty, and by Lemma~\ref{Lem:U_TT}\ref{Item:UTT_Lem:UTT}, it is $V^2$-TT, so in particular $U$-TT.
\end{proof}

When $G$ is Polish, a Baire category argument shows that any $G$-ED flow satisfying point \ref{Item:Q_U-open-dense_lem:RC-equivs} of the above lemma has a $G$-isolated point. 
For an arbitrary topological group, it is an a priori stronger property, though we do not construct a counterexample. 
\begin{question}
    \label{Que:G-iso-dense}
   Does there exist a $G$-ED flow, perhaps under additional set-theoretic hypotheses, without $G$-isolated points which satisfies the Rosendal criterion?
\end{question}

Theorem~\ref{Thm:Rosendal_Minimal} is one of the main theorems of this paper. It gives a variety of conditions all equivalent to asserting that a given minimal $G$-ED flow $X$ satisfies the Rosendal criterion. One of these conditions is that any ultracopower contains a unique minimal subflow, so Proposition~\ref{Prop:WCAP_Ultracoproducts} tells us that when $X = \rmM(G)$, these conditions are also equivalent to the meets topology on $\rm{Min}_G(Z)$ being compact Hausdorff for every $G$-flow $Z$.  We then define the class $\sf{TMD}$ of topological groups $G$ with $\rmM(G)\in \cal{RC}_G$.

\begin{theorem}
    \label{Thm:Rosendal_Minimal}
    If $G$ is a topological group and $X$ is a minimal $G$-ED flow, then the following are equivalent.
    \begin{enumerate}[label=\normalfont(\arabic*)]
        \item \label{Item:RC_Thm:RM}
        $X\in \cal{RC}_G$. 
        \item \label{Item:UTT_Thm:RM}
        For each $U\in \cN_G$, $X$ is somewhere $U$-TT. 
        \item \label{Item:Thick_Thm:RM}
        $X$ satisfies the thickness condition.
        \item \label{Item:Ultra_Thm:RM}
        Any ultracopower of $X$ contains a unique minimal subflow.
        \item \label{Item:HP_Thm:RM}
        For any ultracopower of $X$, the ultracopower map is highly proximal.
        \item \label{Item:ContAll_Thm:RM}
        For each $\sigma\in \rm{SN}(G)$, $\partial_\sigma$ has a dense set of compatibility points.
        \item \label{Item:ContQuot_Thm:RM}
        For each $\sigma\in \rm{SN}(G)$, $\ol{\partial}_\sigma$ has a dense set of compatibility points. 
        \item \label{Item:ContBase_Thm:RM}
        For some base of semi-norms $\cB\subseteq \rm{SN}(G)$ and for each $\sigma\in \cB$, there is some $\partial_\sigma$-compatible point.
    \end{enumerate}
    $\ref{Item:RC_Thm:RM}\Leftrightarrow \ref{Item:ContAll_Thm:RM} \Leftrightarrow \ref{Item:ContQuot_Thm:RM}$ still hold with minimality weakened to $G$-TT.
\end{theorem}
       
\begin{proof}
    \noindent
    $\ref{Item:RC_Thm:RM}\Leftrightarrow \ref{Item:ContAll_Thm:RM}$: By  \Cref{lem:RC-equivs}, for each $\epsilon>0$, the set $\{x\in X: x\in \rm{Int}(Q_{\sigma, \epsilon}[x])\}$ contains an open dense set, and thus so does $\{x\in X: x\in \rm{Int}(R_{\sigma, \epsilon}[x])\}$.
    Then \ref{Item:ContAll_Thm:RM} follows from the Baire category theorem.
   
   Conversely, if $x\in X$ is a $\partial_\sigma$-compatability point, then $x\in \Int(R_{\sigma, c}[x])$ for every $c > 0$. As $X$ is $G$-ED, we have for any $c, \epsilon > 0$ that  $R_{\sigma, c}[x]\subseteq Q_{\sigma, c+\epsilon}$. The result now follows from Lemma~\ref{lem:RC-equivs}.
    \vspace{3 mm}

    \noindent
    $\ref{Item:ContAll_Thm:RM}\Leftrightarrow \ref{Item:ContQuot_Thm:RM}$: Proposition~\ref{Prop:Adequate} gives $\Rightarrow$. For the reverse, fix $\sigma\in \rm{SN}(G)$ and $A\in\op(X)$. We can find $\sigma\leq \tau\in \rm{SN}(G)$ and $B\in \op(X/\partial_\tau)$ with $\pi_{\partial_\tau}^{-1}[B]\subseteq A$. Find $y\in B$ which is $\ol{\partial}_\tau$-compatible. Letting $x\in A$ satisfy $\pi_{\partial_\tau}(x) = y$, then $x$ is $\partial_\tau$-compatible. As $\sigma\leq \tau$, we have that $x$ is $\partial_\sigma$-compatible. 
    \vspace{3 mm}
    
     \noindent
    $\ref{Item:RC_Thm:RM}\Rightarrow\ref{Item:Thick_Thm:RM}$: Fix $U\in \cN_G$ and find $C\in \op(X)$ which is $U$-TT. Using minimality, find $F\in \fin{G}$ such that $F^{-1}C = X$. Hence if $A, B\in \op(X)$, then $FA\cap C\neq \emptyset$ and $FB\cap C\neq \emptyset$. Since $C$ is $U$-TT, we must have $C\subseteq \ol{UFA}$, so in particular $UFA\cap UFB\neq \emptyset$.
    \vspace{3 mm}

    \noindent
    $\ref{Item:Thick_Thm:RM}\Rightarrow \ref{Item:UTT_Thm:RM}$: Suppose $V\in \cN_G$ witnesses the failure of $\ref{Item:UTT_Thm:RM}$ and let $U\in \cN_G$ be such that $U^2 \subseteq V$. 
    By \Cref{Lem:U_TT} and the assumption that $X$ is $G$-ED, this implies that for all $x \in X$, $R_U[x]$ is nowhere dense. 
    Fix $F\in \fin{G}$. We claim that there are $x, y\in X$ so that $(F{\cdot}x\times F{\cdot}y)\cap R_U = \emptyset$. Suppose towards a contradiction that this were not true, and fix $x\in X$. Given $g, h\in F$, set $D_{g, h} = \{y\in X: (gx, hy)\in R_U\}$. By our assumption, $X = \bigcup_{g, h\in F} D_{g, h}$, so for some $g, h\in F$, the set $D_{g, h}\subseteq X$ is somewhere dense. But as $h{\cdot}D_{g, h}\subseteq R_U[gx]$ and as the latter is nowhere dense, this is a contradiction. 

    Thus, for each $F\in \fin{G}$, fix $x_F, y_F\in X$ with $(F{\cdot}x_F\times F{\cdot}y_F)\cap R_U = \emptyset$. Letting $V\in \cN_G$ be such that $V^2\subseteq U$, this implies that we can find $A\in \op(x_F, X)$ and $B\in \op(y_F, X)$ with $VFA\cap VFB = \emptyset$.  
    \vspace{3 mm}

    \vspace{3 mm}
    
    \noindent
    $\ref{Item:UTT_Thm:RM}\Rightarrow \ref{Item:RC_Thm:RM}$: Suppose $X\not\in \cal{RC}_G$; as $X$ is minimal, hence $G$-TT, we can find $U\in \cN_G$ and $A\in \op(X)$  such that $A$ is nowhere $U$-TT. Now observe that if $g\in G$, then $gA$ is nowhere $gUg^{-1}$-TT. 
    By minimality, find $F\in \fin{G}$ with $X = F^{-1}A$. Thus, for a suitably small $V\in \cN_G$, $X$ is nowhere $V$-TT.
    \vspace{3 mm}

    \vspace{3 mm}
        
    \noindent
    $\ref{Item:ContAll_Thm:RM}\Rightarrow \ref{Item:ContBase_Thm:RM}$: Clear. 
    \vspace{3 mm}

    \noindent
    $\ref{Item:ContBase_Thm:RM}\Rightarrow \ref{Item:UTT_Thm:RM}$: This follows from \Cref{Lem:U_TT} and the $G$-ED hypothesis on $X$. 
    \vspace{3 mm}

    \noindent
    $\ref{Item:Thick_Thm:RM}\Rightarrow\ref{Item:Ultra_Thm:RM}$:  Suppose for some set $I$ and some $\cU\in \beta I$ that $\Sigma_\cU^GX$ contains two minimal subflows $M$ and $N$. Given $F\in \fin{G}$, we can find $A, B\in \op(\rmS_G(I\times X))$ and $U \in \cN_G$ with $M\subseteq A$, $N\subseteq B$, and with $UFA\cap UFB = \emptyset$. For each $i\in I$, let $A_i\in \op(X)$ be such that $A = \bigcup_{i\in I} (\{i\}\times A_i)$, and similarly for $B$. Upon finding $i\in I$ with both $A_i$ and $B_i$ non-empty, we have $UFA_i\cap UFB_i = \emptyset$, showing that $X$ fails the thickness condition.
    \vspace{3 mm}
    
    \noindent
    $\ref{Item:Ultra_Thm:RM}\Rightarrow \ref{Item:Thick_Thm:RM}$:  Suppose that $X$ fails the thickness condition, as witnessed by $U\in \cN_G$. Let $\cU$ be a cofinal ultrafilter on $\fin{G}:= I$. We will show that $\Sigma_\cU^G \rmS_G(X)$ contains two disjoint subflows. For each $F\in I$, let $X_F$ denote $\set{F} \times X\subseteq I \times X$, and let $A_F, B_F \sub X_F$ be such that $UFA_F\cap UFB_F = \emptyset$. Given $S \sub \fin{G}$, write $A_S \coloneqq \bigcup_{F \in S} FA_F$, and similarly for $B_S$. Whenever $S\subseteq \fin{G}$ is cofinal, $A_S$ and $B_S$ are thick members of $\op(X)$ with $UA_S\cap UB_S = \emptyset$. In $\rmS_G(I\times X)$, this implies that $\rmC_{A_S}\cap \rmC_{B_S} = \emptyset$, and Lemma~\ref{Lem:Subflows_of_SGX} tells us that each of $\rmC_{A_S}$ and $\rmC_{B_S}$ contains a subflow. As the property of containing a subflow is preserved under a directed intersection of compact sets, it follows that both $Y= \bigcap_{S\in \cU} \rmC_{A_S}$ and $Z= \bigcap_{S\in \cU} \rmC_{B_S}$ contain subflows, and we have $Y, Z\subseteq \Sigma_\cU^G X$.
    \vspace{3 mm}

    \noindent
    $\ref{Item:Ultra_Thm:RM}\Rightarrow \ref{Item:HP_Thm:RM}$:    Let $I$ be a set, and let $\cU\in \beta I$, and write $\pi:= \pi_{X, \cU}$. Let $Y\subseteq \Sigma_\cU^G X$ denote the unique minimal subflow. Given $p\in X$ and considering the inclusion $\Sigma_\cU^GX\subseteq \rmS_G(I\times X)$, the fiber $\pi^{-1}[p]$ is exactly described by the near filter $\cF_p:= \la S\times B: S\in \cU, B\in \op(p, X)\ra$. Fix $A\in \op(I\times X)$ such that $\pi_{\rm{fib}}(\rmN_A\cap \Sigma_\cU^G X) = \emptyset$, i.e.\ such that for every $p\in X$, $A$ is compatible with $\cF_p$. Equivalently, this means that for every $V\in \cN_G$, $S\in \cU$, and $B\in \op(X)$, we have $VA\cap (S\times B)\neq \emptyset$. We will show that $Y\cap N_A = \emptyset$, i.e.\ that  $Y\subseteq \rmC_A$, and conclude by \Cref{Prop:HP_Onto_Minimal}. By Lemma~\ref{Lem:Subflows_of_SGX}, $Y$ is described by the near filter 
    $$\{B\in \op(I\times X): \forall S\in \cU\,\, B\cap (S\times X) \text{ is pre-thick}\}.$$
    Hence it is enough to show that for every $S\in \cU$, we have $A\cap (S\times X)\subseteq I\times X$ is pre-thick. Towards showing this, fix $S\in \cU$, $U\in \cN_G$, and $F\in \fin{G}$. Let $V\in \cN_G$ satisfy $V^2\subseteq U$. Since $\ref{Item:Ultra_Thm:RM}\Rightarrow \ref{Item:Thick_Thm:RM}\Rightarrow \ref{Item:UTT_Thm:RM}\Rightarrow \ref{Item:RC_Thm:RM}$, we have $X\in \cal{RC}_G$. Using this, we can find $B\in \op(X)$ such that for every $g\in F$, $gB$ is $V$-TT. It follows that writing $A_i\in \op(X)$ for the set with $A\cap (\{i\}\times X) = \{i\}\times A_i$, there are $\cU$-many $i\in I$ with $FB\subseteq \ol{UA_i}$. It follows that $UA\cap (S\times X)$ is thick as desired. 
    \vspace{3 mm}
    
    \noindent
    $\ref{Item:HP_Thm:RM}\Rightarrow \ref{Item:Ultra_Thm:RM}$: If some ultracopower of $X$ contains at least two minimal subflows, then by Proposition~\ref{Prop:HP_Onto_Minimal}, the ultracopower map cannot be highly proximal.\qedhere

\end{proof}

By combining Theorem~\ref{Thm:Rosendal_Minimal} and Proposition~\ref{Prop:WCAP_Ultracoproducts}, we can now state the main definition of this paper.

\begin{defin}
    \label{Def:Dyn_Tractable}
    A topological group $G$ has \emph{tractable minimal dynamics} iff $\rmM(G)$ satisfies the equivalent conditions of Theorem~\ref{Thm:Rosendal_Minimal}, equivalently if for every $G$-flow $Z$, the meets topology on $\rm{Min}_G(Z)$ is compact Hausdorff (Proposition~\ref{Prop:WCAP_Ultracoproducts}). Write $\sf{TMD}$ for the class of topological groups with tractable minimal dynamics. 
\end{defin}

Let us immediately observe that $\sf{CMD}\subseteq \sf{TMD}$, and all of the characterizations of $\sf{TMD}$ are natural weakenings of statements that characterize $\sf{CMD}$. One thing that is missing is a $\sf{TMD}$ version of the CAP property. Let us call a topological group \emph{weakly CAP} if for any $G$-flow $X$ and $y\in \ol{\rm{AP}_G(X)}$, the flow $\ol{G\cdot y}$ contains a unique minimal subflow; write $\sf{WCAP}$ for this class of topological groups. Thus $\sf{TMD}\subseteq \sf{WCAP}$. For separable groups, the next theorem shows that these coincide. 

\begin{theorem}
    \label{Thm:Polish_WCAP}
    If $G$ is separable and $X$ is a minimal, $G$-ED flow with $X\not\in \cal{RC}_G$, then there are a countable set $I$, $\cU\in \beta I$, and $y\in \Sigma_\cU^GX$ so that $\ol{G\cdot y}$ contains multiple minimal subflows.
\end{theorem}

\begin{proof}
    Replacing $G$ by a countable dense subgroup, we will assume that $G$ is countable. Suppose $U\in \cN_G$ witnesses the failure of $\ref{Item:UTT_Thm:RM}$ and let $V\in \cN_G$ be such that $V^2 \subseteq U$. 
    By \Cref{Lem:U_TT} and the assumption that $X$ is $G$-ED, this implies that for all $x \in X$, $R_V[x]$ is nowhere dense. 
    Fix $F\in \fin{G}$ and $x\in X$. We claim that for some $k\in G$, we have $(F{\cdot}x\times F{\cdot}kx)\cap R_V = \emptyset$. Towards a contradiction, suppose that this were not true. Given $g, h\in F$, set $D_{g, h} = \{k\in G: (gx, hkx)\in R_V\}$. By our assumption, $G = \bigcup_{g, h\in F} D_{g, h}$, so for some $g, h\in F$, the set $D_{g, h}\subseteq G$ is piecewise syndetic. But as $h{\cdot}D_{g, h}\cdot x\subseteq R_V[gx]$ and as the latter is nowhere dense, this contradicts Corollary~\ref{Cor:Orbit_Fragment}.

    Using the above, and the countability of $G$, we can inductively in a back-and-forth fashion construct thick subsets $T_0, T_1\subseteq G$ so that $(T_0{\cdot}x\times T_1{\cdot}x)\cap R_V = \emptyset$. Fix $W\in \cN_G$ with $W^2\subseteq V$. For each $F\in \fin{T_0\cup T_1}$, writing $F_i = F\cap T_i$, we can find $A_F, B_F\in \op(x, X)$ with $WF_0A_F\cap WF_1B_F = \emptyset$. Set $I = \fin{T_0\cup T_1}$, and let $\cU\in \beta I$ be cofinal. Set $A = \bigcup_{F\in I} \{F\}\times F_0A_F\subseteq I\times X$ and $B = \bigcup_{F\in I} F_1B_F\subseteq I\times X$. Let $x_\cU = \displaystyle\lim_{F\to \cU} (F, x)\in \Sigma_\cU^GX$. We claim that for every $g\in T_0$, we have $gx_\cU\in \rmC_A$; likewise for $g\in T_1$ and $\rmC_B$. This follows since for each $F\in I$ with $g\in F$, we have $x\in g^{-1}F_0A_F$, and since $\cU\in \beta I$ is cofinal. We conclude that $\ol{G\cdot x}\cap \rmC_A$ and $\ol{G\cdot x}\cap \rmC_B$ both contain $G$-subflows, and $\rmC_A\cap \rmC_B = \emptyset$.
\end{proof}

It seems highly likely that $\sf{TMD}$ and $\sf{WCAP}$ coincide at least among the class of ``ambitable'' groups discussed in \cite{Pachl}. It is an open question if every non-precompact topological group is ambitable.

\begin{question}
    \label{Que:WCAP}
    Do we have $\sf{TMD} = \sf{WCAP}$?
\end{question}

\section{Some closure properties of the class $\sf{TMD}$}
\label{Section:ClosureProps}

This section contains some closure properties of the class $\sf{TMD}$ that we can prove at this point in time. In particular, we show that this class is closed under group extensions and surjective inverse limits. More closure properties of $\sf{TMD}$ appear in Section~\ref{Section:Abs} whose proofs use more sophisticated set-theoretic techniques.

We begin with an easy fact.

\begin{prop}
    \label{Prop:Homomorphic_Image}
    Let $\alpha\colon K \to G$ be a continuous group homomorphism with dense image. If $K \in \sf{TMD}$ then $G \in \sf{TMD}$.
\end{prop}
\begin{proof}
    $\rmM(G)$ is a minimal $K$-flow, so by hypothesis $\rmM(G) \in \cal{RC}_K$. We conclude by \Cref{Lem:Rosendal_Maps}. 
\end{proof}

By Definition~\ref{Def:Dyn_Tractable}, a topological group has tractable minimal dynamics exactly when for any flow, the meets topology on the space of minimal subflows is compact Hausdorff. We will make extensive use of the meets topology working towards the proof of Theorem~\ref{Thm:WCAP_Group_Extensions}. Given a net $(X_i)_{i\in I}$ of minimal subflows of some given flow, we write $X_i\to Y$ to indicate Vietoris convergence and $X_i\xrightarrow{m} Y$ to indicate meets convergence.

Recall that for a $G$-flow $X$, $\rm{AP}_G(X)$ denotes the set of points in $X$ belonging to minimal subflows.

\begin{lemma}
\label{Lem:AP_Normal_subgroup}
    Let $G$ be a topological group and $H\trianglelefteq^c  G$ a closed normal subgroup. If $X$ is a minimal $G$-flow, then $\rm{AP}_H(X)\subseteq X$ is dense.
\end{lemma}

\begin{proof}
    As $X$ is a $G$-flow, hence also an $H$-flow, we have $\rm{AP}_H(X)\neq \emptyset$. If $x\in \rm{AP}_H(X)$, then by the normality of $H$ in $G$, so is $gx$ for any $g\in G$, with $\ol{H\cdot gx} = g\cdot \ol{H\cdot x}$. 
\end{proof}

\begin{lemma}
\label{Lem:WCAP_Closed_Subgroup}
    Let $G$ be a topological group and $H\leq^c G$ with $H\in \sf{TMD}$. If $X$ is a $G$-flow, then for every $x\in \ol{\rm{AP}_H(X)}$, $\ol{H\cdot x}$ contains a unique minimal $H$-subflow $Z_x$, and the map $\psi\colon \ol{\rm{AP}_H(X)}\to \rm{Min}_H(X)$ given by $\psi(x) = Z_x$ is meets continuous.
\end{lemma}
       
\begin{proof}
    Given $x\in \ol{\rm{AP}_H(X)}$, let $(x_i)_{i\in I}$ be a net from $\rm{AP}_H(X)$ with $x_i\to x$. Passing to a subnet, we can assume that $(\ol{H\cdot x_i})_{i\in I}$ is Vietoris convergent, with limit $Y$. Then $x\in Y$, so $\ol{H\cdot x}\subseteq Y$, and $Y$ contains a unique minimal subflow since $H\in \sf{TMD}$, by Proposition~\ref{Prop:WCAP_Ultracoproducts}.

    Keeping in mind Fact~\ref{Fact:Continuity}, let $(x_i)_{i\in I}$ be a net from $\rm{AP}_H(X)$ with $x_i\to x\in \ol{\rm{AP}_H(X)}$. Write $Z_i = Z_{x_i} = \ol{H\cdot x_i}$. Passing to a subnet if needed, we can assume $Z_i\to Z\in \rm{Sub}_G(X)$. By Proposition~\ref{Prop:Meets_Compact}, if $Y\subseteq Z$ is the unique minimal subflow of $Z$, then $Z_i\xrightarrow{m} Y$. However, we also have $x\in Z$, so $Z_x\subseteq Z$. Hence $Z_x = Y$. 
\end{proof}

\begin{lemma}
    \label{Lem:Kflow}
    Let $G$ be a topological group, $H\trianglelefteq^c G$ with $H\in\sf{TMD}$, and let $K = G/H$. Fix $X$ a $G$-flow. 
    \begin{enumerate}[label=\normalfont(\arabic*)]
        \item \label{Item:KFlow_Lem:KFlow}
        $\rm{Min}_H(X)$ with the meets topology is a $K$-flow under the action induced from $G$.
        
        \item \label{Item:Min_Lem:KFlow}
        If $X$ is a minimal $G$-flow, we have that $\rm{Min}_H(X)$ with the meets topology is a minimal $K$-flow. 

        \item \label{Item:Vietoris_Lem:KFlow}
        If $(Y_i)_{i\in I}$ is a net from $\rm{Min}_G(X)$ and $Y_i\to Y\in \rm{Sub}_G(X)$, then in the Vietoris topology on $\rm{Sub}_K(\rm{Min}_H(X))$, we have $\rm{Min}_H(Y_i)\to \rm{Min}_H(Y)\in \rm{Sub}_K(\rm{Min}_H(X))$.

        \item \label{Item:MG_Lem:KFlow}
        $\rm{Min}_H(\rmM(G))\cong \rmM(K)$.
    \end{enumerate}
\end{lemma}

\begin{proof}
    \ref{Item:KFlow_Lem:KFlow}: Say that $g_iH\to gH\in K$ and $Z_i\xrightarrow{m} Z\in \rm{Min}_H(X)$. Passing to a subnet if needed, we can obtain:
    \begin{itemize}
        \item 
        $h_i\in H$ with $g_ih_i\to g$,
        \item 
        $z_i\in Z_i$ and $z\in Z$ with $z_i\to z$,
    \end{itemize}
    Hence $g_ih_iz_i\to gz$, and by Lemma~\ref{Lem:WCAP_Closed_Subgroup}, it follows that $g_iH{\cdot}Z_i\xrightarrow{m} gH{\cdot}Z$. 
    \vspace{3 mm}

    \noindent
    \ref{Item:Min_Lem:KFlow}: Suppose that $Y, Z\in \rm{Min}_H(X)$, and pick some $y\in Y$ and $z\in Z$. By the minimality of $X$, find a net  $(g_i)_{i\in I}$ from  $G$ with $g_iy\to z$. By Lemma~\ref{Lem:WCAP_Closed_Subgroup}, we have $g_iH{\cdot}Y\xrightarrow{m} Z$, so $Z\in \ol{K\cdot Y}$.
    \vspace{3 mm}
    
    \noindent
     \ref{Item:Vietoris_Lem:KFlow}: Passing to a subnet, assume that $Q\in \rm{Sub}_K(\rm{Min}_H(X))$ is the Vietoris limit of $(\rm{Min}_H(Y_i))_{i\in I}$. We show $Q = \rm{Min}_H(Y)$. First suppose $Z\in Q$. Passing to a subnet, find $Z_i\in \rm{Min}_H(Y_i)$ with $Z_i\xrightarrow{m} Z\in \rm{Min}_H(X)$. Then since $Z_i\subseteq Y_i$, we have $Z\subseteq Y$, i.e.\ $Z\in \rm{Min}_H(Y)$. Conversely, fixing $Z\in \rm{Min}_H(Y)$, we can, upon passing to a subnet and using Lemma~\ref{Lem:AP_Normal_subgroup}, find $y_i\in \rm{AP}_H(Y_i)$ and $z\in Z$ with $y_i\to z$. By Lemma~\ref{Lem:WCAP_Closed_Subgroup}, we have $\ol{H\cdot y_i}\xrightarrow{m} Z$. As $\ol{H\cdot y_i}\in \rm{Min}_H(Y_i)$, we have $Z\in Q$.

     \ref{Item:MG_Lem:KFlow}: Letting $\pi\colon G\to K$ denote the quotient map and  $\wt{\pi}\colon \sa(G)\to \sa(K)$ the continuous extension, pick a minimal subflow $M\subseteq \sa(G)$. Then $\wt{\pi}[M]\subseteq \sa(K)$ is a minimal $K$-subflow, so isomorphic to $\rmM(K)$. If $Z\in \rm{Min}_H(M)$, then $\wt{\pi}[Z]$ is a singleton. This gives us a $K$-equivariant map $\phi\colon \rm{Min}_H(M)\to \wt{\pi}[M]$. AS $\phi$ is continuous, $\phi$ is a $G$-map from a minimal flow onto $\wt{\pi}[M]$, so we must have $\rm{Min}_H(M)\cong \rmM(K)$.  
\end{proof}

\begin{theorem}
    \label{Thm:WCAP_Group_Extensions}
    The class $\sf{TMD}$ is closed under group extensions: If $G$ is a topological group, $H\trianglelefteq G$ is a closed normal subgroup, and both $H$, $G/H\in \sf{TMD}$, then also $G\in \sf{TMD}$.
\end{theorem}

\begin{proof}
    Let $X$ be a $G$-flow, and suppose that $(Y_i)_{i\in I}$ is a net from $\rm{Min}_G(X)$ with $Y_i\to Y\in \rm{Sub}_G(X)$. We want to show that $Y$ contains a unique minimal subflow. By Lemma~\ref{Lem:Kflow}\ref{Item:Min_Lem:KFlow}, each $\rm{Min}_H(Y_i)$ is a minimal $K$-flow, and by Lemma~\ref{Lem:Kflow}\ref{Item:Vietoris_Lem:KFlow}, we have $\rm{Min}_H(Y_i)\to \rm{Min}_H(Y)\in \rm{Sub}_K(\rm{Min}_H(X))$. As $K\in \sf{TMD}$, $\rm{Min}_H(Y)$ contains a unique minimal $K$-subflow. If $Y$ contained distinct minimal $G$-subflows $Z_0$ and $Z_1$, then $\rm{Min}_H(Z_0)$ and $\rm{Min}_H(Z_1)$ would be distinct minimal $K$-subflows of $\rm{Min}_H(Y)$, a contradiction.  
\end{proof}

As an application, we note that this gives the first examples of Polish groups in $\sf{GPP}\setminus \sf{PCMD}$ which are not non-archimedean; for instance, take any $H\in \sf{GPP}\setminus \sf{PCMD}$, let $K$ be a non-trivial connected compact group, and let $G = H\times K$. However, the following is still open (see \cites{GTZmanifolds, Surfaces}).

\begin{question}
    \label{Que:Connected}
    Does the class $\sf{GPP}\setminus \sf{PCMD}$ contain any connected group? For instance, is the group $\homeo^+(S^2)$ of orientation-preserving homeomorphisms of the sphere a member of this class?
\end{question}

We next consider \emph{surjective inverse limits} of topological groups. An \emph{inverse system} of topological groups, denoted $(G_i, \pi^j_i)_I$, is a net $(G_i)_{i\in I}$ of topological groups along with a continuous homomorphism $\pi^j_i\colon G_j\to G_i$ for each $i\leq j\in I$. The \emph{inverse limit} is the topological group
\begin{equation*}
    \varprojlim (G_i, \pi^j_i):= \{(g_i)_{i\in I}\in \prod_{i\in I} G_i: \forall i\leq j\in I\, \pi^j_i(g_j) = g_i\}.
\end{equation*}
This is a closed subgroup of $\prod_i G_i$. When the $\pi^j_i$ are understood, we can just write $\varprojlim_i G_i$. We say that the inverse system is \emph{surjective} if for each $j\in I$, the projection map $\pi_j\colon \varprojlim_i G_i\to G_j$ is onto. This implies that each $\pi^j_i$ is onto. For more discussion, see Section~7 of \cite{BassoZucker}.

The proof of Proposition~\ref{Prop:Inverse_Limits} is very similar to that of Proposition~7.3 from \cite{BassoZucker}. We extract a lemma contained in that proof. 

\begin{lemma}
    \label{Lem:Topometric_Samuel_Hom}
    Let $\pi\colon G\to H$ be a continuous surjective homomorphism of topological groups, and extend to $\wt{\pi}\colon \sa(G)\to \sa(H)$. Given $\sigma\in \rm{SN}(H)$, form $\sigma\circ \pi\in \rm{SN}(G)$. Then $$\partial_{\sigma\circ \pi}^{\sa(G)} = \partial_\sigma^{\sa(H)}\circ (\wt{\pi}\times \wt{\pi}).$$
\end{lemma}

We also recall Proposition~\ref{Prop:Level_Maps}, which implies that for any $\sigma\in \rm{SN}(G)$ and minimal subflow $M\subseteq \sa(G)$, we have $\partial_\sigma^M = \partial_\sigma^{\sa(G)}|_{M\times M}$. 

\begin{prop}
\label{Prop:Inverse_Limits}
    If $(G_i, \pi^j_i)_I$ is a surjective inverse system of topological groups with $G_i\in \sf{TMD}$ for each $i\in I$, then $G:= \varprojlim_i G_i\in \sf{TMD}$.
\end{prop}

\begin{proof}
    By Lemma~7.2 of \cite{BassoZucker}, we have $\sa(G) = \varprojlim_i \sa(G_i)$ along the maps $\wt{\pi}_i^j\colon \sa(G_j)\to \sa(G_i)$. If $M\subseteq \sa(G)$ is minimal, then $\wt{\pi}_i[M]:= M_i\subseteq \sa(G_i)$ is minimal, and $\wt{\pi}_i^j[M_j] = M_i$. If $\sigma\in \rm{SN}(G_i)$, then Lemma~\ref{Lem:Topometric_Samuel_Hom} and the fact immediately below combine to yield that $\partial_{\sigma\circ \pi_i}^M = \partial_{\sigma}^{M_i}\circ (\wt{\pi}_i\times \wt{\pi}_i)$. For the moment, fix $i\in I$. By assumption $G_i \in \sf{TMD}$, so $\partial_\sigma^{M_i}$ has a compatibility point, hence $\partial_{\sigma\circ \pi_i}^M$ does as well. Now letting $i\in I$ vary, as $\{\sigma\circ \pi_i: \sigma\in \rm{SN}(G_i), i\in I\}\subseteq \rm{SN}(G)$ forms a base of semi-norms on $G$, we see from Theorem~\ref{Thm:Rosendal_Minimal}\ref{Item:ContBase_Thm:RM} that $G\in \sf{TMD}$. 
\end{proof}

\begin{corollary}
    The class $\sf{TMD}$ is preserved under arbitrary products. 
\end{corollary}

In \cite{BassoZucker}*{Proposition~7.7}, it is shown that if $H$ and $K$ are topological groups with $H\in \sf{CMD}$, then $\rmM(H\times K)\cong \rmM(H)\times \rmM(K)$. Furthermore, if $\la G_i: i\in I\ra$ are topological groups with $G_i\in \sf{CMD}$ for every $i\in I$, the universal minimal flow of their product is simply the product of the corresponding universal minimal flows. By \cite{BassoZucker}*{Theorem~5.8}, the first result is also sharp; if both $H, K\not\in \sf{CMD}$, then it is shown that $\rmM(H)\times \rmM(K)$ is not $(H\times K)$-ED. The next natural question is to understand situations where $\rmM(H\times K)\cong \rmS_{H\times K}(\rmM(H)\times \rmM(K))$. By allowing the ``correction'' of taking the Gleason completion, we obtain an exact analogue of \cite{BassoZucker}*{Proposition~7.7} for the class $\sf{TMD}$, though with a much more intricate proof.

\begin{theorem}\mbox{}
    \label{Thm:Products}
    \vspace{-2 mm}

    \begin{enumerate}[label=\normalfont(\arabic*)]
        \item \label{Item:Finite_Thm:Products}
        Suppose $H$ and $K$ are topological groups with $H\in \sf{TMD}$. Write $G = H\times K$. Then $\rmM(G) \cong \rmS_G(\rmM(H)\times \rmM(K))$.
        \item \label{Item:Infinite_Thm:Products}
        Suppose $\la G_i: i\in I\ra$ is a tuple  of topological groups with $G_i\in \sf{TMD}$ for every $i\in I$. Then writing $G = \prod_i G_i$, we have $\rmM(\prod_i G_i)\cong \rmS_G(\prod_i \rmM(G_i))$.
    \end{enumerate}  
\end{theorem}

\begin{proof}
\ref{Item:Finite_Thm:Products}: Viewing $\rmM(H)$ and $\rmM(K)$ as $G$-flows in the obvious way, fix $G$-maps $\pi_K\colon \rmM(G)\to \rmM(K)$ and $\pi_H\colon \rmM(G)\to \rmM(H)$. As $\rmM(K)$ is a Gleason complete $G$-flow, Proposition~\ref{Prop:MHP_Open} tells us that $\pi_K$ is open. By Lemma~\ref{Lem:Kflow}, $\rmM(K)\cong \rm{Min}_H(\rmM(G))$ equipped with the Meets topology. It follows that for every $p\in \rmM(K)$, $\pi_K^{-1}[p]:= X_p$ is an $H$-flow which contains a unique minimal $H$-subflow which we denote by $Y_p$. Since $\pi_H[Y_p] = \rmM(H)$, it follows that $\pi_H|_{Y_p}$ is an isomorphism of $H$-flows. 

\begin{claim}
    $X_p$ is a Vietoris limit of members of $\rm{Min}_H(\rmM(G))$. Furthermore, there is an ultracopower $\Sigma_\cU^H \rmM(H)$ and a $G$-map $\phi\colon \Sigma_\cU^H \rmM(H)\to X_p$ with $\pi_{\rmM(H), \cU} = \pi_H\circ \phi$.
\end{claim}

\begin{proof}[Proof of claim]
    The continuity of $\pi_K$ tells us that dealing with the ``contains'' part of the Vietoris topology is vacuous. So it is enough to prove that whenever $\{A_i: i< n\}\subseteq \op(\rmM(G))$ and $A_i\cap X_p\neq \emptyset$ for every $i< n$, then there is some $Y\in \rm{Min}_H(\rmM(G))$ with $A_i\cap Y\neq \emptyset$ for every $i< n$. For each $i< n$, consider the set $B_i:= \{q\in \rmM(K): A_i \cap Y_q\neq \emptyset\}$. We show that $B_i\subseteq \pi_K[A_i]$ is open dense. Density follows since $\rm{AP}_H(\rmM(G))\subseteq \rmM(G)$ is dense. For openness, if $(q_j)_{j\in J}$ is a net from $\rmM(K)$ with $q_j\to q\in \rmM(K)$ and $Y_{q_j}\cap A_i = \emptyset$ for each $j\in J$, then $Y_q$ is contained in any Vietoris cluster point of the net $Y_{q_j}$, so must also avoid $A_i$. Writing $D = \bigcap_{i< n} \pi_K[A_i]$, then $D\in \op(p, \rmM(K))$. It follows that $B = \bigcap_{i< n} B_i$ is dense in $D$. Picking any $q\in B\cap D$, the $H$-minimal flow $Y_q$ is as desired. 

    For the furthermore, let $(Y_i)_{i\in I}$ be a net from $\rm{Min}_H(\rmM(G))$ with $Y_i\to X_p$. For each $i\in I$, let $\phi_i\colon \rmM(H)\to Y_i$ be an $H$-flow isomorphism. Then $\bigcup_{i\in I} \phi_i\colon I\times \rmM(H)\to \rmM(G)$ continuously extends to $\alpha_H(I\times \rmM(H))$, and we let $\phi$ be the restriction of this extension to $\Sigma_\cU^H \rmM(H)$.
\end{proof}

It follows from Theorem~\ref{Thm:Rosendal_Minimal}\ref{Item:HP_Thm:RM} that $\pi_H|_{X_p}\colon X_p\to \rmM(H)$ is highly proximal. This implies that $\pi_H\times \pi_K\colon \rmM(G)\to \rmM(H)\times \rmM(K)$ is highly proximal. As highly proximal maps between minimal flows are irreducible, this finishes the proof. 
\vspace{3 mm}

\noindent
\ref{Item:Infinite_Thm:Products}: We have the following isomorphisms which will be explained after:
\begin{align*}
    \rmM(G) &\cong \varprojlim_{F\in \fin{I}} \rmM(\prod_{i\in F} G_i)\\
    &\cong \varprojlim_{F\in \fin{I}} \rmS_G(\prod_{i\in F} \rmM(G_i))\\
    &\cong \rmS_G(\varprojlim_{F\in \fin{I}} \prod_{i\in F} \rmM(G_i))\\
    &\cong \rmS_G(\prod_i \rmM(G_i)). 
\end{align*}
The first line follows from \cite{BassoZucker}*{Lemma~7.4}. To obtain the second line using item \ref{Item:Finite_Thm:Products}, we note that $\rmS_G(X\times Y)\cong \rmS_G(X\times \rmS_G(Y))$ for any $G$-flows $X$ and $Y$. We also briefly discuss obtaining the third line. Typically, one can only conclude that the Gleason completion of an inverse limit (line 3) factors onto the inverse limit of the Gleason completions (line 2). Here however, we know that both are minimal and line 2 is $\rmM(G)$.
\end{proof}

Given that the productivity result characterizes the class $\sf{CMD}$ \cite{BassoZucker}*{Corollary~7.8}, it is natural to hope that Theorem~\ref{Thm:Products} characterizes $\sf{TMD}$. 
\begin{question}
    \label{Que:Products}
    Suppose $H$ and $K$ are topological groups with $H, K\not\in \sf{TMD}$, and write $G = H\times K$. Do we have $\rmM(H\times K)\not\cong \rmS_G(\rmM(H)\times \rmM(K))$?
\end{question}
This seems to be a very challenging problem, as it requires constructing ``interesting'' minimal $G$-flows which somehow can't be built just from $H$-flows and $K$-flows. In upcoming work with Barto\v{s}ov\'a, Kwiatkowska, and the second author, and making use of recent results in \cite{FSZ}, it is shown that Question~\ref{Que:Products} has an affirmative answer in the case that $H$ and $K$ are countable discrete groups.

We finish the subsection by considering open subgroups. Let us quickly observe that there is no hope of transferring results up from a subgroup; indeed, if $G$ is countable discrete, then the trivial subgroup is open and extremely amenable. Going down, we show that the class $\sf{CMD}$ is preserved upon passing to open subgroups. This is known for Polish groups, see for example \cite{MR4549426}*{Fact 2.8}, but the general question was not addressed in \cite{BassoZucker}.

\begin{theorem}
\label{Thm:CMD_OpenSub}
    If $G$ is a topological group, $H \leq^c G$ is an open subgroup, and $G\in \sf{CMD}$, then $H\in \sf{CMD}$.
\end{theorem}

\begin{proof}
    By co-induction (see \cite{MR3156511}*{Lemma 13}) there is an $H$-map $\varphi$ from a minimal $G$-flow $Y$ onto $\rmM(H)$. 
    Since $G\in \sf{CMD}$, all points of $Y$ are $G$-isolated in $Y$, thus they are also $H$-isolated in $Y$.
    Fix a minimal $H$-subflow $Z$ of $Y$. 
    It is isomorphic to $\rmM(H)$, so $\varphi$ induces a retraction of $Y$ onto $Z$, therefore all points of $Z$, and thus $\rmM(H)$ are $H$-isolated.
\end{proof}

It was asked in \cite{MR4549426}*{Question 2.9} if the same is true for the class $\sf{GPP}$.
The machinery presented in that paper can be used to construct a counterexample.

\begin{prop}
    The class of $\sf{GPP}$ (hence also $\sf{TMD}$) groups is not closed under taking open subgroups.
\end{prop}
\begin{proof}
    We refer to \cite{MR4549426} for all notation and terminology. 
    Let $\Gamma$ be a countably infinite discrete group, and consider the left action by translations on itself. 
    By \cite{MR4549426}*{Corollary 1.5}, the kaleidoscopic group $\cK(\Gamma)$ is GPP. 
    If $x$ is a branching point of $W_\infty$, then $\cK(\Gamma)_x$ is an open subgroup of $\cK(\Gamma)$. 
    Suppose towards contradiction that $\cK(\Gamma)_x$ is GPP. Since there is a continuous surjective homomorphism $\cK(\Gamma)_x \to \Gamma$, by \Cref{Prop:Homomorphic_Image}, $\Gamma$ would be GPP, a contradiction (see \Cref{Prop:Polish_LC}).
\end{proof}

\section{Polish groups}
\label{Section:Polish}

In this section, we strengthen the main results of the previous section in the case that $G$ is Polish. Along the way, we present some general results about presyndetic subgroups. These arise since a Polish group has the generic point property exactly when it contains a closed, presyndetic, extremely amenable subgroup \cite{ZucMHP}.

\subsection{Presyndetic subgroups}

We quickly recall some notions from Section~\ref{Section:Background}. Given a topological group $G$ and $H\leq^c G$, we equip $G/H$ with its right uniformity, and $\sa(G/H)$ naturally becomes a $G$-flow which is universal among $G$-flows containing $H$-fixed-points. We call $H\leq^c G$ \emph{presyndetic} if for every $U\in \cN_G$, $UH$ is syndetic, i.e.\ there is $F\in \fin{G}$ with $FUH = G$. Let us record the following results from \cite{ZucMHP}.

\begin{fact}
\label{Fact:SGmodH}
    Let $G$ be a topological group and $H\leq^c G$. 
    \begin{enumerate}[label=\normalfont(\arabic*)]
        \item \label{Item:Min_Fact:SGmodH}
        $\sa(G/H)$ is minimal iff $H$ is presyndetic. 
        \item \label{Item:MG_Fact:SGmodH}
        $\sa(G/H)\cong \rmM(G)$ iff $H$ is presyndetic and extremely amenable.
    \end{enumerate}
\end{fact}

\begin{proof}
    \ref{Item:Min_Fact:SGmodH}: See \cite{ZucMHP}*{Proposition~6.6}. While the result there is stated for Polish groups, the proof in general is exactly the same. Note the switch in left/right conventions between this paper and \cite{ZucMHP}.
    \vspace{3 mm}

    \noindent
    \ref{Item:MG_Fact:SGmodH}: See \cite{ZucMHP}*{Theorem 7.5(1)}. Again, the extra assumptions that appear there are not needed for the given proof. 
\end{proof}

\begin{prop}
\label{Prop:Presynd_Inclusions}
    Let $K\leq^c H\leq^c G$ be topological groups. Then $K\leq^c G$ is presyndetic iff both $K\leq^c H$ and $H\leq^c G$ are presyndetic.
\end{prop}

\begin{proof}
    First assume $K\leq^c G$ is presyndetic. Then clearly $H\leq^c G$ is presyndetic. To see that $K\leq^c H$ is presyndetic, fix $U\in \cN_G$. There is $F\in \fin{G}$ with $FUK = G$. For each $g\in F$ with $gU\cap H\neq \emptyset$, fix $u_g\in U$ with $gu_g\in H$, and set $F'= \{gu_g: g\in F \text{ with } gU\cap H\neq \emptyset\}$. Given $h\in H$, there are $g\in F$, $u\in U$, and $k\in K$ with $h = guk$. Since $gu = hk^{-1}\in H$, we can write $h = gu_g(u_g^{-1}u)k$. Noting that $u_g^{-1}u\in H$, we have $F'(U^2\cap H)K = H$. 

    The other direction is \cite{MR4549426}*{Proposition 2.2}.
\end{proof}

Recall that given a topological group $G$, $\wt{G}$ denotes its Raikov completion, a topological group which contains $G$ as a dense subgroup. For every $G$-flow, the action extends to $\wt{G}$ continuously, so we can think of $G$ and $\wt{G}$ as having the exact same flows.

\begin{prop}
    \label{Prop:Presynd_Raikov}
    Given topological groups $H\leq^c G$, then $H\leq^c G$ is presyndetic iff $\wt{H}\leq^c \wt{G}$ is presyndetic.
\end{prop}

\begin{proof}
    Noting that $\wt{H}$ is simply the closure of $H$ in $\wt{G}$, we have natural inclusions of uniform spaces $G/H\hookrightarrow \wt{G}/\wt{H}\hookrightarrow \wh{G/H}$, from which it follows that $\wh{\wt{G}/\wt{H}}\cong \wh{G/H}$, so also $\sa(G/H)\cong \sa(\wt{G}/\wt{H})$ as $G$-flows. Hence $\sa(G/H)$ is a minimal $G$-flow iff $\sa(\wt{G}/\wt{H})$ is a minimal $\wt{G}$-flow.
\end{proof}

We take the opportunity to record some similar facts for co-precompact subgroups, where we recall that $H\leq^c G$ is \emph{co-precompact} iff $\sa(G/H) = \wh{G/H}$, i.e.\ iff for every $U\in \cN_G$, there is $F\in \fin{G}$ with $UFH = G$. 

\begin{prop}
    \label{Prop:CoPre_Inclusions}
    Let $K\leq^c H\leq^c G$ be topological groups. Then $K\leq^c G$ is co-precompact iff both $K\leq^c H$ and $H\leq^c G$ are co-precompact. 
\end{prop}

\begin{proof}
    First assume $K\leq^c G$ is co-precompact. Then clearly $H\leq^c G$ is co-precompact. The proof that $K\leq^c H$ is co-precompact is very similar to the first part of the proof of Proposition~\ref{Prop:Presynd_Inclusions}. 

    Now suppose both $H\leq^c G$ and $K\leq^c H$ are co-precompact. Fix $U\in \cN_G$. Find $F_G\in \fin{G}$ with $UF_GH = G$. Find $V\in \cN_G$ small enough so that $gV\subseteq Ug$ for every $g\in F_G$. Then find $F_H\in \fin{H}$ with $(V\cap H)F_HK = H$. Then $U^2(F_GF_H)K \supseteq UF_GVF_HK \supseteq  UF_GH = G$.
\end{proof}

\begin{prop}
    \label{Prop:CoPre_Raikov}
    Given topological groups $H\leq^c G$, then $H\leq^c G$ is co-precompact iff $\wt{H}\leq^c \wt{G}$ is co-precompact. 
\end{prop}

\begin{proof}
    As in the proof of Proposition~\ref{Prop:Presynd_Raikov}, we have a canonical isomorphism $\wh{\wt{G}/\wt{H}}\cong \wh{G/H}$.  
\end{proof}

\subsection{Comeager orbits}

Recall that $\sf{GPP}$ denotes the class of Polish groups with the \emph{generic point property}, i.e.\ such that every minimal flow has a comeager orbit. By Proposition~14.1 of \cite{AKL}, this is equivalent to the universal minimal flow having a comeager orbit. Our first task is to see how the Rosendal criterion relates to comeager orbits for actions of Polish groups. Recall that the \emph{$\pi$-weight} of a topological space $X$ is the smallest cardinality of a \emph{$\pi$-base} for $X$, a set $\cB\subseteq \op(X)$ such that for every $A\in \op(X)$, there is $B\in \cB$ with $B\subseteq A$.  
  
    \begin{theorem}
        \label{thm:Rosendal-general}
        Fix a Polish group $G$ and a Baire $G$-space $X$. The following are equivalent:
         \begin{enumerate}[label=(\arabic*)]
            \item \label{Item:Orbit_Fact:RG}
            $X$ has a comeager orbit.
            \item \label{Item:RC_Fact:RG}
            $X\in \cal{RC}_G$, is $G$-TT, and has countable $\pi$-weight. 
         \end{enumerate}
         If moreover $X$ has a comeager set of points with a dense orbit (in particular implying $X$ is $G$-TT), the above are also equivalent to:
         \begin{enumerate}[label=(\arabic*)]
          \setcounter{enumi}{2}
             \item \label{Item:RC-alone_Fact:RG}
            $X\in \cal{RC}_G$. 
            \item \label{Item:G-iso_Fact:RG}
            $X$ has a comeager set of $G$-isolated points.
         \end{enumerate}
    \end{theorem}
    \begin{proof}
        $\ref{Item:Orbit_Fact:RG} \Rightarrow \ref{Item:RC_Fact:RG}$: 
        Let $x_0\in X$ be such that $G\cdot x_0\subseteq X$ is comeager. We first show that $X$ has countable $\pi$-weight. For a moment, fix $U\in \cN_G$. As $G$ is Polish, countably many translates of $U$ cover $G$, implying that $U\cdot x_0$ is non-meager. As $X$ is Baire, we then have $\Int_X(\ol{U\cdot x_0})\neq \emptyset$.
        
        Fix a countable basis $\set{U_n : n \in \omega}$ at $1_G$ and a countable dense subgroup $D \sub G$. 
        Let $A \in \op(X)$ be given, and find $g \in D$ such that $gx_0 \in A$.
        Using regularity, find $n \in \omega$ such that $g\cdot \Int_X(\ol{U_n \cdot x_0}) \sub A$, 
        which shows that $X$ has countable $\pi$-weight. 
    
        Since $X$ has a dense orbit it is $G$-TT. The rest of the argument, showing that $X\in \cal{RC}_G$, follows $(i)\Rightarrow (ii)$ of Proposition~3.2 from \cite{BYMT} word for word; we note that here, one can drop the regularity assumption.
        \vspace{3 mm}
        
        \noindent
        $\ref{Item:RC_Fact:RG} \Rightarrow \ref{Item:Orbit_Fact:RG}$: This direction follows the proof in \cite{BYMT} word for word, replacing ``countable basis'' in the last sentence with ``countable $\pi$-base.'' 
        \vspace{3 mm}
        
        \noindent
        For the remaining directions, we assume that $X$ has a comeager set of points with a dense orbit. In particular, $X$ is $G$-TT.
        $\ref{Item:RC_Fact:RG} \Rightarrow \ref{Item:RC-alone_Fact:RG}$ is clear.
        \vspace{3 mm}
        
        \noindent
        $\ref{Item:RC-alone_Fact:RG} \Rightarrow \ref{Item:G-iso_Fact:RG}$: 
        By \Cref{lem:RC-equivs}, for each $V \in \cN_G$, $I_V \coloneqq \set{x \in X : \Int(Q_V[x]) \ne \emptyset}$ contains an open dense set.
        Since $\cN_G$ is countable and $X$ is Baire, $\bigcap_{V \in \cN_G} I_V$ is comeager, and is the set we were looking for. This implication doesn't use the hypothesis on the comeager set of points with dense orbit.
        \vspace{3 mm}
        
        \noindent
        $\ref{Item:G-iso_Fact:RG} \Rightarrow \ref{Item:RC_Fact:RG}$: 
        By \Cref{lem:RC-equivs}, it is enough to prove that $X$ has countable $\pi$-weight. 
        By hypothesis, there is some $x_0$ which is $G$-isolated and moreover has a dense orbit.
        Then $\set{g \cdot \Int(Q_V[x_0]) : g \in D, \, V \in \cN_G}$, for $D$ a countable dense subgroup of $G$, is a $\pi$-basis of $X$.
        To see this, let $A \in \op(X)$ be given, and find $g \in D$, $B \in \op(x_0, X)$, and $V \in \cN_G$ such that $\ol{VgB} \sub A$. 
        Then, if $U \in \cN_G$ is such that $U \sub g^{-1} V g$, we have $g \cdot \Int(Q_U[x_0]) \sub g \ol{UB} \sub \ol{VgB} \sub A$.
    \end{proof}
    
    \begin{remark}
        The argument of $\ref{Item:Orbit_Fact:RG}\Rightarrow \ref{Item:RC_Fact:RG}$ makes use of Effros's theorem, which is typically stated only in the metrizable setting, but is discussed in greater generality in Fact~2.5 of \cite{MR4549426}: If $G$ is a Polish group, $X$ is a $G$-space, and $x\in X$ has dense orbit, then $G\cdot x\subseteq X$ is non-meager iff, letting $G_x\subseteq G$ be the stabilizer of $x\in X$, the map from $G\cdot x$ to $G/G_x$ sending $gx$ to $gG_x$ is a homeomorphism.
    \end{remark}

    \begin{remark}
        We do not know of any example of a Polish group $G$ and a Baire $G$-space $X$ with a comeager orbit, but with uncountable $\pi$-weight. Such an example would necessarily be non-regular.
    \end{remark}

When $G$ is Polish, we can add to the equivalent conditions of Theorem~\ref{Thm:Rosendal_Minimal}. Some of these equivalences have appeared in Theorem~5.5 of \cite{ZucMHP}. Notice that the minimality assumption on $X$ can be weakened to ``has a comeager set of points with a dense orbit'', if we drop the presyndeticity requirement for $H$ in Item \ref{Item:Structure_Thm:PGWUEB}.

\begin{theorem}
\label{Thm:Polish_Group_WUEB}
    If $G$ is Polish and $X$ is a minimal $G$-ED flow, then the following are equivalent:
    \begin{enumerate}[label=\normalfont(\arabic*)]
        \item \label{Item:RC_Thm:PGWUEB}
        $X\in \cal{RC}_G$,
        \item \label{Item:Orbit_Thm:PGWUEB}
        $X$ has a comeager orbit.
        \item \label{Item:FC_Thm:PGWUEB}
        $X$ has a point of first-countability. 
        \item \label{Item:GIso_Thm:PGWUEB}
        $X$ has a $G$-isolated point.
        \item \label{Item:Compatibility_Thm:PGWUEB}
        For any compatible right-invariant metric $d$ on $G$, there is a $\partial_d$-compatible point. 
        \item \label{Item:Structure_Thm:PGWUEB}
        For some closed, presyndetic subgroup $H\leq G$, we have $X\cong \sa(G/H)$.
    \end{enumerate}
    In particular, for Polish $G$, we have $G\in \sf{TMD}$ iff $G\in \sf{GPP}$. 
\end{theorem} 

\begin{remark}
    If $X$ satisfies the conditions of Theorem~\ref{Thm:Polish_Group_WUEB}, then every point in the comeager orbit will be $G$-isolated, but the converse need not hold. 
\end{remark}
    
\begin{proof}    
    $\ref{Item:RC_Thm:PGWUEB}\Leftrightarrow \ref{Item:Orbit_Thm:PGWUEB}$ is \Cref{thm:Rosendal-general}.
    \vspace{3 mm}

    \noindent
    $\ref{Item:Orbit_Thm:PGWUEB}\Rightarrow \ref{Item:FC_Thm:PGWUEB}$: Any point in the comeager orbit is a point of first-countability. This is a consequence of the Effros theorem and the fact that if $X$ is regular, $Y$ is dense in $X$, and $\cA$ is a local base at $y$ in $Y$, then $\set{\Int_X(\rm{Cl}_X(A)) : A \in \cA}$ is a local base at $y$ in $X$. 
    \vspace{3 mm}

     \noindent
    $\ref{Item:FC_Thm:PGWUEB}\Rightarrow \ref{Item:GIso_Thm:PGWUEB}$ is \Cref{thm:ED-first-countable}.
    \vspace{3 mm}

    \noindent
    $\ref{Item:GIso_Thm:PGWUEB}\Rightarrow \ref{Item:Compatibility_Thm:PGWUEB}$ is clear.
    \vspace{3 mm}

    \noindent
    $\ref{Item:Compatibility_Thm:PGWUEB}\Rightarrow \ref{Item:RC_Thm:PGWUEB}$ follows from $\ref{Item:ContAll_Thm:RM} \Rightarrow \ref{Item:RC_Thm:RM}$ in \Cref{Thm:Rosendal_Minimal} (which still holds when weakening minimality to $G$-TT). 
     \vspace{3 mm}

    \noindent
    $\ref{Item:Orbit_Thm:PGWUEB}\Rightarrow  \ref{Item:Structure_Thm:PGWUEB}$: This is $(3) \Rightarrow (4)$ in \cite{ZucMHP}*{Theorem~5.5}. 
    \vspace{3 mm}

     \noindent
    $\ref{Item:Structure_Thm:PGWUEB}\Rightarrow  \ref{Item:Orbit_Thm:PGWUEB}$: The orbit $G/H$ is comeager in $\sa(G/H)$; more generally, any Polish space is \v{C}ech-complete, so is a $G_\delta$ subset of any compactification. 
\end{proof}

Together with \Cref{thm:Rosendal-general}, we have the following dichotomy. 
\begin{corollary}
    \label{cor:dichotomy-polish-pi-weight}
    If $G$ is Polish, then any minimal $G$-ED flow has either: 
    \begin{itemize}
        \item countable $\pi$-weight, or 
        \item has no points of first-countability. 
    \end{itemize}
    In particular, if $G \in \sf{GPP}$, then all minimal flows have countable $\pi$-weight. 
\end{corollary}

We ask whether the converse to the last sentence of \Cref{cor:dichotomy-polish-pi-weight} holds.

\begin{question}
    \label{Que:Polish_PiWt}
    Let $G$ be a Polish group with $G\not\in \sf{GPP}$. Must $\rmM(G)$ have uncountable $\pi$-weight? Continuum $\pi$-weight?
\end{question}

Question~\ref{Que:Polish_PiWt} is known to have an affirmative answer in the case that $G$ is countable discrete \cite{GTWZ}, but in full generality this question seems quite difficult.

An affirmative answer to the question would imply that a Polish group is $\sf{GPP}$ iff  $\rmM(G)$ have countable $\pi$-weight iff every minimal $G$-flow has countable $\pi$-weight.
To add some context to this question, we offer the following characterization of flows of separable groups with countable $\pi$-weight. 

\begin{prop}
    \label{Prop:PiW}
    Let $G$ be a separable group and $X$ a $G$-flow. Then $X$ has countable $\pi$-weight iff $X$ is an irreducible extension of a metrizable $G$-flow.
\end{prop}

\begin{proof}
    Since irreducible maps between compact spaces preserve $\pi$-weight, one direction is clear. For the other, fix $D\subseteq G$ a countable dense subgroup and $\{B_n: n< \omega\}\subseteq \op(X)$ a countable $\pi$-base for $X$. For each $n< \omega$, let $f_n\colon X\to [0, 1]$ be a non-zero function with $\rm{supp}(f_n)\subseteq B_n$. Considering the right $G$-action on $\rm{CB}(X)$ discussed in Subsection~\ref{Subsection:Top_Dyn}, let $\cA\subseteq \rm{CB}(X)$ be the smallest $D$-invariant sub-Banach-lattice of $\rm{CB}(X)$ with $\{f_n: n< \omega\}\subseteq \cA$. Letting $Y$ denote the Kakutani dual of $\cA$, then $Y$ is metrizable since $\cA$ is separable, and the inclusion $\cA\subseteq \rm{CB}(X)$ gives us a factor map $\pi\colon X\to Y$. To see that $\pi$ is irreducible, we recall that open subsets of $Y$ are in one-one correspondence with closed ideals of $\rm{CB}(Y)$ by identifying an open set $C$  with the set of continuous functions $I_C$ supported on that open set. Hence we want to show that for any $B\in \op(X)$, there is a closed ideal $J\subseteq \rm{CB}(Y)$ with $J\subseteq I_B$. For some $n< \omega$, we have $B_n\subseteq B$, and we can let $J\subseteq \rm{CB}(Y)$ be the closed ideal generated by $f_n$.
\end{proof}

Many interesting properties of $G$-flows are preserved under irreducible extensions or irreducible factors. \Cref{Prop:PiW} thus tells us that, for separable groups, flows with countable $\pi$-weight can be  understood, to a great extent, by only studying metrizable flows. 
\bigskip

We recover one of the results from \cite{ZucMHP}.
\begin{corollary}
    For Polish $G$, we have $G\in \sf{GPP}$ iff there is $H\leq^c G$ which is presyndetic and extremely amenable.
\end{corollary}
\begin{proof}
    Combine Theorem~\ref{Thm:Polish_Group_WUEB} and Fact~\ref{Fact:SGmodH}.
\end{proof}

We can also recover the main results from \cite{MNVTT} and \cite{BYMT} about Polish $\sf{CMD}$ groups.

\begin{theorem}
    \label{Thm:Polish_CMD}
    Let $G$ be a Polish group. Then the following are equivalent:
    \begin{enumerate}[label=\normalfont(\arabic*)]
        \item \label{Item:CMD_Thm:PCMD}
        $G\in \sf{CMD}$, i.e.\ $\rmM(G)$ is metrizable.
        \item \label{Item:CoPre_Thm:PCMD}
        There is $H\leq^c G$ which is co-precompact and extremely amenable.
        \item \label{Item:PreSynd_Thm:PCMD}
        There is $H\leq^c G$ which is co-precompact, presyndetic, and extremely amenable.
    \end{enumerate}
\end{theorem}

\begin{remark}
    In Item~\ref{Item:PreSynd_Thm:PCMD}, note that by Fact~\ref{Fact:SGmodH} Item~\ref{Item:MG_Fact:SGmodH}, we have $\rmM(G)\cong \wh{G/H}$.
\end{remark}

\begin{proof}
    $\ref{Item:PreSynd_Thm:PCMD} \Rightarrow \ref{Item:CoPre_Thm:PCMD}$: Clear.
    \vspace{3 mm}
    
    \noindent
    $\ref{Item:CMD_Thm:PCMD}\Rightarrow \ref{Item:PreSynd_Thm:PCMD}$: As $G\in \sf{CMD}$, also $G\in \sf{TMD}$, so by Theorem~\ref{Thm:Polish_Group_WUEB}\ref{Item:Structure_Thm:PGWUEB}, we have $\rmM(G)\cong \sa(G/H)$ for some presyndetic, extremely amenable $H\leq^c G$. Since $\sa(G/H)$ is assumed metrizable, we have $H\leq^c G$ co-precompact as well.
    \vspace{3 mm}

    \noindent
    $\ref{Item:CoPre_Thm:PCMD}\Rightarrow \ref{Item:CMD_Thm:PCMD}$: Since $H$ is extremely amenable, $\sa(G)$ has an $H$-fixed point. Thus there is a $G$-map $\phi\colon \wh{G/H}\to \sa(G)$. Hence any minimal subflow of the metrizable flow $\wh{G/H}$ must be isomorphic to $\rmM(G)$.
\end{proof}

We end the section with three propositions (\ref{Prop:Polish_LC}, \ref{Prop:Polish_Presynd}, and \ref{Prop:Polish_CoPre}) that we will eventually prove for all topological groups, but the proof of the general result will use the Polish case via a forcing and absoluteness argument. Our first proposition concerns locally compact non-compact Polish groups, where we obtain a new result about the underlying space of $\mg$. 

\begin{prop}
    \label{Prop:Polish_LC}
    Let $G$ be a Polish, locally compact group. If $G \in \sf{GPP}$, then $G$ is compact. In particular, if $G$ is non-compact, then $\mg$ contains no points of first countability. 
\end{prop}

\begin{proof}
    By \Cref{Thm:Polish_Group_WUEB}, $\mg \cong \sa(G/H)$ for some extremely amenable, presyndetic, closed subgroup $H \le^c G$. 
    Since the only extremely amenable locally compact group is the trivial group, this implies that $G$ is covered by finitely many translates of any open neighborhood of the identity, i.e., is precompact. But it is Polish, hence Raikov complete and therefore compact. \qedhere

\end{proof}

Moreover, we can strengthen the dichotomies of \Cref{cor:Polish-dichotomy} and \Cref{cor:dichotomy-polish-pi-weight} when assuming that $G$ is locally compact, and the flow is minimal. 

\begin{corollary}
    \label{cor:dichotomy-polish-lc}
    If $G$ is a Polish, locally compact group, then any minimal $G$-ED flow $X$ is either:
    \begin{itemize}
        \item metrizable, i.e., second-countable, or
        \item has no points of first-countability.
    \end{itemize}
    Moreover, $X$ is metrizable if and only if the action is transitive. 
\end{corollary}
\begin{proof}
    Suppose that $x \in X$ is a point of first-countability. 
    By \Cref{Thm:Polish_Group_WUEB}, $X$ has a comeager orbit $Z$. 
    Since $G$ is locally compact, $Z$ is $K_\sigma$, so its complement is $G_\delta$.
    If the complement were nonempty, it would consist of a union of orbits, so it would also be dense, hence comeager, a contradiction.
    Therefore, there is a unique orbit. By the Effros theorem, $X \cong G/H$, for some $H \le^c G$, so $X$ is metrizable.  
\end{proof}

The second proposition shows that presyndetic subgroups have tractable minimal dynamics iff the ambient group does.

\begin{prop}
    \label{Prop:Polish_Presynd}
    Let $G$ be Polish and $H\leq^c G$ be presyndetic. Then $G\in \sf{GPP}$ iff $H\in \sf{GPP}$. 
\end{prop}

\begin{proof}
    If $H\in \sf{GPP}$, then there is $K\leq^c H$ a presyndetic, extremely amenable subgroup. By Proposition~\ref{Prop:Presynd_Inclusions}, $K\leq^c G$ is presyndetic, so $G\in \sf{GPP}$. 

    If $G\in \sf{GPP}$, then there is $K\leq^c G$ a presyndetic, extremely amenable subgroup. Then $\rmM(G)\cong \sa(G/K)$, so let $\phi\colon \sa(G/K)\to \sa(G/H)$ be a $G$-map. By \cite{AKL}*{Proposition~14.1}, $\phi(K)\in G/H$, say $\phi(K) = gH$. As $\phi$ is a $G$-map, for any $k\in K$, we have $kgH = gH$, implying $g^{-1}Kg\leq^c H$. As $g^{-1}Kg\cong K$ is extremely amenable, we have $H\in \sf{GPP}$.
\end{proof}

The last result shows that if a co-precompact subgroup has concrete minimal dynamics, then so does the ambient group.

\begin{prop}
    \label{Prop:Polish_CoPre}
    Let $G$ be Polish and $H\leq^c G$ be co-precompact. If $H\in \sf{CMD}$, then $G\in \sf{CMD}$. If $G$ is $\sf{CMD}$ and $H$ is also presyndetic, then $H\in \sf{CMD}$.
\end{prop}

\begin{remark}
    For $G = S_\infty$, a subgroup $H\leq^c S_\infty$ is co-precompact exactly when it is \emph{oligomorphic}, i.e.\ for every $n< \omega$, the natural $G$-action on $\omega^n$ has finitely many orbits. For an example of an oligomorphic group $G$ with $G\not\in \sf{TMD}$, see \cite{evans2019automorphism}. 
\end{remark}

\begin{proof}
    If $H\in \sf{CMD}$, we have by Theorem~\ref{Thm:Polish_CMD} that there is $K\leq^c H$ a co-precompact, extremely amenable subgroup. By Proposition~\ref{Prop:CoPre_Inclusions}, we have $K\leq^c G$ co-precompact.

    If $G\in \sf{CMD}$ and $H\leq^c G$ is both co-precompact and presyndetic, the argument that $H\in \sf{CMD}$ is very similar to that of the second part of Proposition~\ref{Prop:Polish_Presynd}.
\end{proof}

\section{The automorphism group of $\rmM(G)$}
\label{Section:AutMG}

This section considers the \emph{tau topology} on the automorphism group of $\rmM(G)$, which we denote $\bbG$, and more generally on the automorphism group of any \emph{minimal ideal flow} (introduced by Auslander as ``regular'' flows \cite{Aus_Reg}). Defined by Furstenberg \cite{Furstenberg_Distal} and further explored in \cites{Ellis_book, EGS_1975, Glasner_Prox, Auslander, KP_2017, Rze_2018}, this is a compact T1 topology on $\bbG$ with separately continuous multiplication and continuous inversion. We show that the tau-topology on $\bbG$ can be characterized using ultracopowers of $\rmM(G)$ (Theorem~\ref{Thm:Build_M_Ultralimit}). This characterization combined with Theorem~\ref{Thm:Rosendal_Minimal}\ref{Item:Ultra_Thm:RM} yields that for $G\in \sf{TMD}$, the tau-topology on $\bbG$ is Hausdorff. We then give a structural characterization of when the \emph{Ellis group} of a minimal flow has Hausdorff tau-topology (Theorem~\ref{Thm:Tau_Hausdorff}), thus yielding a structure theorem for $\rmM(G)$ when $G\in \sf{TMD}$. For $\sf{GPP}$ groups, we provide a stronger structure theorem (Theorem~\ref{Thm:MG_Structure_GPP}), and in Section~\ref{Section:Abs}, we will use set-theoretic methods to show that the strong version holds for all $\sf{TMD}$ groups.

\subsection{Minimal ideal flows and the tau topology}
\label{Subsection:MinIdealFlows}

We review some dynamical background from \cites{Aus_Reg, Auslander, Glasner_Prox, KP_2017, Rze_2018}, though some aspects of our presentation are new, in particular Proposition~\ref{Prop:Tau_Base}. Even though many of those references assume that the acting group $G$ is discrete, all of the results that we will refer to are valid even for $G$ a topological group; simply replace all instances of $\beta G$ with $\sa(G)$.

Fix a group $G$; all flows mentioned in this section are $G$-flows unless explicitly mentioned otherwise. Historically, a minimal flow $X$ is called ``regular'' if for any $x, y\in X$, there is $p\in \aut(X)$ with the pair $(p(x), y)$ proximal \cite{Aus_Reg}. Since the word ``regular'' is so over-loaded, we rename this property of $X$ and call $X$ a \emph{minimal ideal flow}. The reason for this name is that we can start with any flow $X$ and form its enveloping semigroup 
$$\rmE(X) := \ol{\{(gx)_{x\in X}: g\in G\}}\subseteq X^X.$$
Then $\rmE(X)$ is a CRTS under function composition, and there is a canonical semigroup homomorphism $G\to \rmE(X)$ with dense image. Given $p\in \rmE(X)$ and $x\in X$, we typically write $px$ or $p\cdot x$ for $p(x)$, and given a net $(g_i)_{i\in I}$ from $G$, we write $g_i\to p$ if for every $x\in X$, we have $g_ix\to px$. By standard facts about compact right-topological semigroups \cite{HS_Alg_SC}, any minimal subflow $M\subseteq \rmE(X)$ (equivalently, minimal left ideal) is a minimal ideal flow, and all minimal subflows of $\rmE(X)$ are isomorphic. Starting with a minimal flow $X$, we define its \emph{enveloping ideal} $M_X$ to be any minimal flow isomorphic to a minimal subflow of $\rmE(X)$.

\begin{defin}
    Given a minimal flow $X$, the \emph{Ellis group} of $X$ is the group $\bbG_{X}:= \aut(M_X)$. When $X$ is a minima ideal flow, we freely identify $\bbG_X$ with $\aut(X)$. When $X = \rmM(G)$, we just write $\bbG$. Given $X$ a minimal flow, $M$ a minimal ideal flow, and $\pi\colon M\to X$ a $G$-map, we write $\bbG_{\pi} = \{p\in \bbG_M: \pi\circ p = \pi\}$. 
\end{defin}

\begin{fact}
\label{Fact:MinIdealFlow}
    Fix $X$ a minimal flow and $M$ a minimal ideal flow.
    \begin{itemize}
        \item 
        One can find a set $I$ and a minimal subflow $Z\subseteq X^I$ with $Z\cong M_X$. Call $S\subseteq X$ \emph{almost periodic} (AP) if whenever $I$ is a set and $f\in X^I$ satisfies $f[I] = S$, then $\ol{G\cdot f}$ is minimal. A set $S\subseteq X$ is AP iff every finite subset of $S$ is, so by Zorn's lemma, there are maximal AP sets. If $S\subseteq X$ is maximal AP, $I$ is a set, and $f\colon I\to X$ satisfies $f[I] = S$, we have $M_X\cong \ol{G\cdot f}$. If $M\subseteq \rmE(X)$ is a minimal subflow, then there is an idempotent $u\in M$ with $uX = S$. 
        \item 
        If $Y$ is another minimal flow and $\phi\colon X\to Y$ is a $G$-map, then there is a $G$-map $\wt{\phi}\colon M_X\to M_Y$ (which doesn't really depend on $\phi$).  
        \item 
        $X$ is a minimal ideal flow iff $X\cong M_X$. In this case, the maximal AP subsets of $X$ are exactly $\bbG_X$-orbits. 
        \item 
        Consider a \emph{proximal set} $S\subseteq M_X$, i.e.\ for any finite $S_0\subseteq S$ and $A\in \op(M_X)$, there is $g\in G$ with $gS_0\subseteq A$, and fix $y\in S$. Let $u\in \rmE(X)$ be a minimal idempotent with $ux = y$ for every $x\in S$. Write $N = \ol{G\cdot u}$, so that $\rho_x\colon N\to M_X$ given by $\rho_x(p)= px$ is an isomorphism. We claim that if $p\in N$ and $\rho_x(p)\in S$, then $p$ is idempotent. We have $$\rho_x(pp) = \rho_x(pup) = pu\rho_x(p) = py = p\rho_x(u) = \rho_x(pu) = \rho_x(p)$$
        and $\rho_x$ is injective on $N$. 
        Conversely, any set $J\subseteq N$ consisting of idempotents is a proximal set. To summarize this discussion, the proximal subsets of $M_X$ are exactly the ``potentially idempotent'' subsets.

        Going forward, we often identify $M$ with some arbitrary choice of minimal subflow of $\rmE(M)$, and by the above, we have flexibility in which elements become idempotents under this identification.
        \item 
        If $\phi\colon M\to X$ and $\psi\colon M\to X$ are $G$-maps, then there is $p\in \bbG_M$ with $\psi\circ p = \phi$. Hence we can identify the set of $G$-maps from $M$ to $X$ with the right coset space $\bbG_{\phi}\backslash \bbG_M$.
        \item 
        If $\pi\colon M\to X$ is a $G$-map, then $X$ is a minimal ideal flow iff for all $r, s\in M$ with $\pi(r) = \pi(s)$ and $p\in \bbG_M$, we have $\pi\circ p(r) = \pi\circ p(s)$. This implies that $\bbG_\pi\trianglelefteq \bbG_M$ and  $\bbG_X \cong \bbG_M/\bbG_\pi$, and we write $\wt{\pi}\colon \bbG_M\to \bbG_X$ for the quotient.

        In particular, if $\pi_0, \pi_1\colon M\to X$ are both $G$-maps, then there is $q\in \bbG_X$ with $q\circ \pi_0 = \pi_1$, so $\pi_0$ and $\pi_1$ induce the same equivalence relation on $M$. Hence when discussing $G$-maps between minimal ideal flows, we sometimes refer to ``the canonical'' $G$-map to emphasize that the choice of $G$-map doesn't really matter.
        
        \item 
        If $M_X\subseteq \rmE(X)$ and $x_0\in X$, we obtain a $G$-map $\rho_{x_0}\colon M_X\to X$ given by $\rho_{x_0}(p) = px_0$. Fixing an arbitrary idempotent $u\in M_X$, then $\pi\colon M_X\to X$ is a $G$-map iff $\pi = \rho_{x_0}$ for some $x_0\in X$ iff $\pi = \rho_{x_0}$ for some $x_0$ with $ux_0 = x_0$. Hence for any $r\neq s\in M_X$, there is $p\in \bbG_M$ with $\pi\circ p(r)\neq \pi\circ p(s)$. Thus $\bbG_\pi\subseteq \bbG_M$ is malnormal, i.e.\ $\bigcap_{p\in \bbG_M} p^{-1}\bbG_\pi p = \{e_{\bbG_M}\}$.

        \item 
        View $M\subseteq \rmE(M)$, and let $u\in M$ be idempotent. Then by the previous bullet, every $G$-map $\pi\colon M\to M$ has the form $\rho_{p}$ for some $p\in uM$. As $uM$ is a group, we obtain that $M$ is \emph{coalescent}, meaning that every $G$-map $\pi\colon M\to M$ is an automorphism. We have $\bbG_M^{opp}\cong uM$, where ``opp'' stands for opposite group. 
        
        \item 
        If $\phi\colon M\to X$ and $\psi\colon M_X\to X$ are $G$-maps, then there is a $G$-map $\pi\colon M\to M_X$ with $\psi\circ \pi = \phi$.
    \end{itemize}
\end{fact}

Fix a $G$-flow $X$, $M\subseteq \rmE(X)$ a minimal subflow, and $u\in M$ an idempotent. On $uM\cong \bbG_M^{opp}$, one can define the \emph{tau-topology}, a compact T1 topology with separately continuous multiplication and continuous inversion. An equivalent topology was defined by Furstenberg in \cite{Furstenberg_Distal}, though for our presentation we refer to  \cites{Glasner_Prox, Auslander, KP_2017}. 
\begin{defin}
    If $A\subseteq M$ and $p\in \rmE(X)$, we define $p\bullet A\in \rmK(M)$ where given $x\in M$, we declare $x\in p\bullet A$ iff there are nets $(g_i)_{i\in I}$ from $G$ and $(a_i)_{i\in I}$ from $A$ with $g_i\to p$ and $g_ia_i\to x$. Given $A\subseteq uM$, one sets $\rm{cl}_\tau(A) = (u\bullet A)\cap uM$.
\end{defin}
\begin{fact} \label{Fact:Circle} Fix $A, B\subseteq M$ and $p, p_i, q\in \rmE(X)$.
    \begin{enumerate}[label=\normalfont(\arabic*)]
        \item 
        $p\bullet (A\cup B) = p\bullet A\cup p\bullet B$.
        \item 
        $pA\subseteq p\bullet A$.
        \item 
        If $p_i\to p$ and $p_i\bullet A\to K\in \rmK(M)$ in the Vietoris topology, then $K\subseteq p\bullet A$.
        \item 
        $p\bullet(q\bullet A)\subseteq pq\bullet A$.
        \item \label{Item:Tau_Fact:Circle}
        if $A\subseteq uM$, we have $\rm{cl}_\tau(A) = u(u\bullet A)\supseteq A$, and $\rm{cl}_\tau$ is a topological closure operator on $uM$.
    \end{enumerate}
\end{fact}
\begin{remark}
    When $\rmE(X)\cong \beta G$, then the ``revised circle operation'' from \cite{KP_2017} is just the extended action of $\beta G$ on $\rmK(M)$. In this situation -- or more generally if $\rmE(X)\cong \rmE(\rmK(X))$ -- we can improve some of the items above; in particular, $p_i\bullet A\to p\bullet A$ and $p\bullet (q\bullet A) = pq\bullet A$. Be cautious when reading references primarily focused on $\rmM(G)$, such as \cite{Auslander}, as they will sometimes make use of these stronger properties.
\end{remark}

A priori, given a minimal ideal flow $M$, it is not clear to what extent the tau-topology on $\bbG_M$ depends on the various choices made along the way. Auslander \cite{Auslander}*{p.\ 148} gives an alternate definition of the tau-topology for $M = \rmM(G)$ which is completely intrinsic to the dynamics of $M$ and works just fine for any minimal ideal flow. Given $p\in \bbG_M$, let $\rm{Graph}(p) = \{(x, p(x)): x\in M\}\subseteq M\times M$, and given $A\subseteq \bbG_M$, let $\rm{Graph}(A) = \bigcup_{q\in A} \rm{Graph}(q)$. If $M\subseteq \rmE(X)$ is a minimal subflow, $u\in M$ is idempotent, and $p\in uM$, we put $\rm{Graph}(p)  = \rm{Graph}(\rho_p)$, and similarly for $A\subseteq uM$.

\begin{prop}
\label{Prop:Tau_Graph}
     Given $M\subseteq \rmE(X)$ as above, $u\in M$ idempotent, $p\in uM$, and $A\subseteq uM$, the following are equivalent.
     \begin{enumerate}[label=\normalfont(\arabic*)]
         \item \label{Item:Tau_Prop:TauGraph}
         $p\in \rm{cl}_\tau(A)$, 
         \item \label{Item:GraphNets_Prop:TauGraph}
         There are nets $(g_i)_{i\in I}$ from $G$ and $(a_i)_{i\in I}$ from $A$ with $g_iu\to u$ and $g_ia_i\to p$,
         \item \label{Item:GraphCont_Prop:TauGraph}
          $\rm{Graph}(p)\subseteq \ol{\rm{Graph}(A)}$,
          \item \label{Item:GraphMeet_Prop:TauGraph}
          $\rm{Graph}(p)\cap \ol{\rm{Graph}(A)} \neq \emptyset$. 
     \end{enumerate}
\end{prop}

\begin{proof}
    $\ref{Item:GraphCont_Prop:TauGraph}\Leftrightarrow\ref{Item:GraphMeet_Prop:TauGraph}$: Immediate from the minimality of $M$. 
    \vspace{3 mm}
    
    $\ref{Item:GraphNets_Prop:TauGraph}\Leftrightarrow \ref{Item:GraphCont_Prop:TauGraph}$: We have $\rm{Graph}(p)\subseteq \ol{\rm{Graph}(A)}$ iff $(u, up) = (u, p)\in \ol{\rm{Graph}(A)}$. Since $\ol{\rm{Graph}(A)} = \ol{\{(gu, gq): g\in G: q\in A\}}$, we have $(u, p)\in \ol{\rm{Graph}(A)}$ iff \ref{Item:GraphNets_Prop:TauGraph} holds.
    \vspace{3 mm}
    
    $\ref{Item:Tau_Prop:TauGraph}\Leftrightarrow \ref{Item:GraphNets_Prop:TauGraph}$: If $p\in \rm{cl}_\tau(A)$ witnessed by $(g_i)_{i\in I}$ and $(a_i)_{i\in I}$, then $g_i\to u$ implies $g_iu\to uu = u$. Conversely, suppose there are nets $(g_i)_{i\in I}$ from $G$ and $(a_i)_{i\in I}$ from $A$ with $g_iu\to u$ and $g_ia_i\to p$. We construct nets witnessing that $p\in \rm{cl}_\tau(A)$ with index set $\op(u, \rmE(X))\times \op(p, M)$, ordered by reverse inclusion. Fix $U\in \op(u, \rmE(X))$ and $V\in \op(p, M)$. Let $(h_j)_{j\in J}$ be a net from $G$ with $h_j\to u$. Fix a large enough $i\in I$ with $g_iu\in U$ and $g_ia_i = g_iua_i\in V$. Then find large enough $j\in J$ with $g_ih_j\in U$ and $g_ih_ja_i\in V$. We set $g_{U, V} = g_ih_j$ and $q_{U, V} = q_i$.  
\end{proof}

Fix a minimal ideal flow $M$. Given  $p\in \bbG_M$ and $A, B\in \op(M)$, we define 
\begin{align*}
\rm{Nbd}_{A, B} &:= \{q\in \bbG_M: q[A]\cap B\neq \emptyset\}\\
\rm{Nbd}_A(p) &:= \rm{Nbd}_{A, p[A]}.
\end{align*}

\begin{prop}
    \label{Prop:Tau_Base}
    For any $A, B\in \op(M)$, $\rm{Nbd}_{A, B}$ is tau-open. For any $w\in M$ and $p\in \bbG_M$, the set $\{\rm{Nbd}_A(p): A\in \op(w, M)\}$ is a base of tau-open neighborhoods of $p$.
\end{prop}

\begin{proof}
    If $q\in \rm{Nbd}_{A, B}$, find $C\in \op(A)$ with $q[C]\subseteq B$. Then if $x\in C$, we have $(x, q(x))\not\in \ol{\rm{Graph}(\bbG_M\setminus \rm{Nbd}_{A, B})}$. Hence $\rm{Nbd}_{A, B}$ is tau-open. Now fix $w\in X$, and suppose $Q\subseteq \bbG_X$ is any set with $p\in \rm{cl}_\tau(Q)$. By minimality of $M$, this happens iff $(w, p(w))\in \ol{\rm{Graph}(Q)}$, which happens iff for every $A\in \op(w, M)$, we have $Q\cap \rm{Nbd}_A(p)\neq \emptyset$. 
\end{proof}

A \emph{compact T1 group} is a compact T1 space equipped with a group operation with separately continuous multiplication, continuous inversion, and such that the product of two closed sets is closed.

\begin{corollary}
    \label{Cor:Tau_Cont}
    If $M$ is a minimal ideal flow, then $\bbG_M$ equipped with the tau-topology is a compact T1 group. 
\end{corollary}

\begin{proof}
    Points are clearly closed. Compactness follows from Fact~\ref{Fact:Circle}\ref{Item:Tau_Fact:Circle}; one can show that if $p_i\in uM$ and $p_i\to q\in M$, then $p_i\xrightarrow{\tau} uq$. Separately continuous multiplication and continuous inversion are immediate from Proposition~\ref{Prop:Tau_Base}; given $p, q\in \bbG_M$ and $A\in \op(M)$, we have $q\cdot \rm{Nbd}_A(p) = \rm{Nbd}_A(qp)$, $\rm{Nbd}_A(p)\cdot q = \rm{Nbd}_{q^{-1}[A]}(pq)$, and $(\rm{Nbd}_A(p))^{-1} = \rm{Nbd}_{p[A]}(p^{-1})$.

    We refer to \cite{Auslander}*{p.\ 198, Lemma 4} for the argument that the product of two tau-closed sets is tau-closed.
\end{proof}

Going forward, topological vocabulary applied to $\bbG_M$ always refers to the tau-topology.

\begin{prop}
\label{Prop:Tau_Maps}
    Fix $X$ a minimal flow, $M$ a minimal ideal flow, and $\pi\colon M\to X$ a $G$-map.
    \begin{enumerate}[label=\normalfont(\arabic*)]
        \item \label{Item:Closed_Prop:TauMaps} 
        $\bbG_\pi\subseteq \bbG_M$ is tau-closed.

        \item \label{Item:Hom_Prop:TauMaps}
        If $X$ is also a minimal ideal flow and $\wt{\pi}\colon \bbG_M\to \bbG_X$ is the homomorphism induced by $\pi$, then $\wt{\pi}$ is continuous (in fact a quotient map).
    \end{enumerate}
\end{prop}

\begin{proof}
     \ref{Item:Closed_Prop:TauMaps}: To see this, write $R_\pi = \{(y, z)\in M\times M: \pi(y) = \pi(z)\}$. Given $p\in \bbG_M$, we have $p\in \bbG_\pi$ iff $\rm{Graph}(p)\subseteq R_\pi$. As $R_\pi$ is closed, we have $\ol{\rm{Graph}(\bbG_\pi)}\subseteq R_\pi$.
     \vspace{3 mm}

     \ref{Item:Hom_Prop:TauMaps}: Suppose $Q\subseteq \bbG_X$ is closed and $p\in \rm{cl}_\tau(\wt{\pi}^{-1}[Q])$. This happens iff for some (any) $y\in M$, we have $(y, p(y))\in \ol{\rm{Graph}(\wt{\pi}^{-1}[Q])}$. This implies $(\pi(y), \wt{\pi}(p)(\pi(y)))\in \ol{\rm{Graph}(Q)}$. So $\wt{\pi}(p)\in Q$, i.e.\ $p\in \wt{\pi}^{-1}[Q]$. We refer to \cite{Rze_2018}*{Prop.\ 5.41} for the proof that $\wt{\pi}$ is a quotient map. 
\end{proof}

We note that when $\bbG_X$ is Hausdorff, the argument that $\wt{\pi}$ is a quotient map is much simpler, as then $\bbG_X$ and $\bbG_M/\bbG_\pi$ are compact Hausdorff spaces with one finer than the other, so the topologies must coincide.

Towards proving Theorem~\ref{Thm:Tau_Hausdorff}, we will need to review some more content from \cite{Auslander}*{Ch.\ 14}. Given a compact T1 group $\Gamma$, let $\cN_\Gamma$ denote a base of $\tau$-open nbds of $e_\Gamma$. The \emph{Hausdorff kernel} of $\Gamma$ is defined via
$$\rmH(\Gamma) := \bigcap_{U\in \cN_\Gamma} U^2.$$
This is a subgroup preserved under any topological automorphism of $\Gamma$ (in particular it is normal), and $\Gamma/\rmH(\Gamma)$ is a compact Hausdorff group. Furthermore, given any tau-closed subgroup $K\subseteq \Gamma$, we have that $\Gamma/K$ equipped with the quotient topology is Hausdorff iff $\rmH(\Gamma)\subseteq K$. In particular, we have:

\begin{lemma}
\label{Lem:Haus}
    Let $M$ be a minimal ideal flow, $X$ a minimal flow, and $\pi\colon M\to X$ a $G$-map. Then $\bbG_X$ is Hausdorff iff $\rmH(\bbG_M)\subseteq \bbG_\pi$.
\end{lemma}

\begin{proof}
    Writing $\bbH_\pi = \bigcap_{p\in \bbG_M} p^{-1}\bbG_\pi p$, we have $\bbG_X\cong \bbG_M/\bbH_\pi$, and $\rmH(\bbG_M)\subseteq \bbG_\pi$ iff $\rmH(\bbG_M)\subseteq \bbH_\pi$.
\end{proof}

The tau-topology allows for the development of a detailed structure theory for minimal flows. This allows us to give a complete characterization of those minimal flows with Hausdorff Ellis group. To state it, we need some definitions.

\begin{defin} \label{Def:Extension_Props}
    Fix minimal flows $X$ and $Y$ and a $G$-map $\phi\colon X\to Y$.
    \begin{itemize}
        \item 
        $\phi$ is \emph{proximal} if for any $x_0, x_1\in X$ with $\phi(x_0) = \phi(x_1)$, there is $p\in \sa(G)$ with $px_0 = px_1$.
        \item 
        $\phi$ is \emph{distal} if for any $x_0\neq x_1\in X$ with $\phi(x_0) = \phi(x_1)$ and any $p\in \sa(G)$, we have $px_0\neq px_1$.
        \item 
        $\phi$ is \emph{equicontinuous} if for any $A\in \cU_X$, there is $B\in \cU_X$ so that if $x_0, x_1\in X$ with $(x_0, x_1)\in A$ and $\phi(x_0) = \phi(x_1)$, we have $(gx_0, gx_1)\in B$ for every $g\in G$.
        \item 
        $\phi$ is a \emph{compact group extension} if there is a compact Hausdorff topological group $K$ acting continuously on $X$ by $G$-flow automorphisms so that $\phi(x) = \phi(y)$ iff there is $k\in K$ with $k(x) = y$. In particular, $Y \cong X/K$.
    \end{itemize}
    If $Y$ is trivial, we simply call $X$ proximal, distal, equicontinuous, or a compact group flow, respectively. We note that compact group extensions are equicontinuous, which in turn are distal.
\end{defin}

\begin{lemma}
    \label{Lem:Prox}
    Suppose $M$ is a minimal ideal flow, $X$ and $Y$ are minimal flows, and $\pi\colon M\to X$ and $\phi\colon X\to Y$ are $G$-maps. 
    \begin{enumerate}[label=\normalfont(\arabic*)]
        \item $\phi$ is proximal iff $\bbG_{\pi} = \bbG_{\phi\circ \pi}$.
        \item If $\phi$ is distal, then $\phi$ is equicontinuous iff $\rmH(\bbG_{\phi\circ \pi})\subseteq \bbG_\phi$, and $\phi$ is a compact group extension iff also $\rmH(\bbG_{\phi\circ \pi})\trianglelefteq \bbG_{\phi}$.
    \end{enumerate}
\end{lemma}

\begin{proof}
    First assume $\phi$ is proximal. Fix $p\in \bbG_{\phi\circ \pi}$. Towards a contradiction, suppose $\pi\circ p\neq \pi$. For any $x\in M$, we must have $\pi\circ p(x)\neq \pi(x)$ is a distal pair. But $\phi\circ \pi\circ p(x) = \phi\circ \pi(x)$, and $\phi$ is proximal. Hence $\bbG_\pi = \bbG_{\phi\circ \pi}$.

    Now assume $\phi$ is not proximal, witnessed by $x\neq y\in X$ with $\phi(x) = \phi(y)$. We implicitly fix a $G$-map $M\to \rmE(X)$. Fix $u\in M$ idempotent. Replacing $x, y$ with $ux, uy$, we may assume $ux = x$ and $uy = y$. Thus there is $p\in uM$ with $px = y$, so we have $\rho_p\in \bbG_{\phi\circ \pi}\setminus \bbG_\pi$.

    For the second item, see \cite{Auslander}*{p.\ 204}.
\end{proof}

\begin{lemma}
\label{Lem:EnvMap}
    Let $\phi\colon X\to Y$ be a $G$-map, and let $\wt{\phi}\colon M_X\to M_Y$ be the induced $G$-map. 
    \begin{enumerate}[label=\normalfont(\arabic*)]
        \item \label{Item:Prox_Lem:EnvMap}
        If $\phi$ is proximal, then so is $\wt{\phi}$.
        \item \label{Item:Equi_Lem:EnvMap}
        If $\phi$ is equicontinuous, then $\wt{\phi}$ is a compact group extension.
    \end{enumerate}
\end{lemma}

\begin{proof}
    \ref{Item:Prox_Lem:EnvMap}: Let $\pi_X\colon M_X\to X$ and $\pi_Y\colon M_Y\to Y$ be $G$-maps, arranging that $\pi_Y\circ \wt{\phi} = \phi\circ \pi_X$. As $\phi$ is proximal, we have from Lemma~\ref{Lem:Prox} that  $\bbG_{\pi_Y\circ \wt{\phi}} = \bbG_{\pi_X}$. We have by Fact~\ref{Fact:MinIdealFlow} that $\bigcap_{p\in \bbG_{M_X}} p^{-1}\bbG_{\pi_X}p = \{\rm{id}_{M_X}\}$, but also $\bbG_{\wt{\phi}} = \bigcap_{p\in \bbG_{M_X}} p^{-1}\cdot \bbG_{\pi_Y\circ \wt{\phi}}\cdot p = \{\rm{id}_{M_X}\}$. Hence $\wt{\phi}$ is proximal by Lemma~\ref{Lem:Prox}.
    \vspace{3 mm}

    \noindent
    \ref{Item:Equi_Lem:EnvMap}: It is routine to check that $\wt{\phi}$ is distal, so by Lemma~\ref{Lem:Prox}, it suffices to show that $\rmH(\bbG_{\wt{\phi}}) = \{\rm{id}_{M_X}\}$. Lemma~\ref{Lem:Prox} gives us $\rmH(\bbG_{\phi\circ \pi_X})\subseteq \bbG_{\pi_X}$. We then have 
    \begin{align*}
        \rmH(\bbG_{\wt{\phi}}) &= \rmH\left(\bigcap_{p\in \bbG_{M_X}} p^{-1}\cdot \bbG_{\pi_Y\circ \wt{\phi}}\cdot p\right)\\
        &=  \rmH\left(\bigcap_{p\in \bbG_{M_X}} p^{-1}\cdot \bbG_{\phi\circ \pi_X}\cdot p\right) \\
        &= \bigcap_{p\in \bbG_{M_X}} p^{-1}\rmH(\bbG_{\phi\circ \pi_X})p\\
        &\subseteq \bigcap_{p\in \bbG_{M_X}} p^{-1}\cdot \bbG_{\pi_X}\cdot p\\
        &= \{\rm{id}_{M_X}\} \qedhere 
    \end{align*}
\end{proof}

\begin{theorem}
    \label{Thm:Tau_Hausdorff}
    Fix a minimal flow $X$. The following are equivalent.
    \begin{enumerate}[label=\normalfont(\arabic*)]
        \item \label{Item:Haus_Thm:TauHaus}
        $\bbG_X$ is Hausdorff.
        \item \label{Item:Basic_Thm:TauHaus}
        There are minimal flows $X^*$, $Y$, $Z$ and $G$-maps $\eta\colon X^*\to X$, $\sigma\colon Z\to Y$, and $\iota\colon X^*\to Z$ with $Y$, $\eta$, and $\iota$ proximal and with $\sigma$ equicontinuous.
        \begin{center}
        \begin{tikzcd}
        & X^*\arrow[d, "\iota"] \arrow[ld, "\eta"']\\
        X & Z\arrow[d, "\sigma"] \\
        & Y
        \end{tikzcd}
        \end{center}
        \item \label{Item:Weak_Thm:TauHaus}
        Same as Item~\ref{Item:Basic_Thm:TauHaus}, but with no demand on $\eta$.
        \item \label{Item:Group_Thm:TauHaus}
        Same as Item~\ref{Item:Weak_Thm:TauHaus}, but with $X^*$ and $Z$ (and $Y$) minimal ideal flows and $\sigma$ a compact group extension.
    \end{enumerate}
\end{theorem}

\begin{proof}
$\ref{Item:Weak_Thm:TauHaus}\Rightarrow \ref{Item:Haus_Thm:TauHaus}$: This direction is implicit in \cite{DC3}; by Proposition~\ref{Prop:Tau_Maps}\ref{Item:Hom_Prop:TauMaps}, it is enough to check that $\bbG_{X^*}$ is Hausdorff. Fix a $G$-map $\psi\colon \rmM(G)\to X^*$. We have $\bbG = \bbG_{\sigma\circ \iota\circ\psi}$ as $Y$ is proximal, and we have $\bbG_{\psi} = \bbG_{\iota\circ \psi}$ since $\iota$ is proximal. As $\sigma$ is equicontinuous, we have $\rmH(\bbG)\subseteq \bbG_\psi$, and we conclude by Lemma~\ref{Lem:Haus}. 
\vspace{3 mm}

\noindent
$\ref{Item:Haus_Thm:TauHaus}\Rightarrow \ref{Item:Basic_Thm:TauHaus}$: Fix $X$ any minimal flow, and let $\pi\colon \rmM(G)\to X$ be a $G$-map. We can find a proximal extension $\eta\colon X^*\to X$, a minimal proximal flow $Y$, and a \emph{relatively incontractible} $G$-map $\phi\colon X^*\to Y$ \cite{Auslander}*{p.\ 207--211}. There is a $G$-map $\psi\colon \rmM(G)\to X^*$ with $\eta\circ \psi = \pi$. As $\eta$ is proximal, we have $\bbG_{\pi} = \bbG_{\psi}$, and as $Y$ is proximal, we have $\bbG_{\phi \circ \psi} = \bbG$. Then we can find a minimal flow $Z$ and $G$-maps $\iota\colon X^*\to Z$ and $\sigma\colon Z\to Y$ such that $\phi = \sigma\circ \iota$,  $\sigma$ is equicontinuous, and  $\bbG_{\iota\circ\psi} = \rmH(\bbG)\cdot \bbG_\psi$ \cite{Auslander}*{p.\ 205--207}. It remains to show that if $\bbG_X$ is Hausdorff, then $\iota\colon X^*\to Z$ is proximal. This happens iff $\bbG_{\iota\circ \psi} = \bbG_\psi$; since $\rmH(\bbG)\subseteq \bbG_\pi = \bbG_\psi$, this clearly holds.
\vspace{3 mm}

\noindent
$\ref{Item:Weak_Thm:TauHaus}\Rightarrow \ref{Item:Group_Thm:TauHaus}$: Simply replace $X^*$, $Z$, and $Y$ by their enveloping ideals and use Lemma~\ref{Lem:EnvMap}. 
\end{proof}

\subsection{$\aut(\rmM(G))$ for $\sf{TMD}$ groups}

If $G\in \sf{PCMD}$, it is shown in \cite{MNVTT} that if $H\leq_c G$ is a precompact, presyndetic, extremely amenable subgroup, i.e.\ so that $\rmM(G)\cong \sa(G/H)$, then as groups we have $\bbG^{opp}\cong \rmN_G(H)/H$, where $\rmN_G(H)$ is the normalizer of $H$ in $G$. As $H\leq_c G$ is precompact, $\rmN_G(H)/H$ is a precompact Polish group, hence compact, and it is natural to then view $\bbG$ as a compact metrizable group. 

For $G\in \sf{GPP}$, almost the exact same argument gives that $\bbG\cong \rmN_G(H)/H$ for $H\leq_c G$ presyndetic and extremely amenable, but now $H$ is not necessarily precompact, and we cannot immediately conclude that $\rmN_G(H)/H$ is compact. However, in this section, we will prove that this is still the case by way of the tau topology on $\bbG$. Along the way, we provide a new characterization of the tau-topology on $\bbG$ using ultracopowers of $\rmM(G)$.

For the rest of this section, \textbf{we fix} $w\in \rmM(G)$, and we write $O = \op(w, \rmM(G))$. Hence by Proposition~\ref{Prop:Tau_Base}, $\{\rm{Nbd}_A(p): A\in \op(w, \rmM(G))\}$ is a base of $\tau$-open neighborhoods of $p$.

Consider a set $I$, and form $\rmS_G(I\times \rmM(G)) =: Y$. Fix $(p_i)_{i\in I}\in \bbG^I$, and we let $\phi\in \aut(Y)$ denote the automorphism of $Y$ satisfying $\phi(i, x) = (i, p_i(x))$ for every $i\in I$ and $x\in \rmM(G)$. Write $\pi\colon Y\to \rmM(G)$ for the projection to $\rmM(G)$. Let $M\subseteq Y$ be a minimal subflow, and note that $\pi|_M\colon M\to \rmM(G)$ is an isomorphism. Define
$$\lim_M p_i =: \pi\circ \phi \circ (\pi|_M)^{-1} \in \aut(\rmM(G)).$$

\begin{lemma}
	\label{Lem:M_Limit_in_closure}
	If $p = \displaystyle\lim_M p_i$, then $p\in \rm{cl}_\tau(\{p_i: i\in I\})$. 
\end{lemma}
	
\begin{proof}
	Fix $A\in O$. Write $x = (\pi|_M)^{-1}(w)$ and $B = \pi^{-1}[A]\in \op(x, Y)$. Then we have $\{\pi^{-1}[p[A]], \phi[B]\}\subseteq \op(\phi(x), Y)$. In particular, we have $\pi^{-1}[p[A]]\cap \phi[B]\cap (I\times \rmM(G))\neq \emptyset$. Thus for some $i\in I$, we have $p[A]\cap p_i[A]\neq \emptyset$.
\end{proof}
	
For the rest of the section, \textbf{we fix} $\cV$ a cofinal ultrafilter on $O$, where $O$ is a directed partial order under reverse inclusion. Let us write $X := \Sigma_\cV^G \rmM(G)\subseteq \alpha_G(O\times \rmM(G))$.
	
\begin{lemma}
	\label{Lem:Build_M_closure}
	Suppose $S\subseteq \bbG$ and $p\in \rm{cl}_\tau(S)$. Then there are $(p_A)_{A\in O}\in S^O$ and a minimal $M\subseteq X$ with $\displaystyle \lim_M p_A = p$. 
\end{lemma}
	
\begin{proof}
	Let $\phi\in \aut(X)$ denote the automorphism induced by $(p_A)_{A\in O}$. We first show it suffices to find any $x\in X$ with $p\circ \pi(x) = \pi\circ \phi(x)$. Given such an $x$, let $q\in \sa(G)$ belong to a minimal subflow. Then $\sa(G)\cdot qx =: M\subseteq X$ is a minimal subflow. Since $p\circ \pi$ and $\pi\circ \phi$ are $G$-maps, we have $p\circ \pi(qx) = \pi\circ \phi(qx)$, implying that $p\circ \pi|_M = \pi\circ \phi|_M$.
		
	For each $A\in O$, since $p\in \rm{cl}_\tau S$, we can find $p_A \in S$ and $B_A\in \op(A)$ with $p_A[B_A]\subseteq p[A]$. Find $x\in X$ with $\bigcup_{i\in J} (\{i\}\times B_i)\in x$ for every $J\in \cU$ (viewing elements of $X$ as near ultrafilters). Since $\cV$ is cofinal and since $\{p[A]: A\in O\} = \op(p(w), \rmM(G))$, we have $\pi(x) = w$ and $\pi\circ \phi(x) = p(w).$
\end{proof}
	
We next turn to characterizing tau-ultralimits. Given a set $I$, $\cU\in \beta I$,  and $(p_i)_{i\in I}\in \bbG^I$, we write 
$$\displaystyle \lim_\cU p_i = \bigcap_{S\in \cU} \rm{cl}_\tau \{p_i: i\in S\}.$$ 
Since $\tau$ is compact, this set is always non-empty; the $\tau$-topology on $\bbG$ is Hausdorff exactly when this set is always a singleton. By enlarging the index set, we can establish a tight connection between ultralimits and the $M$-limits defined before Lemma~\ref{Lem:M_Limit_in_closure}. Instead of working with $\cU\in \beta I$, we form the \emph{Fubini product} $\cV\otimes \cU\in \beta(O\times I)$ where given $S\subseteq O\times I$, we have 
$$S\in \cV\otimes \cU\Leftrightarrow \forall^\cV A\in O \, \forall^\cU i\in I\,\, (A, i)\in S.$$
	
\begin{theorem}
	\label{Thm:Build_M_Ultralimit}
	Fix a set $I$ and $(p_i)_{i\in I}\in \bbG^I$. Given $A\in O$ and $i\in I$, define $p_{A, i} := p_i$. Given $p\in \bbG$, we have $p\in \displaystyle\lim_\cU p_i$ iff for some minimal subflow $M\subseteq \Sigma_{\cV\otimes \cU}^G \rmM(G)$, we have $p = \displaystyle \lim_M p_{A, i}$.
\end{theorem} 
	
\begin{proof}
	The reverse direction follows from Lemma~\ref{Lem:M_Limit_in_closure}. For the forward direction, we mimic the proof of Lemma~\ref{Lem:Build_M_closure}. Given $A\in O$, Lemma~\ref{Lem:M_Limit_in_closure} tells us that 
		$$I_A:= \{i\in I: p_i\in \rm{Nbd}_A(p)\}\in \cU.$$
	For each $i\in I_A$, find $B_{A, i}\in \op(A)$ with $p_i[B_{A, i}]\subseteq p[A]$. Let $x\in \Sigma_{\cV\otimes \cU}^G \rmM(G)$ satisfy 
    $$\bigcup_{(A, i)\in O\times I} (\{(A, i)\}\times B_{A, i})\in x.$$ This is possible by the definition of $\cV\otimes \cU$ and since $I_A\in \cU$ for every $A\in O$. Arguing as in the proof of Lemma~\ref{Lem:Build_M_closure}, we have $p\circ \pi(x) = \pi\circ \phi(x)$, where $\phi$ here denotes the automorphism of $\Sigma_{\cV\otimes \cU}^G \rmM(G)$ induced by $(p_{A, i})_{A, i}$.
\end{proof}

\begin{corollary}
    \label{Cor:Tau_CHaus}
    If $G\in \sf{TMD}$, then $(\bbG, \tau)$ is a compact Hausdorff topological group. 
\end{corollary}

\begin{proof}
    Theorem~\ref{Thm:Build_M_Ultralimit}, Theorem~\ref{Thm:Rosendal_Minimal}\ref{Item:Ultra_Thm:RM}, and Ellis's joint continuity theorem (see \cite{Auslander}). 
\end{proof}

\begin{question}
    \label{Que:AutMGHausTMD}
    Suppose $G$ is a topological group with $\bbG$ Hausdorff. Do we have $G\in \sf{TMD}$?
\end{question}

\subsection{The structure of $\rmM(G)$ for $\sf{GPP}$ groups}

We can now generalize and strengthen the characterization of $\bbG$ for $\sf{PCMD}$ groups from \cite{MNVTT} to $\sf{GPP}$ groups.

\begin{theorem}
    \label{Thm:Tau_Polish}
    If $G\in \sf{GPP}$ and $H\leq^c G$ is a presyndetic, extremely amenable subgroup, then $(\bbG^{opp}, \tau)\cong \rmN_G(H)/H$. In particular $H\trianglelefteq^c \rmN_G(H)$ is cocompact.
\end{theorem}

\begin{proof}
    The argument that $\bbG^{opp}\cong \rmN_G(H)/H$ as abstract groups is identical to \cite{MNVTT}. Let $\sigma\in \sn(G)$ be a compatible norm. The action of $\bbG$ on $G/H\subseteq \sa(G/H) = \rmM(G)$ is a free action by $\partial_\sigma$-isometries by Proposition~\ref{Prop:Level_Maps}, and $\partial_\sigma$ is compatible with the Polish topology on $G/H$. Thus the Polish topology $\tau'$ on $\bbG$ inherited from $\rmN_G(H)/H$ is the topology of pointwise convergence on $G/H$, and $\tau'\supseteq \tau$. Furthermore, if $x\in G/H$ and $B\in \op(\rmM(G))$, we have 
    $$\{p\in \bbG: p(x)\in \ol{B}\} = \bigcap_{A\in \op(x, \rmM(G))} \rm{Nbd}_{A, B}.$$
    As points in $G/H$ are points of first countability in $\rmM(G)$ by Theorem~\ref{Thm:Polish_Group_WUEB}, we see that these sub-basic $\tau'$-closed sets are $\tau$-Borel. Standard Baire category techniques now show that $\tau' = \tau$ (see \cite{Kechris_Classical}*{Theorem~9.10}).  
\end{proof}

Using Theorem~\ref{Thm:Tau_Hausdorff}, Corollary~\ref{Cor:Tau_CHaus}, and Theorem~\ref{Thm:Tau_Polish}, we can provide a detailed structure theorem for $\mg$ when $G\in\sf{GPP}$, which we will later generalize to $\sf{TMD}$ . We need the following definition from \cite{Glasner_Prox}.

\begin{defin}
    \label{Def:Compactification_Flow}
    A minimal flow $Z$ is called a \emph{compactification flow} if it is a compact group extension of a proximal flow.  
\end{defin}
Note that any compactification flow is a minimal ideal flow. Glasner in \cite{Glasner_Prox} constructs for any topological group $G$ the \emph{universal compactification flow} of $G$, which we denote by $\rm{UCF}(G)$. By definition, this is the largest compact group extension of the \emph{universal minimal proximal flow} $\Pi(G)$. Upon fixing a minimal subflow $M\subseteq \sa(G)$ and a $G$-map $\pi\colon M\to \Pi(G)$ (as $\Pi(G)$ is proximal, $\pi$ is unique), Glasner proves that the following equivalence relation $\sim$ on $M$ is closed, $G$-invariant, and that $M/{\sim} \cong \rm{UCF}(G)$:
\begin{align*}
    x\sim y \Leftrightarrow \pi(x) = \pi(y) \text{ and } \rho_x\circ \rho_y^{-1}\in \rmH(\bbG).
\end{align*}
Notice that upon knowing that $M/{\sim}\cong \rm{UCF}(G)$ is a minimal ideal flow, it follows that $\sim$ did not depend on the choice of $M$, and we can unambiguously refer to $\sim$ on $\rmM(G)$.  We have $\aut(\rm{UCF}(G))\cong \bbG/\rmH(\bbG)$, and the quotient of $\rm{UCF}(G)$ by this compact group action is $\Pi(G)$. We take the opportunity to show that $\rm{UCF}(G)$ as defined above is indeed universal.
\begin{lemma}
    \label{Lem:UCF}
    $\rm{UCF}(G)$ admits a $G$-map onto any compactification flow.
\end{lemma}

\begin{proof}
    Fix a minimal flow $M\subseteq \sa(G)$, fix a $G$-map $\pi\colon M\to \Pi(G)$, and view $\rm{UCF}(G) = M/{\sim}$ as above. Let $Z$ be a compactification flow, witnessed by the proximal flow $X$ and  compact group extension $\sigma\colon Z\to X$. Hence $\bbG_Z$ is a compact group acting continuously on $Z$ by $G$-flow automorphisms. Fix $G$-maps $\psi\colon M\to Z$ and $\eta\colon \Pi(G)\to X$; we can arrange that $\eta\circ \pi = \sigma\circ \psi$. Since $\bbG_Z \cong \bbG/\bbG_\psi$ is Hausdorff, we must have $\rmH(\bbG)\subseteq \bbG_\psi$. 
    
    Suppose $x, y\in M$ with $x\sim y$; we want to show that $\psi(x) = \psi(y)$. As $\pi(x) = \pi(y)$, we must have $\sigma\circ \psi(x) = \sigma\circ \psi(y)$. Fix $k\in \bbG_Z$ with $k(\psi(x)) = \psi(y)$. Let $u\in M$ be idempotent with $ux = x$, and let $p\in M$ be such that $\rho_p\in \bbG$ is a representative of $k$. We have
    \begin{align*}
         \psi(xp) &= \psi(y)\\[3 mm]
        \text{and }\psi(xp) &= \psi(uxp)  = u\psi(xp) = u\psi(y) 
        = \psi(uy)\\
        &= \psi\circ \rho_x\circ \rho_y^{-1}(uy) = \psi(ux) = \psi(x). \qedhere
    \end{align*}    
\end{proof}
	
Upon analyzing the proof of Theorem~\ref{Thm:Tau_Hausdorff} (in particular the references to \cite{Auslander} therein) in the case that $X = \rmM(G)$, we obtain the following.
\begin{theorem}
    \label{Thm:AutMG_Haus}
    For $G$ a topological group, $\bbG = \aut(\rmM(G))$ is Hausdorff iff the canonical $G$-map $\iota\colon \rmM(G)\to \rm{UCF}(G)$ is proximal.
\end{theorem}
In particular, the above theorem applies for $G\in \sf{TMD}$. Eventually, we will use set-theoretic methods to show for $G\in \sf{TMD}$ that $\iota$ is highly proximal. As a first step, we show this for $G\in \sf{GPP}$. Recall that if $H\leq^c G$ is a presyndetic, extremely amenable subgroup then $\sa(G/H)\cong \rmM(G)$.

\begin{theorem}
    \label{Thm:UnivProxGPP}
    Let $G\in \sf{GPP}$ and let $H\leq^c G$ a presyndetic, extremely amenable subgroup. 
    Then $\sa(G/\rmN_G(H))\cong \Pi(G)$. In particular, if $G$ is $\sf{GPP}$ and strongly amenable, then $G\in \sf{PCMD}$
\end{theorem}

\begin{proof}
    Consider the canonical map $\pi\colon \sa(G/H)\to \sa(G/\rmN_G(H))$. Each fiber of $\pi$ is invariant under the action of $\rmN_G(H)/H\cong \bbG^{opp}$ (Theorem~\ref{Thm:Tau_Polish}), implying that $\bbG_\pi = \bbG$ and hence $\sa(G/\rmN_G(H))$ is proximal (Lemma~\ref{Lem:Prox}). If $X$ is another minimal proximal flow and $\phi\colon \sa(G/H)\to X$ is a $G$-map, then we have $\phi\circ p = \phi$ for any $p\in \bbG$. It follows that $\phi$ induces a map on $G/\rmN_G(H)$, and this map continuously extends to $\sa(G/\rmN_G(H))$.

    If $G$ is strongly amenable, then $G = \rmN_G(H)$, and as $H\trianglelefteq^c \rmN_G(H) = G$ is cocompact by Theorem~\ref{Thm:Tau_Polish}, we have $G\in \sf{PCMD}$. 
\end{proof}

\begin{theorem}
    \label{Thm:MG_Structure_GPP}
    Let $G\in \sf{GPP}$ and let $H\leq^c G$ a presyndetic, extremely amenable subgroup. We have $\rmM(G)\cong \rmS_G(\rm{UCF}(G))$, i.e.\ in Theorem~\ref{Thm:AutMG_Haus}, we can take $\iota$ to be highly proximal.
\end{theorem}
      
\begin{proof}
    Let $\pi\colon \sa(G/H)\to \sa(G/\rmN_G(H)) \cong \Pi(G)$ be the canonical map, and identify $\sa(G/H)$ with a minimal left ideal $M\subseteq \sa(G)$ to form the relation $\sim$ defined after Definition~\ref{Def:Compactification_Flow}. We show that on $G/H$, each $\sim$-class is a singleton.  and suppose $x = xH, y = yH\in G/H$ satisfy $\pi(x) = \pi(y)$. Then $x\rmN_G(H) = y\rmN_G(H)$, and hence by Theorem~\ref{Thm:Tau_Polish}, $x = p(y)$ for some $p\in \bbG$. It follows that $\rho_x\circ \rho_y^{-1}\in \rmH(\bbG) = \{e_\bbG\}$ iff $p = e_\bbG$, i.e.\ $x\sim y$ implies $x = y$.  
\end{proof}

For $G$ a topological group with $\bbG$ Hausdorff, when do we have $\rmM(G)\cong \rm{UCF}(G)$? This happens exactly when the action of $\bbG$ on $\rmM(G)$ is continuous, since then $\rmM(G)/\bbG\cong \Pi(G)$. Thus we obtain:

\begin{theorem}
    \label{Thm:MG_Structure_CMD}
    For $G\in \sf{CMD}$, we have $\rmM(G)\cong \rm{UCF}(G)$.
\end{theorem}

\begin{proof}
    For each $\sigma\in \sn(G)$, $\partial_\sigma$ is a continuous pseudometric on $\rmM(G)$ by Theorem~\ref{Thm:Concrete}, and $\bbG$ acts on $\rmM(G)$ by $\partial_\sigma$-isometries by Proposition~\ref{Prop:Level_Maps}. Thus the tau-topology on $\bbG$ is simply the topology of pointwise convergence as $\sigma$ varies over all semi-norms, and the $\bbG$-action on $\rmM(G)$ is therefore continuous.
\end{proof}
Though for $\sf{GPP}$ groups, we cannot rule out the possibility that one needs a highly proximal extension of $\rm{UCF}(G)$ to obtain $\rmM(G)$, we have no examples where this is needed.

\begin{question}
    \label{Que:Not_UCF}
    Is there $G\in \sf{GPP}$ with $\rmM(G)\not\cong \rm{UCF}(G)$?
\end{question}
We note that there are examples of $\sf{GPP}\setminus \sf{PCMD}$ groups with $\rmM(G)\cong \rm{UCF}(G)\cong \Pi(G)$, for instance many of the kaleidoscopic groups from \cite{MR4549426}.

\section{Unfolded flows and abstract KPT correspondence}
\label{Section:KPT}

Motivated by the example of an automorphism group of a \fr structure acting on a space of expansions, we define for any topological group $G$ the notion of an \emph{unfolded flow}. Instead of associating a set of expansions to each finite substructure of a \fr structure, we associate a topometric space to each member of $\rm{SN}(G)$. Various ``bonding maps'' between these topometric spaces allow us to \emph{fold} the unfolded flow into a $G$-flow. Of particular importance is the case when each topometric space appearing in an unfolded flow is simply a metric space; we call these \emph{$G$-skeletons}. For $G$-skeletons, we define a finite combinatorial statement which we call the \emph{Ramsey property} which characterizes when the folded flow of a $G$-skeleton is isomorphic to $\mg$, thus providing an abstract KPT correspondence that gives another characterization of the class $\sf{TMD}$.

\subsection{Unfolded $G$-flows and $G$-skeletons}

A \emph{$G$-join-semi-lattice} is a join semi-lattice $\cB$ equipped with a \emph{right} $G$-action by automorphisms. Given $\sigma\in \cB$ and $F\in \fin{G}$, we write $\sigma F = \bigvee\{\sigma g: g\in F\}$. Given $\sigma\leq \tau\in \cB$, we write $\rmP(\tau, \sigma):= \{g\in G: \sigma g\leq \tau\}$. The prototypical example is $\cB = \sn(G)$.  Given $\sigma, \tau\in\rm{SN}(G)$, we write $\sigma\leq \tau$ if $\sigma(g)\leq \tau(g)$ for every $g\in G$, and we let $\sigma\vee\tau\in \rm{SN}(G)$ denote the pointwise maximum. Given $g, h\in G$ and $\sigma\in \rm{SN}(G)$, we define $\sigma g\in \sn(G)$ via $(\sigma g)(h) = \sigma(ghg^{-1})$.

Given two $G$-join-semi-lattices $\cB$ and $\cB'$, a \emph{$G$-map} $\theta\colon \cB\to \cB'$ is a $G$-equivariant, order-preserving, join-preserving map.

\begin{defin}
    \label{Def:Unfolded_Flow}
    Fix $\cB$ a $G$-join-semi-lattice and a $G$-map $\theta\colon \cB\to \rm{SN}(G)$. 
    
    A \emph{$\theta$-unfolded $G$-space} $\Omega$ consists of the following data: 
    \begin{itemize}
        \item 
        For each $\sigma\in \cB$,  a topometric space $(X_\sigma, d_\sigma)$. We call this \emph{level $\sigma$} of the unfolded $G$-space.
        \item 
        For each  $\sigma \leq \tau\in \cB$, a continuous Lipschitz surjection $\pi^\tau_\sigma\colon X_\tau\to X_\sigma$ such that whenever $\sigma\leq \tau\leq \nu\in \cB$, we have $\pi^\tau_\sigma\circ \pi_\tau^\nu = \pi^\nu_\sigma$.
        \item 
        For each $\sigma\in \cB$ and $g\in G$, a homeomorphism and isometry $\lambda_{g, \sigma}\colon X_{\sigma g}\to X_\sigma$ satisfying:
        \begin{itemize}
            \item (Equivariance): For every $g\in G$ and $\sigma\leq \tau\in \cB$, we have \newline $\lambda_{g, \sigma}\circ \pi^{\tau g}_{\sigma g} = \pi^\tau_\sigma \circ \lambda_{g, \tau}$. 
            \item (Action): For every $g, h\in G$ and $\sigma\in \cB$, $\lambda_{g, \sigma}\circ \lambda_{h, \sigma g} = \lambda_{gh, \sigma}$. Given $\sigma, \tau\in \cB$, $g\in \rmP(\tau, \sigma)$, and $x\in X_\tau$, we write $g\cdot_\sigma x = \lambda_{g, \sigma}\circ \pi^\tau_{\sigma g}(x)$.
            \item (Continuity): For every $\sigma, \tau\in \cB$ and $g\in \rmP(\tau, \sigma)$, and $x\in X_\tau$, we have $d_\sigma(g\cdot_\sigma x, \pi^\tau_\sigma(x)) \leq \theta(\sigma)(g)$. When $\theta$ is understood, we typically suppress it and just write $\sigma(g)$.
        \end{itemize}
    \end{itemize}
    An unfolded $G$-space with $X_\sigma$ compact for each $\sigma \in \cB$ is called an \emph{unfolded $G$-flow}.
\end{defin}

Given an unfolded $G$-flow $\Omega$ as above, we set $X_\Omega := \varprojlim (X_\sigma, \pi^\tau_\sigma)$. Write $\pi_\sigma\colon X_\Omega\to X_\sigma$ for the projection. We remark that for each $\sigma\in \cB$, $d_\sigma\circ \pi_\sigma$ is an lsc pseudometric on $X_\Omega$, and for each $\sigma\leq \tau\in \cB$, $d_\sigma\circ \pi_\sigma^\tau$ is an lsc pseudometric on $X_\tau$, and we can identify both $X_\Omega/(d_\sigma\circ \pi_\sigma)$ and $X_\tau/(d_\sigma\circ \pi_\sigma^\tau)$ with $X_\sigma$ in the obvious way. Given $y\in X_\Omega$ and $g\in G$, we define $gy\in X_\Omega$ to be the unique element satisfying $\pi_\sigma(gy) = g\cdot_\sigma \pi_{\sigma g}(y)$ for every $\sigma\in \cB$; we introduce the shorthand notation $g\cdot_\sigma y = \pi_\sigma(gy)$.

\begin{lemma}
    This is a well-defined and continuous action, turning $X_\Omega$ into a $G$-flow.
\end{lemma}

We therefore call $X_\Omega$ the \emph{folded flow} of $\Omega$.

\begin{proof}
    It is straightforward to see using the equivariance and action properties of $\Omega$ that this is a well-defined action by homeomorphisms of $X_\Omega$. To see how the continuity property of $\Omega$ ensures that $G\times X_\Omega\to X_\Omega$ is continuous, let $(g_i)_{i\in I}$ and $(y_i)_{i\in I}$ be nets from $G$ and $X_\Omega$, respectively, with $g_i\to 1_G\in G$ and $y_i\to y\in X_\Omega$. To see that $g_iy_i\to y$, fix $\sigma\in \cB$. The continuity property of $\Omega$ implies that $d_\sigma(g_i\cdot_\sigma y_i, \pi_\sigma(y_i))\leq \sigma(g_i)$. Let $z\in X_\sigma$ be any limit point of the net $(g_i\cdot_\sigma y_i)_{i\in I}$. As $d_\sigma$ is lsc and since $\sigma(g_i)\to 0$, we have $d_\sigma(z , \pi_\sigma(y)) = 0$, implying $z = \pi_\sigma(y)$.
\end{proof}

Given a $G$-flow $Y$, there are typically several different unfolded flows $\Omega$ with $X_\Omega \cong Y$. However, there is a canonical choice when $Y$ is $G$-ED.

\begin{defin}
    \label{Def:Unfolded_GED}
    Given a $G$-ED flow $Y$, a $G$-join-semi-lattice $\cB$, and a $G$-map $\theta\colon \cB\to \sn(G)$, the \emph{$\theta$-unfolding} of $Y$ is the unfolded $G$-flow $\Omega_{Y, \theta}$, or just $\Omega_Y$ if $\theta$ is understood, defined as follows:
    \begin{itemize}
        \item 
        Given $\sigma\in \cB$, level $\sigma$ of $\Omega_Y$ is $(Y/\partial_{\theta(\sigma)}, \ol{\partial}_{\theta(\sigma)})$. We typically suppress $\theta$ and just write $\partial_\sigma$.
        \item 
        Given $\sigma\leq \tau\in \cB$, the map $\pi_\sigma^\tau\colon Y/\partial_\tau\to Y/\partial_\sigma$ is defined where given $y\in Y$, we have  $\pi_\sigma^\tau([y]_{\partial_\tau}) := [y]_{\partial_\sigma}$.
        \item 
        Given $\sigma\in \cB$ and $g\in G$, the map $\lambda_{g, \sigma}\colon Y/\partial_{\sigma g}\to Y/\partial_{\sigma}$ is defined where given $y\in Y$, we have $\lambda_{g, \sigma}([y]_{\partial_{\sigma g}}) := g\cdot [y]_{\partial_{\sigma g}}= [gy]_{\partial_{\sigma}}$. 
    \end{itemize}
\end{defin}

The case when $Y$ is a $G$-TT,  $G$-ED flow with $Y\in \cal{RC}_G$ is of particular interest. Theorem~\ref{Thm:Rosendal_Minimal} tells us that this happens iff for each $\sigma\in \rm{SN}(G)$, the set $D_\sigma\subseteq Y$ of $\partial_\sigma$-compatible points is dense. Whenever $\sigma\leq \tau\in \rm{SN}(G)$, note that $D_\tau\subseteq D_\sigma$, and by Lemma~\ref{Prop:Adequate}\ref{Item:CompDown_Prop:Adequate}, $\pi_\sigma[D_\sigma]\subseteq Y/\partial_\sigma$ is exactly the set of $\ol{\partial_\sigma}$-compatible points. We therefore express the salient features of the ``continuous part'' of the unfolded $G$-flow of $Y$ in the following definition, thinking of this as the ``skeleton'' around which the rest of the unfolded $G$-flow can be formed.

\begin{defin}
    \label{Def:Skeleton}
    Given $\theta\colon \cB\to \sn(G)$ as usual, a \emph{$\theta$-skeleton} $\Omega$ is a $\theta$-unfolded $G$-space such that for each $\sigma\in \cB$, $d_\sigma$ is a compatible metric on $X_\sigma$. A \emph{$G$-skeleton} is a $\theta$-skeleton for some $\theta$ with $\theta[\cB]\subseteq \sn(G)$ a base of semi-norms. A $G$-skeleton is called \emph{precompact} if each $(X_\sigma, d_\sigma)$ is a precompact metric space.
\end{defin}

Upon taking a $G$-skeleton $\Omega$, taking the Samuel compactification of each level, and continuously extending the relevant maps, we obtain an unfolded $G$-flow we denote by $\sa(\Omega)$.  We abuse notation and freely extend $\pi^\tau_\sigma$ and $\lambda_{g, \sigma}$ to the Samuel compactifications $(\sa(X_\sigma), \partial_{d_\sigma})$, and we write $X_\Omega$ for $X_{\sa(\Omega)}$.  

If $Y\in \cal{RC}_G$ is a $G$-TT, $G$-ED flow, we let $\Omega^0_{Y, \theta}$, or just $\Omega^0_Y$, denote the $G$-skeleton obtained from taking the unfolded $G$-flow $\Omega_Y$ and keeping only the $\ol{\partial_\sigma}$-compatible points of $Y/\partial_\sigma$.

Not every unfolded $G$-flow is the unfolded flow of a $G$-ED flow.
We now isolate properties of unfolded $G$-spaces which, in the case of unfolded $G$-flows and $G$-skeletons, correspond to dynamical properties of the corresponding folded flows.

\begin{defin}
\label{Def:GED_Unfolded}
    Let $\Omega$ be an unfolded $G$-space, with all the notations from \Cref{Def:Unfolded_Flow}.
    \begin{itemize}
        \item $\Omega$ has the \emph{$G$-TT property} if for any $\sigma\in \cB$ and any $A, B\in\op(X_\sigma)$, there is $g\in G$ so that for any $\tau\in \cB$ with $g\in \rmP(\tau, \sigma)$, we have $(\pi_\sigma^\tau)^{-1}[A]\cap (\lambda_{g, \sigma}\circ \pi_{\sigma g}^\tau)^{-1}[B] \neq \emptyset$. 
        \item $\Omega$ has the \emph{minimality property} if for every $\sigma\in \cB$ and $A\in \op(X_\sigma)$, there is $F\in \fin{G}$ such that for every $y\in X_{\sigma F}$, there is $g\in F$ with $g\cdot_\sigma y\in A$. 
        \item $\Omega$ satisfies the \emph{ED property} if all of the following:
        \begin{itemize}
            \item 
            For every $\sigma\in \cB$, the pseudometric $d_\sigma$ is open-compatible.
            \item 
            For every $\sigma\leq \tau\in \cB$, the pseudometric $d_\sigma\circ \pi_\sigma^\tau$ on $X_\tau$ is adequate. 
            
            If $\Omega$ is an unfolded $G$-flow, this happens exactly when $d_\sigma \circ \pi_\sigma$ on $X_\Omega$ is adequate. 
            \item 
            Whenever  $\sigma\leq \tau\in \cB$, $0< \epsilon < 1$ and $A, B\in \op(X_{\tau})$ with $d_\sigma\circ \pi_\sigma^\tau(A, B) < \epsilon$, then there is $g\in \rmB_\sigma(\epsilon)$ so that for any $\upsilon\in \cB$ with $g\in \rmP(\upsilon, \tau)$, we have $(\pi^\upsilon_\tau)^{-1}[A]\cap (\lambda_{g, \tau}\circ \pi^\upsilon_{\tau, g})^{-1}[B] \neq \emptyset$. 
        \end{itemize}
    \end{itemize}
\end{defin}

\begin{prop}
\label{Prop:GED_Flow_gives_ED_Unfolded}
    If $Y$ is a $G$-ED flow, then $\Omega_Y=: \Omega$ has the ED property.
\end{prop}

\begin{proof}
    If $Y$ is a $G$-ED flow, then Lemma~\ref{Lem:ED_Pseudo} and Proposition~\ref{Cor:ED_Adequate} show that the $d_\sigma$'s are adequate and open-compatible, and by Proposition \ref{Prop:Adequate}\ref{Item:Quotient_Prop:Adequate}, so are $d_\sigma \circ \pi^\tau_\sigma$.
    Now suppose that $\sigma\leq \tau\in \rm{SN}(G)$, $0< \epsilon < 1$, and $A, B\subseteq \op(Y_\tau/\partial_\tau)$ satisfy $\ol{\partial}_\sigma\circ \pi^\tau_\sigma(A, B) < \epsilon$. Then $\partial_\sigma(\pi_\tau^{-1}[A], \pi_\tau^{-1}[B])< \epsilon$. By definition of $\partial_\sigma$, we have $\pi_\tau^{-1}[A]\cap \rmB_\sigma(\epsilon)\cdot\pi_\tau^{-1}[B]\neq \emptyset$. Fix $g\in \rmB_\sigma(\epsilon)$ with $\pi_\tau^{-1}[A]\cap g^{-1}\cdot \pi_\tau^{-1}[B]\neq \emptyset$. If $\upsilon\in \rm{SN}(G)$ with $g\in \rmP(\upsilon, \tau)$, then $(\pi^\upsilon_\tau)^{-1}[A]\cap (\lambda_{g, \tau}\circ \pi^\upsilon_{\tau, g})^{-1}[B] \neq \emptyset$ as desired.
\end{proof}

\begin{prop}
\label{Prop:GED_Unfolded}
    Let $\Omega$ be an unfolded $G$-flow, with all the notations from \Cref{Def:Unfolded_Flow}, and form the folded flow $Y:= X_\Omega$.
    Then: 
    \begin{enumerate}[label=\normalfont(\arabic*)]
        \item \label{Item:GTT_Prop:Unfolded} $Y$ is $G$-TT iff $\Omega$ has the TT property. 
        \item \label{Item:Min_Prop:Unfolded} $Y$ is minimal iff $\Omega$ has the minimality property.
        \item \label{Item:GED_Prop:Unfolded} If $\Omega$ has the ED property, then $Y$ is $G$-ED, and furthermore $(Y/\partial_{\sigma}, \ol{\partial}_{\sigma})\cong (X_\sigma, d_\sigma)$.
    \end{enumerate}
\end{prop}

\begin{proof}[Proof of \Cref{Prop:GED_Unfolded}]
    Both directions of the $G$-TT statement are straightforward.

    Assume $\Omega$ has the minimality property. Fix $B\in \op(Y)$. We may suppose $B = \pi_\sigma^{-1}[A]$ for some $\sigma\in \cB$ and  $A\in \op(X_\sigma)$. Let $F\in \fin{G}$ be as promised by the assumption that $\Omega$ is minimal. Now consider some $y\in Y$. Considering $\pi_{\sigma F}(y)$, we can find $g\in F$ so that $g\cdot_\sigma \pi_{\sigma F}(y)\in A$. In particular, we have $gy\in B$. Hence $Y$ is minimal.

    Now assume $Y$ is minimal. Fix $\sigma\in \cB$ and $A\in \op(X_\sigma)$. As $Y$ is minimal, let $F\in \fin{G}$ be such that $F^{-1}\pi_\sigma^{-1}[A] = Y$. Given $y\in X_{\sigma F}$, find $z\in Y$ with $\pi_{\sigma F}(z) = y$. Find $g\in F$ with $gz\in \pi_\sigma^{-1}[A]$. Hence $g\cdot_\sigma y\in A$.

    Assume that $\Omega$ has the ED-property and fix $U\in \cN_G$ and $A\in \op(Y)$. We may assume that for some $\sigma\in \cB$ and $\epsilon > 0$, we have $U = \rmB_\sigma(2\epsilon)$. For each $c> 0$, write $A_c:= \rmB_{d_\sigma\circ \pi_\sigma}(A, c)$. As $d_\sigma\circ \pi_\sigma$ is adequate, $A_c$ is open, and we also note that $[A_c]_{d_\sigma\circ \pi_\sigma} = A_c$.
    \vspace{2 mm}
    
    \begin{claim}
        $\ol{A_{2\epsilon}}\subseteq \ol{U A}$. In particular, $A_{2\epsilon}\subseteq \Int(\ol{U A})$. 
    \end{claim}

    \begin{proof}[Proof of claim]
        Fix $y_0\in \ol{A_{2 \epsilon}}$ and $D_0\in \op(y_0, Y)$. Pick $y_1\in A$ with $d_\sigma\circ \pi_\sigma(y_1, D_0)< 2\epsilon$, and fix $D_1\in \op(y_1, A)$. We can suppose that there are $\sigma\leq \tau\in \cB$ and $E_i\in \op(\pi_{\tau}(y_i), X_{\tau})$ that $D_i = \pi_\tau^{-1}[E_i]$. Since $\Omega$ has the ED property, there is $g\in \rmB_\sigma(2\epsilon) = U$ so that for any $\upsilon\in \cB$ with $g\in \rmP(\upsilon, \tau)$, we have $(\pi_\tau^\upsilon)^{-1}[E_0]\cap (\lambda_{g, \tau}\circ \pi_{\tau g}^\upsilon)^{-1}[E_1]\neq \emptyset$. In $X_\Omega$, this implies that $D_0\cap g^{-1}D_1\neq \emptyset$. In particular, $D_0\cap U\cdot A\neq \emptyset$, and $D_0\in \op(y, Y)$ was arbitrary. 
    \end{proof}

    Since $d_\sigma$ is open-compatible and $\pi_\sigma[A_\epsilon]\in \op(X_\sigma)$, we have $\ol{\pi_\sigma[A_\epsilon]}\subseteq \rmB_{d_\sigma}(\pi[A_\epsilon], \epsilon)\subseteq \pi[A_{2\epsilon}]$. This implies that $\ol{A_\epsilon}\subseteq A_{2\epsilon}$. Putting everything together, we have $\ol{A}\subseteq \ol{A_\epsilon}\subseteq A_{2\epsilon}\subseteq \Int(\ol{U\cdot A})$.

    To see that $(Y/\partial_\sigma, \ol{\partial_\sigma})\cong (X_\sigma, d_\sigma)$, let $x, y\in Y$, and fix $0\leq c< 1$. If $d_\sigma(\pi_\sigma(x), \pi_\sigma(y))> c$, find $f\in \rmC_{d_\sigma}(X_\sigma)$ with $|f(\pi_\sigma(x))- f(\pi_\sigma(y))| > c$ (see \cite{Ben_Yaacov_2013}). Since $f\circ \pi_\sigma\in \rmC_{\partial_\sigma}(Y)$ by the continuity hypothesis on $\Omega$, we obtain $\partial_\sigma(x, y) > c$. In the other direction, suppose $\partial_\sigma(x, y)> c$. Find $\epsilon > 0$ and $A\in \op(x, Y)$ so that $y\not\in \ol{\rmB_\sigma(c+\epsilon)\cdot A}$. Write $B = \pi_\sigma[A]$ and $B(c+\epsilon) = \rmB_{d_\sigma}(B, c+\epsilon)$. Since $\ol{\rmB_\sigma(c+\epsilon)\cdot A}\supseteq \pi_\sigma^{-1}[\ol{B(c+\epsilon)}]$ using the $G$-ED property, we see that $d_\sigma(\pi_\sigma(x), \pi_\sigma(y))> c$.  
\end{proof}

\begin{prop}
    \label{Prop:Synd_Skeleton}
        Given a property from Definition~\ref{Def:GED_Unfolded}, $\Omega$ has it iff $\sa(\Omega)$ does.
    \end{prop}
    
\begin{proof}
        Suppose $\Omega$ has the TT property. Fix $\sigma\in \cB$ and $A, B\in \op(\sa(X_\sigma))$, and write $A_0 = A\cap X_\sigma$, $B_0 = B\cap X_\sigma$. Using the $G$-TT property for $\Omega$, find $g\in G$ so that, writing $\tau = \sigma\vee \sigma g$, we have $(\pi_\sigma^\tau)^{-1}[A_0]\cap (\lambda_{g, \sigma}\circ \pi_{\sigma g}^\tau)^{-1}[B_0]\cap X_\tau \neq \emptyset$. This implies $(\pi_\sigma^\tau)^{-1}[A]\cap (\lambda_{g, \sigma}\circ \pi_{\sigma g}^\tau)^{-1}[B] \neq \emptyset$. 
        
        Suppose $\sa(\Omega)$ has the TT property. Fix $\sigma\in \cB$ and $A, B\in \op(X_\sigma)$. Using the TT property for $\sa(\Omega)$, find $g\in G$ so that, writing $\tau = \sigma\vee \sigma g$, we have $(\pi_\sigma^\tau)^{-1}[\wt{\rmC}_A]\cap (\lambda_{g, \sigma}\circ \pi_{\sigma g}^\tau)^{-1}[\wt{\rmC}_B]\neq \emptyset$. As this set is open, it meets $X_\tau\subseteq \sa(X_\tau)$, implying $(\pi_\sigma^\tau)^{-1}[A]\cap (\lambda_{g, \sigma}\circ \pi_{\sigma g}^\tau)^{-1}[B]\cap X_\tau \neq \emptyset$.

        Suppose $\Omega$ has the minimality property. Fix some $\sigma\in \cB$ and $A\in \op(\sa(X_\sigma))$. Find $B\in \op(\sa(X_\sigma))$ with $\ol{B}\subseteq A$. Let $F\in \fin{G}$ witness the minimality property for $B\cap X_\sigma\in \op(X_\sigma)$. Given $y\in \sa(X_{\sigma F})$ and considering some net $(y_i)_{i\in I}$ from $X_{\sigma F}$ with $y_i\to y$, we can find $g\in F$ with $g\cdot \pi_{\sigma g}^{\sigma F}(y)\in \ol{B}\subseteq A$. 
    
        Suppose $\sa(\Omega)$ has the minimality property. Fix $\sigma\in \cB$ and $A\in \op(X_\sigma)$. Let $F$ witness the minimality property for $\wt{\rmC}_A\in \op(\sa(X_\sigma))$. This $F$ will also witness the minimality property for $A$. 
    
        Suppose $\Omega$ has the ED property. By Lemma~\ref{Lem:ED_Pseudo}, Corollary~\ref{Cor:ED_Adequate}, and Proposition~\ref{Prop:Adequate}\ref{Item:Quotient_Prop:Adequate}, we have for every $\sigma\leq \tau\in \cB$ that $\partial_{d_\sigma}$ is open-compatible and that $\partial_{d_\sigma}\circ \pi^\tau_\sigma = \partial_{d_\sigma\circ \pi^\tau_\sigma}$ is adequate. Fix $0< \epsilon < 1$ and $A, B\in \op(\sa(X_\tau))$ with $\partial_{d_\sigma}\circ \pi^\tau_\sigma(A, B)< \epsilon$. Write $A_0 = A\cap X_\tau$, $B_0 = B\cap X_\tau$. By Proposition~\ref{Prop:Adequate}\ref{Item:Dense_Prop:Adequate}, we have $d_\sigma \circ \pi^\tau_\sigma(A_0, B_0)< \epsilon$. Using the ED property for $\Omega$, find $g\in \rmB_\sigma(\epsilon)$ so that writing $\upsilon = \tau\vee \tau g^{-1}$, we have $(\pi_\tau^\upsilon)^{-1}[A_0]\cap (\lambda_{g, \tau}\circ \pi_{\tau g}^\upsilon)^{-1}[B_0]\cap X_\upsilon \neq \emptyset$. This implies $(\pi_\tau^\upsilon)^{-1}[A]\cap (\lambda_{g, \tau}\circ \pi_{\tau g}^\upsilon)^{-1}[B]\neq \emptyset$.
    
        Suppose $\sa(\Omega)$ has the ED property. Fix $0< \epsilon < 1$ and $A, B\in \op(X_\tau)$ with $d_\sigma\circ \pi^\tau_\sigma(A, B)<\epsilon$. Then $\partial_{d_\sigma}\circ \pi^\tau_\sigma(\wt{\rmC}_A, \wt{\rmC}_B)< \epsilon$. Using the ED property for $\sa(\Omega)$, find $g\in \rmB_\sigma(\epsilon)$ so that writing $\upsilon = \tau\vee \tau g^{-1}$, we have $(\pi_\tau^\upsilon)^{-1}[\wt{\rmC}_A]\cap (\lambda_{g, \tau}\circ \pi_{\tau g}^\upsilon)^{-1}[\wt{\rmC}_B] \neq \emptyset$. As this set is open, it meets $X_\upsilon\subseteq \sa(X_\upsilon)$, implying $(\pi_\tau^\upsilon)^{-1}[A]\cap (\lambda_{g, \tau}\circ \pi_{\tau g}^\upsilon)^{-1}[B]\cap X_\upsilon \neq \emptyset$ 
    \end{proof}

\begin{corollary}
    \label{Cor:Rosendal_unfolded}
    If the $G$-skeleton $\Omega$ satisfies the TT and ED properties, then $X_\Omega\in \cal{RC}_G$. 
\end{corollary}

\begin{proof}
    \Cref{Prop:GED_Unfolded}, \Cref{Prop:Synd_Skeleton}, and Theorem~\ref{Thm:Rosendal_Minimal}.
\end{proof}

\subsection{Automorphism groups}
\label{Subsection:AutGroups}

A major motivation for our definition of unfolded $G$-flow comes from considering automorphism groups of structures. Thus, before proceeding further with general considerations, we discuss these groups and introduce some relevant notation.  

By \emph{structure}, we mean a first-order structure in some relational language. To emphasize the language $\cL$, we can say \emph{$\cL$-structure}. We typically use bold letters for structures, possibly with decoration, and use the unbolded letter for the underlying set. So $A, B, C$ are the underlying sets of $\bA$, $\bB^*$, $\bC'$, etc. An \emph{embedding} from $\bA$ to $\bB$ is an injection from $A$ to $B$ which preserves all relations and non-relations. We write $\emb(\bA, \bB)$ for the set of embeddings from $\bA$ to $\bB$, and we write $\rmP(\bB, \bA) = \{f^{-1}: f\in \emb(\bA, \bB)\}$ for the set of \emph{partial isomorphisms} of $\bB$ onto $\bA$ (the similarity to the notation $\rmP(\tau, \sigma)$ from the last subsection is deliberate; see Definition~\ref{Def:Unfolded_Structural}). We write $\bA\leq \bB$ when $\rmP(\bB, \bA)\neq \emptyset$. A bijective embedding from $\bA$ to itself is an \emph{automorphism}, and we write $\aut(\bA)$ for the group of automorphisms of $\bA$. 
We say $\bA$ is a \emph{substructure} of $\bB$ if $A\subseteq B$ and the inclusion map is an embedding of $\bA$ into $\bB$. If $\bB$ is a structure and $A\subseteq B$, then $\bB|_A$ is the \emph{induced substructure} of $\bB$ on $A$.   Given $f\in \rmP(\bB, \bA)$, write $\b{dom}(f)\subseteq \bB$ for the induced substructure of $\bB$ on $\dom(f)$. 

 \begin{defin}
    \label{Def:Age_Classes}
    Let $\cL$ be a relational language and $\bK$ an infinite $\cL$-structure. Write $\fin{\bK}$ for the set of finite substructures of $\bK$. The \emph{age} of $\bK$ is the \emph{proper class} $\age(\bK):= \{\bA \text{ a finite $\cL$-structure}: \emb(\bA, \bK)\neq \emptyset\}$. Thus, up to isomorphism, the members of $\age(\bK)$ are exactly the members of $\fin{\bK}$. Conversely, given a class $\cK$ of finite $\cL$-structures, we call $\cK$ an \emph{age class} if $\cK = \age(\bK)$ for some infinite $\cL$-structure $\bK$. 
    
    $\bK$ is \emph{$\omega$-homogeneous} if for any $\bA\in \fin{\bK}$ and $f\in \rmP(\bK, \bA)$, there is $g\in \aut(\bK)$ with $g|_{\dom(f)} = f$. If $\bK$ is countable and $\omega$-homogeneous, we call $\bK$ \emph{ultrahomogeneous} or a \emph{\fr structure}. We call $\cK$ an \emph{$\omega$-homogeneous age class} if there is an $\omega$-homogeneous structure $\bK$ with $\cK = \age(\bK)$; if there is a \fr structure $\bK$ with $\cK = \age(\bK)$, we call $\cK$ a \emph{\fr class}.
\end{defin}

If $\cK$ is an age class, then $\cK$ satisfies the following properties.
\begin{itemize}
    \item 
    $\cK$ is invariant under isomorphism and contains arbitrarily large finite structures.
    \item 
    \emph{Hereditary property} (HP): If $\bA\in \cK$ and $\bB\subseteq \bA$, then $\bB\in \cK$.
    \item 
    \emph{Joint embedding property} (JEP): If $\bA, \bB\in \cK$, then there is $\bC\in \cK$ with both $\bA\leq \bC$ and $\bB\leq \bC$.
\end{itemize}

An $\omega$-homogeneous age class also satisfies:
\begin{itemize}
    \item 
    \emph{Amalgamation property} (AP): Whenever $\bA, \bB, \bC\in \cK$, $f\in \rmP(\bB, \bA)$, and $g\in \rmP(\bC, \bA)$, then there are $\bD\in \cK$, $r\in \rmP(\bD, \bB)$, and $s\in \rmP(\bD, \bC)$ with $f\circ r = g\circ s\in \rmP(\bD, \bA)$.
\end{itemize}

\fr \cite{fraisse_1954} proves that if $\cK/\cong$ is countable and satisfies the above conditions, then $\cK$ is a \fr class. Furthermore, there is up to isomorphism a unique \fr structure $\bK$ with $\age(\bK) = \cK$, which we call the \emph{\fr limit} of $\cK$ and denote by $\flim(\cK)$. Upon considering uncountable structures, there are many more examples of $\omega$-homogeneous structures, for instance saturated structures.

Recall that a topological group $G$ is \emph{non-archimedean} if it admits a base at $e_G$ consisting of open (clopen) subgroups. By considering the action of $G$ on $G/U$ as $U$ varies over clopen subgroups, we can encode $G$ as a dense subgroup of $\aut(\bK)$ for some $\omega$-homogeneous structure $\bK$. Here we equip $\aut(\bK)$ with the topology of pointwise convergence. When $G$ is Raikov complete, we can ensure that $G\cong \aut(\bK)$. 

\textbf{Fix for the remainder of this subsection}  an $\omega$-homogeneous $\cL$-structure $\bK$ with $\cK = \age(\bK)$ and $G = \aut(\bK)$. For each $\bA\in \fin{\bK}$, write $U_\bA\leq^c G$ for the pointwise stabilizer of $\bA$, and write $\sigma_\bA\in \rm{SN}(G)$ for the function $1-\chi_{U_\bA}$. Then $\{\sigma_\bA: \bA\in \fin{\bK}\}$ is a conjugation-invariant base of semi-norms closed under finite pointwise maximums.

If $Y$ is a $G$-ED flow and $\bA\subseteq \bB\in \fin{\bK}$, we typically write $Y_\bA = Y/\partial_{\sigma_\bA}$, $\pi^\bB_\bA = \pi^{\sigma_\bB}_{\sigma_\bA}$, etc. We also write $\rm{clop}_\bA(Y) = \{Q\subseteq Y \text{ clopen}: U_\bA\cdot Q = Q\}$. Observe that the pseudometric $\partial_\bA$ is discrete $\{0, 1\}$-valued. This has easy but important consequences that we collect in the next proposition. 

\begin{prop}
    \label{Prop:NA_Levels}
    If $Y$ is a $G$-ED flow and $\bA\in \fin{\bK}$, then $Y_\bA$ is a compact ED space and is the Stone dual of $\rm{clop}_\bA(Y)$. The map $\pi_\bA\colon Y\to Y_\bA$ is open.
\end{prop}

\begin{proof}
    That $Y_\bA$ is a compact ED space and the Stone dual of $\rm{clop}_\bA(Y)$ follows from  $\ol{\partial_{\sigma_\bA}}$ being a discrete lsc metric on $Y_\bA$, combined with Proposition~\ref{Prop:Adequate}. That $\pi_\bA$ is open follows from the adequacy of $\partial_{\sigma_\bA}$.
\end{proof}

When working with automorphism groups, the unfolded $G$-flows and $G$-skeletons we work with will have a particularly nice form. First, note that $\fin{\bK}$ is a join semi-lattice with the join given by union and the order given by subset inclusion. The natural left action on $\fin{\bK}$ can be turned into a right action by defining $\bA g := g^{-1}\bA$. The map sending $\bA\in \fin{\bK}$ to $\sigma_\bA\in \rm{SN}(G)$ is then order preserving, join preserving, and $G$-equivariant.

\begin{defin}
\label{Def:Unfolded_Structural}
A \emph{structural} unfolded $G$-space $\Omega$ is an unfolded $G$-space where $\cB = \fin{\bK}$, $\theta\colon \fin{\bK}\to \rm{SN}(G)$ is given by $\theta(\bA) = \sigma_\bA$, and each $d_\bA$  is just the discrete $\{0, 1\}$-valued lsc metric on the space $X_\bA$. 
$\Omega$ is  a \emph{structural} unfolded $G$-\emph{flow} if each $X_\bA$ is compact, and is a \emph{structural $G$-skeleton} if each $X_\bA$ is discrete. 
The notation $\rmP(\bB, \bA)$ is now overloaded for $\bA, \bB\in \fin{\bK}$, since it can either be a set of partial isomorphisms or, per Definition~\ref{Def:Unfolded_Flow}, the set of group elements $\{g\in G: g|_{g^{-1}A} \in \rmP(\bB, \bA)\}$. We can typically just view $\rmP(\bB, \bA)$ as a set of partial isomorphisms, even in situations when the set of group elements is intended.

Given $\bA, \bB\in \fin{\bK}$, $f\in \rmP(\bB, \bA)$, and $g\in G$ extending $f$, we sometimes write  $\pi^\bB_f\colon X_\bB\to X_\bA$ for the map $\lambda_{g, \bA}\circ \pi^\bB_{g^{-1}\bA}$, $\pi_f\colon X_\Omega\to X_\bA$ for the map $\lambda_{g, \bA}\circ \pi_{g^{-1}\bA}$, and $\lambda_{f, \bA}$ for $\lambda_{g, \bA}$. Note that the continuity axiom (Definition~\ref{Def:Unfolded_Flow}) ensures that this is independent of our choice of $g\in G$ extending $f$. Given $x\in X_\bB$ and $y\in X_\Omega$, we often just write $f\cdot x$, $f\cdot y$ for $\pi_f^\bB(x)$, $\pi_f(y)$ respectively. Note that this ``action'' notation is associative whenever the compositions make sense. 
\end{defin}

The prototypical example of a structural $G$-skeleton we will encounter comes from considering expansion classes of $\cK$. A straightforward coding argument (Proposition~\ref{Prop:Expansion_Skeleton}) shows that this is entirely general.
    
\begin{defin}
    \label{Def:Expansion_Classes}
    If $\cL^*\supseteq \cL$ is a relational language and $\bB^*$ is an $\cL^*$-structure, then $\bB^*|_\cL$ is the \emph{reduct} of $\bB^*$ to $\cL$, i.e.\ the structure obtained from $\bB^*$ by forgetting the interpretations of relational symbols in $\cL^*\setminus \cL$, and conversely, we call $\bB^*$ an \emph{$\cL^*$-expansion} of $\bB$.

    An \emph{$\cL^*$-expansion class} of $\cK$ is an isomorphism-invariant class $\cK^*$ of $\cL^*$-structures for some relational language $\cL^*\supseteq \cL$ with $\cK^*|_\cL:= \{\bA^{\!*}|_\cL: \bA^{\!*}\in \cK^*\} = \cK$. We can write $\cK^*/\cK$ to emphasize that $\cK^*$ is an expansion of $\cK$.  Given $\cS\subseteq \cK$, we write $\cK^*(\cS) = \{\bA^{\!*}\in \cK^*: \bA^{\!*}|_\cL \in \cS\}$. This is a set provided $\cS$ is. If $S = \{\bA\}$ for some $\bA\in \cK$, we write $\cK^*(\bA)$ for $\cK^*(\{\bA\})$. 

    We call $\cK^*/\cK$ \emph{reasonable} if for any $\bA\subseteq \bB\in \cK$ and any $\bA^{\!*}\in \cK^*(\bA)$, there is $\bB^*\in\cK^*(\bB)$ with $\bB^*|_A = \bA^{\!*}$.  
    
    Given a reasonable $\cL^*$-expansion class $\cK^*/\cK$, the structural $G$-skeleton $\Omega_{\cK^*}$ is defined by setting $X_\bA = \cK^*(\bA)$ for each $\bA\in \fin{\bK}$. We write $X_{\cK^*}$ for $X_{\sa(\Omega)_{\cK^*}}$. Given $\bA, \bB\in \fin{\bK}$, $f\in \rmP(\bB, \bA)$, and $\bB^*\in \cK^*(\bB)$, we define $\pi_f^\bB(\bB^*) = f\cdot \bB^*$ to be the unique member of $\cK^*(\bA)$ such that $f\in \rmP(\bB^*, f\cdot \bB^*)$.
\end{defin}

\begin{remark}
    Inside $X_{\cK^*}$, one can consider the (possibly empty) subspace $X_{\cK^*}^0:= \{y\in X_{\cK^*}: \forall \bA\in \fin{\bK}\, \pi_\bA(y)\in \cK^*(\bA)\}$, viewing this as the ``standard part'' of $X_{\cK^*}$. When $\bK$ is countable, then reasonability ensures that $X_{\cK^*}^0\neq \emptyset$ (and in fact is dense $G_\delta$). 
\end{remark}

\begin{prop}
    \label{Prop:Expansion_Skeleton}
    For any structural $G$-skeleton $\Omega$, there are a relational language $\cL^*$ and a reasonable $\cL^*$-expansion class $\cK^*/\cK$ so that $\Omega\cong \Omega_{\cK^*}$.
\end{prop}

\begin{proof}
     To form $\cL^*\supseteq \cL$, for each $\bA\in \fin{\bK}$, enumeration $e\colon |A|\to A$, and $s\in X_\bA$, add a new $|A|$-ary relation symbol $R_{e,s}$. We now turn to defining $\cK^*$. For each $\bB\in \cK$, $f\in \rmP(\bK, \bB)$, and $t\in X_{\b{dom}(f)}$, we define an $\cL^*$-expansion $\bB^{f, t}$ of $\bB$ as follows. Given $\bA\in \fin{\bK}$, $s\in X_{\bA}$, an enumeration $e\colon |A|\to A$, and $b_0,..., b_{|A|-1}\in B$, we declare that $R_{e, s}(b_0,..., b_{|A|-1})$ holds in $\bB^{f, t}$ iff the map $(b_i\to e(i))_{i< |A|}$ is a partial isomorphism $g\in \rmP(\bB, \bA)$, and $(g\circ f)\cdot t = s$. In the case that $\bB\in \fin{\bK}$ and $f\in \rmP(\bK, \bB)$ is the partial identity map, we just write $\bB^t$. 
    
    We collect some observations about these objects:
    \begin{itemize}
        \item 
        Fixing $f\in \rmP(\bK, \bB)$, the map sending $t\in X_{\b{dom}(f)}$ to the $\cL^*$-structure $\bB^{f, t}$ is injective.
        \item 
        If $f_0, f_1\in \rmP(\bK, \bB)$ and $h\in \rmP(\bK, \b{dom}(f_1))$ is such that $f_1\circ h = f_0$, then we have $\bB^{f_0, t} = \bB^{f_1, h\cdot t}$.
        \item 
        If $\bA, \bB\in \cK$, $h\in \rmP(\bB, \bA)$, $f\in \rmP(\bK, \bB)$, $q\in \rmP(\b{dom}(f), \b{dom}(h\circ f))$ denotes the partial identity map, and $t\in X_{\b{dom}(f)}$, then $h{\cdot}\bB^{f, t} = \bA^{h\circ f, q\cdot t}$ 
    \end{itemize}
      We define $\cK^*$ by setting $\cK^*(\bB) = \{\bB^{f, t}: t\in S_{\b{dom}(f)}\}$ for some (any) $f\in \rmP(\bK, \bB)$, then extending to make $\cK^*$ invariant under isomorphism. 
\end{proof}

Proposition~\ref{Prop:Expansion_Skeleton} allows us to work with expansion classes rather than structural $G$-skeletons. The various properties we have defined for $G$-skeletons then become familiar combinatorial properties of expansion classes. 

\begin{defin}
    \label{Def:Synd_Expansion}
    We call $\cK^*/\cK$ \emph{syndetic} if for any $\bA^{\!*}\in \cK^*$, there is $\bB\in \cK$ so that for any $\bB^*\in \cK^*(\bB)$, we have $\bA^{\!*}\leq \bB^*$.

    We remark that this has gone by many names in the literature, including the \emph{order property} \cite{KPT}, the \emph{expansion property} \cite{Lionel_2013}, and the \emph{minimal property} \cite{ZucThesis}.
\end{defin}

\begin{prop}
\label{Prop:Unfolded_Expansion}
    Fix a reasonable $\cL^*$-expansion class $\cK^*/\cK$. Write $\Omega = \Omega_{\cK^*}$.
    \begin{enumerate}[label=\normalfont(\arabic*)]
        \item \label{Item:GTT_Prop:UE}
        $\Omega$ has the TT-property iff $\cK^*$ has the JEP.
        \item \label{Item:Synd_Prop:UE}
        $\Omega$ has the minimality property iff $\cK^*/\cK$ is syndetic.
        \item \label{Item:GED_Prop:UE}
        $\Omega$ has the ED property iff $\cK^*$ has the AP. 
    \end{enumerate} 
\end{prop}

\begin{proof}
    \ref{Item:GTT_Prop:UE}: Suppose $\cK^*$ has the JEP. Towards showing $\Omega$ has the TT-property, fix $\bA\in \fin{\bK}$ and $\bA^{\!*}$, $\bA'\in \cK^*(\bA)$. Use JEP to find $\bB^*\in \cK^*$ with $\bA^{\!*}, \bA'\leq \bB^*$. We may assume that $\bB^*|_\cL\in \fin{\bK}$ and that $\bA^{\!*}\subseteq \bB^*$. Find $f\in \rmP(\bB, \bA)$ with $f\cdot \bB^* = \bA'$, and let $g\in G$ extend $f$. Then  $\bB^*\in (\pi_\bA^\bB)^{-1}[\bA^{\!*}]\cap (\lambda_{g, \sigma}\circ \pi^\bB_{g^{-1}\bA})^{-1}[\bA']\neq \emptyset$. 
    
    Now suppose $\Omega$ has the TT-property. Towards showing $\cK^*$ has the JEP, fix $\bA^{\!*}, \bB^*\in \cK^*$. Writing $\bA, \bB$ for the $\cL$-reducts, we may assume $\bA, \bB\in \fin{\cK}$, so write $\bC = \bA\cup \bB\in \fin{\bK}$. Using reasonability, find $\bC^*, \bC'\in \cK^*(\bC)$ with $\bC^*|_A = \bA^{\!*}$, $\bC'|_B = \bB^*$. Using the TT-property of $\Omega$, find $g\in G$ so that, writing $\bD = \bC\cup g^{-1}\bC\in \fin{\bK}$, we have $(\pi_\bC^\bD)^{-1}[\bC^*]\cap (\lambda_{g, \sigma}\circ \pi^\bD_{g^{-1}\bC})^{-1}[\bC']\neq \emptyset$. Any $\bD^*$ in this set will witness JEP for $\bA^{\!*}, \bB^*$.
    \vspace{3 mm}

    \noindent
    \ref{Item:Synd_Prop:UE}:  Suppose $\cK^*/\cK$ is syndetic. Towards showing $\Omega$ has the minimality property, fix $\bA\in \fin{\bK}$ and $\bA^{\!*}\in \cK^*(\bA)$. Find $\bB\in \fin{\bK}$ with $\bA\subseteq \bB$ witnessing that $\cK^*/\cK$ is syndetic for $\bA^{\!*}$. Find $F\in \fin{G}$ with the property that for any $f\in \rmP(\bB, \bA)$, there is $g\in F$ extending $f$. Then $\bB\supseteq \bigcup_{g\in F} g^{-1}\bA := \bC$. Given $\bC^*\in \cK^*(\bC)$, use reasonability to lift to $\bB^*\in \cK^*(\bB)$, then find $g\in F$ with $g|_{g^{-1}A}\cdot \bB^* = g|_{g^{-1}A}\cdot \bC^* = \bA^{\!*}$, showing the minimality property for $\Omega$.

    Now suppose $\Omega$ has the minimality property. Towards showing $\cK^*/\cK$ is syndetic, fix $\bA\in \fin{\bK}$ and $\bA^{\!*}\in \cK^*(\bA)$. By the minimality property of $\Omega$, find $F\in \fin{G}$ so that, writing $\bB = \bigcup_{g\in F}g^{-1}\bA$, we have for any $\bB^*\in \cK^*(\bB)$ that there is $g\in F$ with $g|_{g^{-1}A}\cdot \bB^* = \bA^{\!*}$. In particular, there is $f\in \rmP(\bB, \bA)$ with $f\cdot \bB^* = \bA^{\!*}$, showing that $\cK^*/\cK$ is syndetic.
    \vspace{3 mm}
    
    \noindent
    \ref{Item:GED_Prop:UE}: Suppose $\cK^*$ has the AP. Towards showing that $\Omega$ has the ED-property, fix $\bA\subseteq \bB\in \fin{\bK}$ and $\bB^*, \bB'\in \cK^*(\bB)$ with $d_\bA\circ \pi^\bB_\bA(\bB^*, \bB')< 1$, i.e.\ with $\bB^*|_A = \bB'|_A =: \bA^{\!*}$. Find $\bC^*\in \cK^*$ witnessing the AP for $\bA^{\!*}\hookrightarrow \bB^*$ and $\bA^{\!*}\hookrightarrow \bB'$; we may assume $\bC^*|_\cL =: \bC\in \fin{\bK}$, that $\bB^*\subseteq \bC^*$, and for some $f\in \rmP(\bC, \bB)$ with $f\cdot \bC^* = \bB'$ and $f|_A = \rm{id}_\bA$, we have $\bC = \bB\cup \b{dom}(f)$. Let $g\in G$ extend $f$, and note that $\sigma_\bA(g) = 0$. Then $\bC^*\in (\pi_\bB^\bC)^{-1}[\bB^*]\cap (\lambda_{g, \bB}\circ \pi^\bC_{g^{-1}\bB})^{-1}[\bB']$, showing that $\Omega$ has the ED property. 
    
    Now suppose $\Omega$ has the ED-property. Towards showing that $\cK^*$ has the AP, fix $\bA^{\!*}, \bB^*, \bC^*\in \cK^*$ with $\bA^{\!*}\subseteq \bB^*$ and $\bA^{\!*}\subseteq \bC^*$. Letting $\bA, \bB, \bC$ denote the $\cL$-reducts, we may also assume that $\bA, \bB, \bC\in \fin{\bK}$. 
    Let $\bD = \bB\cup \bC$, and use reasonability to lift $\bB^*$ and $\bC^*$ to $\bD^*, \bD'\in \cK^*(\bD)$, respectively. Since $\bD^*|_A =\bD'|_A = \bA^{\!*}$, we use the ED property of $\Omega$ to find $g\in G$ with $\sigma_\bA(g)< 1$ (i.e.\ $g\in U_\bA$) and, writing $\bE = \bD\cup g^{-1}\bD$, with $(\pi_\bD^\bE)^{-1}[\bD^*]\cap (\lambda_{g, \bD}\circ \pi^\bE_{g^{-1}\bD})^{-1}[\bD']\neq \emptyset$. Any $\bE^*$ in this set satisfies $\bE^*|_B = \bB^*$ and $g|_{g^{-1}C}\cdot \bE^* = \bC^*$, hence witnesses the amalgamation property for $\bA^{\!*}, \bB^*, \bC^*, \bA^{\!*}\subseteq \bB^*, \bA^{\!*}\subseteq \bC^*$. 
\end{proof}

\subsection{Abstract KPT correspondence}
\label{Subsection:KPT}

In this subsection, $G$ is an arbitrary topological group and $\Omega$ is a $G$-skeleton, with notation as in Definitions~\ref{Def:Unfolded_Flow} and \ref{Def:Skeleton}. 

Given $x\in X_\sigma$, $x'\in X_\tau$, and $\epsilon > 0$, we write $\rmP(x', x, \epsilon) = \{h\in \rmP(\tau, \sigma): d_\sigma(h\cdot_\sigma x', x)< \epsilon\}$. Similarly, if $y\in X_\Omega$, we write $\rmP(y, x, \epsilon) = \{h\in G: d_\sigma(\pi_\sigma(hy), x) < \epsilon)\}$. Given a cofinal net $(\tau_i)_{i\in I}$ from $\cB$, $x_i\in X_{\tau_i}$, and $y\in X_\Omega$, we write $x_i\to y$ if $\pi_{\tau_i}^{-1}[x_i]\to \{y\}$ in the Vietoris topology, equivalently if for any choice of $y_i\in X_\Omega$ with $\pi_{\tau_i}(y_i) = x_i$, we have $y_i\to y$.  Fixing $x_i\to y$ and $g\in G$, then if $g\not\in \rmP(x_i, x, \epsilon)$ for frequently many $i\in I$, we have $g\not\in \rmP(y, x, \epsilon)$, and if $g\in \rmP(x_i, x, \epsilon)$ for frequently many $i\in I$, we have $g\in \rmP(y, x, \epsilon+\eta)$ for any $\eta > 0$.

The next definition will make use of some dynamical properties of Lipschitz functions on $G$, which we briefly review. We equip $\rmC_\sigma(G, [0, 1])$ with the topology of pointwise convergence, making it a compact space. We can turn it into a $G$-flow where given $f\in \rmC_\sigma(G, [0, 1])$ and $g\in G$, $g\cdot f\in \rmC_\sigma(G, [0, 1])$ is defined via $(g\cdot f)(h) = f(hg)$. In particular, given $u\in \sa(G)$, $u\cdot f$ is defined as usual for $G$-flows, i.e.\ $u\cdot f = \lim_{g\to u} g\cdot f$. Notice that for each $g\in G$, we have $(u\cdot f)(g) = \wt{f}(gu)$, where $\wt{f}\colon \sa(G)\to [0, 1]$ continuously extends $f$.

\begin{defin}
    \label{Def:NewRP}
    We say that $\Omega$ has the \emph{Ramsey property} if 
    \begin{align*}
    &\forall \sigma\in \cB \forall x\in X_\sigma \forall \epsilon, \delta > 0 \forall F\in \fin{G}\\ 
    &\exists \tau\in \cB\exists x'\in X_\tau\\
    &\forall f\in \rmC_\sigma(G, [0, 1])\\ 
    &\exists h\in G\\ 
    &Fh\subseteq \rmP(\tau, \sigma) \text{ and }\rm{diam}(f[Fh\cap \rmP(x', x, \epsilon)])< 2\epsilon+\delta.
\end{align*}
    Note by a compactness argument that the Ramsey property is equivalent to the following finitary version (finitary RP):
    \begin{align*}
    &\forall \sigma\in \cB \forall x\in X_\sigma \forall \epsilon, \delta > 0 \forall F\in \fin{G}\\ 
    &\exists \tau\in \cB\exists x'\in X_\tau\exists S\in \fin{G}\\
    &\forall f\in \rmC_\sigma(FS, [0, 1])\\ 
    &\exists h\in S\\ 
    &Fh\subseteq \rmP(\tau, \sigma) \text{ and }\rm{diam}(f[Fh\cap \rmP(x', x, \epsilon)])< 2\epsilon+\delta.
\end{align*}
    We say that $\Omega$ has the \emph{strong Ramsey property} if:
    \begin{align*}
    &\forall \sigma\in \cB \forall x\in X_\sigma \forall \epsilon, \delta > 0 \forall F\in \fin{G}\\ 
    &\exists \tau\in \cB \exists S\in \fin{G}\\
    &\forall x'\in X_\tau \forall f\in \rmC_\sigma(FS, [0, 1])\\ 
    &\exists h\in S\\ 
    &Fh\subseteq \rmP(\tau, \sigma) \text{ and }\rm{diam}(f[Fh\cap \rmP(x', x, \epsilon)])< 2\epsilon+\delta.
\end{align*}
    We say that $\Omega$ is \emph{finitely oscillation stable} if:
    \begin{align*}
          \quad &\forall y\in X_\Omega\forall \sigma\in \cB \forall x\in X_\sigma \forall \epsilon > 0  \forall f\in \rmC_\sigma(G, [0, 1])\\
        &\exists p\in \sa(G)\\
        &\rm{diam}((p\cdot f)[\rmP(py, x, \epsilon)])\leq 2\epsilon.
    \end{align*}
    We say that $\Omega$ is \emph{weakly finitely oscillation stable} if: 
    \begin{align*}
        &\exists y\in X_\Omega\\ &\forall \sigma\in \cB \forall x\in X_\sigma \forall \epsilon > 0 \forall f\in \rmC_\sigma(G, [0, 1])\\
        &\exists p\in \sa(G)\\
        &\rm{diam}((p\cdot f)[\rmP(py, x, \epsilon)])\leq 2\epsilon.
    \end{align*}
\end{defin}

The name ``Ramsey property'' is borrowed from the eponymous property that a class $\cK$ of finite structures might enjoy:

\begin{itemize}
    \item 
    \emph{Ramsey property} (RP): Whenever $\bA\leq \bB\in \cK$ and $0< r< \omega$, there is $\bC\in \cK$ with $\bB\leq \bC$ and
    $$\bC\to (\bB)^\bA_{r}.$$
    The arrow notation means that for any $\gamma\colon \rmP(\bC, \bA)\to r$, there is $x\in \rmP(\bC, \bB)$ with $\gamma$ monochromatic on $\rmP(\bB, \bA)\circ x$. It is equivalent to demand this just for $r = 2$.
\end{itemize}

\begin{prop}
    \label{Prop:Ramsey}
    Suppose $\cK$, $\bK$, $G = \aut(\bK)$ are as in the previous subsection, $\cK^*/\cK$ is syndetic, and $\Omega = \Omega_{\cK^*}$. Then $\Omega$ has the Ramsey property iff $\cK^*$ has the Ramsey property. 
\end{prop}

\begin{proof}
    We first translate what it means for $\Omega$ to have the Ramsey property in more combinatorial language:
    \begin{itemize}
        \item 
        For any $\bA\subseteq \bB\in \fin{\bK}$, $\bA^{\!*}\in \cK^*(\bA)$, and $r < \omega$, there are $\bC\in \fin{\bK}$ and $\bC^*\in \cK^*(\bC)$ such that for any $\gamma\colon \rmP(\bC^*, \bA^{\!*})\to r$, there are $f\in \rmP(\bC, \bB)$ and $i< r$ such that $(\rmP(\bB, \bA)\circ f)\cap \rmP(\bC^*, \bA^{\!*}) \subseteq \gamma^{-1}[\{i\}]$. 
    \end{itemize}
    In this translation, $x\in X_\sigma$ corresponds to $\bA^{\!*}\in \cK^*(\bA)$, $\epsilon$ and $\delta$ correspond to the number of colors $r$, $F$ corresponds to $\bB$, $x'\in X_\tau$ corresponds to $\bC^*\in \cK^*(\bC)$, and $S$ corresponds to $\rmP(\bC, \bB)$.
    
    Suppose $\cK^*$ has the RP. Fix $\bA\subseteq \bB\in \fin{\bK}$, $\bA^{\!*}\in \cK^*(\bA)$, and $r< \omega$. Enlarging $\bB$ if needed, we can suppose that $\bB$ witnesses that $\cK^*/\cK$ is syndetic for $\bA^{\!*}$. Pick any $\bB^*\in \cK^*(\bB)$, and find $\bC^*\in \cK$ with $\bC^*\to (\bB^*)^{\bA^{\!*}}_r$. We may suppose $\bC^*|_\cL=: \bC\in \fin{\bK}$. Fix $\gamma\colon \rmP(\bC^*, \bA^{\!*})\to r$. We can find $f\in \rmP(\bC^*, \bB^*)$ and $i< r$ so that $(\rmP(\bB^*, \bA^{\!*})\circ f)\subseteq \gamma^{-1}[\{i\}]$. We conclude by observing that if $h\in \rmP(\bB, \bA)\setminus \rmP(\bB^*, \bA^{\!*})$, then $h\circ f\not\in \rmP(\bC^*, \bA^{\!*})$. 
    
    Now suppose $\Omega$ has the RP. Fix $\bA^{\!*}\subseteq \bB^*\in \cK^*$. Letting $\bA, \bB$ denote the $\cL$-reducts, we can suppose that $\bA\subseteq \bB\in \fin{\bK}$. Let $\bC\in \fin{\bK}$ with $\bB\subseteq \bC$ witness that $\cK^*/\cK$ is syndetic for $\bB^*$. Find $\bD\in \fin{\bK}$ and $\bD^*\in \cK^*(\bD)$ as guaranteed by $\Omega$ having RP for $\bA\subseteq \bC$ and $\bA^{\!*}$. We claim that $\bD^*\to (\bB^*)^{\bA^{\!*}}_2$. Let $\gamma\colon \rmP(\bD^*, \bA^{\!*})\to 2$ be a coloring. By assumption, there are $f\in \rmP(\bD, \bC)$ and $i< 2$ with $(\rmP(\bC, \bA)\circ f)\cap \rmP(\bD^*, \bA^*)\subseteq \gamma^{-1}[\{i\}]$. Write $\bC^* = x\cdot \bD^*$. Pick some $h\in \rmP(\bC^*, \bB^*)$, which is non-empty by our assumption on $\bC$. Then $h\circ f\in \rmP(\bD^*, \bB^*)$, and $\gamma$ is monochromatic on $\rmP(\bB^*, \bA^{\!*})\circ (h\circ f)$.
\end{proof}

\begin{remark}
    A result due to Ne\v{s}et\v{r}il and R\"odl \cite{NR_1977} says that if $\cK$ is a class of structures satisfying the HP, JEP, and RP, then $\cK$ satisfies the AP. In analogy, the following question is natural: Given a $G$-skeleton $\Omega$, does having the TT-property and the RP imply the ED property? 
\end{remark}

The next theorems describe an ``abstract KPT correspondence'' for topological groups with tractable minimal dynamics. We believe that the class $\sf{TMD}$ is the most general possible setting where any form of KPT correspondence can be used to describe $\rmM(G)$. However, we emphasize that the next theorem is new even for Polish groups with metrizable universal minimal flow. 

\begin{theorem}
    \label{Thm:NewRP}
    Fix $\Omega$ a $G$-skeleton with the ED and minimality properties, and write $Y:= X_\Omega$. The following are equivalent:
    \begin{enumerate}[label=\normalfont(\arabic*)]
        \item \label{Item:RP_Thm:NRP}
        $\Omega$ has the Ramsey property.
        \item \label{Item:RP*_Thm:NRP}
        $\Omega$ has the strong Ramsey property.
        \item \label{Item:FOS_Thm:NRP}
        $\Omega$ is finitely oscillation stable.
        \item \label{Item:WFOS_Thm:NRP}
        $\Omega$ is weakly finitely oscillation stable.
        \item \label{Item:MG_Thm:NRP}
        $Y\cong \mg$.
    \end{enumerate}
\end{theorem}

\begin{proof}
    $\ref{Item:RP_Thm:NRP}\Rightarrow \ref{Item:RP*_Thm:NRP}$: Suppose \ref{Item:RP_Thm:NRP} holds and that $\sigma$, $x$, $\epsilon$, $\delta$, and $F$ are given. Fix some very small $\eta > 0$, and apply the finitary version of the Ramsey property with $\sigma$, $x$, $\epsilon + \eta$, $\delta - 2\eta$, and $F$ to find $\tau$, $x'$, and $S$. Let $F'\in \fin{G}$ be as guaranteed by the minimality property for $\rmB_{d_\tau}(x', \eta)$. Writing $\upsilon = \tau F'$, we will show that $\upsilon\in \cB$ and $SF'\in \fin{G}$ witness the strong Ramsey property, so fix $x'' \in X_\upsilon$ and $f\in \rmC_\sigma(G, [0,1])$. Find $g\in F'$ with $g\cdot_\tau x''\in \rmB_{d_\tau}(x', \eta)$. Considering the function $g\cdot f\in \rmC_\sigma(G, [0,1])$, find $h\in S$ satisfy $Fh\subseteq \rmP(\tau, \sigma)$ and $\rm{diam}((g\cdot f)[Fh\cap \rmP(x', x, \epsilon+\eta)])< 2\epsilon+\delta$. Then $Fhg\subseteq \rmP(\upsilon, \sigma)$, and to see that $\rm{diam}(f[Fhg\cap \rmP(x'', x, \epsilon)])< 2\epsilon+\delta$, suppose $k\in F$ satisfies $khg\in \rmP(x'', x, \epsilon)$. This implies that $kh\in \rmP(x', x, \epsilon+\eta)$, hence $f[Fhg\cap \rmP(x'', x, \epsilon)]\subseteq (g\cdot f)[Fh\cap \rmP(x', x, \epsilon+\eta)]$. 
    \vspace{3 mm}
    
    \noindent
    $\ref{Item:RP*_Thm:NRP}\Rightarrow \ref{Item:FOS_Thm:NRP}$: Fix $y$, $\sigma$, $x$, $\epsilon$, and $f$ as in \ref{Item:FOS_Thm:NRP}. Using this $\sigma$, $x$, and $\epsilon$, and given $\delta> 0$ and $F\in \fin{G}$, let $\tau_{\delta, F}\in \cB$ and $S_{\delta, F}\in \fin{G}$ be as guaranteed by \ref{Item:FOS_Thm:NRP}. We remark that any $\tau\in \cB$ with $\tau\geq \tau_{\delta, F}$ will also witness \ref{Item:RP*_Thm:NRP} with the same $S_{\delta, F}$, since given $x'\in X_\tau$, one simply considers $\pi^\tau_{\tau_{\delta, F}}(x')\in X_{\tau_{\delta, F}}$. We can find $(I, \leq)$ a directed partial order, $\tau_i\in \cB$, and $x_i\in \tau_i$ for $i\in I$ satisfying both:
    \begin{itemize}
        \item 
        $x_i\to y$
        \item 
        For all $\delta > 0$ and $F\in \fin{G}$, eventually $\tau_i\geq \tau_{\delta, F}$. 
    \end{itemize}
    Fix a particular $F$ and $\delta$ for the moment. For all large enough $i\in I$, we can find $h_i\in S_{\delta, F}$ with $Fh_i\subseteq \rmP(\tau_i, \sigma)$ and $\rm{diam}(f[Fh_i\cap \rmP(x_i, x, \epsilon)]) < 2\epsilon + \delta$. Passing to a subnet, we can assume $h_i = h_{\delta, F}\in S_{\delta, F}$ is constant. This implies that $\rm{diam}(f[Fh_{\delta, F}\cap \rmP(y, x, \epsilon)])< 2\epsilon + \delta$. Taking $\delta\to 0$ and $F\to G$, let $h_{\delta, F}\to p\in \sa(G)$, yielding $\rm{diam}((p\cdot f)[\rmP(py, x, \epsilon)]\leq 2\epsilon$.
    \vspace{3 mm}

    \noindent
     $\ref{Item:FOS_Thm:NRP}\Rightarrow \ref{Item:WFOS_Thm:NRP}$: Clear.
    \vspace{3 mm}
    
    \noindent
    $\ref{Item:WFOS_Thm:NRP}\Rightarrow \ref{Item:RP_Thm:NRP}$: Suppose $\neg$\ref{Item:RP_Thm:NRP}, witnessed by bad $\sigma$, $x$, $\epsilon$, $\delta$, $F$. Towards showing $\neg \ref{Item:WFOS_Thm:NRP}$, fix $y\in X_\Omega$. Then $\sigma$, $x$, and $\epsilon+ \eta$ for any $0<\eta < \delta/2$ will witness $\neg$\ref{Item:WFOS_Thm:NRP}. Fixing $y\in Y$, we need to find a bad $f$. Find a directed partial order $(I, \leq)$ and $\tau_i\in \cB$, $x_i\in X_{\tau_i}$ for each $i\in I$ so that $x_i\to y$. For each $i\in I$, let $f_i$ witness $\neg$\ref{Item:RP_Thm:NRP} with respect to $x_i\in X_{\tau_i}$. Passing to a subnet, let $f = \lim_i f_i\in \rmC_\sigma(G, [0, 1])$. 
     
     To show this $f$ does indeed witness $\neg$\ref{Item:WFOS_Thm:NRP}, fix $p\in \sa(G)$.  Write $B = \{z\in X_\Omega: \partial_{d_\sigma}(\pi_\sigma(z), x)< \epsilon\}\in \op(X_\Omega)$ and $F_0 = F\cap \rmP(py, x, \epsilon)$. Then $F_0py\subseteq B$, so find $A\in \op(py, X_\Omega)$ such that $F_0A\subseteq B$. Fix $h\in G$ with $hy\in A$. Passing to a subnet if needed, find $x_i\in X_{\tau_i}$ with $x_i\to y$. For large enough $i\in I$, we have $Fh\subseteq \rmP(\tau_i, \sigma)$. By our assumption on $f_i$, this implies $\rm{diam}(f_i[Fh\cap \rmP(x_i, x, \epsilon)])\geq 2\epsilon + \delta$.  Hence for any $\eta > 0$ and any $h\in G$ with $hy\in A$, we have $\rm{diam}(f[Fh\cap \rmP(y, x, \epsilon+\eta)])\geq 2\epsilon +\delta$. Now for any $h\in G$, we have $f[Fh\cap \rmP(y, x, \epsilon+\eta)] = (h\cdot f)[F\cap \rmP(hy, x, \epsilon + \eta)]$, so taking $h\to p$, we have for any $\eta > 0$ that  $\rm{diam}((p\cdot f)[F\cap \rmP(py, x, \epsilon +\eta)])\geq 2\epsilon +\delta$.
    \vspace{3 mm}

    \noindent
    $\ref{Item:WFOS_Thm:NRP}\Rightarrow \ref{Item:MG_Thm:NRP}$: Suppose $\neg$\ref{Item:MG_Thm:NRP}. Fix a minimal subflow $M\subseteq \sa(G)$. Fix $y\in Y$, and let $\phi\colon M\to Y$ be a $G$-map such that $\phi^{-1}[y]$ contains an idempotent $u\in M$. Note that $y = \phi(u) = \phi(uu) = u\phi(u) = uy$. By \Cref{Prop:Level_Maps} and \Cref{Prop:GED_Unfolded}\ref{Item:GED_Prop:Unfolded}, we obtain for each $\sigma\in \cB$ a continuous surjection $\phi_\sigma\colon M/\partial_\sigma\to \sa(X_\sigma) $.
    Since $Y$ is $G$-ED and $\phi$ is not an isomorphism, we have that $\phi$ is not irreducible. Hence for some $\sigma$, the map $\phi_\sigma$ is not weakly almost 1-1, implying that for some $x\in X_\sigma$, we have $|\phi_\sigma^{-1}[x]|> 1$. Find $p, q\in M$ with $\pi_{\partial_\sigma}(p)\neq \pi_{\partial_\sigma}(q)\in \phi_\sigma^{-1}[x]$. Pick $\epsilon > 0$ satisfying $\partial_\sigma(p, q) > 2\epsilon$. Let $f\in \rmC_\sigma(G, [0, 1])$ be such that upon continuously extending to $\sa(G)$, we have $|f(p)-f(q)| > 2\epsilon$. Let $A\in \op(p, M)$, $B\in \op(q, M)$ satisfy $\partial_\sigma(A, B)> 2\epsilon$ and $\phi_\sigma\circ \pi_{\partial_\sigma}[A], \phi_\sigma\circ \pi_{\partial_\sigma}[B]\subseteq \rmB_{\ol{\partial}_{d_\sigma}}(x, \epsilon)$. Letting $g, h\in G$ satisfy $gu\in A$ and $hu\in B$, we have $g, h\in \rmP(y, x, \epsilon) = \rmP(uy, x, \epsilon)$ and $|(u\cdot f)(g)-(u\cdot f)(h)|> 2\epsilon$.   
    \vspace{3 mm}
        
    \noindent
    $\ref{Item:MG_Thm:NRP}\Rightarrow \ref{Item:FOS_Thm:NRP}$: Fix $y$, $\sigma$, $x$, $\epsilon$ and $f$ as in $\ref{Item:FOS_Thm:NRP}$. One can find a $G$-map $\phi\colon Y\to \sa(G)$ with $\phi(y)$ an idempotent (namely, the inverse of $\rho_y|_M$); to simplify notation, let us identify $Y\subseteq \sa(G)$ and assume $y\in Y$ is idempotent. To see that $(y\cdot f)[\rmP(yy, x, \epsilon)] = (y\cdot f)[\rmP(y, x, \epsilon)] \leq 2\epsilon$, continuously extend $f$ to $\sa(G)$, and note for $g\in G$ that $(y\cdot f)(g) = f(gy)$. If $g, h\in \rmP(y, x, \epsilon)$, then $\partial_\sigma(gy, hy)< 2\epsilon$, and since $f\in \rmC_\sigma(G, [0, 1])$, we have $|f(gy)-f(hy)|< 2\epsilon$. 
\end{proof}

\begin{theorem}
    \label{Thm:Abstract_KPT}
    Fix a topological group $G$. 
    \begin{enumerate}[label=\normalfont(\arabic*)]
    
        \item \label{Item:TMD_Thm:KPT} 
        $G\in \sf{TMD}$ iff there is a $G$-skeleton $\Omega$ with the minimality, ED, and Ramsey properties.
        \item \label{Item:CMD_Thm:KPT}
        $G\in \sf{CMD}$ iff there is a precompact $G$-skeleton $\Omega$ with the minimality, ED, and Ramsey properties.
        \item \label{Item:EA_Thm:KPT}
        $G\in \sf{EA}$ iff the trivial $G$-skeleton has the Ramsey property.
    \end{enumerate}
\end{theorem}

\begin{proof}
    \ref{Item:TMD_Thm:KPT}: Assume $\Omega$ has the minimality, ED, and Ramsey properties. Write $Y:= X_\Omega$. By Corollary~\ref{Cor:Rosendal_unfolded}, $Y\in \cal{RC}_G$, and by Theorem~\ref{Thm:NewRP}, we have $Y\cong \mg$, yielding $G\in \sf{TMD}$.

    Conversely, suppose $G\in \sf{TMD}$, and fix a minimal $M\subseteq \sa(G)$. For each $\sigma\in \rm{SN}(G)$, we let $D_\sigma \subseteq M/\partial_\sigma$ denote the set of $\ol{\partial}_\sigma$-compatible points, which is dense by Theorem~\ref{Thm:Rosendal_Minimal}. In particular, we have $M/\partial_\sigma \cong \sa(D_\sigma)$ by Corollary~\ref{Prop:Topo_Converse}. By Proposition~\ref{Prop:Synd_Skeleton}, $\Omega_M^0$ is syndetic and has the $G$-ED property, and by Theorem~\ref{Thm:NewRP}, it has the Ramsey property.
\vspace{3 mm}

\noindent
\ref{Item:CMD_Thm:KPT}: If $G\in \sf{CMD}$, then for every $\sigma\in \rm{SN}(G)$, $\rmM(G)/\partial_\sigma$ is just a compact metric space by Theorem~\ref{Thm:Concrete}\ref{Item:PsAll_Thm:Concrete}. In the other direction, given such a $G$-skeleton $\Omega$, we can take the completion of each level and simply assume each $(X_\sigma, d_\sigma)$ is a compact metric space. Then Theorem~\ref{Thm:NewRP} tells us that $\rmM(G)\cong X_\Omega$, and \Cref{Prop:GED_Unfolded}\ref{Item:GED_Prop:Unfolded} tells us that $\rmM(G)/\partial_\sigma\cong (X_\sigma, d_\sigma)$ for every $\sigma\in \rm{SN}(G)$. Hence $G\in \sf{CMD}$ by Theorem~\ref{Thm:Concrete}\ref{Item:PsAll_Thm:Concrete}. 
\vspace{3 mm}

\noindent
\ref{Item:EA_Thm:KPT}: Clear.
\end{proof}

\begin{remark}
    Compare Theorem~\ref{Thm:Abstract_KPT}~\ref{Item:EA_Thm:KPT} to the characterization of extreme amenability given in \cite{Pestov_book}*{Theorem~2.1.11}.
\end{remark}

\section{Absoluteness results for minimal dynamics}
\label{Section:Abs}

This section develops new set-theoretic methods which allow us to consider $G$-flows in different transitive models of set theory (see also \cite{Zapletal_2015}). Rather than working with flows, we develop some theory related to \emph{preflows}, precompact uniform spaces equipped with a $G$-action such that everything continuously extends to the completion. The main reason for this is that while the statement ``$\cB\subseteq \cP(X)$ is a base for a compact topology on $X$'' is a $\Pi_1$ statement, the statement ``$\cB\subseteq \cP(X\times X)$ is a base for a precompact uniformity on $X$'' is $\Delta_1$. This drop in complexity enables us to phrase many properties of $G$-flows as $\Delta_1$ properties of $G$-preflows. We use the abstract KPT correspondence to provide $\Sigma_1$ definitions of membership in $\sf{EA}$, $\sf{CMD}$, and $\sf{TMD}$, enabling us to show that these classes are all $\Delta_1$. In particular, $G\in \sf{TMD}$ iff in some forcing extension of the set-theoretic universe, the Raikov completion of $G$ is $\sf{GPP}$ in the extension. This allows us to generalize several results that we have proven for $\sf{GPP}$ to all of $\sf{TMD}$, in particular the $\mg$ structure theorem (Theorem~\ref{Thm:MG_Structure_TMD}). In the final subsection, we show that when phrased for preflows, whether or not the Ellis group of a minimal flow $X$ is Hausdorff is absolute between transitive models of set-theory. We will apply this in Section~\ref{Section:Definable} to prove the revised Newelski conjecture.

\subsection{$\sf{CMD}$, $\sf{TMD}$, and the L\'evy hierarchy}
\label{Subsection:Levy}

We use some simple facts about the L\'evy hierarchy of formulas; see \cite{Jech}. Recall that a formula $\phi(\ol{x})$ of set theory is $\Delta_0 = \Sigma_0 = \Pi_0$ if it is equivalent over ZFC to one that has no unbounded quantifiers. Inductively, $\psi(\ol{y})$ is $\Sigma_{n+1}$ or $\Pi_{n+1}$, respectively, iff $\psi$ is equivalent to $\exists \ol{x}\phi(\ol{x}, \ol{y})$ for some $\Pi_n$ formula (respectively, equivalent to $\forall \ol{x} \phi(\ol{x})$ for some $\Sigma_n$ formula). A formula is $\Delta_n$ if it is both $\Sigma_n$ and $\Pi_n$. For instance, the property ``$x$ is finite'' is $\Delta_1$. $\Delta_1$ formulas are absolute between transitive models of set theory.  We collect below some relevant examples of $\Delta_1$ formulas for our purposes (which are all in fact $\Delta_0$ except for precompactness):
\begin{itemize}
    \item 
    ``$\cB_X$ is a base for a topology on the set $X$.''

    \item 
    Given $\cB_X$ a base for a topology on the set $X$ and $A\subseteq X$, the formula ``$A\in \op(X, \cB_X)$.''
    
    \item 
    ``$\cU_X$ is a base for a (precompact) uniform structure on $X$.''

    \item 
    ``$d$ is a pseudometric on $X$.'' 
    \item 
    ``$\cN_G$ is a base at $e_G$ for a group topology on the group $G$.''

    \item 
    ``$\sigma$ is a semi-norm on $(G, \cN_G)$.'' 
    \item 
    ``$\cB$ is a base of semi-norms on $(G, \cN_G)$.''
    \item 
    ``$(X, \cB_X)$ is a $(G, \cN_G)$-space.''
    \item 
    ``$\Omega$ is a $(G, \cN_G)$-skeleton.''
\end{itemize}
As hinted at above, it will be helpful to think about topological spaces and topological groups as being given by some choice of basis rather than working with all open sets whenever possible. Thus when $X$ is a topological space, we let $\cB_X$ be a base of non-empty regular open sets with $X\in \cB_X$, and when $X$ is a uniform space, we let $\cU_X$ be a base for the uniformity consisting of regular open sets and with $X\times X\in \cU_X$,  and we set $\cB_X := \{\bigcup_{x\in F} U_x[x]: F\in \fin{X}, U_x\in \cU_X\}$ (note that the definition of $\cB_X$ is $\Delta_1$ in $\cU_X$). 
If $wh{X}$ denotes the completion of $(X, \cU_X)$, given $A\in \cB_X$, we set $\wh{A} = \Int_{\wh{X}}(\rm{cl}_{\wh{X}}(A))$, and given $U\in \cU_X$, we set $\wh{U} = \Int_{\wh{X}\times \wh{X}}(\rm{cl}_{\wh{X}\times \wh{X}}(U))$. Then we have $\wh{A}\cap X = A$ and $\wh{U}\cap (X\times X) = U$. The sets $\wh{\cB}_X := \{\wh{A}: A\in \cB_X\}$ and $\wh{\cU}_X:= \{\wh{U}: U\in \cU_X\}$ are bases for the topology and the uniformity on $\wh{X}$, respectively.

\begin{lemma}
    \label{Lem:Levy}
    Fix a topological group $G$, a $G$-space $X$, and a $G$-skeleton $\Omega$. 
    \begin{enumerate}[label=\normalfont(\arabic*)]
        \item \label{Item:UTT_Lem:Levy}
        Given $U\in \cN_G$, and $A\in \op(X)$, the formula ``$A$ is $U$-TT'' is $\Delta_0$
        \item \label{Item:RC_Lem:Levy}
        ``$X\in \cal{RC}_G$'' is $\Delta_0$.
        
        \item \label{Item:ED_Lem:Levy}
        ``$\,\Omega$ has the ED property'' is $\Delta_0$
   
        \item \label{Item:Min_Lem:Levy}
        ``$\,\Omega$ has the minimality property'' is $\Delta_1$.
        \item \label{Item:RP_Lem:Levy}
        ``$\,\Omega$ has the Ramsey property'' is $\Delta_1$.
        \item \label{Item:DB_Lem:Levy}
        The formula ``$X$ is $G$-finite'' is $\Pi_1$.
     \end{enumerate}
\end{lemma}

\begin{proof}
    \ref{Item:UTT_Lem:Levy}: $A$ is $U$-TT iff for every $B_0, B_1\in \cB_X$ with $B_i\subseteq A$, there is $g\in U$ with $gB_0\cap B_1\neq \emptyset$.
    \vspace{3 mm}

    \noindent
    \ref{Item:RC_Lem:Levy}:  $X\in \cal{RC}_G$ iff $X$ is $G$-TT and for every $A\in \tau$ and $U\in \cN_G$, there is $B\in \tau$ with $B\subseteq A$ and with $B$ $U$-TT.
    \vspace{3 mm}

    \noindent
    \ref{Item:ED_Lem:Levy}: When considering Definition~\ref{Def:GED_Unfolded}, we can quantify over \emph{rational} $0< \epsilon < 1$ and take $A$ and $B$ to be metric balls.  
    \vspace{3 mm}

    \noindent
    \ref{Item:Min_Lem:Levy}: Again, when considering Definition~\ref{Def:GED_Unfolded}, we can quantify over $A$ a metric ball in $X_\sigma$.
    \vspace{3 mm}

    \noindent
    \ref{Item:RP_Lem:Levy}: First note that given $\epsilon > 0$, the definition of the set $\rmP(x', x, \epsilon)$ is $\Delta_0$. Then considering the finitary RP from Definition~\ref{Def:NewRP}, it suffices to consider rational $\epsilon, \delta > 0$ and rational valued $f\in \rmC_\sigma(FS, [0, 1])$.
    \vspace{3 mm}

    \noindent
    \ref{Item:DB_Lem:Levy}: $X$ is $G$-finite iff for every $U\in \cN_G$ and every function $f\colon \bbN\to \cB_X$, there are $m< n\in \bbN$ with $U\cdot f(m)\cap U\cdot f(n)\neq \emptyset$. 
\end{proof}

\begin{defin}
    \label{Def:Preflow}
    Given a topological group $G$, a \emph{$G$-preflow} is a precompact uniform space $(X, \cU_X)$ such that $G$ acts on $X$ by uniform isomorphisms and for any $U\in \cU_X$, there is $V\in \cN_G$ with $Vx\subseteq U[x]$ for every $x\in X$. This definition is $\Delta_0$.

    If $X$ and $Y$ are $G$-preflows, a \emph{preflow $G$-map} from $X$ to $Y$ is a uniformly continuous, $G$-equivariant map. 
\end{defin}

The next lemma shows that $G$-preflows are exactly the dense, $G$-invariant subspaces of $G$-flows, and every preflow $G$-map continuously extends to the corresponding flows. In particular, every $G$-preflow is a $G$-space.

\begin{lemma}
    If $X$ is a $G$-preflow, then the $G$-action continuously extends to $\wh{X}$, turning it into a $G$-flow.
\end{lemma}

\begin{proof}
    Given $g\in G$, the map $x\to gx$ is uniformly continuous by assumption, so continuously extends to $\wh{X}$. Since, given $g, h\in G$, we have $g(hx) = (gh)x$ on $X$, this is also true on $\wh{X}$. Keeping in mind Fact~\ref{Fact:Continuity}, let $(x_i)_{i\in I}$ be a Cauchy net from $X$, and suppose $(g_i)_{i\in I}$ is a net from $G$ with $g_i\to e_G$. It suffices to show that for any entourage $U\in \cU_X$, we eventually have $(x_i, g_ix_i)\in U$. But for suitably large $i\in I$, $g_i$ will be contained in a suitably small neighborhood of $e_G$ such that  $(x, g_ix)\in U$ for every $x\in X$. 
\end{proof}

The reason we work with preflows is that many properties of $G$-flows and $G$-maps live very low in the L\'evy hierarchy when phrased for preflows.

\begin{lemma}
\label{Lem:Preflow_Props}
    Fix $G$-preflows $X$ and $Y$ and a preflow $G$-map $\pi\colon X\to Y$.
    \begin{enumerate}[label=\normalfont(\arabic*)]
        \item \label{Item:Min_Lem:Preflow}
        $\wh{X}$ is minimal iff for every $A\in \cB_X$, there is $F\in \fin{G}$ with $FA = X$. Hence minimality of $\wh{X}$ is a $\Delta_1$ property of $X$.
        \item \label{Item:RC_Lem:Preflow}
        $\wh{X}\in \cal{RC}_G$ iff $X\in \cal{RC}_G$ as a $G$-space. Hence $\wh{X}\in \cal{RC}_G$ is a $\Delta_0$ property of $X$.
        \item \label{Item:DB_Lem:Preflow}
        $\wh{X}$ is $G$-finite iff $X$ is $G$-finite as a $G$-space. Hence $G$-finiteness of $\wh{X}$ is a $\Pi_1$ property of $X$.
    \end{enumerate}
\end{lemma}

\begin{proof}
    Item~\ref{Item:Min_Lem:Preflow} follows from Lemma~\ref{Lem:Min}. Items~\ref{Item:RC_Lem:Preflow} and \ref{Item:DB_Lem:Preflow} follow from Lemmas~\ref{Lem:Dyn_Bounded}, \ref{Lem:Rosendal_Maps}, and \ref{Lem:Levy}.
\end{proof}

Let us say that a preflow $X$ is \emph{minimal} exactly when $\wh{X}$ is. We now have one of our main applications of the abstract KPT correspondence.

\begin{theorem}
    \label{Thm:TMD_Absolute}
    The classes $\sf{EA}$, $\sf{CMD}$ and $\sf{TMD}$ are $\Delta_1$.
\end{theorem}

\begin{proof}
    First we note the following $\Delta_1$ characterization of $\sf{EA}$: $G\in \sf{EA}$ iff the trivial $G$-skeleton has the Ramsey property (Theorem~\ref{Thm:Abstract_KPT}\ref{Item:EA_Thm:KPT}). For the remaining classes, we use Theorem~\ref{Thm:Abstract_KPT} and the preceding lemmas of this section to give $\Sigma_1$ formulas equivalent to membership and also for non-membership:
    \begin{itemize}
        \item 
        $G\not\in \sf{CMD} \Leftrightarrow$ there is a minimal preflow which is not $G$-finite.
        \item 
        $G\not\in \sf{TMD} \Leftrightarrow$ there is a minimal preflow $X$ with $X\not\in \cal{RC}_G$.
        
        \item 
        $G\in \sf{CMD} \Leftrightarrow$ there is a precompact $G$-skeleton with the minimality, ED, and Ramsey properties.

        \item 
        $G\in \sf{TMD} \Leftrightarrow$ there is a $G$-skeleton with the minimality, ED, and Ramsey properties. \qedhere
    \end{itemize} 
\end{proof}

Using Theorem~\ref{Thm:TMD_Absolute}, we can now generalize Propositions~\ref{Prop:Polish_LC}, \ref{Prop:Polish_Presynd}, and \ref{Prop:Polish_CoPre} as well as Theorem~\ref{Thm:MG_Structure_GPP} to all $\sf{TMD}$ groups. \textbf{For the rest of the section}, we let $\sfV$ denote our universe of set theory and $\sfW\supseteq \sfV$ denote some outer model with the same ordinals, for instance a forcing extension by $\rm{Coll}(\omega, \kappa)$ for some cardinal $\kappa$. We make no claim that this is the most general setting for this discussion, but it will suffice for our purposes. As mentioned at the beginning of the section, various objects have $\Delta_1$ definitions when suitably presented, such as topological/uniform spaces, topological groups, $G$-spaces, etc. Thus given such an object from the ground model $\sfV$, it will remain as such in $\sfW$. For a more detailed discussion of interpreting topological spaces in different universes, see \cite{Zapletal_2015}. 

 If $X\in \sfV$ is a uniform space, we write $\wh{X}^\sfW$ and $\sa^\sfW(X)$ for its completion and Samuel compactification, respectively, as computed in $\sfW$, similarly $\beta^\sfW(X)$ for $X\in \sfV$ a set. If $X\in \sfV$ compact and $\cU_X$ is the unique compatible uniform structure on $X$, we typically write $X^\sfW$ in place of $\wh{X}^\sfW$, thinking of the compact space $X^\sfW$ as the ``interpretation'' in $\sfW$ of the compact space $X$. Similarly, if $X\in \sfV$ is locally compact, then $X$ inherits a uniform structure from its one-point compactification, and we write $X^\sfW$ for its completion in $\sfW$. Likewise, if $G\in \sfV$ is a topological group, we write $\rmM^\sfW(G)$ for the universal minimal flow as computed in $\sfW$. We write $\wt{G}^\sfW$ for the Raikov completion of $G$ in $\sfW$. If $G$ is a (locally) compact group in $\sfV$, then $\wt{G}^\sfW$ will be a (locally) compact group in $\sfW$. If $X\in \sfV$ is a $G$-flow, then $X$ becomes a preflow in $\sfW$, and by Lemma~\ref{Lem:Preflow_Props}, $X$ is minimal iff $X^\sfW$ is, and $X\in \cal{RC}_G$ iff $X^\sfW\in \cal{RC}_G^\sfW$. 

\begin{theorem}
    \label{Thm:Abs_Apps} Fix a topological group $G$.
    \begin{enumerate}[label=\normalfont(\arabic*)]
        \item \label{Item:LC_Thm:AbsApps}
        If $G$ is locally compact, non-compact, then $G\not\in \sf{TMD}$.
        \item \label{Item:PreSynd_Thm:AbsApps}
        If $H\leq^c G$ is presyndetic, then $H\in \sf{TMD}$ iff $G\in \sf{TMD}$.
        \item \label{Item:CoPre_Thm:AbsApps}
        If $H\leq^c G$ is co-precompact and $H\in \sf{CMD}$, then $G\in \sf{CMD}$.
    \end{enumerate}
\end{theorem}

\begin{proof}
    \ref{Item:LC_Thm:AbsApps}: In some forcing extension $\sfW\supseteq \sfV$, we have that $G$ is second countable, so in particular, $\wt{G}^\sfW$ is a locally compact, non-compact Polish group, hence $\wt{G}^\sfW\not\in \sf{TMD}$ by Proposition~\ref{Prop:Polish_LC}. Then, working in $W$, since $G\subseteq \wt{G}^\sfW$ is a dense subgroup, we have $G\not\in \sf{TMD}$. Then apply Theorem~\ref{Thm:TMD_Absolute}.
    \vspace{3 mm}

    \noindent
    \ref{Item:PreSynd_Thm:AbsApps}: In some forcing extension $\sfW$ of $\sfV$, we have that $G$ is second countable, so in particular, $\wt{G}^\sfW$ is Polish. Noting that the formula which asserts that $H\leq^c G$ is presyndetic is $\Delta_1$ and using Proposition~\ref{Prop:Presynd_Raikov} in $\sfW$, we have that $\wt{H}^\sfW\leq^c \wt{G}^\sfW$ is presyndetic. Working in $\sfW$, using Proposition~\ref{Prop:Polish_Presynd} and considering dense subgroups, we have that $H\in \sf{TMD}$ iff $\wt{H}^\sfW\in \sf{TMD}$ iff $\wt{G}^\sfW\in \sf{TMD}$ iff $G\in \sf{TMD}$. Then apply Theorem~\ref{Thm:TMD_Absolute}.
    \vspace{3 mm}

    \noindent
    \ref{Item:CoPre_Thm:AbsApps}: In some forcing extension $\sfW$ of $\sfV$, we have that $G$ is second countable and $\wt{G}^\sfW$ is Polish. Noting that the formula which asserts that $H\leq^c G$ is co-precompact is $\Delta_1$ and using Proposition~\ref{Prop:CoPre_Raikov} in $\sfW$, we have that $\wt{H}^\sfW\leq^c \wt{G}^\sfW$ is co-precompact. If $H\in \sf{CMD}$ in $\sfV$, then by Theorem~\ref{Thm:TMD_Absolute} this holds in $\sfW$, so also $\wt{H}^\sfW\in \sf{CMD}$. Proposition~\ref{Prop:Polish_CoPre} in $\sfW$ gives us that $\wt{G}^\sfW\in \sf{CMD}$, so also $G\in \sf{CMD}$ in $\sfW$. Then use Theorem~\ref{Thm:TMD_Absolute}.
\end{proof}

\begin{remark}
    For Item~\ref{Item:LC_Thm:AbsApps}, it is also possible to construct, given a locally compact, non-compact $G$, a $G$-flow where the meets topology on minimal subflows is not Hausdorff. However, we do not know of any ``direct'' proof for Items~\ref{Item:PreSynd_Thm:AbsApps} and \ref{Item:CoPre_Thm:AbsApps}. Item~\ref{Item:CoPre_Thm:AbsApps} solves a question that was implicit in \cite{BassoZucker} (see the discussion before \cite{BassoZucker}*{Proposition~8.6}).  
\end{remark}

In addition to absoluteness for the classes $\sf{CMD}$ and $\sf{TMD}$, we also have absoluteness results for $\mg$ whenever $G\in \sf{TMD}$. If $\Omega\in \sfV$ is a $G$-skeleton, we write $X_{\Omega, \sfW}$ for the folded flow as computed in $\sfW$, whereas $X_\Omega^\sfW = (X_\Omega)^\sfW$ is the interpretation of $\sfV$'s version of $X_\Omega$ in $\sfW$.

\begin{lemma}
    \label{Lem:Fold_Skeleton_W}
    Fix a forcing extension $\sfW\supseteq \sfV$. Then for any $G$-skeleton, we have in $\sfW$ that  $X_{\Omega, \sfW}\cong \rmS_G(X_\Omega^\sfW)$. If $\Omega$ is precompact, then $X_{\Omega, \sfW}\cong X_\Omega^\sfW$.
\end{lemma}

\begin{proof}
    In $\sfV$ we have $X_\Omega = \displaystyle\varprojlim_{\sigma\in \cB} \sa(X_\sigma)$. Moving to $\sfW$, we have 
    \begin{equation*}
    X_{\Omega}^\sfW = \left(\displaystyle\varprojlim_{\sigma\in \cB} \sa(X_\sigma)\right)^\sfW = \displaystyle\varprojlim_{\sigma\in \cB} \sa(X_\sigma)^\sfW
    \end{equation*}
    whereas $X_{\Omega, \sfW} =  \displaystyle\varprojlim_{\sigma\in \cB} \sa^\sfW(X_\sigma)$. Letting $\xi_\sigma\colon \sa^\sfW(X_\sigma)\to \sa(X_\sigma)^\sfW$ be the unique continuous extension of the identity map on $X_\sigma$, we obtain a $G$-map $\xi\colon X_{\Omega, \sfW}\to X_\Omega^\sfW$. To check that $\xi$ is irreducible, it is enough to note that each $\xi_\sigma$ is irreducible.

    In the case that $\Omega$ is precompact, we have $\sa^\sfW(X_\sigma) = \wh{X_\sigma}^\sfW = \sa(X_\sigma)^\sfW$.
\end{proof}

\begin{theorem}
    \label{Thm:Abs_MG}
    Fix $G$ a topological group.
    \begin{enumerate}[label=\normalfont(\arabic*)]
        \item \label{Item:CMD_Thm:AbsMG}
        $G\in \sf{CMD}$ iff for every forcing extension $\sfW\supseteq \sfV$, we have in $\sfW$ that $\rmM(G)^\sfW\cong \rmM^\sfW(G)$.
        \item \label{Item:TMD_Thm:AbsMG}
        If $G\in \sf{TMD}$, then for every forcing extension $\sfW\supseteq \sfV$, we have in $\sfW$ that  $\rmS_G(\rmM(G)^\sfW)\cong \rmM^\sfW(G)$.
    \end{enumerate}
\end{theorem}

\begin{proof}
    \ref{Item:CMD_Thm:AbsMG}: First assume $G\in \sf{CMD}$. Let $\Omega\in \sfV$ be a $G$-skeleton with the ED, minimality, and Ramsey properties such that each $(X_\sigma, d_\sigma)$ is precompact. By Lemma~\ref{Lem:Levy}, $\Omega$ has all of these properties in $\sfW$. Then by Lemma~\ref{Lem:Fold_Skeleton_W}, we have in $\sfW$ that $\rmM(G)^\sfW \cong X_\Omega^\sfW \cong X_{\Omega, \sfW}\cong \rmM^\sfW(G)$. 

    Conversely, suppose $G\not\in \sf{CMD}$. 
    Let $\sfW\supseteq \sfV$ be any forcing extension adding a new subset of $\omega$. Then $\xi\colon \beta^\sfW\omega\to (\beta\omega)^\sfW$ is not injective (see \cite{BartJud}*{Theorem 6.2.2}). Fix $X\in \sfV$ a $G$-ED flow which is not $G$-finite, with $U$, $A_n$, and $\phi$ as above. Working in $\sfW$, we will show that $X^\sfW$ is not $G$-ED by showing that $\pi:= \pi_G(X^\sfW)\colon \rmS_G(X^\sfW)\to X^\sfW$ is not injective. Let $B_n\in \op(X^\sfW)$ be open with $B_n\cap X = A_n$. Note that $\{UB_n: n< \omega\}$ is pairwise disjoint. Now let $C_n = \pi^{-1}[B_n]\in \op(\rmS_G^\sfW(X^\sfW))$, and for each $n< \omega$, pick $y_n\in C_n$ with $\pi(y_n) = x_n$. Letting $\psi\colon \beta^\sfW \omega\to \rmS_G(X^\sfW)$ be the continuous map with $\psi(n) = y_n$ for every $n< \omega$. Then $\pi\circ \psi = \phi\circ \xi$. By \Cref{Fact:BetaN} $\psi$ is injective, but $\xi$ is not, we must have that $\pi$ is not injective.
    \vspace{3 mm}

    \noindent
    \ref{Item:TMD_Thm:AbsMG}: Assume $G\in \sf{TMD}$. Let $\Omega\in \sfV$ be a $G$-skeleton with the ED, minimality, and Ramsey properties. By Lemma~\ref{Lem:Levy}, $\Omega$ has all of these properties in $\sfW$. Then Lemma~\ref{Lem:Fold_Skeleton_W} gives us in $\sfW$ that $\rmM^\sfW(G)\cong X_{\Omega, \sfW}\cong \rmS_G(X_\Omega^\sfW) \cong \rmS_G(\rmM(G)^\sfW)$.
\end{proof}

As an example, we show how to recover Barto\v{s}ov\'a's result about the universal minimal flow of $G = \rm{Sym}(\kappa)$ for $\kappa$ uncountable. In a forcing extension $\sfW\supseteq \sfV$ where $\kappa$ becomes countable, we have $\wt{G}^\sfW\cong S_\infty$. Hence by Theorem~\ref{Thm:TMD_Absolute}, we have $G\in \sf{CMD}$. Since $\rm{LO}(\kappa)$ satisfies $\rm{LO}(\kappa)^\sfW \cong \rm{LO}(\omega) = \rmM(\wt{G}^\sfW)$, we have by Theorem~\ref{Thm:Abs_MG} that $\mg \cong \rm{LO}(\kappa)$. 

It is natural to ask if Theorem~\ref{Thm:Abs_MG}\ref{Item:TMD_Thm:AbsMG} can be strengthened to an if and only if. 

\begin{question}
    If $G\notin\sf{TMD}$, is there a forcing extension $\sfW\supseteq \sfV$ in which $\rmS_G^\sfW(\rmM(G)^\sfW) \not\cong \rmM^\sfW(G)$?
\end{question}

We can find such a forcing extension conditional to an affirmative answer to Question~\ref{Que:Polish_PiWt} as follows. Let us first recall that if $X$ and $Y$ are spaces and $\pi\colon X\to Y$ is continuous and irreducible, then $\pi w(X) = \pi w(Y)$. Suppose $G\not\in \sf{TMD}$. Find a forcing extension $\sfW\supseteq \sfV$ in which $G$ is separable metrizable and $\rmM(G)^{\sfW}$ and $\rmM(G)^\sfW$ is metrizable. Working in $\sfW$ and assuming an affirmative answer to Question~\ref{Que:Polish_PiWt}, $\rmM^{\sfW}(G)$ has uncountable $\pi$-weight, implying that any $G$-map $\pi\colon \rmM^{\sfW}(G)\to \rmM(G)^{\sfW}$ is not irreducible. 

\subsection{The structure of $\rmM(G)$ for $\sf{TMD}$ groups}

We now generalize Theorem~\ref{Thm:MG_Structure_GPP} to all $\sf{TMD}$ groups. We fix $G\in \sf{TMD}$ and the $G$-skeleton $\Omega = \Omega_{\rmM(G)}^0$; in this subsection we simply take $\cB = \sn(G)$ and $\theta$ the identity map. Hence given $\sigma\in \sn(G)$, we have $\sa(X_\sigma, d_\sigma) \cong (\rmM(G)/\partial_\sigma, \ol{\partial}_\sigma)$, $X_\sigma$ is exactly the set of continuity points of $\ol{\partial}_\sigma = d_\sigma$, and $(X_\sigma, d_\sigma)$ is complete.

As in Section~\ref{Section:AutMG}, we write $\bbG$ for $\aut(\rmM(G))$ equipped with the tau-topology, which is compact Hausdorff by Corollary~\ref{Cor:Tau_CHaus}. By Proposition~\ref{Prop:Level_Maps}, the action of $\bbG$  on $\rmM(G)$ induces an action on $(X_\sigma, d_\sigma)$ by isometries, hence also an action on  $\sa(X_\sigma, d_\sigma)$ by homeomorphisms and isometries, and for any $\sigma\leq \tau\in \sn(G)$ and any $g\in G$, we have that $\pi_\sigma^\tau\colon X_\tau\to X_\sigma$, $\lambda_{g, \sigma}\colon X_{\sigma g}\to X_\sigma$, and $\pi_\sigma\colon \rmM(G)\to \sa(X_\sigma)$ are all $\bbG$-equivariant.

\begin{lemma}
\label{Lem:AutMG_Level}
    The $\bbG$-action on $X_\sigma$ is continuous.
\end{lemma}

\begin{proof}
    As the action is by isometries, it is enough to show that for each $x\in X_\sigma$, the map $\bbG\to X_\sigma$ given by $p\to p(x)$ is continuous. Fix $\epsilon > 0$, and let $A\in \op(\sa(X_\sigma))$ satisfy $A\cap X_\sigma = \rmB_{d_\sigma}(x, \epsilon)$, and write $B = \pi_\sigma^{-1}[A]$. Consider $W = \rm{Nbd}_B(e_\bbG)\subseteq \bbG$. If $p\in W$, we have $p[B]\cap B\neq \emptyset$, implying $p[A]\cap A\neq \emptyset$, so also $p[\rmB_{d_\sigma}(x, \epsilon)]\cap \rmB_{d_\sigma}(x, \epsilon) \neq \emptyset$. As $p$ is an isometry on $X_\sigma$, we see that $p[\rmB_{d_\sigma}(x, \epsilon)]\subseteq \rmB_{d_\sigma}(x, 3\epsilon)$.  
\end{proof}

Lemma~\ref{Lem:AutMG_Level} implies that $d_\sigma$ induces a metric on $X_\sigma/\bbG$. Since the maps $\pi_\sigma^\tau$ and $\lambda_{g, \sigma}$ are $\bbG$-equivariant for every $\sigma\leq \tau\in \sn(G)$ and $g\in G$, we can make sense of $\Omega/\bbG$ as a $G$-skeleton. This $G$-skeleton describes the universal minimal proximal flow.   

\begin{prop}
    \label{Prop:Univ_Prox_Skeleton}
    $X_{\Omega/\bbG}\cong \Pi(G)$.
\end{prop}

\begin{proof}
    For each $\sigma\in \sn(G)$, let $\phi_{\sigma}\colon X_\sigma\to X_\sigma/\bbG$ be the quotient map. Then $\phi_\sigma$ is Lipschitz, so extends to a continuous Lipschitz map $\phi_\sigma\colon \sa(X_\sigma)\to \sa(X_\sigma/\bbG)$. These maps commute with the various bonding maps in $\Omega$ and $\Omega/\bbG$, hence we get a $G$-map $\phi\colon \rmM(G)\to X_{\Omega/\bbG}$. Since $\bbG_\phi = \bbG$, we see that $X_{\Omega/\bbG}$ is proximal.

    As $\Pi(G)\cong \rmS_G(\Pi(G))$, we have that $\Pi(G)$ is $G$-ED, and as $G\in\sf{TMD}$, we have $\Pi(G)\in \cal{RC}_G$, so set $\Theta = \Omega^0_{\Pi(G)}$, with levels $\sa(Y_\sigma, d_\sigma^\Theta)$. We can suppose $(Y_\sigma, d_\sigma^\Theta)$ is complete. Suppose $\psi\colon \rmM(G)\to \Pi(G)$ is a $G$-map. For each $\sigma\in \sn(G)$, $\psi$ induces a map $\psi_\sigma\colon \sa(X_\sigma)\to \sa(Y_\sigma)$, and we have $\psi_\sigma[X_\sigma] \subseteq Y_\sigma$. As $\bbG_\psi = \bbG$, we have for every $p\in \bbG$ that have $\psi_\sigma\circ p = \psi_\sigma$. In particular, on $\psi_\sigma|_{X_\sigma}$ must factor through the quotient map $X_\sigma\to X_\sigma/\bbG$. Hence $X_{\Omega/\bbG}$ admits a $G$-map to $\Pi(G)$.    
\end{proof}

We now formulate preflow versions of proximality and of being a compactification flow. Given a compact group $K$, let us call a $(G\times K)$-flow $X$ a \emph{$K$-compactification flow} if the $G$-action is minimal, the $K$-action is faithful (equivalently free), and $X/K$ is a proximal $G$-flow. Given $K$ a precompact group and $X$ a $(G\times K)$-preflow, then $\wh{X}$ is a $(G\times K)$-flow, so also a $(G\times \wt{K})$-flow, and we can ask when $\wh{X}$ is a $\wt{K}$-compactification flow. We give characterizations that are similar in spirit to Lemma~\ref{Lem:Min}.

\begin{prop}
    \label{Prop:Prox_CF}
    Fix a precompact group $K$ and a $(G\times K)$-flow $X$.
    \begin{enumerate}[label=\normalfont(\arabic*)]
        \item \label{Item:Prox_Prop:ProxCF}
        $X$ is proximal as a $G$-flow iff for every $A\in \cU_X$, there is $F\in \fin{G}$ with $FA\subseteq X\times X$ dense, iff the same with $FA = X\times X$.
        \item \label{Item:CF_Prop:ProxCF}
        $X$ is a $\wt{K}$-compactification flow iff the following all hold:
        \begin{itemize}
            \item $X$ is a minimal $G$-flow.
            \item For every $U\in \cN_K$, there is $A\in \cU_X$ so that for every $p\in K\setminus U$ and $x\in X$, we have $(x, px)\not\in A$.
            \item For every $A\in \cU_X$, there are $F_0\in \fin{G}$ and $F_1\in \fin{K}$ so that we have $\{(x, y)\in X\times X: \exists (g, p)\in F_0\times F_1 (pgx, gy)\in A\}\subseteq X\times X$ dense, iff the same with equality.
        \end{itemize} 
    \end{enumerate}
\end{prop}

\begin{proof}
    \ref{Item:Prox_Prop:ProxCF}: If $X$ is proximal and $A\in \op(X\times X)$ with $\Delta_X\subseteq A$, then by compactness of $X$ we can find $F\in \fin{G}$ with $FA = X\times X$. Conversely, if $x, y\in X$ and $(x, y)$ is not proximal, find $A\in \op(X\times X)$ with $\Delta_X\subseteq A$ and with $(gx, gy)\not\in \ol{A}$ for any $g\in G$. Then for any $F\in \fin{G}$, we have $F\ol{A}\neq X\times X$, so $FA\subseteq X\times X$ is not dense.
    \vspace{3 mm}

    \noindent
    \ref{Item:CF_Prop:ProxCF}: The second bullet is equivalent to $\wt{K}$ acting freely on $X$. We focus on the third bullet. First observe that we can identify $X/\wt{K}$ with the subspace of the Vietoris space $\rmK(X)$ consisting of $\wt{K}$-orbits. Hence if $X/\wt{K}$ is a compactification flow, $A\in \cU_X$ (which by shrinking we may assume is invariant under the diagonal action of $\wt{K}$), and $x, y\in X$, we first find $g\in G$ with $(g\wt{K}x\times g\wt{K}y)\cap A\neq \emptyset$. Then, as $A$ is $\wt{K}$-invariant and $K\subseteq \wt{K}$ is dense, we find $p\in K$ with $(gpx, gy)\in A$. Compactness of $X$ then gives us the desired condition with equality. Conversely, suppose $\wt{K}a, \wt{K}b\in X/\wt{K}$ and that $(\wt{K}a, \wt{K}b)$ is not proximal. We may find $A\in \cU_X$ invariant under the diagonal $\wt{K}$-action and so that $(gpa, gb)\not\in \ol{A}$ for any $g\in G$ or $p\in K$. It follows that for any $F_0\in \fin{G}$ and $F_1\in \fin{K}$, the set $\{(x, y)\in X\times X: \exists (g, p)\in F_0\times F_1 (pgx, gy)\in A\}\subseteq X\times X$ cannot be dense, as its closure does not contain $(a, b)$.  
\end{proof}

\begin{corollary}
\label{Cor:Prox_CF_Levy}
    If $K$ is a precompact group and $X$ is a $(G\times K)$-preflow, then the properties of $\wh{X}$ being a proximal $G$-flow or a $\wt{K}$-compactification flow are $\Delta_1$ properties of $X$.
\end{corollary}

\begin{proof}
    $\wh{X}$ has the relevant property iff the condition with equality in the relevant item of Proposition~\ref{Prop:Prox_CF} holds verbatim for $X$ 
\end{proof}

We will want similar results for describing when a preflow $G$-map $\pi\colon X\to Y$ continuously extends to a proximal $G$-map $\wh{\pi}\colon \wh{X}\to \wh{Y}$. It takes a somewhat more complicated form since for the fiber products, we might not have $X\times_\pi X\subseteq \wh{X}\times_{\wh{\pi}}\wh{X}$ dense. As a bonus, we obtain a $\Delta_0$ statement. In fact, both parts of Corollary~\ref{Cor:Prox_CF_Levy} can be strengthened to $\Delta_0$.

\begin{prop}
    \label{Prop:Preflow_Prox}
    Let $X$ and $Y$ be $G$-preflows, let $\pi\colon X\to Y$ a preflow $G$-map, and let $\wh{\pi}\colon \wh{X}\to \wh{Y}$ be the continuous extension of $\pi$. 
    \begin{enumerate}[label=\normalfont(\arabic*)]
        \item \label{Item:Prox_Prop:PreflowProx}
        $\wh{\pi}$ is proximal iff for every $U_0\in \cU_X$, there are $V\in \cU_Y$ and $U\in \cU_X$ so that 
        for any $x, y\in X$, if $(\pi(x), \pi(y))\in V$, then there is $g\in G$ with $gU[x]\times gU[y] \subseteq U_0$. 

        In particular, $\wh{X}$ is proximal iff for every $U_0\in \cU_X$, there is $U\in \cU_X$ so that for any $x, y\in X$, there is $g\in G$ with $gU[x]\times gU[y]\subseteq U_0$.
        
        \item \label{Item:Irred_Prop:PreflowProx}
        $\wh{\pi}$ is irreducible iff for every $A\in \cB_X$, there is $B\in \cB_Y$ with $\pi^{-1}[B]\subseteq A$.
    \end{enumerate}
    Hence all of these are $\Delta_0$ properties.
\end{prop}

\begin{proof}
    \ref{Item:Prox_Prop:PreflowProx}: First suppose the criterion does not hold, witnessed by some ``bad'' $U_0\in \cU_X$. For each $U\in \cU_X$ and $V\in \cU_Y$, we can find $x_{U, V}, y_{U, V}\in X$ with $(\pi(x), \pi(y))\in V$, but with $gU[x_{U, V}]\times gU[y_{U, V}]\not\subseteq U_0$. Pick $x_{g, U, V}\in U[x_{U, V}]$ and $y_{g, U, V}\in U[y_{U, V}]$ with $(gx_{g, U, V}, gy_{g, U, V})\not\in U_0$. Viewing $\cU_X\times \cU_Y$ as a directed partial order in the usual way, we can pass to a subnet so that $x_{U, V}\to x'\in \wh{X}$ and $y_{U, V}\to y'\in \wh{Y}$. We have $\pi(x') = \pi(y')$. Note for every $g\in G$ that also $x_{g, U, V}\to x'$ and $y_{g, U, V}\to y'$. It follows that $(gx', gy')\not\in U_0$ for any $g\in G$, i.e.\ $(x', y')$ is not a proximal pair. 

    Now suppose $\wh{\pi}$ is not proximal, witnessed by $x, y\in \wh{X}$ with $\pi(x) = \pi(y)$. Let $U_0\in \cU_X$ be such that $(gx, gy)\not\in \wh{U_0}$ for any $g\in G$. Fix $V\in \cU_Y$ and $U\in \cU_X$. Find $x', y'\in X$ with $(x, x'), (y, y')\in \wh{U}$ and with $(\pi(x'), \pi(y'))\in V$. It follows that for any $g\in G$, we cannot have $gU[x']\times gU[y']\subseteq U_0$
    \vspace{3 mm}

    \noindent
    \ref{Item:Irred_Prop:PreflowProx}
    The forward direction follows from Lemma~\ref{Lem:Fiber_Image_Open}. For the reverse, suppose the criterion holds, and consider $C\in \op(\wh{X})$. Find $A\in \cB_X$ with $\wh{A}\subseteq C$, and let $B\in \cB_Y$ satisfy $\pi^{-1}[B]\subseteq A$. Fix $y\in B$; we claim $\wh{\pi}^{-1}[y]\subseteq C$. If $z\in \wh{\pi}^{-1}[y]$, suppose $z_i\in X$ with $z_i\to z$. Eventually we have $\wh{\pi}(z_i) = \pi(z_i)\in B$. Hence eventually $z_i\in A$, implying that $z\in \wh{A}\subseteq C$. 
\end{proof}

We can now generalize Theorem~\ref{Thm:MG_Structure_GPP} to all $\sf{TMD}$ groups. 

\begin{theorem}
    \label{Thm:MG_Structure_TMD}
    For $G\in \sf{TMD}$, the canonical $G$-map $\iota\colon \rmM(G)\to \rm{UCF}(G)$ is highly proximal, equivalently irreducible. 
\end{theorem}

\begin{proof}
    Let $\sfV$ be our universe of set theory, and let $\sfW\supseteq \sfV$ be a forcing extension in which $G$ is second countable, i.e.\ separable metrizable. Then Theorem~\ref{Thm:MG_Structure_CMD} applies to $\wt{G}^\sfW$, hence also to $G$ in $\sfW$. We write $\rmM^\sfW(G)$, $\rm{UCF}^\sfW(G)$, and $\Pi^\sfW(G)$ for the relevant universal dynamical object computed in $\sfW$, and we write $\rmM(G)^\sfW$, $\rm{UCF}(G)^\sfW$, and $\Pi(G)^\sfW$ for the interpretations of the corresponding universal objects from $\sfV$. We have $\rmM^\sfW(G) = X_{\Omega, \sfW}$ and $\rmM(G)^\sfW = X_\Omega^\sfW$. We write $\wt{\bbG}^\sfW$ for the Raikov completion in $\sfW$ of $\bbG$, and we write $\bbG_\sfW = \aut(\rmM^\sfW(G))$ (though we will eventually see that these are equal). Let $\psi\colon \rm{UCF}(G)\to \Pi(G)$ be the canonical $G$-map. In $\sfW$, we have the following diagram.
    \begin{center}
    \begin{tikzcd}
        \rmM(G)^\sfW \arrow[d, "\wh{\iota}"'] & \rmM^\sfW(G) \arrow[l, "\theta_2"] \arrow[d, "\iota'"]\\
        \rm{UCF}(G)^\sfW \arrow[d, "\wh{\psi}"'] & \rm{UCF}^\sfW(G) \arrow[l, "\theta_1"] \arrow[d, "\psi'"]\\
        \Pi(G)^\sfW & \Pi^\sfW(G) \arrow[l, "\theta_0"]
    \end{tikzcd}
    \end{center}
In this diagram, $\theta_2$ is irreducible by Lemma~\ref{Lem:Fold_Skeleton_W}, $\iota'$ is irreducible by Theorem~\ref{Thm:MG_Structure_GPP}, and $\wh{\iota}$ is proximal by Theorem~\ref{Thm:AutMG_Haus} and Proposition~\ref{Prop:Preflow_Prox}\ref{Item:Prox_Prop:PreflowProx}. 
To show that $\iota$ is irreducible/highly proximal, it is enough to show that $\wh{\iota}$ is irreducible by Proposition~\ref{Prop:Preflow_Prox}\ref{Item:Irred_Prop:PreflowProx}. Since $\wh{\iota}\circ \theta_2 = \theta_1\circ \iota'$, it suffices to prove that $\theta_1$ is irreducible. 

First note that $\Pi(G)^\sfW = X_{\Omega/\bbG}^\sfW$ is proximal and $\rm{UCF}(G)^\sfW$ is a $\wt{\bbG}^\sfW$-compactification flow by Corollary~\ref{Cor:Prox_CF_Levy}. Since $\wh{\iota}\circ \theta_2$ is proximal, we have $\wt{\bbG}^\sfW = \bbG_\sfW$. Writing $\Omega^\sfW$ for the $G$-skeleton formed with levels $X_\sigma^\sfW$ (the completion in $\sfW$ of $X_\sigma$), we have that $\Omega^\sfW/\bbG_\sfW = \Omega^\sfW/\wt{\bbG}^\sfW = (\Omega/\bbG)^\sfW$. Hence $\Pi^\sfW(G)\cong X_{\Omega/\bbG, \sfW}$, and Lemma~\ref{Lem:Fold_Skeleton_W} then implies that $\theta_0$ is irreducible.

We can now show that $\theta_1$ is irreducible/highly proximal. We can view $\rm{UCF}(G)^\sfW$ and $\rm{UCF}^\sfW(G)$ as $(G\times \bbG_\sfW)$-flows in such a way that $\theta_1$ is $\bbG_\sfW$-equivariant. Pick $x\in \rm{UCF}(G)^\sfW$, let $u\in \sa(G)$ belong to a minimal subflow, and form $Y = u\bullet \theta_1^{-1}[x]\in \rmK(\rm{UCF}^\sfW(G))$. We want to show that $Y$ is a singleton. We have that $Y$ is a subset of a $\theta_1$-fiber as this is a Vietoris-closed property. We also have that $Y$ is a subset of a $\psi'$-fiber since $\theta_0$ is highly proximal. But $\psi'$-fibers are the same as $\bbG_\sfW$-orbits, and the $\bbG_\sfW$-action on both $\rm{UCF}(G)^\sfW$ and $\rm{UCF}^\sfW(G)$ are both free. Hence $Y$ must be a singleton as desired.
\end{proof}

\begin{question}
\label{Que:MGPropsAbs}
    Are any of the following statements $\Delta_1$ properties of a topological group $G$?
    \begin{itemize}
        \item ``$\rmM(G)\cong \Pi(G)$''
        \item ``$\rmM(G)\cong \rm{UCF}(G)$''
        \item ``$\bbG$ is Hausdorff.''
    \end{itemize}
\end{question}
Of course, an affirmative answer to Question~\ref{Que:AutMGHausTMD} would imply an affirmative answer for the third bullet.

\subsection{Absoluteness results for Ellis groups}

The main result of this subsection will be applied in the next section to prove the revised Newelski conjecture (Theorem~\ref{Thm:Newelski}). Fix a topological group $G$. Given a minimal preflow $X$, we define its Ellis group to be $\bbG_X:= \bbG_{\wh{X}}$, i.e.\ the automorphism group of the enveloping ideal $M_{\wh{X}}$, which we equip with the tau-topology. We can now state the main theorem of the subsection.

\begin{theorem}
    \label{Thm:Ellis_H_Absolute}
    The formula $\Phi(X) = $``$X$ is a minimal preflow and $\bbG_X$ is Hausdorff'' is $\Delta_1$.
\end{theorem}

Let us note that upon translating Theorem~\ref{Thm:Tau_Hausdorff} to a statement about preflows, and using Propositions~\ref{Prop:Prox_CF} and ~\ref{Prop:Preflow_Prox}, we have almost everything to conclude that $\Phi(X)$ is $\Sigma_1$. The last piece we need is to discuss equicontinuous extensions.

\begin{prop}
    \label{Prop:Preflow_Equi}
    Let $X$ and $Y$ be minimal preflows and $\pi\colon X\to Y$ a preflow $G$-map. Then $\wh{\pi}\colon \wh{X}\to \wh{Y}$ is equicontinuous iff for every $U_0\in \cU_X$, there is $U\in \cU_X$ so that for every $V\in \cU_X$ and $x, y\in X$ with $V[x]\times V[y]\subseteq U_1$ and $\Int_X\left(\rm{cl}_X(\pi[V[x]])\cap \rm{cl}_X(\pi[V[y]])\right)\neq \emptyset$, we have for every $g\in G$ that $(gV[x]\times gV[y])\cap U_0\neq \emptyset$. Hence equicontinuity of $\wh{\pi}$ is $\Delta_0$.
\end{prop}

\begin{proof}
    Suppose $\wh{\pi}$ is equicontinuous. Fix $U_0\in \cU_X$. We can find $U\in \cU_X$ so that for any $x, y\in \wh{X}$ with $(x, y)\in \wh{U}$, we have for every $g\in G$ that $(gx, gy)\in \wh{U_0}$. Hence if $x, y\in X$ and $V\in \cU_X$ are such that $V[x]\times V[y]\subseteq U_1$ and  $\Int_X\left(\rm{cl}_X(\pi[V[x]])\cap \rm{cl}_X(\pi[V[y]])\right)\neq \emptyset$, then we have $\wh{\pi}[\wh{V}[x]]\cap \wh{\pi}[\wh{V}[y]]\neq \emptyset$. We then have for any $g\in G$ that $(g\wh{V}[x]\times g\wh{V}[y])\cap \wh{U_0}\neq \emptyset$, implying that $gV[x]\times gV[y]\cap U_0\neq \emptyset$.

    Conversely, suppose the criterion holds. Fix $U_0\in \cU_X$, and let $U_1\in \cU_X$ be as guaranteed by the criterion. Let $x, y\in \wh{X}$ with  $(x, y)\in \wh{U_1}$ and $\wh{\pi}(x) = \wh{\pi}(y)$. Fix $g\in G$; we show that $(gx, gy)\in \wh{U}_0^2$. Fix $V\in \cU_X$, which we can suppose is small enough with $\wh{V}^2[x]\times \wh{V}^2[y]\subseteq \wh{U_1}$. Pick $x'\in \wh{V}[x]\cap X$ and $y'\in \wh{V}[y]\cap X$. Then $\Int_X\left(\rm{cl}_X(\pi[V[x']])\cap \rm{cl}_X(\pi[V[y']])\right)\neq \emptyset$, so by the criterion $(gV[x']\times gV[y'])\cap U_0\neq \emptyset$. Hence for each $V$, we can pick $x_V\in \wh{V}^2[x]$ and $y_V\in \wh{V}^2[y]$ with $(gx_V, gy_V)\in \wh{U_0}$. Hence viewing $(x_V)_{V\in \cU_X}$ and $(y_V)_{V\in \cU_X}$ as nets, we have $x_V\to x$, $y_V\to y$, so $(gx, gy)\in \wh{U}_0^2$. 
\end{proof}

So the main work in this subsection is proving that $\Phi(X)$ is $\Pi_1$. One way of phrasing the main difficulty is that given a forcing extension $\sfW\supseteq \sfV$ and a $G$-flow $X\in \sfV$ with enveloping ideal $M_X$, we in general will not have that $M_X^\sfW$ is the enveloping ideal of $X^\sfW$, nor that it is even a minimal ideal flow. Writing $M_{X, \sfW}$ for the enveloping ideal of $X^\sfW$, we need to understand the relationship between $M_{X, \sfW}$ and $M_X^\sfW$.

\begin{lemma}
    \label{Lem:Same_Ellis}
    Let $X$ be a minimal flow, and suppose $Y$ is a minimal flow such that $Y\subseteq X^I$ for some set $I$. Then $M_Y\cong M_X$. 
\end{lemma}

\begin{proof}
    As $Y$ factors onto $X$ via projection onto coordinate $i\in I$, we have that $M_X$ is a factor of $M_Y$. Keeping in mind the first bullet of Fact~\ref{Fact:MinIdealFlow}, we have that $M_Y$ is a factor of $M_X$, hence by coalescence they are isomorphic. 
\end{proof}

This implies that $M_X^\sfW$ is a factor of $M_{X, \sfW}$. We next record a $\Sigma_1$-consequence of Proposition~\ref{Prop:Tau_Base} via the notion of \emph{wild automorphisms}. This will tell us that $M_X^\sfW$ is poorly enough behaved that $M_{X, \sfW}$ cannot have Hausdorff Ellis group.

\begin{defin}
\label{Def:Wild}
    A preflow $X$ has \emph{wild automorphisms} if there are a net $(p_i)_{i\in I}$ from $\aut(X)$ and $\rm{id}_X\neq p\in \aut(X)$ so that for any $A, B\in \cB_X$ with $B\subseteq A$, then for all large enough $i\in I$, we have $p_i[A]\cap B\neq \emptyset$ and $p_i[A]\cap p[B]\neq \emptyset$. This is a $\Sigma_1$ formula of $X$.
\end{defin}

Note that if $X$ has wild automorphisms, then so does $\wh{X}$. 

\begin{prop}
    \label{Prop:Non_Hausdorff}
    If $X$ is a minimal ideal flow, then $\bbG_X$ is non-Hausdorff iff $X$ has wild automorphisms.
\end{prop}

\begin{proof}
    Keeping in mind Proposition~\ref{Prop:Tau_Base}, a net $(p_i)_{i\in I}$ from $\bbG_X$ converges to both $\rm{id}_X$ and $p\in \bbG_X\setminus \{\rm{id}_X\}$ iff $(p_i)_{i\in I}$ and $p$  witness that $X$ has wild automorphisms.  
\end{proof}

\begin{prop}
    \label{Prop:Wild_NonHaus}
    If $X$ is a minimal preflow with wild automorphisms, then $\bbG_X$ is not Hausdorff. 
\end{prop}

\begin{proof}
    Towards a contradiciton, suppose $\bbG_X$ was Hausdorff. We can replace $X$ with $\wh{X}$ and assume $X$ is a flow. Using Theorem~\ref{Thm:Tau_Hausdorff}\ref{Item:Group_Thm:TauHaus}, we fix minimal ideal flows $X^*$, $Z$, and $Y$ along with $G$-maps $\eta\colon X^*\to X$, $\iota\colon X^*\to Z$, and $\sigma\colon Z\to Y$ such that $\iota$ and $Y$ are proximal and $\sigma$ is a compact group extension, witnessed by the compact group $K\cong \bbG_Z\cong \bbG_{X^*}$. Let us also fix $(p_i)_{i\in I}$ and $p\neq \rm{id}_{X}$ from $\aut(X)$ witnessing that $X$ has wild automorphisms. As $X^*$ is a minimal ideal flow, for each $i\in I$, we can find $q_i\in K$ with $p_i\circ \eta = \eta\circ q_i$. 

    Fix $y\in X^*$. We claim that there is a distal pair $w\neq z\in X^*$ so that for any $A\in \op(y, X^*)$, $B\in \op(w, X^*)$ and $C\in \op(z, X^*)$, then for large enough $i\in I$, we have $q_i[A]\cap B\neq \emptyset$ and $q_i[A]\cap C\neq \emptyset$. For each $A\in \op(y, X^*)$, let 
    $K_A\in \rmK(X^*)$ be a limit point of the net $(q_i[\ol{A}])_{i\in I}$ from $\rmK(X^*)$. We can arrange that $A'\subseteq A$ implies $K_{A'}\subseteq K_A$. Set $K = \bigcap_{A\in \op(y, X^*)} K_A$. 
    Now fix $A\in \op(y, X^*)$. As $\eta$ is a map between minimal flows, it is pseudo-open, so $\Int(\eta[A])\neq \emptyset$. We also note that $\ol{\Int(\eta[A])} = \eta[\ol{A}]$. Pick some $x\in \Int(\eta[A])$, and let $L_A\in \rmK(X)$ be a limit point of the net $(p_i[\eta[\ol{A}]])_{i\in I}$. We can arrange that $L_A\subseteq \eta[K_A]$. Then as $(p_i)_{i\in I}$ and $p$ witness wild automorphisms for $X$, we have $\{x, p(x)\}\subseteq L_A$, so in particular, the set $\{(w, z)\in K_A\times K_A: p\circ \eta(w) = z\}$ is closed and non-empty, and every $(w, z)$ in this set is a distal pair. Intersecting over all $A\in \op(y, X^*)$, we obtain $(w, z)$ as desired.

    As $\iota$ is proximal, $\iota(w)\neq \iota(z)$, and for any $A\in \op(\iota(y), Z)$, $B\in \op(\iota(w), Z)$ and $C\in \op(\iota(z), Z)$, then for large enough $i\in I$, we have $q_i[A]\cap B\neq \emptyset$ and $q_i[A]\cap C\neq \emptyset$. But as the action of $K$ on $Z$ is continuous, this cannot happen. 
\end{proof}

Given a minimal preflow $X$, we now combine the results of this subsection to obtain a $\Sigma_1$ statement equivalent to $\neg \Phi(X)$:
\begin{itemize}
    \item 
    There are a preflow $Y$ with $\wh{X}\cong \wh{Y}$, a set $I$, and a preflow $Z\subseteq Y^I$ such that $Z$ has wild automorphisms.
\end{itemize}
Let us briefly discuss two aspects of this statement. First, the formula $\Psi(X, Y)$ which asserts that $X$ and $Y$ are preflows with $\wh{X}\cong \wh{Y}$ can be seen to be $\Sigma_1$ as follows: $\Psi(X, Y)$ holds iff $X$ and $Y$ are preflows and there are a preflow $Z$ and preflow $G$-maps $\phi\colon X\to Z$ and $\psi\colon Y\to Z$ such that $\im(\phi)$ is dense and so that for every $U\in \cU_X$, there is $V\in \cU_Z$ so that for every $x, y\in X$ with $(x,y)\not\in U$, we have $(\phi(x), \phi(y))\not\in V$, and likewise for $Y$ and $\psi$. Second, we view $Y^I$ as a preflow with the diagonal action and the product uniformity. While the statement ``$Z= Y^I$'' is $\Pi_1$, the statement $Z\subseteq Y^I$ is $\Delta_1$: Every $z\in Z$ is a function from $I$ to $Y$, $Z$ is $G$-invariant,  and for every $U\in \cU_Z$, there are $F\in \fin{I}$ and $V\in \cU_Y$ so that for any $z_0, z_1\in Z$, we have $(z_0, z_1)\in U$ iff for every $i\in F$, we have $(z_0(i), z_1(i))\in V$. 

Having shown that $\Phi(X)$ is both $\Sigma_1$ and $\Pi_1$, this proves Theorem~\ref{Thm:Ellis_H_Absolute}.

\begin{remark}
    One can show that the formula $\Theta(X, K)$ which asserts that ``$X$ is a minimal preflow and $K$ is a precompact group such that $\wt{K}\cong \bbG_X$'' is $\Sigma_1$ and not $\Pi_1$. So it is somewhat amusing that the formula $\Phi(X)$, which is equivalent to $\exists K \Theta(X, K)$, is $\Delta_1$.
\end{remark}

\section{Connections with definable groups}
\label{Section:Definable}

The study of definable groups has been a central interest of model theorists for decades; see \cites{Poizat_Stable, Wag_2000, HPP_2008, Newelski_2009, CS_2018}. Given a model $\bM$ of a first-order theory $T$, a \emph{group definable in $\bM$} is an definable (with parameters from $\bM$) subset $G\subseteq \bM^n$ for some $n< \omega$ along with formulas which define multiplication and inversion on $G$. The formulas which define the underlying set and the group operations of $G$ in $\bM$ will make sense in any model of $T$, so one often writes $G^\bM$, $G^\bN$, etc.\ to discuss the group defined by these formulas in the given model. One way to think of a definable group is as a discrete group $G$ equipped with a distinguished $2$-sided-invariant Boolean algebra $\cB\subseteq \cP(G)$ which contains all singletons. The sets in $\cB$ should be thought of as exactly the definable (or \emph{externally} definable, see below) subsets of $G$.

Newelski in \cite{Newelski_2009} introduces the main concepts of the \emph{definable} topological dynamics of definable groups; see also \cites{GPP_2014, KP_2017, CS_2018}. Given a definable group $(G, \cB)$, one can form $\st(\cB)$ the Stone space of $\cB$, which is naturally a $G$-flow. Since $\cB$ contains all singletons, we can view $G$ as a subspace of $\st(\cB)$. When $G$ is definable in $\bM$ and $\cB = \cB_{G, \bM}$ is the algebra of definable subsets of $G$, we typically write $\rmS_G(\bM)$ for $\st(\cB)$ (with context distinguishing this use of $\rmS_G$ from the Gleason completion of a $G$-space). Thus $\rmS_G(\bM)$ is the space of $n$-types over $\bM$ which are finitely satisfiable in $G$. Another common situation uses externally definable subsets. If $G$ is definable in $\bM$, a subset $S\subseteq G$ is \emph{externally definable} if $S = G\cap D$, where $D\subseteq \frak{C}^n$ is definable with parameters in some monster model $\frakC\succ \bM$. If $\cB = \cB_{G, \bM}^{ext}$ is the algebra of externally definable subsets of $G$, we write $\rmS_{G, ext}(\bM)$ for $\st(\cB)$;  this is the space of \emph{global} $n$-types which are finitely satisfiable in $\bM$. Note that $\rmS_{G, ext}(\bM)$ factors onto $\rmS_G(\bM)$.

Given a definable group $(G, \cB)$ and a $G$-flow $X$, we call $X$ \emph{definable} if for every $x\in X$, the map $\rho_x\colon G\to X$ given by $\rho_x(g) = gx$ continuously extends to $\st(\cB)$. When $G$ is definable in $\bM$, we call a $G$-flow $X$ definable or externally definable depending on which subalgebra $\cB\subseteq \cP(G)$ is intended. It is not always the case that $\st(\cB)$ is definable; this happens exactly when $\st(\cB)\cong \rmE(\st(\cB))$. Let us call $(G, \cB)$ \emph{strongly definable} in this case. We can characterize when $(G, \cB)$ is strongly definable using the shift action on $2^G$, where given $x\in 2^G$ and $g, h\in G$, we set $(g\cdot x)(h) = x(hg)$.
\begin{lemma}
    Let $(G, \cB)$ be a definable group. Then it is strongly definable iff whenever $S, T\subseteq G$ with $S\in \cB$ and $\chi_T\in \ol{G\cdot \chi_S}\subseteq 2^G$, then $T\in \cB$.
\end{lemma}

\begin{proof}
    We have that $\st(\cB)\cong \rmE(\st(\cB))$ iff the following multiplication on $\st(\cB)$ is well defined, where given $p, q\in \st(\cB)$ and $S\in \cB$, we have
    $$S\in pq \Leftrightarrow \{g\in G: S\in gq\}\in p$$
    Set $T = \{g\in G: S\in gq\}$; the above is well defined provided $T\in \cB$, and we have $\chi_T = \displaystyle\lim_{h\to q} h\cdot \chi_S$.
\end{proof}

For $G$ definable in $\bM$, $(G, \cB_{G, \bM}^{ext})$ is always strongly definable \cite{Newelski_2009}. In certain situations, $(G, \cB_{G, \bM})$ is strongly definable. This happens whenever $\bM$ is stable (as then externally definable sets are definable), or more generally whenever types over $\bM$ are definable. When $\bM$ is NIP, one can pass to the \emph{Shelah expansion} $\bM^{ext}$, which adds predicates for every externally definable subset of $\bM$. Shelah proves \cite{Shelah_2009} that $\bM^{ext}$ eliminates quantifiers, is NIP, and that externally definable subsets of $\bM^{ext}$ are definable, and hence  $\rmS_G(\bM^{ext}) = \rmS_{G, ext}(\bM)$. For more on NIP theories, see \cite{Simon_NIP}. 

Let $G$ be a group definable in an NIP structure $\bM$. We can view $\rmS_{G, ext}(\bM)\cong \rmS_G(\bM^{ext})$ as the \emph{universal externally definable $G$-ambit}, and any minimal subflow of it will be the \emph{universal externally definable minimal flow}. Letting $\cM\subseteq \rmS_{G, ext}(\bM)$ be a minimal subflow (the font choice distinguishes this from the underlying set of the model $\bM$), then $\cM$ is a minimal ideal flow, and the results of Subsection~\ref{Subsection:MinIdealFlows} apply. In particular, one can form the Ellis group $\bbG_\cM$, equipped with the tau-topology \cite{KP_2017}.  Newelski conjectured in \cite{Newelski_2009}*{p. 69} that for groups definable in a NIP structure, one has $\bbG_\cM\cong G^\frakC/G_{00}$, where $\frakC\succ \bM$ is a monster model, $G_{00}$ is the smallest type-definable subgroup of $G^\frakC$ of bounded index, and $G^\frakC/G_{00}$ is equipped with the \emph{logic topology} \cite{Pillay_2004}. While this is true in the definably amenable case \cite{CS_2018}, it fails for $\rm{SL}(2, \bbR)$ \cite{GPP_2015}. A revised form was asked by Krupi\'nski and Pillay in \cite{KP_2023}: Is $\bbG_\cM$ Hausdorff? This question has come to be called the \emph{revised Newelski conjecture} and was recently proven in the case that $\bM$ is countable by Chernikov, Gannon, and Krupi\'nski \cite{DC3} using deep results of Glasner \cite{Glasner_Tame_Gen} on the structure of minimal, metrizable, and \emph{tame} dynamical systems.

Tame dynamical systems were first isolated by K\"ohler \cite{Kohler_1995} under the name ``regular'' and have many equivalent characterizations. The following combinatorial version of the definition is essentially from \cite{KL_2007}.

\begin{defin}
    \label{Def:Tame}
    Let $G$ be a group (topological, definable, doesn't matter), and fix a $G$-flow $X$. Given $A_0, A_1\in \op(X)$ with $\ol{A_0}\cap \ol{A_1} = \emptyset$ and $S\subseteq G$, we say that $(A_0, A_1)$ \emph{shatters} $S$ if for every partition $S = S_0\sqcup S_1$, there is $x\in X$ so that given $g\in S$, we have $gx\in U_i$ iff $g\in S_i$. We say that $X$ \emph{shatters} $S$ if there are $A_0, A_1$ as above. We remark that the definition works fine for preflows, taking $A_0, A_1\in \cB_X$ to be uniformly apart. 
    
    A $G$-flow $X$ is  \emph{tame} if $X$ does not shatter any infinite $S\subseteq G$.
\end{defin}

One can show (see \cite{GM_2012}) that tame flows are closed under products, factors, and subflows. 

\begin{lemma}
    \label{Lem:Preflow_Tame}
    Given a preflow $X$, $\wh{X}$ is non-tame iff there are uniformly apart $A_0, A_1\in \cB_X$ and an infinite $S\subseteq G$ so that every finite $F\subseteq S$ is shattered by $(A_0, A_1)$. 
\end{lemma}

\begin{proof}
    By taking a limit point in $\wh{X}$ as $F_0\sqcup F_1 = F$ grows to $S_0\sqcup S_1\subseteq S$, we see that the condition on  $X$ implies non-tameness of $\wh{X}$. Say  $\wh{X}$ is non-tame as witnessed by $A_0, A_1\in \op(\wh{X})$ and $S\subseteq G$. By enlarging if needed, $A_0$ and $A_1$ can be chosen so that $A_i\cap X\in \cB_X$. Then, if $F_0, F_1\subseteq S$ are disjoint finite subsets, the set $B:= \{x\in \wh{X}: g\in F_i\Rightarrow gx\in A_i\text{ for each } i< 2\}$ is non-empty open, hence meets $X$. Any $x\in B\cap X$ will witness the desired condition.
\end{proof}

\begin{prop}
\label{Prop:Tame_Levy}
    Given a preflow $X$, the statement ``$\wh{X}$ is tame'' is $\Delta_1$.
\end{prop}

\begin{proof}
    Consider uniformly apart $A_0, A_1\in \cB_X$, and define
    $$T_{A_0, A_1}:= \{s\in G^{<\omega}: \im(s)\text{ is shattered by }(A_0, A_1)\}$$
    We remark that the definition of $T_{A_0, A_1}$ is $\Delta_1$ in $A_0$ and $A_1$. Then by Lemma~\ref{Lem:Preflow_Tame}, $X$ is tame iff for every uniformly apart $A_0, A_1\in \cB_X$, the tree $T_{A_0, A_1}$ is well-founded (see \cite{Jech}*{p.\ 185}). 
\end{proof}

Glasner proves in \cite{Glasner_Tame_Gen} a deep structure theorem for minimal, \emph{metrizable} tame systems. We state a portion of it as Fact~\ref{Fact:Glasner}. 

\begin{defin}
\label{Def:StronglyProx}
    Given a compact space $X$, $\rmP(X)$ denotes the space of regular Borel probability measures on $X$ equipped with the weak${}^*$-topology. Given $G$-flows $X$ and $Y$, a $G$-map $\phi\colon X\to Y$ is \emph{strongly proximal} if for every $\mu\in \rmP(X)$ with $\rm{supp}(\mu)\subseteq \phi^{-1}[y]$ for some $y\in Y$, then there are $x\in X$ and $p\in \sa(G)$ with $p\mu = \delta_x$. If $Y$ is trivial, we call $X$ strongly proximal. 
\end{defin}

\begin{fact}[\cite{Glasner_Tame_Gen}]
    \label{Fact:Glasner}
    Given $X$ a minimal, metrizable, and tame flow, there are $X^*$, $Z$, $Y$, $\eta\colon X^*\to X$, $\iota\colon X^*\to Z$, and $\sigma\colon Z\to Y$ exactly as in Theorem~\ref{Thm:Tau_Hausdorff}\ref{Item:Basic_Thm:TauHaus}, but with $\eta$ and $Y$ both strongly proximal.  
\end{fact}

\begin{theorem}
    \label{Thm:Tame_Tau_Hausdorff}
    The Ellis group of a minimal tame flow is Hausdorff.
\end{theorem}

\begin{proof}
    As in the previous section, $\sfV$ denotes the universe of set theory. Given $X\in \sfV$ a minimal tame flow, let $\sfW\supseteq \sfV$ be a forcing extension in which $X$ is second countable. Hence by Proposition~\ref{Prop:Tame_Levy} and Lemma~\ref{Lem:Preflow_Props}, $X^\sfW\in \sfW$ is a minimal, metrizable, and tame flow. Hence $X^\sfW$ has Hausdorff Ellis group in $\sfW$, so by Theorem~\ref{Thm:Ellis_H_Absolute}, $X$ has Hausdorff Ellis group in $\sfV$.
\end{proof}

\begin{question}
    Using Theorem~\ref{Thm:Tau_Hausdorff}, we obtain a partial form of Glasner's structure theorem for all minimal tame flows. Can this be strengthened to recover the full structure result from \cite{Glasner_Tame_Gen}, or at least the fragment presented in Fact~\ref{Fact:Glasner}? Namely, if $X$ is minimal and tame, then with notation as in Theorem~\ref{Thm:Tau_Hausdorff}, can we demand that $\eta$ and $Y$ are strongly proximal and that $\iota$ is irreducible?
\end{question}

Whenever $G$ is definable in an NIP structure $\bM$, then $\rmS_G(\bM)$ is tame \cite{CS_2018}. Hence $\rmS_{G, ext}(\bM)\cong \rmS_G(\bM^{ext})$ is tame, as is any minimal subflow. Hence we obtain a proof of the revised Newelski conjecture:

\begin{theorem}
    \label{Thm:Newelski}
    Let $G$ be a group definable in a NIP structure $\bM$, and let $\cM\subseteq \rmS_{G, ext}(\bM)$ be a minimal subflow. Then $\bbG_\cM$ is Hausdorff.
\end{theorem}

\begin{proof}
    Immediate from the above discussion and Theorem~\ref{Thm:Tame_Tau_Hausdorff}.
\end{proof}

\begin{question}
    For $G$ a group definable in a NIP structure $\bM$, is the Ellis group of the universal externally definable minimal flow an invariant of the theory? Namely, if $\bN\equiv \bM$ and $\cN\subseteq \rmS_{G, ext}(\bN)$ is a minimal $G^\bN$-subflow, do we have $\bbG_\cN\cong \bbG_\cM$? For definably amenable NIP groups, this is a theorem of Chernikov and Simon \cite{CS_2018}. 
\end{question}

We end with some discussion comparing and contrasting the results in this paper with what is known for definable groups. To do this, it will be helpful to provide weakenings of some of the classes of topological groups we have defined.

\begin{defin}
    \label{Def:Weak_TMD}
    Fix $G$ a topological group.
    \begin{itemize}
        \item 
        $G\in \sf{SCAP}_0$ if $\rm{Min}_G(\sa(G))$ is Vietoris compact.
        \item 
        $G\in \sf{CAP}_0$ iff $\rm{AP}_G(\sa(G))\subseteq \sa(G)$ is closed.
        \item 
        $G\in \sf{TMD}_0$ if $\rm{Min}_G(\sa(G))$ is meets Hausdorff.
        \item 
        $G\in \sf{WCAP}_0$ if for every $y\in \ol{\rm{AP}_G(\sa(G))}$, $\ol{G\cdot y}$ contains a unique minimal subflow. 
    \end{itemize}
\end{defin}
Among ambitable groups, $\sf{CAP}_0$ and $\sf{CMD}$ coincide \cite{BassoZucker}, and among separable groups, $\sf{WCAP}_0$ and $\sf{TMD}$ coincide by Theorem~\ref{Thm:Polish_WCAP}, using the fact that for $G$ a separable and non-precompact group, $\sa(G)$ embeds a copy of $\alpha_G(\omega\times \rmM(G))$ (see \cite{ZucThesis}*{Lemma~2.7.8}). 

We can similarly define classes of strongly definable groups $(G, \cB)$, writing $\sf{SCAP}_0^{def}$, $\sf{SCAP}^{def}$, etc. Let us emphasize that when we write $\sf{SCAP}^{def}$, $\sf{CAP}^{def}$, etc, we mean as in Definiton~\ref{Def:Weak_TMD}, but for all definable $G$-flows. In particular, the analogs of Theorems~\ref{Thm:Concrete} and \ref{Thm:Rosendal_Minimal} need not hold for definable groups. We end by collecting some known properties and examples of definable groups. 
\begin{itemize}
    \item 
    If $G$ is definable in a stable structure $\bM$, then $G\in \sf{SCAP}_0^{def}$, as $\rmS_G(\bM) \cong \rmS_{G, ext}(\bM)$ contains a unique minimal subflow \cite{Newelski_2009}. However, the group $\bbZ$ (viewed as a model) is stable, but not in $\sf{TMD}^{def}$ as follows. Let 
    $$X  = (\bbT\times \{0\}) \cup \{(e^{2i\pi(k/n)}, 1/n): 2\leq n\leq \omega, 0\leq k< n\}\subseteq \bbT\times [0, 1/2].$$
    We turn $X$ into a $\bbZ$-flow via the homeomorphism $f$ given by:
    \begin{align*}
        f(x, y) = \begin{cases}
            (x, y) \quad &\text{if } y = 0\\
            (e^{2i\pi((k+1)/n)}, 1/n) \quad &\text{if }(x, y) = (e^{2i\pi(k/n)}, 1/n).
        \end{cases}
    \end{align*}
    Then $X$ is definable, since the definable subsets of $\bbZ$ are exactly finite unions of arithmetic progressions. However, letting $Y_n$ be the minimal subflow $X\cap (\bbT\times \{1/n\})$, we have $\lim Y_n = \bbT\times \{0\}$, which consists of fixed points. 

    Perhaps $G\in \sf{TMD}^{def}$ or even $\sf{SCAP}^{def}$ provided $\bM$ is suitably saturated. This example also seems related to the lack of a good notion of ``joint continuity'' for definable group actions. 
    \item 
    If $G$ is a definably amenable NIP group definable in a suitably saturated $\bM$ and $(G, \cB_{G, \bM})$ is strongly definable, then $G\in \sf{TMD}_0^{def}$ \cite{CS_2018}*{Proposition~4.3}. Perhaps in the definably amenable NIP context, $(G, \cB_{G, \bM}^{ext})$ is always in $\sf{TMD}_0^{def}$. Perhaps this is true even when $G$ is not definably amenable.
    \item 
    The group $G = \bbR^2\times \bbT$ (defined in the field $\bbR$) is a definably amenable NIP group with $G\not\in \sf{CAP}_0^{def}$ \cite{PY_2016}.
    \item 
    Let us mention a peculiar observation about the orbit structure of some universal definable minimal flows. For both the groups $G = \bbT$ and $G = \rm{SL}_2(\bbR)$ defined in $\bbR$, we have $\cB_{G, \bM} = \cB_{G, \bM}^{ext}$, and the universal definable minimal flow is the natural action of $G$ on the \emph{doubled circle} $X = \{s^-: s\in \bbT\}\sqcup \{s^+: s\in \bbT\}$, where a basic clopen set has the form $[s^+, t^-]$  \cite{GPP_2015}. Hence for both of these definable groups, the universal definable minimal flow has exactly $2$ orbits, both of which are non-meager. 

    It would be interesting to understand which definable NIP groups have the property that their universal externally definable minimal flows have a (possibly several) non-meager orbit.
\end{itemize}

\bibliographystyle{amsplain}
\bibliography{aux/main}

\begin{bibdiv}
\begin{biblist}

\bib{AKL}{article}{
      author={Angel, O.},
      author={Kechris, A.S.},
      author={Lyons, R.},
       title={Random orderings and unique ergodicity of automorphism groups},
        date={2014},
     journal={J. Eur. Math. Soc.},
      volume={16},
      number={10},
       pages={2059\ndash 2095},
}

\bib{Aus_Reg}{article}{
      author={Auslander, J.},
       title={Regular minimal sets. i},
        date={1966},
     journal={Trans. Amer. Math. Soc.},
      volume={123},
      number={2},
       pages={469\ndash 479},
}

\bib{Auslander}{book}{
      author={Auslander, J.},
       title={Minimal flows and their extensions},
      series={Mathematics Studies},
   publisher={North Holland},
        date={1988},
      volume={153},
        ISBN={0-444-70453-1},
}

\bib{Auslander_Glasner}{article}{
      author={Auslander, J.},
      author={Glasner, S.},
       title={Distal and highly proximal extensions of minimal flows},
        date={1977},
     journal={Indiana University Math. J.},
      volume={26},
      number={4},
       pages={731\ndash 749},
}

\bib{BartJud}{book}{
      author={Bartoszynski, T.},
      author={Judah, H.},
       title={Set {T}heory: {O}n the {S}tructure of the {R}eal {L}ine},
   publisher={AK Peters/CRC Press},
        date={1995},
}

\bib{Bartosova_More}{inproceedings}{
      author={Barto\v{s}ov\'a, D.},
       title={More universal minimal flows of groups of automorphisms of
  uncountable structures},
        date={2013},
   booktitle={Asymptotic {G}eometric {A}nalysis, {P}roceedings of the {F}all
  2010 {F}ields {I}nstitute {T}hematic {P}rogram.},
      editor={Ludwig, M.},
      editor={Milman, V.},
      editor={Pestov, V.},
      editor={Tomczak-Jaegermann, N.},
      series={Fields Institute Communications},
      volume={68},
   publisher={Springer},
}

\bib{Bartosova_Thesis}{thesis}{
      author={Barto\v{s}ov\'a, D.},
       title={Topological dynamics of automorphism groups of
  $\omega$-homogeneous structures via near ultrafilters},
        type={Ph.D. Thesis},
        date={2013},
}

\bib{Bartosova_2013}{article}{
      author={Barto\v{s}ov\'a, D.},
       title={Universal minimal flows of groups of automorphisms of uncountable
  structures},
        date={2013},
     journal={Canad. Math. Bull.},
      volume={56},
      number={4},
       pages={709\ndash 722},
}

\bib{Surfaces}{misc}{
      author={Basso, G.},
      author={Codenotti, A.},
      author={Vaccaro, A.},
       title={Surfaces and other peano continua with no generic chains},
        date={2024},
         url={https://arxiv.org/abs/2403.08667},
}

\bib{MR4549426}{article}{
      author={Basso, G.},
      author={Tsankov, T.},
       title={Topological dynamics of kaleidoscopic groups},
        date={2023},
        ISSN={0001-8708},
     journal={Adv. Math.},
      volume={416},
       pages={Paper No. 108915},
         url={https://doi-org.docelec.univ-lyon1.fr/10.1016/j.aim.2023.108915},
      review={\MR{4549426}},
}

\bib{BassoZucker}{article}{
      author={Basso, G.},
      author={Zucker, A.},
       title={Topological dynamics beyond {P}olish groups},
        date={2021},
     journal={Comment. Math. Helv.},
      volume={96},
      number={3},
       pages={589\ndash 630},
}

\bib{Ben_Yaacov_2008}{article}{
      author={Ben~Yaacov, I.},
       title={Topometric spaces and perturbations of metric structures},
        date={2008},
     journal={J. Log. Anal.},
      volume={1},
      number={3-4},
       pages={235\ndash 272},
}

\bib{Ben_Yaacov_2013}{article}{
      author={Ben~Yaacov, I.},
       title={Lipschitz functions on topometric spaces},
        date={2013},
     journal={J. Log. Anal.},
      volume={5},
      number={8},
       pages={1\ndash 21},
}

\bib{BYMT}{article}{
      author={Ben~Yaacov, I.},
      author={Melleray, J.},
      author={Tsankov, T.},
       title={Metrizable universal minimal flows of {P}olish groups have a
  comeager orbit},
        date={2017},
     journal={Geom. funct. anal.},
      volume={27},
      number={1},
       pages={67\ndash 77},
}

\bib{Berberian}{book}{
      author={Berberian, S.},
       title={Lectures in functional analysis and operator theory},
      series={Graduate Texts in Mathematics},
   publisher={Springer-Verlag},
     address={New York},
        date={1974},
        ISBN={0387900802},
}

\bib{MR3156511}{article}{
      author={Bodirsky, M.},
      author={Pinsker, M.},
      author={Tsankov, T.},
       title={Decidability of definability},
        date={2013},
        ISSN={0022-4812},
     journal={J. Symbolic Logic},
      volume={78},
      number={4},
       pages={1036\ndash 1054},
         url={http://projecteuclid.org/euclid.jsl/1388953993},
      review={\MR{3156511}},
}

\bib{DC3}{unpublished}{
      author={Chernikov, A.},
      author={Gannon, K.},
      author={Krupi\'nski, K.},
       title={Definable convolution and idempotent {K}eisler measures {III}.
  {G}eneric stability, generic transitivity, and revised {N}ewelski's
  conjecture},
        note={arXiv:2406.00912},
}

\bib{CS_2018}{article}{
      author={Chernikov, A.},
      author={Simon, P.},
       title={Definably amenable {NIP} groups},
        date={2018},
     journal={J. Amer. Math. Soc.},
      volume={31},
      number={3},
       pages={609\ndash 641},
}

\bib{MR3752609}{book}{
      author={de~Vries, J.},
       title={Topological dynamical systems},
      series={De Gruyter Studies in Mathematics},
   publisher={De Gruyter, Berlin},
        date={2014},
      volume={59},
        ISBN={978-3-11-034073-0; 978-3-11-034240-6},
        note={An introduction to the dynamics of continuous mappings},
      review={\MR{3752609}},
}

\bib{DMW}{article}{
      author={Duchesne, B.},
      author={Monod, N.},
      author={Wesolek, P.},
       title={Kaleidoscopic groups: permutation groups constructed from
  dendrite homeomorphisms},
        date={2019},
     journal={Fund. Math.},
      volume={247},
      number={3},
       pages={229\ndash 274},
}

\bib{Ellis}{article}{
      author={Ellis, R.},
       title={Universal minimal sets},
        date={1960},
     journal={Proc. Amer. Math. Soc.},
      volume={11},
       pages={540\ndash 543},
}

\bib{Ellis_book}{book}{
      author={Ellis, R.},
       title={Lectures on topological dynamics},
   publisher={W.A. Benjamin, New York},
        date={1969},
}

\bib{EGS_1975}{article}{
      author={Ellis, R.},
      author={Glasner, S.},
      author={Shapiro, L.},
       title={Proximal-{I}sometric flows},
        date={1975},
     journal={Adv. Math.},
      volume={17},
       pages={213\ndash 260},
}

\bib{evans2019automorphism}{article}{
      author={Evans, D.M.},
      author={Hubi{\v{c}}ka, J.},
      author={Ne{\v{s}}et{\v{r}}il, J.},
       title={Automorphism groups and {R}amsey properties of sparse graphs},
        date={2019},
     journal={Proceedings of the London Mathematical Society},
      volume={119},
      number={2},
       pages={515\ndash 546},
}

\bib{fraisse_1954}{article}{
      author={Fra\"iss\'e, R.},
       title={Sur l'extension aux relations de quelques propriet\'es des
  ordres},
        date={1954},
     journal={Ann.\ Sci.\ \'Ecole Norm.\ Sup.},
      volume={71},
       pages={363\ndash 388},
}

\bib{FSZ}{article}{
      author={Frisch, J.},
      author={Seward, B.},
      author={Zucker, A.},
       title={Minimal subdynamics and minimal flows without characteristic
  measures},
        date={2024},
     journal={Forum Math. Sigma},
      volume={12},
       pages={e58},
}

\bib{Furstenberg_Distal}{article}{
      author={Furstenberg, H.},
       title={The structure of distal flows},
        date={1963},
     journal={Amer. J. Math.},
      volume={85},
      number={3},
       pages={477\ndash 515},
}

\bib{GPP_2014}{article}{
      author={Gismatullin, J.},
      author={Penazzi, D.},
      author={Pillay, A.},
       title={On compactifications and the topological dynamics of definable
  groups},
        date={2014},
     journal={Ann.\ Pure Appl.\ Log.},
      volume={165},
       pages={552\ndash 562},
}

\bib{GPP_2015}{article}{
      author={Gismatullin, J.},
      author={Penazzi, D.},
      author={Pillay, A.},
       title={Some model theory of $\mathrm{SL}(2,\mathbb{R})$},
        date={2015},
     journal={Fund. Math.},
      volume={229},
      number={2},
       pages={117\ndash 128},
}

\bib{Glasner_Tame_Gen}{article}{
      author={Glasner, E.},
       title={The structure of tame minimal dynamical systems for general
  groups},
        date={2018},
     journal={Invent. Math.},
      volume={211},
       pages={213\ndash 244},
}

\bib{GM_2012}{article}{
      author={Glasner, E.},
      author={Megrelishvili, M.},
       title={Representations of dynamical systems on banach spaces not
  containing $l_1$},
        date={2012},
     journal={Trans. Amer. Math. Soc.},
      volume={364},
      number={12},
       pages={6395\ndash 6424},
}

\bib{GTWZ}{article}{
      author={Glasner, E.},
      author={Tsankov, T.},
      author={Weiss, B.},
      author={Zucker, A.},
       title={Bernoulli disjointness},
        date={2021},
     journal={Duke Math. J.},
      volume={170},
      number={4},
       pages={615\ndash 651},
}

\bib{GW_UMF_S_infty}{article}{
      author={Glasner, E.},
      author={Weiss, B.},
       title={Minimal actions of the group ${S}(\mathbb{Z})$ of permutations of
  the integers},
        date={2002},
     journal={Geom. Funct. Anal.},
      volume={12},
       pages={964\ndash 988},
}

\bib{Glasner_Prox}{book}{
      author={Glasner, S.},
       title={Proximal flows},
      series={Lecture Notes in Mathematics},
   publisher={Springer-Verlag},
        date={1976},
}

\bib{GranLau}{article}{
      author={Granirer, E.},
      author={Lau, A.T.},
       title={Invariant means on locally compact groups},
        date={1971},
     journal={Ill. J. Math.},
      volume={15},
       pages={249\ndash 257},
}

\bib{GroMil}{article}{
      author={Gromov, M.},
      author={Milman, V.D.},
       title={A topological application of the isoperimetric inequality},
        date={1983},
     journal={Amer. J. Math.},
      volume={105},
       pages={843\ndash 854},
}

\bib{Gutman_Li}{article}{
      author={Gutman, Y.},
      author={Li, H.},
       title={A new short proof for the uniqueness of the universal minimal
  space},
        date={2013},
     journal={Proc.\ Amer.\ Math.\ Soc.},
      volume={141},
       pages={265\ndash 267},
}

\bib{GTZmanifolds}{article}{
      author={Gutman, Y.},
      author={Tsankov, T.},
      author={Zucker, A.},
       title={Universal minimal flows of homeomorphism groups of
  high-dimensional manifolds are not metrizable},
        date={2021},
     journal={Math. Ann.},
      volume={379},
       pages={1605\ndash 1622},
}

\bib{HS_Alg_SC}{book}{
      author={Hindman, N.},
      author={Strauss, D.},
       title={Algebra in the {S}tone-\v{C}ech compactification},
     edition={Second Ed.},
   publisher={De Gruyter},
        date={2012},
}

\bib{HPP_2008}{article}{
      author={Hrushovski, E.},
      author={Peterzil, Y.},
      author={Pillay, A.},
       title={Groups, measures and the {NIP}},
        date={2008},
     journal={J. Amer. Math. Soc.},
      volume={21},
      number={2},
       pages={563\ndash 596},
}

\bib{JahelZuckerExt}{article}{
      author={Jahel, C.},
      author={Zucker, A.},
       title={Topological dynamics of {P}olish group extensions},
        date={2022},
     journal={Eur. J. Comb.},
        note={To appear},
}

\bib{Jech}{book}{
      author={Jech, T.},
       title={Set {T}heory: {T}he {T}hird {M}illenium {E}dition},
   publisher={Springer},
        date={2003},
}

\bib{Kechris_Classical}{book}{
      author={Kechris, A.S.},
       title={Classical descriptive set theory},
      series={Graduate Texts in Mathematics},
   publisher={Springer-Verlag},
     address={New York, NY},
        date={1995},
      volume={156},
        ISBN={0-387-94374-9},
}

\bib{KPT}{article}{
      author={Kechris, A.S.},
      author={Pestov, V.G.},
      author={Todor\v{c}evi\'c, S.},
       title={Fra\"iss\'e limits, {R}amsey theory, and topological dynamics of
  automorphism groups},
        date={2005},
     journal={Geom. funct. anal.},
      volume={15},
       pages={106\ndash 189},
}

\bib{KL_2007}{article}{
      author={Kerr, D.},
      author={Li, H.},
       title={Independence in topological and ${C}^*$-dynamics},
        date={2007},
     journal={Math. Ann.},
      volume={338},
       pages={869\ndash 926},
}

\bib{KocakStrauss}{article}{
      author={Ko\c{c}ak, M.},
      author={Strauss, D.},
       title={Near ultrafilters and compactifications},
        date={1997},
     journal={Semigroup Forum},
      volume={55},
      number={1},
       pages={94\ndash 109},
}

\bib{Kohler_1995}{article}{
      author={K\"{o}hler, A.},
       title={Enveloping semigroups for flows},
        date={1995},
     journal={Proc. R. Irish Acad. Sect. A},
      volume={95},
       pages={179\ndash 191},
}

\bib{KP_2023}{unpublished}{
      author={Krupi\'nski, K.},
      author={Pillay, A.},
       title={Generalized locally compact models for approximate groups},
        note={arxiv.org:2310.20683},
}

\bib{KP_2017}{article}{
      author={Krupi\'nski, K.},
      author={Pillay, A.},
       title={Generalized {B}ohr compactification and model-theoretic connected
  components},
        date={2017},
     journal={Math. Proc. Camb. Phil. Soc.},
      volume={163},
       pages={219\ndash 249},
}

\bib{Kw_Dendrites}{article}{
      author={Kwiatkowska, A.},
       title={Universal minimal flows of generalized ważewski dendrites},
        date={2018},
     journal={J. Symb. Log.},
      volume={83},
      number={4},
       pages={1618\ndash 1632},
}

\bib{LeBoudecTsankov}{article}{
      author={Le~Boudec, A.},
      author={Tsankov, T.},
       title={Continuity of the stabilizer map and irreducible extensions},
        date={2024+},
     journal={Comm. Math. Helv.},
        note={To appear},
}

\bib{MNVTT}{article}{
      author={Melleray, J.},
      author={Nguyen Van~Th\'e, L.},
      author={Tsankov, T.},
       title={Polish groups with metrizable universal minimal flow},
        date={2016},
     journal={Int. Math. Res. Not.},
       pages={1285\ndash 1307},
}

\bib{MT}{unpublished}{
      author={Melleray, J.},
      author={Tsankov, T.},
       title={Extremely amenable groups via continuous logic},
        note={arxiv.org/abs/1404.4590},
}

\bib{NR_1977}{article}{
      author={Ne\v{s}et\v{r}il, J.},
      author={R\"odl, V.},
       title={Partitions of finite relational and set systems},
        date={1977},
     journal={J. Comb. Theory (A)},
      volume={22},
       pages={289\ndash 312},
}

\bib{Newelski_2009}{article}{
      author={Newelski, L.},
       title={Topological dynamics of definable group actions},
        date={2009},
     journal={J. Symb. Log.},
      volume={74},
      number={1},
       pages={50–72},
}

\bib{Lionel_2013}{article}{
      author={Nguyen Van~Th\'e, L.},
       title={More on the {K}echris-{P}estov-{T}odorcevic correspondence:
  precompact expansions},
        date={2013},
     journal={Fund.\ Math.},
      volume={222},
       pages={19\ndash 47},
}

\bib{Pachl}{book}{
      author={Pachl, J.},
       title={Uniform spaces and measures},
      series={Fields Institute Monographs},
   publisher={Springer},
     address={New York},
        date={2013},
        ISBN={9781461450580},
}

\bib{Pestov_1998}{article}{
      author={Pestov, V.G.},
       title={On free actions, minimal flows and a problem by {E}llis},
        date={1998},
     journal={Trans. Amer. Math. Soc.},
      volume={350},
      number={10},
       pages={4149\ndash 4165},
}

\bib{Pestov_book}{book}{
      author={Pestov, V.G.},
       title={Dynamics of {I}nfinite-dimensional {G}roups: The
  {R}amsey-{D}voretzky-{M}ilman {P}henomenon},
      series={University lecture series},
   publisher={Amer. Math. Soc.},
        date={2006},
        ISBN={9780821882962},
}

\bib{Pillay_2004}{article}{
      author={Pillay, A.},
       title={Type-definability, compact {L}ie groups, and o-minimality},
        date={2004},
     journal={J. Math. Log.},
      volume={4},
      number={2},
       pages={147\ndash 162},
}

\bib{PY_2016}{article}{
      author={Pillay, A.},
      author={Yao, N.},
       title={On minimal flows, definably amenable groups, and o-minimality},
        date={2016},
     journal={Adv. Math.},
      volume={290},
       pages={483\ndash 502},
}

\bib{Poizat_Stable}{book}{
      author={Poizat, B.},
       title={Stable groups},
      series={Mathematical Surveys and Monographs},
   publisher={American Mathematical Society},
        date={2001},
      volume={87},
        note={translated by Moses Gabriel Klein},
}

\bib{Ramsey}{article}{
      author={Ramsey, F.P.},
       title={On a problem of formal logic},
        date={1930},
     journal={Proceedings of the London Mathematical Society},
      volume={30},
       pages={264\ndash 296},
}

\bib{Rze_2018}{thesis}{
      author={Rzepecki, T.},
       title={Bounded {I}nvariant {E}quivalence {R}elations,},
        type={Ph.D. Thesis},
        date={2018},
}

\bib{Shelah_2009}{article}{
      author={Shelah, S.},
       title={Dependent first order theories, continued},
        date={2009},
     journal={Israel J.\ Math},
      volume={173},
       pages={1\ndash 60},
}

\bib{Simon_NIP}{book}{
      author={Simon, P.},
       title={A guide to nip theories},
      series={Lecture Notes in Logic},
   publisher={Cambridge University Press},
        date={2015},
      volume={44},
}

\bib{Wag_2000}{book}{
      author={Wagner, F.},
       title={Simple theories},
      series={Mathematics and its Applications},
   publisher={Kluwer Academic Publishers},
        date={2000},
      volume={503},
}

\bib{Zapletal_2015}{article}{
      author={Zapletal, J.},
       title={Interpreter for topologists},
        date={2015},
     journal={J.\ Log.\ Anal.},
      volume={7},
       pages={1\ndash 61},
}

\bib{ZucUlts}{unpublished}{
      author={Zucker, A.},
       title={Ultracoproducts and weak containment for flows of topological
  groups},
        note={arxiv.org/abs/2401.08000},
}

\bib{Zucker_Metr_UMF}{article}{
      author={Zucker, A.},
       title={Topological dynamics of automorphism groups, ultrafilter
  combinatorics, and the generic point problem},
        date={2016},
     journal={Trans. Amer. Math. Soc.},
      volume={368},
      number={9},
       pages={6715\ndash 6740},
}

\bib{ZucDirectGPP}{article}{
      author={Zucker, A.},
       title={A direct solution to the generic point problem},
        date={2018},
     journal={Proc. Amer. Math. Soc.},
      volume={146},
      number={5},
       pages={2143\ndash 2148},
}

\bib{ZucThesis}{thesis}{
      author={Zucker, A.},
       title={New directions in the abstract topological dynamics of {P}olish
  groups},
        type={Ph.D. Thesis},
        date={2018},
}

\bib{ZucMHP}{article}{
      author={Zucker, A.},
       title={Maximally highly proximal flows},
        date={2021},
     journal={Ergodic Theory Dyn. Syst.},
      volume={41},
      number={7},
       pages={2220\ndash 2240},
}

\end{biblist}
\end{bibdiv}
\vspace{5 mm}

\noindent
\url{gianluca.basso@protonmail.com}
\vspace{5 mm}

\noindent
University of Waterloo, Department of Pure Mathematics, 200 University Avenue West, Waterloo, Ontario N2L 3G1, Canada.

\noindent
\url{a3zucker@uwaterloo.ca}

\end{document}